\documentclass[review]{elsarticle} 

\usepackage{hyperref}

\journal{Comput.~Methods Appl.~Mech.~Engrg.}

\usepackage{amssymb, amsmath, amsthm}
\usepackage{amsfonts} 
\usepackage{color}
\usepackage{mathrsfs}
\usepackage{comment}
\usepackage{dcolumn}
\usepackage{bm} 

\usepackage[hang,small,bf]{caption}
\usepackage[subrefformat=parens]{subcaption}
\numberwithin{equation}{section}
\newcommand{\e}{\varepsilon}
\renewcommand{\d}{\mathrm{d}}
\newcommand{\R}{\mathbb R}
\newcommand{\N}{\mathbb N}
\newcommand{\dv}{{\rm{div}}}

\newtheorem{thm}{Theorem}[section]
\newtheorem{lem}[thm]{Lemma}
\newtheorem{rmk}[thm]{Remark}
\newtheorem{prop}[thm]{Proposition}
\newtheorem{defi}[thm]{Definition}

\begin{document}

\begin{frontmatter}

\title{Topology optimization method\\ with nonlinear diffusion\tnoteref{mytitlenote}}
\tnotetext[mytitlenote]{This study is partially supported by JSPS KAKENHI Grant Number JP22K20331.}

\author[mymainaddress]{Tomoyuki Oka\corref{mycorrespondingauthor}}
\cortext[mycorrespondingauthor]{Corresponding author}
\ead{tomoyuki-oka@g.ecc.u-tokyo.ac.jp}

\author[mymainaddress]{Takayuki Yamada}
\ead{t.yamada@mech.t.u-tokyo.ac.jp}

\address[mymainaddress]{Graduate School of Engineering, The University of Tokyo, \\Yayoi 2-11-16, Bunkyo-ku, Tokyo 113-8656, Japan}

\begin{abstract}
This paper is concerned with topology optimization based on a level set method using (doubly) nonlinear diffusion equations. Topology optimization using the level set method is called level set-based topology optimization, which is possible to determine optimal configurations that minimize objective functionals by updating level set functions. In this paper, as an update equation for level set functions, (doubly) nonlinear diffusion equations with reaction terms are derived, and then the singularity and degeneracy of the diffusion coefficient are applied to obtain fast convergence of configurations and damping oscillation on boundary structures. In particular, the reaction terms in the proposed method do not depend on the topological derivatives, and therefore, sensitivity analysis to determine a descent direction for objective functionals is relaxed. Furthermore, a numerical algorithm for the proposed method is constructed and applied to typical minimization problems to show numerical validity. This paper is a justification and generalization of the method using reaction-diffusion equations developed by one of the authors in Yamada et al.~(2010).
\end{abstract}


\end{frontmatter}

\section{Introduction} \label{S:intro}
It is a natural question to consider geometries that maximize desired physical properties for given materials, and \emph{topology optimization} is known as one of the methods to answer the question. 
Topology optimization is a type of structural optimization and is also known as the most flexible structural optimization since it allows changes in the shape of materials and in the topology of materials by generating holes; in other words, 
it deeply concerns distribution problems that determine the presence or absence of materials.
Moreover, topological change can be used to create high-performance and even lightweight materials and is therefore attracting attention for industrial applications along with recent improvements in additive manufacturing technology.

In general, topology optimization is formulated as minimization problems for objective functions whose (design) variables are materials (i.e.,~domains). 
In particular, since topological change must be allowed, a fixed design domain is introduced to involve them, and materials implying design variables are represented as in computer graphics with the use of characteristic functions. 
Thus, the design variables are replaced from domains to characteristic functions, and topology optimization is attributed to minimization problems for objective functionals such as usual variational. 
However, the class of design variables also includes discontinuous functions with rapid oscillation, which means that domains with countless holes can also be optimal configurations, and thus, the following research topics arise\/{\rm:}
\begin{itemize}
\item[(i)] The existence of (global) optimal configurations.
\item[(ii)] The specific geometry of the optimal configurations.
\item[(iii)] Reproducibility in manufacturing.
\end{itemize}
These topics are critical 
in various research fields, such as mathematics, physics, engineering and computer science.

As a typical example of topology optimization, the stiffness maximization problem is well known and is formulated as a minimization problem for an objective functional called mean compliance, whose state variables are described by second-order partial differential equations (PDEs) called linearized elastic systems. 
The existence of an optimal configuration (for the generalized problem)
is guaranteed by \emph{homogenization theory}, which is used to specify (weak star) limits of minimizing sequences for design variables (i.e.,~characteristic functions with rapid oscillation). 
Homogenization is a method to replace heterogeneous materials with countless microstructures with an equivalent macrostructure and then provides the replacement of some average quantity for rapidly oscillating functions in the mathematical sense.
Among existing functional analyses, \emph{H-convergence} \cite{MT97}, \emph{two-scale convergence} \cite{N89,A92} and the \emph{unfolding method} \cite{ADH90,CDG08} are well known as typical methods, and in particular, $H$-convergence ensures that there exists at least one global optimal configuration (for the generalized problem). 
Furthermore, two-scale convergence and the unfolding method correspond to mathematical justifications for the asymptotic expansion method \cite{BLP78}, which is classically known as formal computation for periodic homogenization. 
These methods are deeply concerned with \emph{G-closure problem}, which is also related to the existence theory for optimal configurations and optimal conditions. 
Thus, these methods play a crucial role and provide an answer to (i) and form the basis of related
numerical analysis (see, e.g.,~\cite{A02} for details). 

As for (ii),  based on the homogenization theory, the so-called \emph{homogenization design method {\rm (}HDM{\rm\,)}} was first developed (see \cite{BK88,SK91}). In this method, the optimal configuration is regarded as a periodic porous material, and the size and angle of holes are optimized. Furthermore, the \emph{solid isotropic material with penalization {\rm(}SIMP{\rm\,)} method} \cite{B88, BS03,S01}, which optimizes the density of a material by replacing the characteristic function with the density function, has been developed as a simplified version and is still used today. 
However, replacement by density functions generates intermediate domains that are neither material nor void domains, and the so-called \emph{grayscale problem} arises. 
Moreover, \emph{filtering} was developed as a solution strategy; however, issues such as \emph{checkerboarding} and \emph{mesh dependency} sometimes arise if regularization schemes are not used. In \cite{S07}, 
these issues are positively solved by morphology-based density filtering schemes, which are also crucial to manufacturability (see also, e.g.,~\cite{AB93, BJ01, BT01, KY00} for filtering and \cite{ACMOY19} for the resurrection of HDM).

To overcome these numerical issues drastically, the \emph{level set method {\rm(}LSM{\rm\,)}} is also devised (see, e.g.,~\cite{AJT02, OS88, WWG03}). The LSM is an optimization method in which, instead of the characteristic function, a (weak) differentiable function called the \emph{level set function {\rm(}LSF{\rm\,)}} is introduced as a design variable, and the characteristic function is constructed using the sign of the LSF;  details of this method are explained in the next section. 
Thus, the awkward discontinuities are eliminated as design variables, and the use of characteristic functions allows for an explicit material representation again.
As a typical example of the LSF, the signed distance function is well known, and the Hamilton-Jacobi equation is derived by partially differentiating it with respect to a fictitious time variable. 
Thus, the optimal configuration is 
determined by updating the LSF with the aid of the shape derivative and solving it (see also, e.g.,~\cite{AJT04, KWW16, WK18}).
Furthermore, using the bubble method \cite{ba},  improvements that do not rely on the initial configuration are devised (see \cite{AGJ05} for details and also \cite{S99} for the reinitialization of the LSF). 

On the other hand, in \cite{Y10}, the LSF is characterized as a solution to
a \emph{reaction-diffusion equation {\rm(}RDE{\rm\,)}}, and then the complexity of boundary geometries can be controlled by the contribution of diffusion (see, e.g.,~\cite{C11, E16, W22, Z21}).
Thus, the results in \cite{Y10} mean that the method using the RDE also overcomes the above numerical issues and concurrently contributes to (iii).
Based on \cite{Y10}, various applications have been developed from an engineering viewpoint (see, e.g.,~\cite{N22,Y13, YN22}).
Thus, the method in \cite{Y10} can be regarded as a highly effective method not only for (ii) and (iii) under satisfying (i).
However, that report did not mention decreases in objective functionals and optimality for configurations (see also \cite{C11} for a derivation of the RDE). 
Furthermore, many study subjects throughout topology optimization problems are static and linear problems (i.e.,~state variables are solutions to linear elliptic equations). 
Therefore, in a practical sense, it is critical to develop more general and rigorous methods that can adequately manage dynamic and nonlinear problems, particularly for use in future research.

\subsection{Aims and plan of this paper}
 This paper aims to justify the results in \cite{Y10}, and then, in terms of practicality and versatility, we shall generalize the method developed in \cite{Y10} for improving the convergence of optimal configurations. To this end, we shall construct an approximated sensitivity independent of the topological derivative (defined in the next section) to discuss optimality for obtained configurations and show that it is deeply related to the topological derivative derived in \cite{Y10}. Furthermore, we shall focus on nonlinear diffusion (more precisely, singularity and degeneracy of diffusion coefficients) to update LSFs and show numerically that (i)\/{\rm :}~the proposed method converges to an optimized configuration
faster than the method using reaction-diffusion and (ii)\/{\rm :}~the proposed method enables configurations to optimize even in settings where the method using reaction-diffusion cannot converge configurations due to the effect of oscillation near the boundary structures.

This paper is composed of seven sections. In the next section, we shall describe a mathematical overview of level set-based topology optimization and briefly review 
the method using the RDE developed in \cite{Y10}. 
Section \ref{S:nld} will describe the development of methods using (doubly) nonlinear diffusion equations with reaction terms to update the LSF. 
In particular, we shall introduce a reaction term independent of the topological derivative and establish a relaxation method for sensitivity analysis. 
Furthermore, from a mathematical viewpoint, we shall describe how the method using (doubly) nonlinear diffusion equations improves the convergence to optimal configurations. 
Thus, Section \ref{S:nld} is the most contributing. 
Section \ref{S:algo} will describe the numerical algorithm, and we shall apply it to typical minimization problems in Section \ref{S:ex}.
Furthermore, we shall consider an application to Nesterov's accelerated gradient method
in Section \ref{S:app}. The final section then concludes this paper.

\section{Preliminaries}\label{S:pre}

In this section, we describe the level set-based topology optimization developed in \cite{Y10}.
For simplicity, we recall it using the following objective function\/{\rm:}
\begin{equation}\label{eq:mino}
F(\Omega)=\int_{\Omega} f(x,u_\Omega,\nabla u_\Omega)\, \d x.
\end{equation}
Here and henceforth, $\Omega$ is a bounded open set of $\R^{d}$, $d\ge1$, with smooth boundary $\partial\Omega$,
$f:\Omega\times \R\times \R^{d}\to\R$ is a Lebesgue integrable function on $\Omega$,  $u_\Omega\in H^1(\Omega;\R)$ is a given state variable and $\nabla u_{\Omega}\in L^2(\Omega;\R^d)$ is the gradient for $u_\Omega \in H^1(\Omega)$.
In topology optimization, we determine an optimal configuration such that the objective function is minimized. 
Note that $\Omega\mapsto F(\Omega)$ is a set function as a variable $\Omega\subset \R^d$ and is different from the usual variational that deals with objective functionals where functions are variables.
To avoid this difficulty, let $D\subset \R^d$ be a fixed design domain such that $\Omega\subset D$, and
we introduce the following level set function (LSF)
$\phi\in H^1(D;[-1,1])$ and characteristic function $\chi_{\phi}\in L^{\infty}(D;\{0,1\})$, respectively\/{\rm :}
\begin{align*}
\phi(x)
\begin{cases}
>0,&x\in\Omega,\\
=0,\quad &x\in\partial\Omega,\\
<0,\quad &x\in D\setminus \overline{\Omega}\\
\end{cases}
\qquad \text{ and } \qquad
\chi_\phi(x)=
\begin{cases}
1\quad &\text{ if } \phi(x)\ge 0,\\  
0\quad &\text{ if } \phi(x)< 0.
\end{cases}
\end{align*}
Thus, \emph{material domains} and \emph{void domains} can be described as 
$$
[\chi_\phi=1]:=\{x\in D\colon \chi_{\phi}(x)=1\}\quad \text{ and }\quad  
[\chi_\phi=0]:=\{x\in D\colon \chi_{\phi}(x)=0\},
$$
respectively, and then $\Omega\mapsto F(\Omega)$ can be replaced with the following objective functional\/{\rm:} 
\begin{align*}
F(\phi)=\int_{D} f(x,u_\phi,\nabla u_\phi)\chi_\phi(x)\, \d x,
\end{align*}
which implies that topology optimization can be regarded as the distribution problem of materials by formulating it as
\begin{align*}
\inf_{\phi\in H^1(D;[-1,1])} F(\phi).
\end{align*}
Moreover, an optimal configuration $\Omega_{\rm opt}\subset D$ can also be represented as 
$$
\Omega_{\rm opt}:=\{x\in D\colon \chi_{\phi_{\rm opt}}(x)=1\},\quad F(\phi_{\rm opt}):=\inf_{\phi\in H^1(D;[-1,1])} F(\phi)
$$
(i.e.,~$\Omega_{\rm opt}=[\chi_{\phi_{\rm opt}}=1]=D\setminus [\chi_{\phi_{\rm opt}}=0]=[\phi_{\rm opt}\ge 0]$).

\begin{rmk}[Grayscale problem]
\rm 
Since the characteristic function $\chi_{\phi}\in L^{\infty}(D:\{0,1\})$ is constructed by employing the sign of the LSF, an explicit material representation can be obtained, which is an advantage of the level set method (LSM) that does not depend on updating the LSF. Notably, this advantage is not the same as removing the grayscale problem by cutting off the density function $h(\theta)=\theta^r$ for the density $\theta \in L^{\infty}(D:[0,1])$ and $r>0$; indeed, 
in the concept of the original optimal design problem, \eqref{eq:mino} is extended as follows\/{\rm :}
$$
F(\chi_\Omega)=\int_{D} f(x,u_\Omega,\nabla u_\Omega)\chi_\Omega(x)\, \d x,
\quad 
\chi_\Omega(x):=
\begin{cases}
1,\quad &x\in \overline{\Omega},\\  
0,\quad &x\in D\setminus \overline{\Omega}.
\end{cases}
$$  
Thus, $\chi_\Omega\in L^{\infty}(D;\{0,1\})$ is identified with $\chi_\phi\in L^{\infty}(D;\{0,1\})$, and therefore, $[\chi_\Omega=1]=[\chi_\phi=1]=[\phi\ge 0]$. 
On the other hand, since $\chi_\Omega\in L^{\infty}(D;\{0,1\})$
is identified with $h(\theta)\in L^{\infty}(D:[0,1])$ in the SIMP method,
the material domain $[\chi_\Omega=1]$ can be represented by $[h(\theta)=1]$, and then 
$[0<h(\theta)<1]$ indicates the intermediate domain, which implies that $[\chi_\Omega=1]\neq [h(\theta)\ge s]$  for any $0\le s<1$, and the grayscale problem remains unless $[0<h(\theta)<1]$ is excluded in general (see \cite{S07} for the solving strategy). 
However, the use of characteristic functions remains the issue of constructing the rigorous sensitivity that varies the topology (see Proposition \ref{prop} below). 
\end{rmk}

Now, we are in a position to find $\Omega_{\rm opt}\subset D$.
Let $(\phi_n)_{n\in\N}$ be a sequence in $H^1(D;[-1,1])$ of LSFs and let $\phi_0\in L^{\infty}(D;[-1,1])$ be an initial LSF. 
In the LSM, a minimizer for $F:H^1(D;[-1,1])\to \R$ is determined by updating the LSF.
To this end, we recall the following classic well known gradient descent method\/{\rm :}  
\begin{equation}\label{GM}
\phi_{n+1}(x):=\phi_{n}(x)-kF'(\phi_n)\quad \text{ for } n\in \mathbb{N}\cup\{0\}.
\end{equation}
Here $k>0$ is the step width, and $F'(\phi_n)$ is the steepest descent direction of $F$ and stands for the Fr\'echet derivative. 
In \cite{Y10}, the following approximated objective functional is introduced in terms of the regularity for \eqref{GM}\/{\rm :}
\begin{align}\label{eq:modiobj}
F_\e(\phi_n,\phi_{n+1})=F(\phi_n)+\frac{\e}{2}\int_{D}|\nabla \phi_{n+1}(x)|^2\, \d x\quad \text{ for $\e>0$}.
\end{align}
Thus, the second term of the right-hand side implies the regularization term.
Then, by replacing $F$ with $F_\e$ in \eqref{GM}, 
the discretized version of the following (reaction) diffusion equation can be obtained\/{\rm :}
\begin{align}\label{RD}
\begin{cases}
\partial_t \phi-\tau\Delta \phi=\rho \d_{\rm T}F\
\text{ in } D\times (0,+\infty),\\
\phi|_{\partial D}=0,\quad \phi|_{t=0}=\phi_0\in L^{\infty}(D), 
\end{cases}
\end{align}
where $\partial_t=\partial/\partial_t$, $\tau, \rho>0$ and $\d_{\rm T} F$ is the topological derivative of $F$, and is defined as follows (see also \cite{Td1,Td2,Td3,Td4} for details)\/{\rm :} 
\begin{defi}[Topological derivative]\label{D:d_T}
 Let $A=\{\Omega\subset D\colon \Omega \text{ is open in } D\}$. 
 A function $J$ defined on $A$ is said to be topologically differentiable at $\Omega_0$ and at $x\in \Omega_0$ if the following limit exists\/{\rm :}
\begin{align*}
\d_{\rm T}J(\Omega_0,x):=\lim_{\e\to 0_+}\frac{J(\Omega_0\setminus \overline{B_\e(x)})-J(\Omega_0)}{|B_\e(x)|}.
\end{align*}
Here $B_\e(x):=\{ y\in \Omega_0\colon |x-y|<\e \}$ and $|B_\e(x)|$ denotes the Lebesgue measure of $B_\e(x)\subset \Omega_0$. 
\end{defi}
Thus, $\Omega_{\rm opt}\subset D$ will be obtained in specific cases by solving \eqref{RD} and updating $\phi_n\in H^1(D)$ such that $|\phi_n|\le 1$ under $F' \approx -\d_{\rm T} F$. 

\begin{rmk}[Contribution of diffusion]\rm
If $\e>0$ is sufficiently small, the (local) minimizer for the original functional $F: H^1(D;[-1,1])\to \R$ would be characterized by ones for the approximate functional $F_\e$. Conversely, a large $\e>0$ enhances the engineering value from the viewpoint that complex geometries can be avoided as the optimal configuration due to the smoothing effect of diffusion, and practical geometries can be obtained (see \cite{Y10} for details).
\end{rmk}

\begin{rmk}[Boundary condition]
\rm 
In \eqref{RD}, it is assumed that $\phi\in L^2(0,+\infty;$ $H^1_0(D))$ satisfies the homogeneous Dirichlet boundary condition for simplicity, but it is only imposed in terms of the uniqueness for weak solutions. 
Therefore, other boundary conditions can also be allowed.
\end{rmk}

\begin{rmk}[Modified reaction term]\label{R:mrt}
\rm
The combination by \eqref{GM}, \eqref{eq:modiobj} and Definition \ref{D:d_T} is the most innovative idea in \cite{Y10}.
However, the derivation of \eqref{RD} is heuristic, and hence, there is no guarantee that given initial configurations achieve optimal configurations through the proposed method. 
On the other hand, the replacement of $F'$ with $-\d_{\rm T} F$ 
is known to be valid for various topology optimization problems and may be deeply related to the Fr\'echet derivative; for instance, let $J: A \to \R$ be a function given by $J(\Omega):=|\Omega|$.
Then we readily have 
$$
J(\Omega)=\int_D \chi_\phi(x)\, \d x=:J(\chi_\phi),
$$
and hence, $-\d_{\rm T}J(\Omega)=J'(\chi_\phi)$. 
However, even if we can obtain that relation, since the state variable is often a solution to some PDE and the topological derivative needs that information, deriving it becomes hard for complex equations such as nonlinear equations, and issues in the viewpoint of versatility may arise.
\end{rmk}

\section{
Nonlinear diffusion equations for level set functions
}\label{S:nld}

In this study,
under the same assumptions as in \S \ref{S:pre},
we stress that the level set method (LSM) using the following nonlinear diffusion equation as a generalization in \cite{Y10} is more practical and versatile\/{\rm:}
\begin{align}\label{NLD}
\begin{cases}
\partial_{t}\phi^q-\tau \Delta \phi= \rho F_{\eta}'(\phi) \ \text{ in } D\times (0,+\infty),\\
\phi|_{\partial D}=0,\quad \phi|_{t=0}=\phi_0 \in L^{\infty}(D), 
\end{cases}
\end{align}
where 
$q, \tau, \rho> 0$, $\phi^q:=|\phi|^{q-1}\phi$ and 
\begin{equation}\label{eq:Rterm}
F_\eta'(\phi)(x):=f(x,u_\phi,\nabla u_\phi)\delta_\eta(\phi(x))
\end{equation}
for some approximated delta function $\delta_\eta\ge0$ such that $\delta_\eta(\phi)\to \delta(\phi)$  as $\eta\to 0_+$. 

The derivation and rationale are explained below.

\subsection{Reaction term}
As already mentioned in \S \ref{S:intro}, 
an optimal configuration in topology optimization (for the generalized problem) is characterized by employing the weak-star limit (or homogenization limit) for a minimizing sequence of the characteristic function, which means that intermediate domains appear in general. Therefore, in terms of manufacturing, avoiding such domains is a critical issue, even if $\phi_0\in L^{\infty}(D)$ is only locally optimized. Generally, it is not easy to find descent directions of objective functionals that allow for changes in topology (see also Remark \ref{R:mrt}). As one relaxation method, this subsection is devoted to showing that $F'(\phi)$ in \eqref{GM} can be (locally) approximated by $-F_\eta'(\phi)$ under some conditions. We prove the following 
\begin{prop}[Approximated sensitivity with assumptions for initial LSFs]\label{prop}
Let $F:H^1(D;[-1,1])\to\R$ be a functional defined by
$$
F(\phi)=\int_D f_\phi(x)\chi_\phi(x)\, \d x,\quad f_\phi(x):=f(x,u_\phi,\nabla u_\phi).
$$
Let $F_{\eta}'(\phi)$ be a function appeared in \eqref{eq:Rterm} and let
$(\phi_{n})_{n\in\N}$ be a sequence in $H^1(D;[-1,1])$ such that 
\begin{align}\label{eq:reset}
\phi_n(x)
:=
\begin{cases}
{\rm sgn}(\tilde{\phi}_{n}(x))&\text{ if } |\tilde{\phi}_{n}(x)|>1,\\
\tilde{\phi}_{n}(x)&\text{ otherwise}
\end{cases}
\end{align}
and $\tilde{\phi}_{n}:=\phi_{n-1}+kF_\eta'(\phi_{n-1})$ for $k, \eta>0$. Then, under a suitable initial level set function $\phi_0\in L^{\infty}(D)$ such that it is in the neighborhood of critical points,
$F'$ can be at least approximated by $-F'_\eta$ except for boundary structures. 
In addition, suppose that, for $\e>0$ small enough, there exists $N_{\e}\in \N$ such that  
\begin{equation}\label{eq:CC}
\|\phi_{n+1}-\phi_n\|_{L^{\infty}(D)}< k\e 
\quad \text{ for all } n\ge N_{\e}. 
\end{equation}
Then $[\chi_{\phi_n}=1]$ provides a candidate {\rm(}locally{\rm)} optimized configuration. 
\end{prop}

\begin{proof}
we formally deduce that, for any $w\in H^1(D;[-1,1])$,
\begin{align}
\langle F'(\phi),w\rangle_{H^1(D)}
&=\lim_{\eta\to 0_+}\int_{D}\frac{f_{\phi+\eta w}(x)-f_{\phi}(x)}{\eta}\chi_{\phi+\eta w}(x)\, \d x\nonumber\\
&\quad +\lim_{\eta\to 0_+}\int_{D}f_{\phi}(x)\frac{\chi_{\phi+\eta w}(x)-\chi_{\phi}(x)}{\eta}\, \d x.
\label{eq:Rterm1}
\end{align}
By replacing $D$ with $[\phi\neq 0]$, the second term of the right-hand side vanishes. 
On the other hand, the first term of the right-hand side can be represented as
\begin{align}\label{eq:NC}
\int_{[\psi\neq 0]}\frac{f_{\psi+\eta w}(x)-f_{\psi}(x)}{\eta}\chi_{\psi+\eta w}(x)\,\d x
=-
\langle F_\eta'(\psi),w\rangle_{H^1([\psi\neq 0])}+C_\eta
\end{align}
for any $\psi$ in the neighborhood of the critical point.
Here $C_\eta$ is some sufficiently small value depending on $\eta>0$.
Thus, if $\phi_0\in L^{\infty}(D)$ is at least in 
the neighborhood critical ones, $F'(\phi_n)$ can be approximated by $-F_\eta'(\phi_n)+\tilde{F}_\eta'(\phi_n)\chi_{[\phi_n=0]}$ for some $\tilde{F}_\eta'$ (i.e,~$F'\approx -F'_\eta$ except for boundary structures).
Furthermore, \eqref{eq:CC} ensures that  
\begin{align*}
|F_\eta '(\phi_n)(x)|\le k^{-1}\|\phi_{n+1}-\phi_n\|_{L^{\infty}(D)}\le \e\quad \text{ for a.e.~} x\in D, 
\end{align*} 
which together with \eqref{eq:Rterm1} and \eqref{eq:NC} implies that 
$\phi_n$ turns out to be a critical point of $F$ under $|\tilde{F}_\eta'\chi_{[\phi_n=0]}|$ small enough.  
This completes the proof.
\end{proof}

\begin{rmk}[More precise sensitivity analysis]
\rm
By the use of weak forms for governing or state equations, Proposition \ref{prop} also makes sense for objective functionals defined by the integral on the boundary. Otherwise, as one method, the so-called {\it adjoint method} plays a crucial role, and in particular, it is possible to remove the restrictions for initial LSFs and domains (see Lemma \ref{lem} and Remark \ref{R:GR} below).    
\end{rmk}

\begin{rmk}[Removability for local maximization]
\rm
By Proposition \ref{prop}, the following properties are obtained\/{\rm:}
\begin{itemize}
\item[(i)] Since $F'(\phi_n)(x)$ can be approximated by $-F_\eta'(\phi_n)(x)$ for a.e.~$x\in [\phi_n\neq 0]$, it follows that
\begin{align*}
\lefteqn{
\int_{[\phi_n\neq 0]} f_{\phi_{n+1}}(x) \chi_{\phi_{n+1}}(x) \, \d x - \int_{[\phi_n\neq 0]} f_{\phi_{n}}(x) \chi_{\phi_{n}}(x) \, \d x}\\ 
&\qquad=
k\langle F'(\phi_n), F_\eta'(\phi_n)\rangle_{H^1([\phi_n\neq 0])}+o(k)
\le 0 \quad \text{ as $k\to 0_+$.}
\end{align*}
Hence, if one can confirm that
\begin{equation*}
\int_{[\phi_n=0]}f_{\phi_{n+1}}(x)\chi_{\phi_{n+1}}(x) \, \d x \le \int_{[\phi_n=0]} f_{\phi_{n}}(x)\chi_{\phi_{n}}(x)\, \d x, 
\end{equation*}
then the monotonicity $F(\phi_{n+1})\le F(\phi_n)$ is obtained.
\item[(ii)]
By setting 
$$
\delta_\eta(\phi(x))
\begin{cases}
>0 &\text{ if }  x\in [\phi\ge 0],\\
=0 &\text{ otherwise}
\end{cases}
$$
and $m_\eta:=\min_{x\in [\phi_n\ge 0]} \delta_\eta(\phi_n(x))>0$, the assumption \eqref{eq:CC} ensures that 
\begin{align*}
m_\eta |f_{\phi_n}(x)|\le |F_\eta'(\phi_n)|\le \e  \quad \text{ for a.e.~} x\in [\phi_n\ge 0],
\end{align*}
which yields 
\begin{align*}
|F(\phi_n)|
\le  \int_{[\phi_n\ge 0]}|f_{\phi_n}(x)|\, \d x
\le m_\eta^{-1}\e\bigl|[\phi_n\ge 0]\bigl|=: M_{\eta,\e}.
\end{align*}
Thus, if $F\ge 0$ at least, then the objective functional $F:H^1(D;[-1,1])\to\R$ can be minimized 
under the situation where it is possible to take $\e>0$ such that $M_{\eta,\e}$ is sufficiently small.
\end{itemize}

\end{rmk}

\begin{rmk}[Relaxation of sensitivity analysis]
\rm 
Compare \eqref{RD} with \eqref{NLD}. It is noteworthy that the topological derivative $\d_{\rm T}F$ does not appear in \eqref{NLD}. 
In other words, the argument mentioned above corresponds to the relaxation of sensitivity analysis (i.e.,~deriving $F'(\phi)$). 
Therefore, even if deriving it is too hard, the reaction term might be readily obtained by
\begin{align}\label{eq:aRterm}
F'(\phi)(x)\approx -F_\eta'(\phi)(x):=-f(x,u_\phi,\nabla u_\phi)\delta_\eta(\phi(x)),
\end{align}
which implies that this method is more versatile in such specific cases. 
However, we note the cases where \eqref{eq:CC} is no longer satisfied  (see \S \ref{SS:cm} and \S \ref{SS:ex2} for examples where \eqref{eq:CC} is not satisfied due to the effect of boundary oscillation). 
On the other hand, since \eqref{eq:CC} can be verified numerically, it can be adopted as the convergence condition for the numerical algorithm (see \S \ref{S:algo} below).
\end{rmk}

\subsection{Diffusion term}
We can readily derive \eqref{NLD} with the aid of $F_\eta'(\phi)$ mentioned in the previous subsection.  
As in \eqref{GM}, by setting
\begin{align*}
\phi_{n+1}(x)=\phi_{n}(x)+K(\phi_n)F_\eta'(\phi_n) \quad{ and }\quad  
K(\phi_n)=k(q|\phi_n|^{q-1})^{-1},
\end{align*}
the same argument as in the derivation of \eqref{RD} yields \eqref{NLD} formally.

We next explain the motivation for introducing \eqref{NLD}. Recall the following nonlinear diffusion equation\/{\rm:} 
\begin{equation}\label{nld}
\partial_tv=\Delta v^p\quad \text{ in }\ \R^d\times (0,+\infty).
\end{equation}
If $p=1$, \eqref{nld} describes the (linear) diffusion equation. 
For $p\neq 1$, the nonlinear diffusion equation \eqref{nld} is called the \emph{porous medium equation} (or \emph{slow diffusion equation, SDE}) if $1<p<+\infty$ and the \emph{fast diffusion equation {\rm(}FDE\,{\rm)}} if $0<p<1$ (see \cite{V1,V2} for details).
We note that the diffusion term $\Delta v^p$ can be expanded as
\begin{equation}\label{d-coef}
\Delta v^p=\dv\left(p|v|^{p-1}\nabla v\right),
\end{equation}
which implies that $p|v|^{p-1}$ can be regarded as the diffusion coefficient.

In case $p>1$ (i.e.,~SDE) and if $|v|\ll 1$, then the diffusion coefficient is much smaller than that of linear diffusion. As a self-similar solution, the following so-called \emph{Barenblatt solution} (or \emph{Zel'dovich-Kompaneets-Barenblatt solution}) is known\/{\rm :}
$$
\mathcal{B}(x,t) = t^{-\alpha} \left[ C - \kappa (t^{-\alpha/d} |x|)^2 \right]_+^{1/(p-1)} \ \mbox{ for } \ x \in \mathbb R^d, \ t > 0,
$$
where $\alpha := \frac{d}{d(p-1)+2}$, $\kappa := \frac{\alpha(p-1)}{2dp} > 0$ and any $C > 0$. 
Then the support of $\mathcal{B}(\cdot,t)$ can be represented as
$$ 
\text{supp } \mathcal{B}(\cdot,t):=\overline{\{x\in \R^d\colon \mathcal{B}(x,t)>0\}}
=\{x\in \R^d\colon |x|^2\le (C/\kappa)t^{2\alpha/d} \}.
$$
Thus, $\text{supp } \mathcal{B}(\cdot,t)$ is always bounded, and then it spreads at a finite rate with time.  
In particular, the interface spreads at the velocity of $\sqrt{C/\kappa}t^{\alpha/d}>0$, and such a property is called \emph{finite propagation property}, which is a rapidly different property from that of the linear diffusion (i.e., $p=1$); indeed, under the nonnegative initial data $v_0\ge 0$, we assume that there exists $y\in \R^d$ such that $v(y,t)=0$ for any $t>0$ fixed. Then the solution $v(x,t)$ can be represented as
$$
v(y,t)=\int_{\R^d}\frac{1}{(4\pi t)^{d/2}}\exp(-|y-z|^2/4t)v_0(z)\, \d z,
$$
which implies that $v_0\equiv 0$ since the other is positive. This contradicts the assumption of $v_0\ge 0$, that is, $v(x,t)>0$ for all $(x,t)\in \R^d\times (0,+\infty)$.
In other words, the linear diffusion has \emph{infinite propagation property} since $\text{supp } v_0$ spreads instantly.

On the other hand, in case $p<1$ (i.e.,~FDE), different properties from the SDE are derived. 
Actually, we consider the cases where $|v|\ll 1$.
In terms of \eqref{d-coef}, the diffusion coefficient can be regarded as very large near the interface. 
Furthermore, as in linear diffusion, the FDE is known to exhibit infinite propagation property. 

Based on the above facts, we set $v$ as the LSF (i.e.,~$v=\phi$). Then, 
the following effect of nonlinear diffusion will be expected\/{\rm :} 
\begin{itemize}
\setlength{\leftskip}{3mm}
\item[(i-FDE)] Since the diffusion coefficient near boundary structures is very large, it will try to spread out even if the sensitivity is small. In other words, 
the method using fast diffusion is expected to converge to optimal configurations faster than the method using reaction-diffusion as long as boundary structures do not oscillate since the boundary structures still try to spread even near the optimal configuration.
 
\item[(ii-SDE)] Even if the sensitivity is large, boundary structures will not try to spread out as in linear diffusion due to the small diffusion coefficient. 
Therefore, in terms of damping oscillation on boundary structures, the method using slow diffusion is expected to be effective for problems that cannot be converged by the method using reaction-diffusion due to the oscillation on the boundary structures.

\end{itemize}
The above theoretical interpretation is the motivation for introducing \eqref{NLD}.

\begin{rmk}[Novelty in LSM]
\rm 
The novelty of \eqref{RD} is that it applies the effect of diffusion using a perturbed approximate objective functional; in other words, the novelty is the modification of the descent direction for the objective functionals appearing in the gradient descent method.  
On the other hand, \eqref{NLD} is derived by focusing on the step width of the descent direction and making the LSF depend on it.
Therefore, it is noteworthy that the method using \eqref{NLD} is expected to improve convergence to the optimal configuration without relying on other well-known optimization methods  (e.g.,~Newton's method and quasi-Newton method).
Furthermore, \eqref{NLD} corresponds to \eqref{RD} by setting $q=1$, which implies a generalization of the result in \cite{Y10} (see \S \ref{S:ex} below for the reaction term).  
\end{rmk}

\begin{rmk}[Regularity for LSF]
\rm
The LSF $\phi(t)\in H^1_0(D)$ in \eqref{NLD} for $\rho=0$ and $\tau=1$ coincides with $v^p(t)$ in \eqref{nld}, but \eqref{NLD} would be reasonable in terms of regularities; indeed, $v^p(t)$ is in $H^1(\R^d)$ for any $p\ge 2$, but $v(t)$ is not even if it has nonnegative initial data (see \cite[Remark 3.2]{AO2} for counterexamples).
\end{rmk}

\begin{rmk}[Generalization to doubly nonlinear diffusion]
\rm
As another nonlinear diffusion,  the following \emph{parabolic p-Laplace equation} (see, e.g.,~\cite{Di93, Di12}) is well known\/{\rm :}
\begin{align}\label{eq:pLP}
\partial_t v=\Delta_p v,\quad 
\Delta_p v:=\text{div}(|\nabla v|^{p-2}\nabla v).
\end{align}
Here, $\Delta_p$ is called the \emph{$p$-Laplacian}. In particular, $\Delta_p$ is a generalization of Laplacian $\Delta=\Delta_2$, and applied in the SIMP method from a different viewpoint of filtering (see, e.g.,~\cite{W04, Z22}).
Regarding  $|\nabla v|^{p-2}$ as the diffusion coefficient and assuming $|\nabla v|\ll 1$, we can classify \eqref{eq:pLP} as \emph{degenerate diffusion} if $p>2$, and \emph{singular diffusion} if $p<2$ for the same reason as \eqref{NLD}.
Moreover, \eqref{NLD} can be generalized to the following doubly nonlinear diffusion equation\/{\rm :}
\begin{align}\label{DNLD}
\begin{cases}
\partial_{t}\phi^q-\tau \Delta_p \phi= \rho F_{\eta}'(\phi) \ \text{ in } D\times (0,+\infty),\\
\phi|_{\partial D}=0,\quad \phi|_{t=0}=\phi_0 \in L^{\infty}(D), 
\end{cases}
\end{align}
by formally replacing \eqref{GM} with
$$
\phi_{n+1}:=\phi_n-\frac{k}{q|\phi_n|^{q-1}}\left(
F(\phi_n)+\frac{\e}{p}\int_{D}|\nabla \phi_{n+1}(x)|^p\, \d x\right)',
$$
which means that it is expected that convergence to the optimal configuration can be improved by handling information on both the LSF and its gradients.
In particular, \eqref{NLD} is a specific case of \eqref{DNLD} as $p=2$
(see \S \ref{S:app} for a partial application).
\end{rmk}

\section{Numerical algorithm}\label{S:algo}
In this section, based on \S \ref{S:nld}, we construct a numerical algorithm for 
the volume-constrained minimization problem,  
\begin{equation}\label{eq:opt-prob}
\inf_{\phi\in H^1(D;[-1,1])} 
F(\phi)\quad \text{ subject to }\ G(\phi)\le 0,
\end{equation}
where
$F(\phi)$ and $G(\phi)$ are an objective functional and a volume constraint functional, respectively. 
Define the Lagrangian of \eqref{eq:opt-prob} by
$
\mathcal{L}(\phi,\lambda):=F(\phi)+\lambda G(\phi)
$
for the Lagrange multiplier $\lambda\ge 0$. The following steps describe this algorithm\/{\rm :} 
\begin{itemize}
\setlength{\leftskip}{3mm}
\item[\bf Step\,1.]
Set a fixed design domain $D\subset \R^d$, boundary conditions for governing equations and 
an initial LSF $\phi_0\in L^{\infty}(D;[-1,1])$.  
\item[\bf Step\,2.]
Determine a state value (i.e.,~solve governing equations).
\item[\bf Step\,3.]
Compute $F(\phi_n)$ and $G(\phi_n)$ for $n\in\N\cup\{0\}$. 
\item[\bf Step\,4.]
Check for the convergence condition \eqref{eq:CC}. 
Throughout this paper, 
let $k\e=1.0\times 10^{-2}$ in terms of practicality, when no confusion can arise.    
If \eqref{eq:CC} is satisfied, terminate the optimization; 
otherwise, proceed to the next step. 

\item[\bf Step\,5.] 
Compute the reaction term of \eqref{NLD} and $\lambda\ge 0$ by virtue of the augmented Lagrangian method.
\item[\bf Step\,6.] 
Solve the following discretized in time and modified weak form of \eqref{NLD} by
employing the finite element method\/{\rm :}
\begin{align}
&\int_D \tilde{q}(|\phi_{n}(x)|+\xi)^{q-1}\frac{\phi_{n+1}-\phi_{n}}{\varDelta t}(x){\psi}(x) \, \d x
+\int_D\tau \nabla \phi_{n+1}(x)\cdot \nabla {\psi}(x)\, \d x \nonumber\\
&\qquad=
\int_D \rho \mathcal{L}_\eta'(\phi_{n},\lambda) \psi(x)\, \d x
\quad \text{ for all } \psi\in V. \label{discNLD}
\end{align}
Here 
$\varDelta t, \tilde{q}, q,\xi,\tau,\rho \ge 0$, 
$\phi_{n+i}(x)=\phi(x,(n+i)\varDelta t)$, $n\in \N\cup\{0\}$, $i=0,1$, 
$V$ is a subspace of $H^1(D;[-1,1])$ explained later and $\mathcal{L}_\eta'(\phi,\lambda)=F_\eta'(\phi)-\lambda$.
Throughout this paper, let $\rho=0.7$, $\tilde{q}\in\{1,q\}$ and $\xi=1.0\times 10^{-4}$, when no confusion can arise.
\item[\bf Step\,7.] 
Return to {Step\,2} after setting the next given level set function $\phi_n\in V$ as in \eqref{eq:reset}.
This completes the numerical algorithm.
\end{itemize}
As for Steps $1$ and $2$, we need to be modified according to optimization problems, and the remaining steps are common steps if the reaction term $\mathcal{L}_\eta'(\cdot, \cdot)$ in \eqref{discNLD} is determined.

\section{Numerical results}\label{S:ex}

In this section, the numerical algorithm constructed in \S \ref{S:algo} is used to prove numerically for (i-FDE) and (ii-SDE) described in \S \ref{S:nld}. 
To this end, the proposed method using nonlinear diffusion is applied to the so-called cantilever, MBB beam, bridge, compliant mechanism and heat conduction as typical minimization problems in two dimensions mainly (see Figure \ref{fig:DBC} for fixed design domains $D\subset \R^2$ and boundary conditions). 
Here, based on the code in a previous study \cite{O22}, FreeFEM++ \cite{H12} is employed.
\begin{figure}[htbp]
\hspace*{-5mm} 
    \begin{tabular}{cccc}
      \begin{minipage}[t]{0.48\hsize}
        \centering
 \includegraphics[keepaspectratio, scale=0.16]{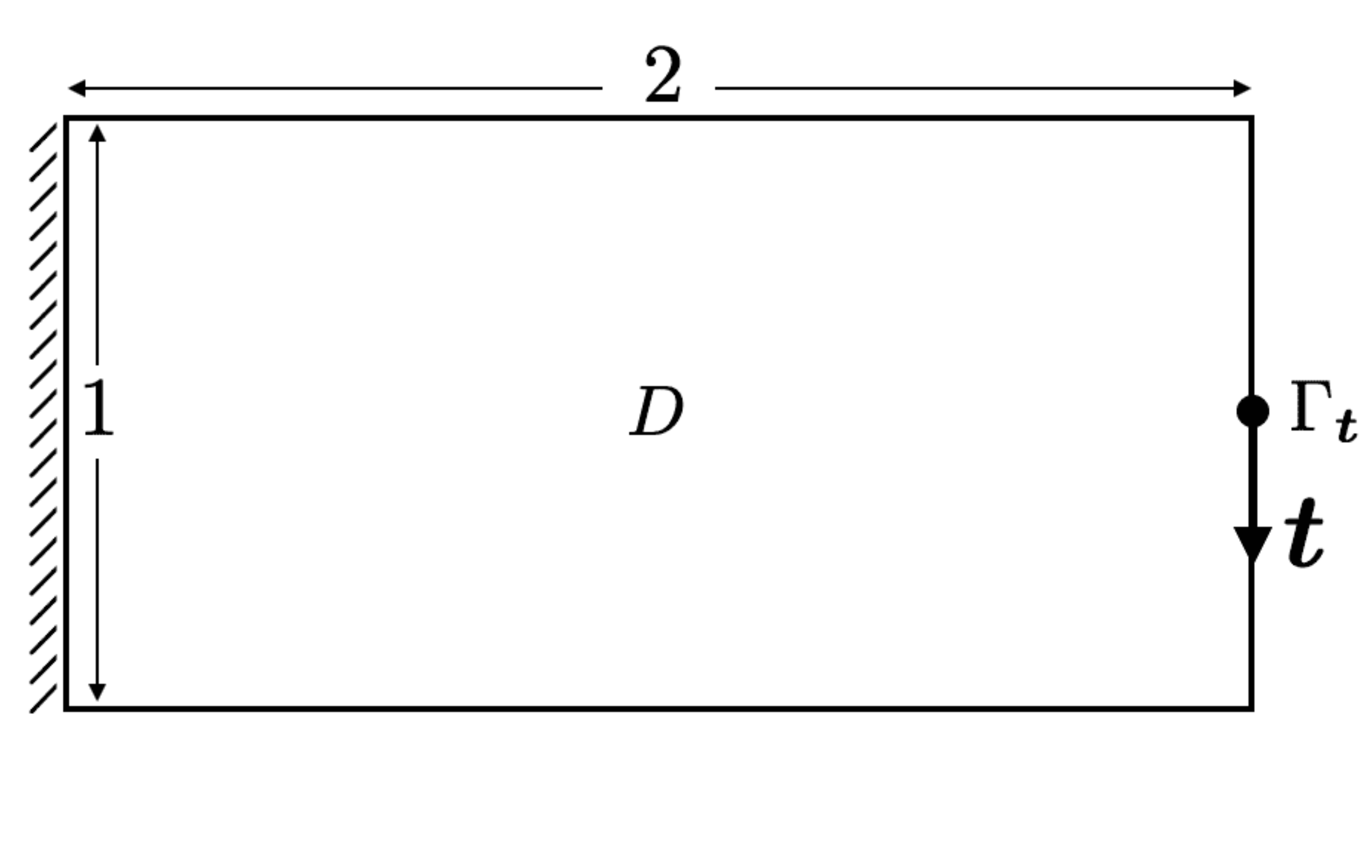}
        \subcaption{Cantilever}
        \label{ica}
      \end{minipage} 
               
    \begin{minipage}[t]{0.48\hsize}
        \centering
        \includegraphics[keepaspectratio, scale=0.16]{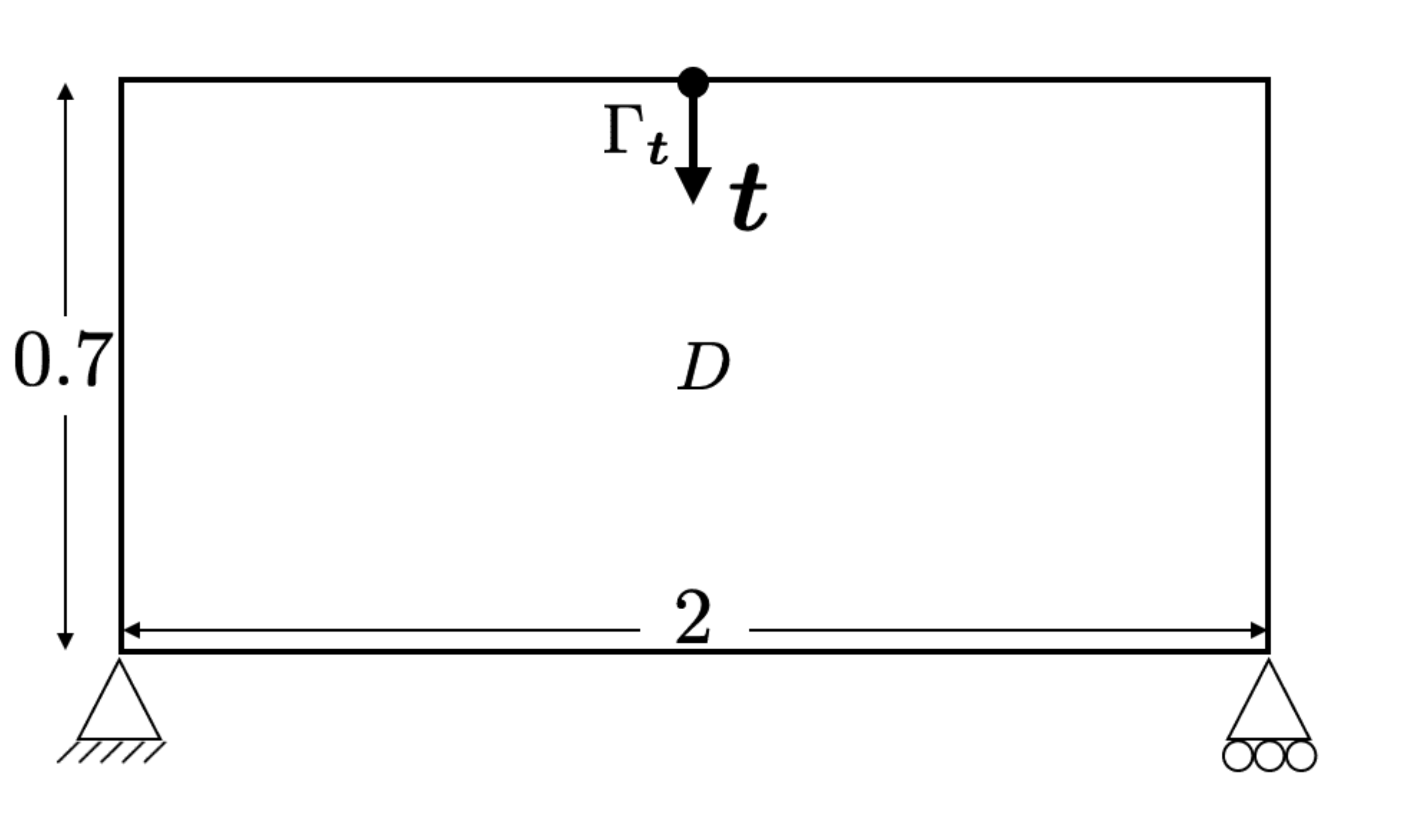}
        \subcaption{MBB beam}
        \label{imbb}
      \end{minipage} 
      \\
       \begin{minipage}[t]{0.48\hsize}
        \centering
        \includegraphics[keepaspectratio, scale=0.16]{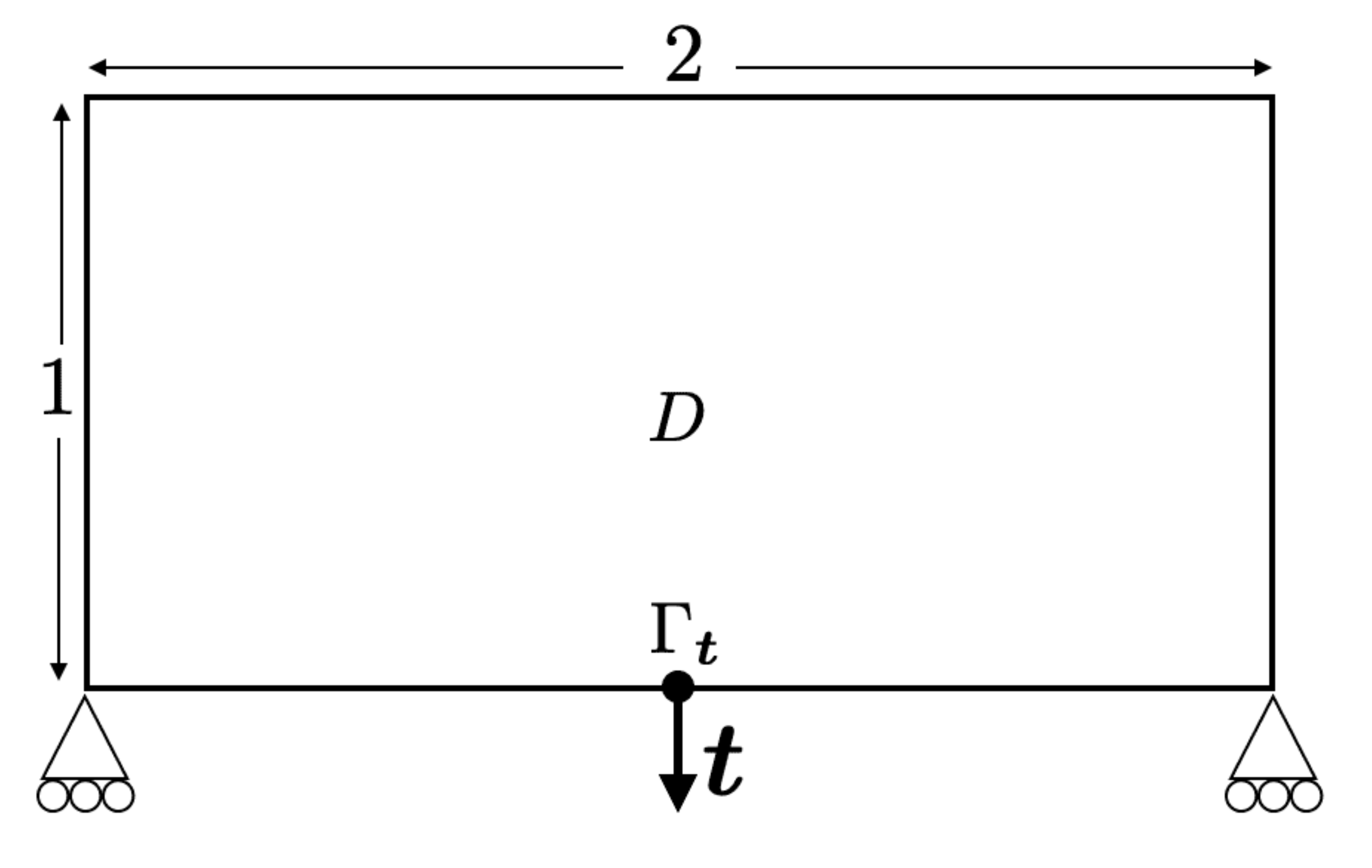}
        \subcaption{Bridge}
        \label{ibri}
      \end{minipage}

      \begin{minipage}[t]{0.24\hsize}
        \centering
        \includegraphics[keepaspectratio, scale=0.16]{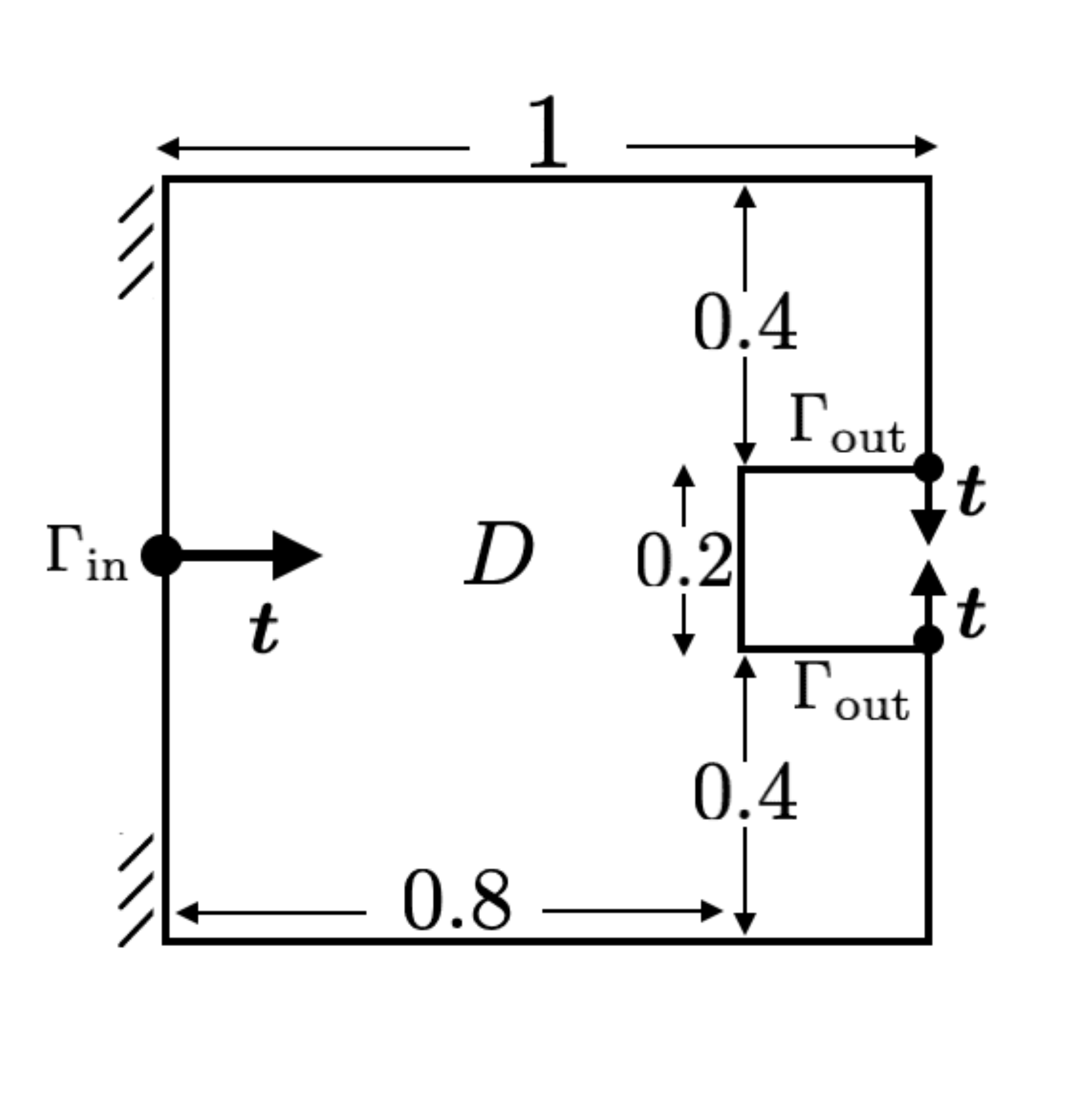}
        \subcaption{Compliant mechanism}
        \label{icm}
      \end{minipage} 
      \begin{minipage}[t]{0.24\hsize}
        \centering
        \includegraphics[keepaspectratio, scale=0.104]{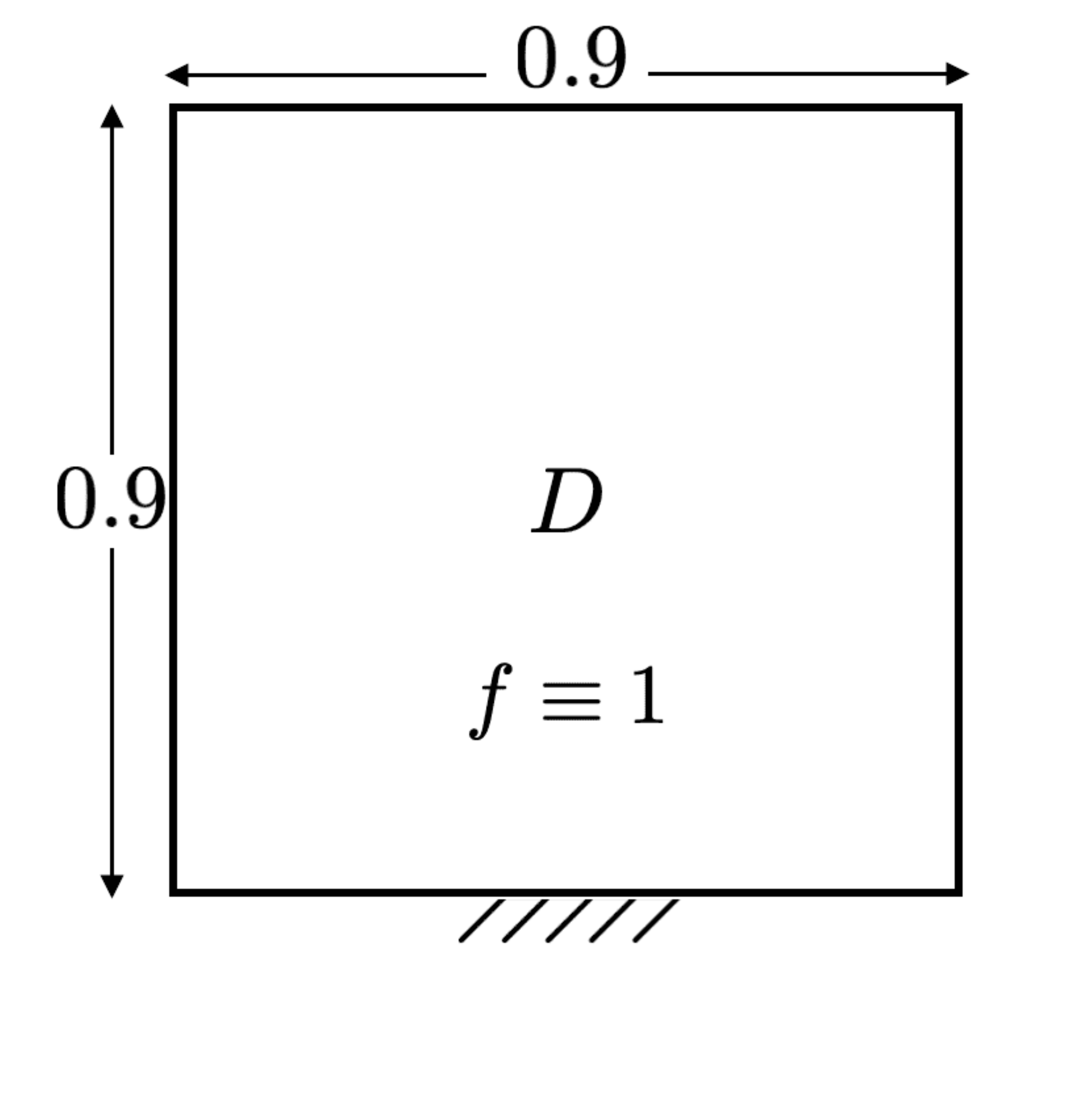}
        \subcaption{Heat conduction}
        \label{ih}
      \end{minipage} 
    \end{tabular}
     \caption{Fixed design domain $D\subset \R^2$ and boundary conditions.}
     \label{fig:DBC}
 \end{figure}

\subsection{Minimum mean compliance problem}\label{SS:ex1}
We first show (i-FDE) for the following minimization problem of 
volume-constrained mean compliance\/{\rm :}
\begin{align}\label{op:MC}
\inf_{\Omega\in \mathcal{U}_{\rm ad}} \left\{F(\Omega):=\langle \boldsymbol{t}, \boldsymbol{u}_\Omega \rangle_{H^{1/2}(\Gamma_t)^d}\right\}, 
\end{align}
where 
$$
\mathcal{U}_{\rm ad}:=\{\Omega \subset D \colon |\Omega|\le G_{\text{max}}|D|  \},
$$
$G_{\text{max}}>0$, the vector-valued state variable ${\boldsymbol u}_\Omega\in V^d$ is a unique  solution to the following linearized elasticity system\/{\rm:}
\begin{align}\label{eq:MCgov}
\int_\Omega
\mathbb{D}{\boldsymbol \varepsilon}({\boldsymbol u}_\Omega)(x) \colon {\boldsymbol \varepsilon}({\boldsymbol v})(x)\, \d x
=\langle \boldsymbol{t}, \boldsymbol{v} \rangle_{H^{1/2}(\Gamma_t)^d}
\quad \text{ for all }\  \boldsymbol{v}\in V^d.
\end{align}
Here and henceforth, $|\Omega|$ stands for the Lebesgue measure of $\Omega\subset \R^d$, 
the forth-order elastic tensor $\mathbb{D}=\mathbb{D}_{ijk\ell}e_i\otimes e_j\otimes e_k\otimes e_\ell$
and the strain tensor ${\boldsymbol \varepsilon}({\boldsymbol u})$ are given by
$$
\mathbb{D}_{ijk\ell}=E\left(\frac{\nu}{(1+\nu)(1-2\nu)}\delta_{ij}\delta_{k\ell}+\frac{1}{2(1+\nu)}(\delta_{ik}\delta_{j\ell}+\delta_{i\ell}\delta_{jk})\right)
$$
for some $E, \nu>0$ and 
$$
{\boldsymbol \varepsilon}({\boldsymbol u})=\frac{1}{2}\left({\boldsymbol \nabla} {\boldsymbol u}+({\boldsymbol \nabla} {\boldsymbol u})^{\bf T}\right), \quad
\boldsymbol \nabla {\boldsymbol u}=\partial_{x_i}u_{j}e_j\otimes e_i,  
$$
respectively, the traction ${\bf t}\in \R^d$ is a constant vector, $\Gamma_t, \Gamma_D\subset \partial D\cap \partial\Omega$ and 
$$
V^d:=\{ \boldsymbol{ v}\in H^1(D)^d\colon \boldsymbol{ v}=0\ \text{ on } \Gamma_D\}.
$$
In particular, $\delta_{ij}$ and $e_k$ stand for the Kronecker delta and 
the $k$-th vector of the canonical basis of $\R^d$, respectively.

The minimization problem \eqref{op:MC} can be replaced by the unconstrained minimization problem \eqref{eq:opt-prob}, where 
\begin{align*}
&
F(\phi)=
\int_{\Gamma_t}
\boldsymbol{t}\cdot \boldsymbol{u}_\phi(x) \, \d \sigma
=
\int_D
\mathbb{D}\chi_\phi(x){\boldsymbol \varepsilon}({\boldsymbol u}_\phi)(x)\colon {\boldsymbol \varepsilon}({\boldsymbol u}_\phi)(x)\, \d x,
\\
&G(\phi)=\int_D \chi_\phi(x)\, \d x- G_{\text{max}}|D|\le 0.
\end{align*}

Now, we perform the numerical analysis for the minimizing problem along with the following steps\/{\rm :} 
\begin{itemize}
\setlength{\leftskip}{3mm}
\item[\bf Step\,1.]
Set Young's modulus $\bm{E}>0$, Poisson's ratio ${\nu}>0$ and
the traction vector $\boldsymbol{t}\in\R^d$  as follows{\rm :}
$$
E=2.1\times 10^{11} , \quad \nu=0.3, \quad \boldsymbol{t} =(0,-1.0\times 10^3).
$$
\item[\bf Step\,2.]
Solve the governing equation \eqref{eq:MCgov}, where 
$\mathbb{D}$ and $\boldsymbol{\e}(\boldsymbol{u}_\phi)$ are represented as 
$$
D:=
\begin{pmatrix}
\tilde\lambda+2\mu & \tilde\lambda& 0 \\
\tilde\lambda & \tilde\lambda+2\mu& 0 \\
0&0&\mu
\end{pmatrix},\quad
\tilde{\lambda}=\frac{E\nu}{(1+\nu)(1-2\nu)},\quad 
\mu=\frac{E}{2(1+\nu)}
$$
and $\boldsymbol{\epsilon}(\boldsymbol{u})=(\partial_{x_1}u_1, \partial_{x_2}u_2, \partial_{x_2}u_1+ \partial_{x_1}u_2)$, respectively.
\item[\bf Step\,5.]
By combining \eqref{eq:MCgov} with \eqref{eq:aRterm}, the reaction term of \eqref{discNLD} is represented by   
\begin{align}\label{MCRterm}
\mathcal{L}_\eta'(\phi,\lambda)=
\mathbb{D}\delta_\eta(\phi){\boldsymbol \varepsilon}({\boldsymbol u}_\phi)\colon {\boldsymbol \varepsilon}({\boldsymbol u}_\phi) - \lambda.
\end{align}
Here we set $\delta_\eta(\phi(x))= C_\eta\chi_{[0\le\phi\le \eta]}(x)/\eta$ and $C_\eta=\eta=1$ for simplicity. 
In particular, $\chi_\phi\in L^{\infty}(D;\{0,1\})$ is treated as $1/2+(15/16)(\phi/\delta)-(5/8)(\phi/\delta)^3+(3/16)(\phi/\delta)^5$ for $|\phi|<\delta$, and we set $\delta=0.8$ below to provide sensitivity in void domains as in {\it topological ligament} (see, e.g.,~\cite{tl}). 
\end{itemize}
The rest of the steps (i.e.,~Steps\,3, 4, 6 and 7) are the same as in \S \ref{S:algo}.

\begin{rmk}[Mathematical justification for sensitivity]
\rm It is noteworthy that the reaction term \eqref{MCRterm} 
coincides with that in \cite[(30)]{Y10}. 
Thus, the approximation $-F_\eta'(\phi)$ for $F'(\phi)$ yields the justification of the result in \cite{Y10}, which also completes the validation of the numerical algorithm in \S \ref{S:algo} (see Remark \ref{R:GR} below for more details).
\end{rmk}

\subsubsection{ Cantilever model}\label{SS:ca}

Under the setting as in Figure \ref{ica}, we set
$(\tau, G_{\rm max},\varDelta t)=(3.0\times 10^{-4}, 0.45, 0.7)$.
Then we can first confirm from Figures \ref{1-d} and \ref{1-h} that the same optimal configurations are obtained. We next see that the method using \eqref{discNLD} for $q>1$ converges to the optimal configuration faster than the method using the RDE (i.e.,~$q=1$); indeed, Figure \ref{fig:mc} shows that only $20$ steps are required to obtain the same shape as the optimal configuration, but for $q=1$, even if the LSF is updated by $325$ steps, the topology is still not optimized. This completes the check for (i-FDE).

\begin{figure}[htbp]
   \hspace*{-5mm} 
    \begin{tabular}{cccc}
      \begin{minipage}[t]{0.24\hsize}
        \centering
        \includegraphics[keepaspectratio, scale=0.1]{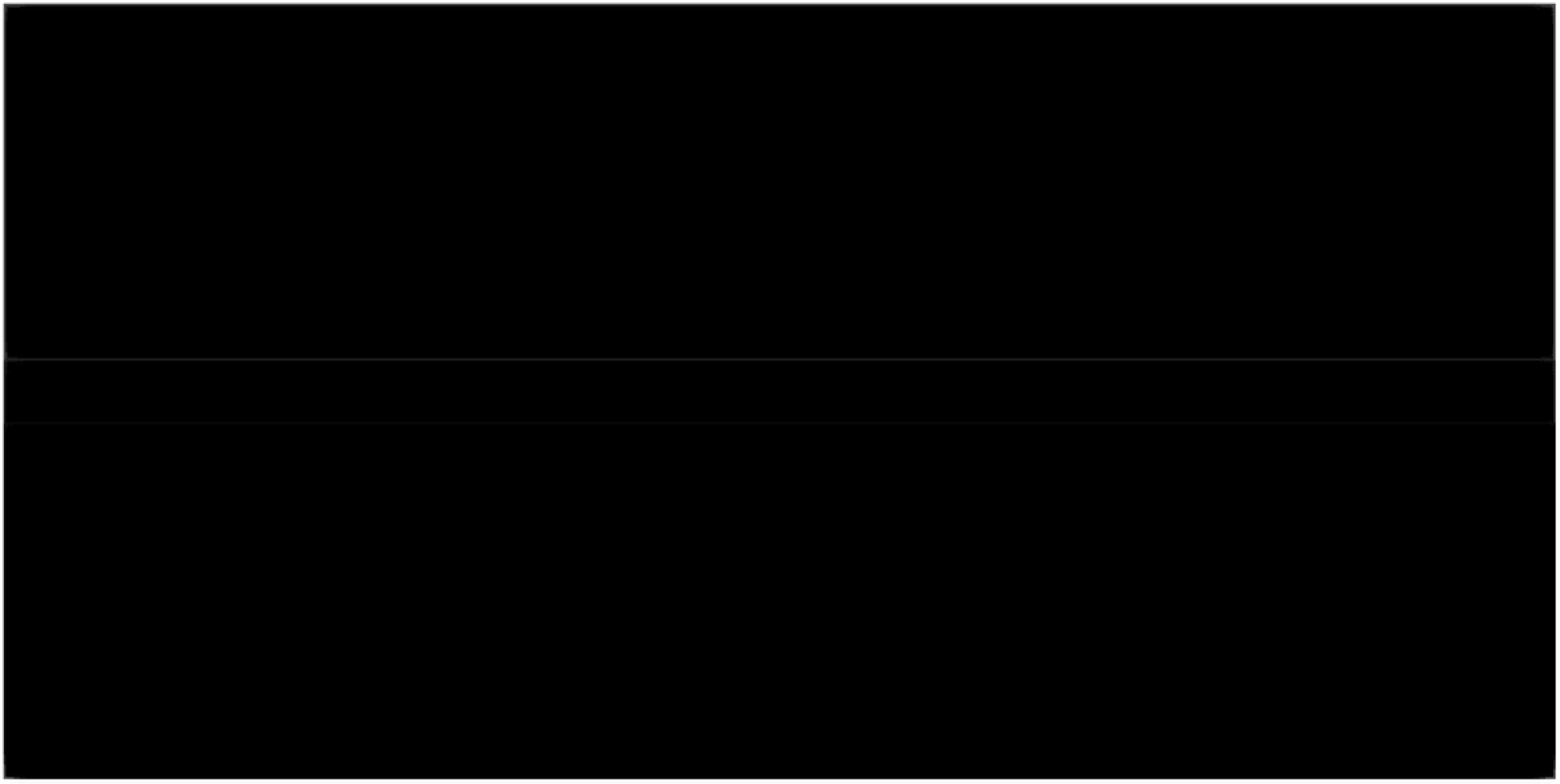}
        \subcaption{Step\,0}
        \label{1-a}
      \end{minipage} 
      \begin{minipage}[t]{0.24\hsize}
        \centering
        \includegraphics[keepaspectratio, scale=0.1]{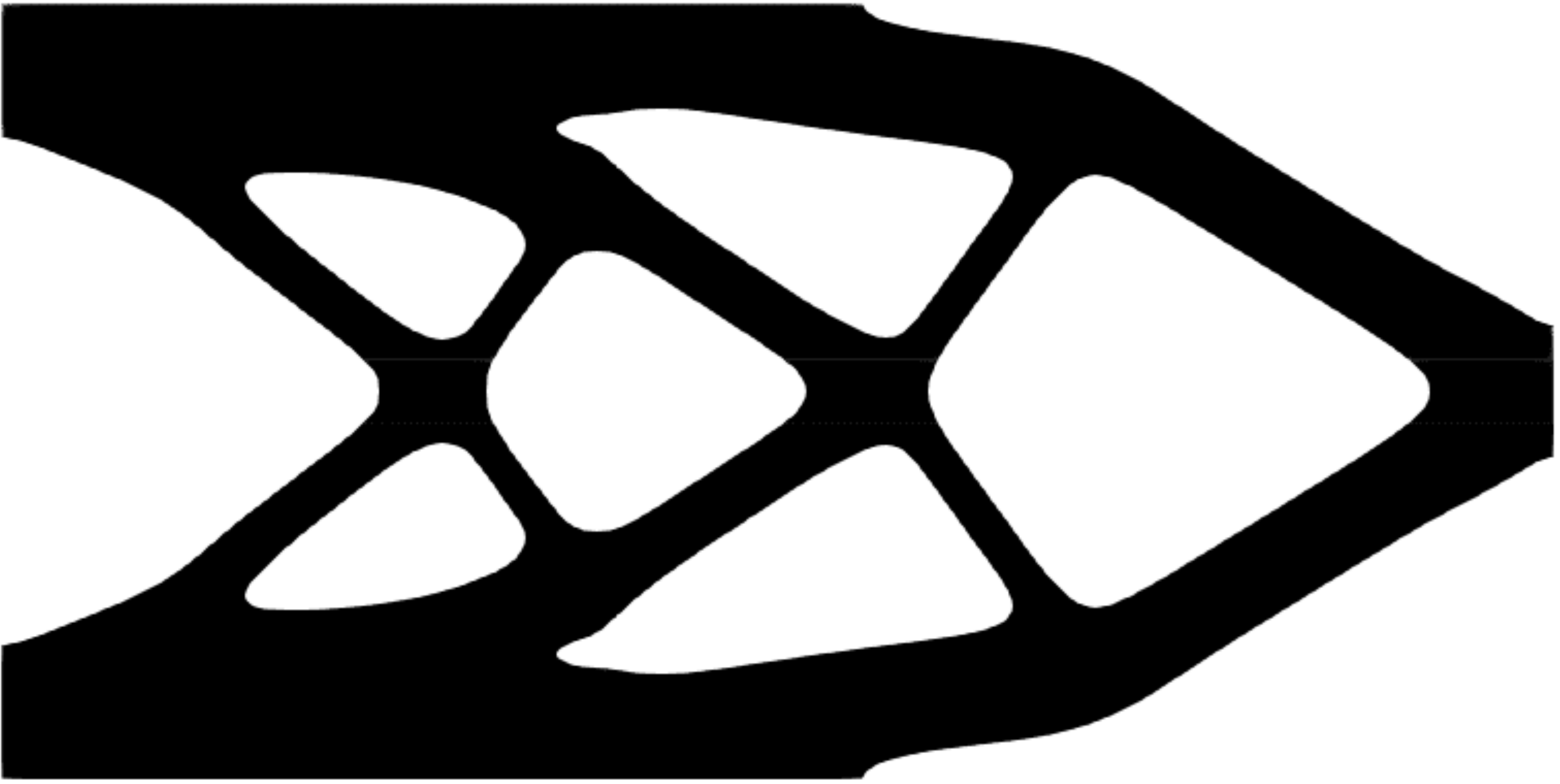}
        \subcaption{Step\,20}
        \label{1-b}
      \end{minipage} 
         \begin{minipage}[t]{0.24\hsize}
        \centering
        \includegraphics[keepaspectratio, scale=0.1]{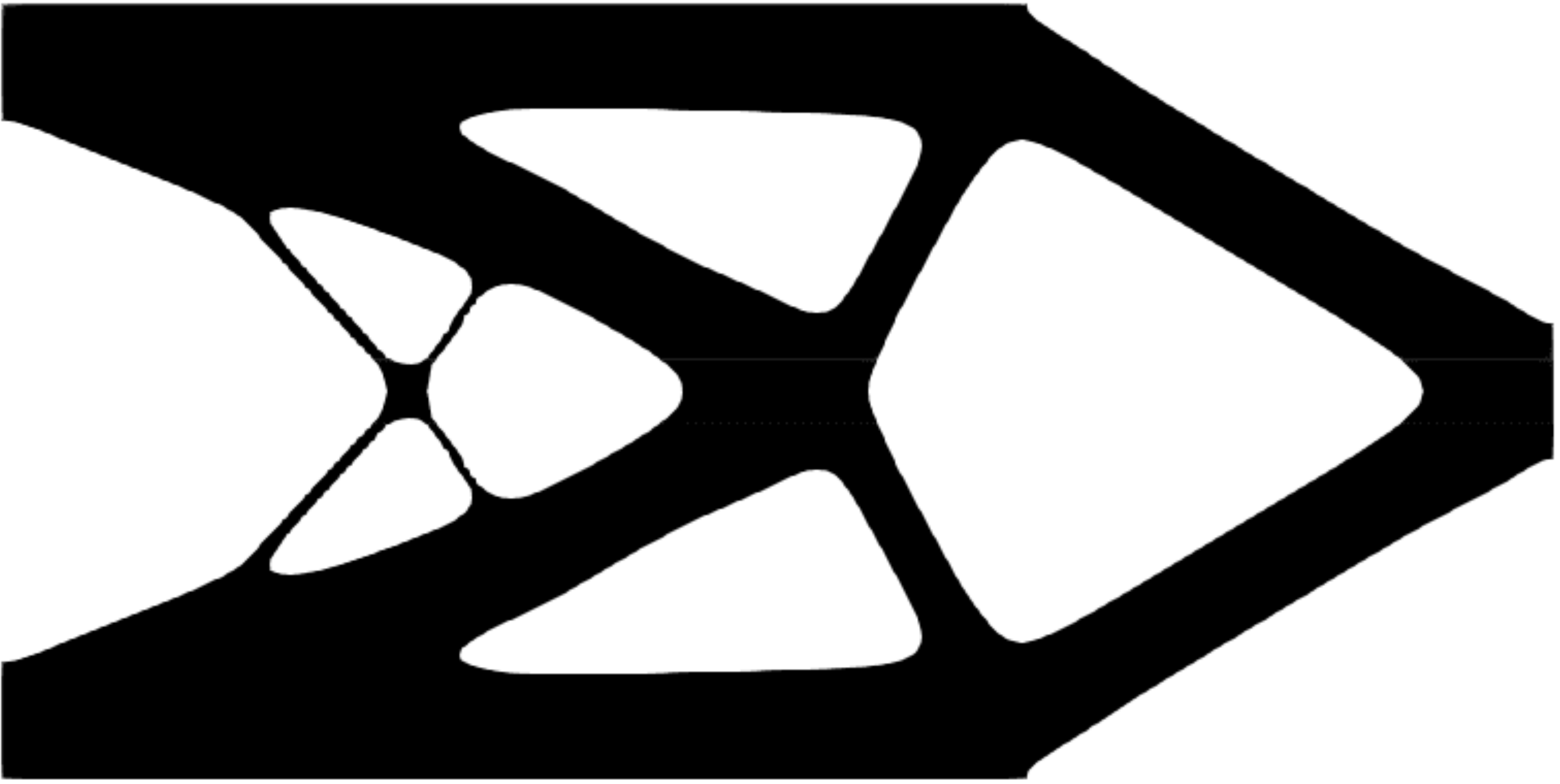}
        \subcaption{Step\,325}
        \label{1-c}
      \end{minipage}
           \begin{minipage}[t]{0.24\hsize}
        \centering
        \includegraphics[keepaspectratio, scale=0.1]{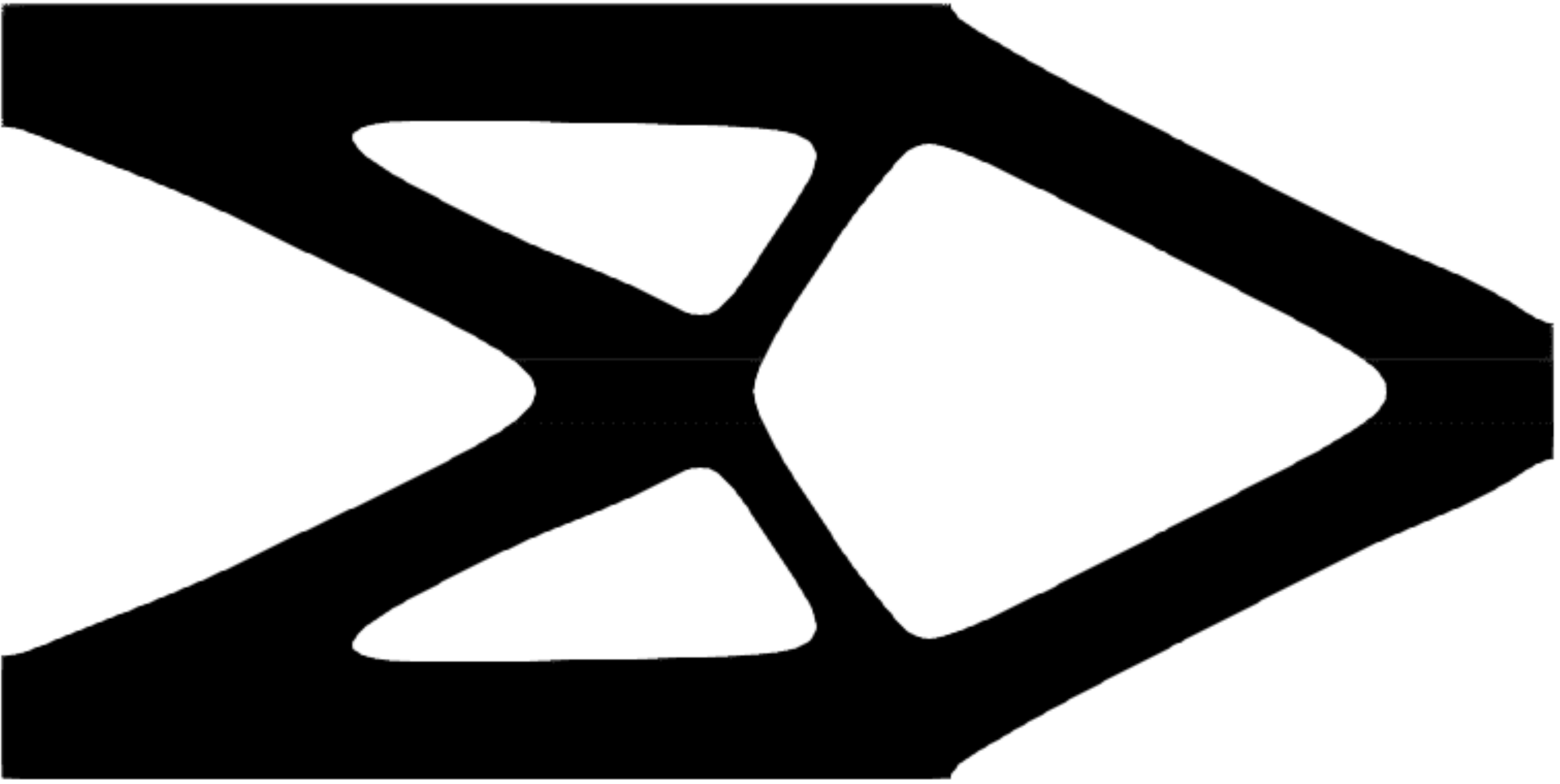}
        \subcaption{Step\,700$^{\#}$}
        \label{1-d}
      \end{minipage}
         \\ 
    \begin{minipage}[t]{0.24\hsize}
        \centering
        \includegraphics[keepaspectratio, scale=0.1]{mc-0.pdf}
        \subcaption{Step\,0}
        \label{1-e}
      \end{minipage} 
      \begin{minipage}[t]{0.24\hsize}
        \centering
        \includegraphics[keepaspectratio, scale=0.1]{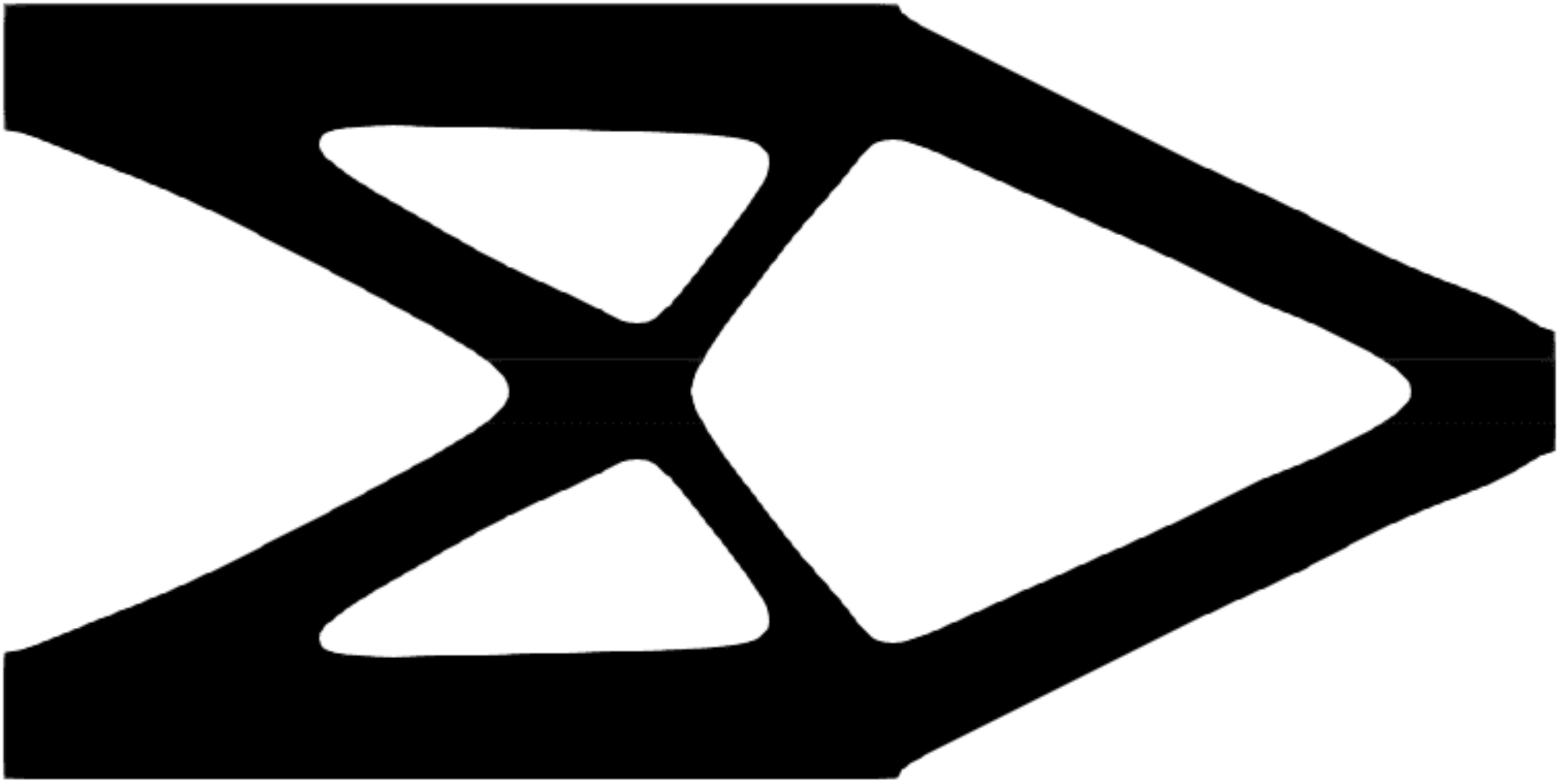}
        \subcaption{Step\,20}
        \label{1-f}
      \end{minipage} 
       \begin{minipage}[t]{0.24\hsize}
        \centering
        \includegraphics[keepaspectratio, scale=0.1]{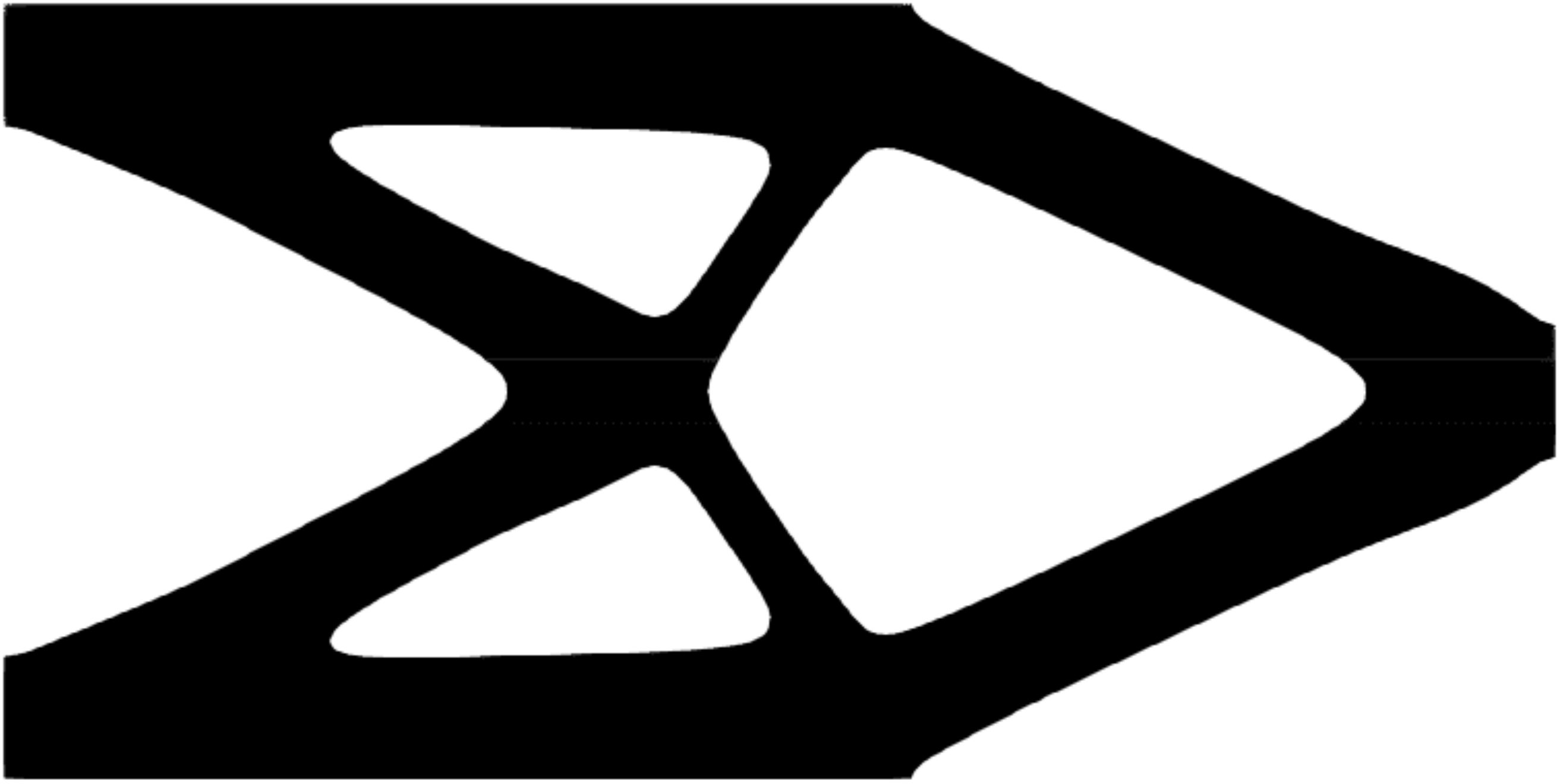}
        \subcaption{Step\,60}
        \label{1-g}
      \end{minipage} 
           \begin{minipage}[t]{0.24\hsize}
        \centering
        \includegraphics[keepaspectratio, scale=0.1]{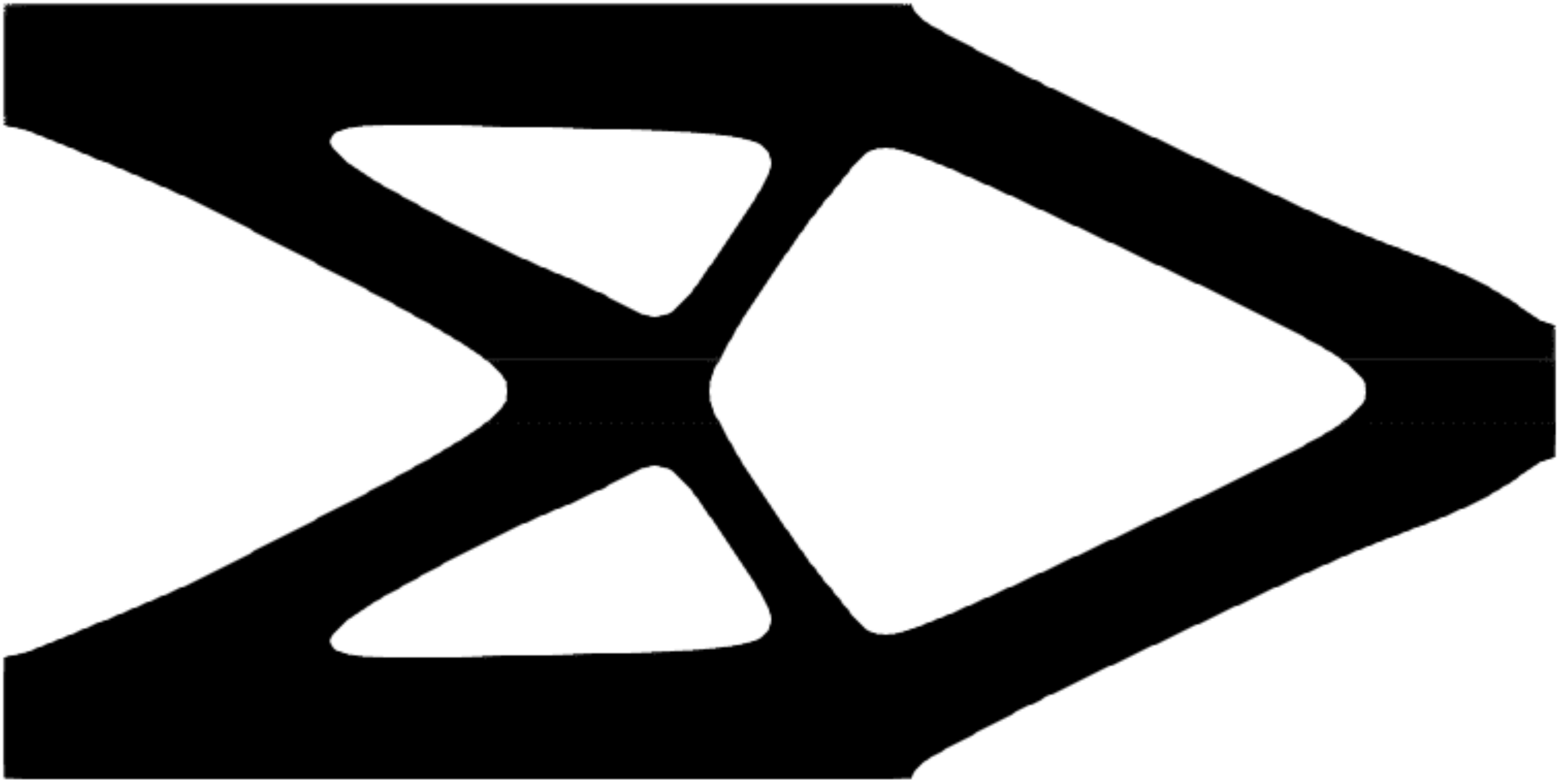}
        \subcaption{Step\,64$^{\#}$}
        \label{1-h}
        \end{minipage} 
    \end{tabular}
     \caption{Configuration $\Omega_{\phi_n}\subset D$ for the case where the initial configuration is the whole domain. 
Figures (a)--(d) and (e)--(h) represent $\Omega_{\phi_n}\subset D$ for $q=1$ and $q=4$ in \eqref{discNLD}, respectively. 
The symbol ${}^{\#}$ implies the final step.      
     }
     \label{fig:mc}
 \end{figure}

\begin{figure}[htbp]
        \centering
        \includegraphics[keepaspectratio, scale=0.35]{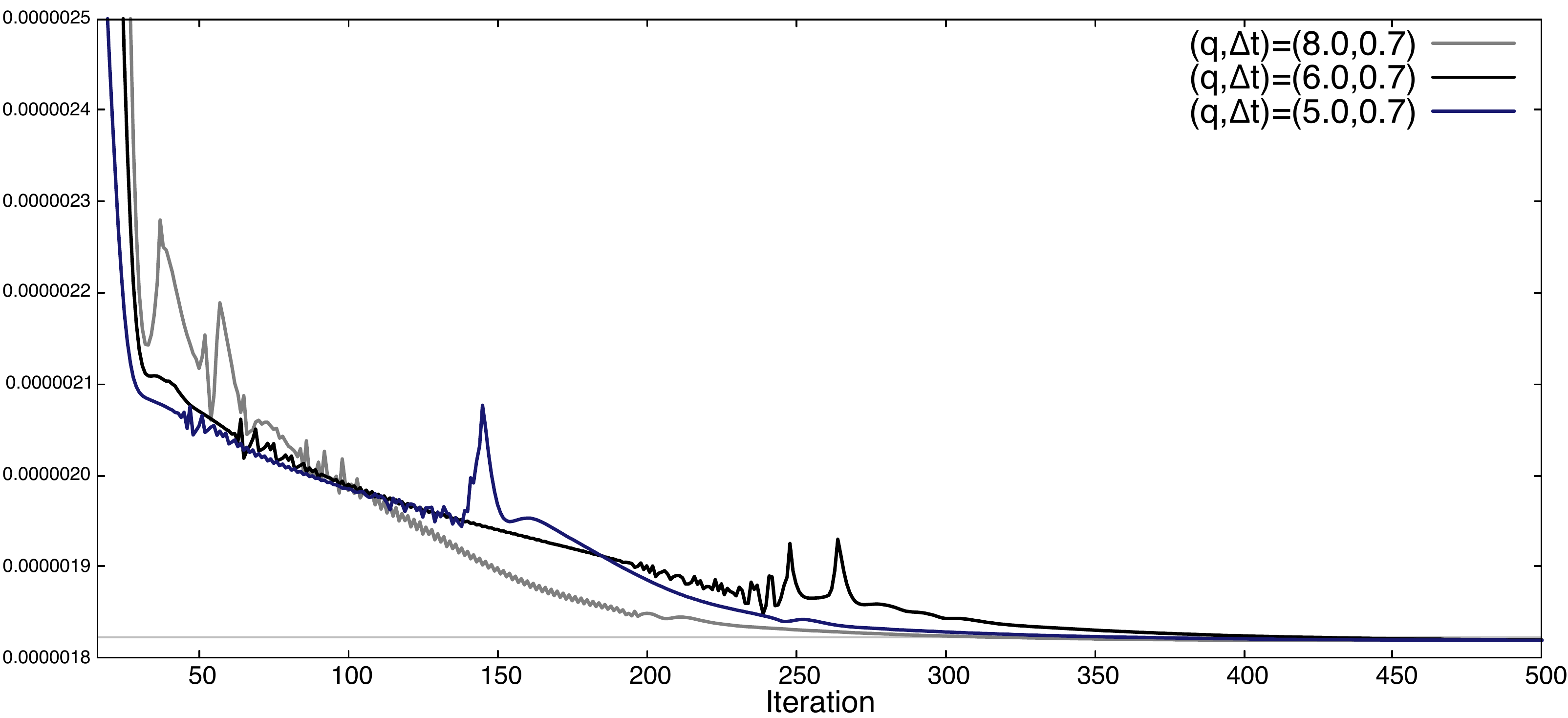}
   \caption{Objective functional $F(\phi_n)$ using fast diffusion.}
\label{fig:mcobj}
\end{figure}       
  
\begin{figure}[htbp]
        \centering
        \includegraphics[keepaspectratio, scale=0.37]{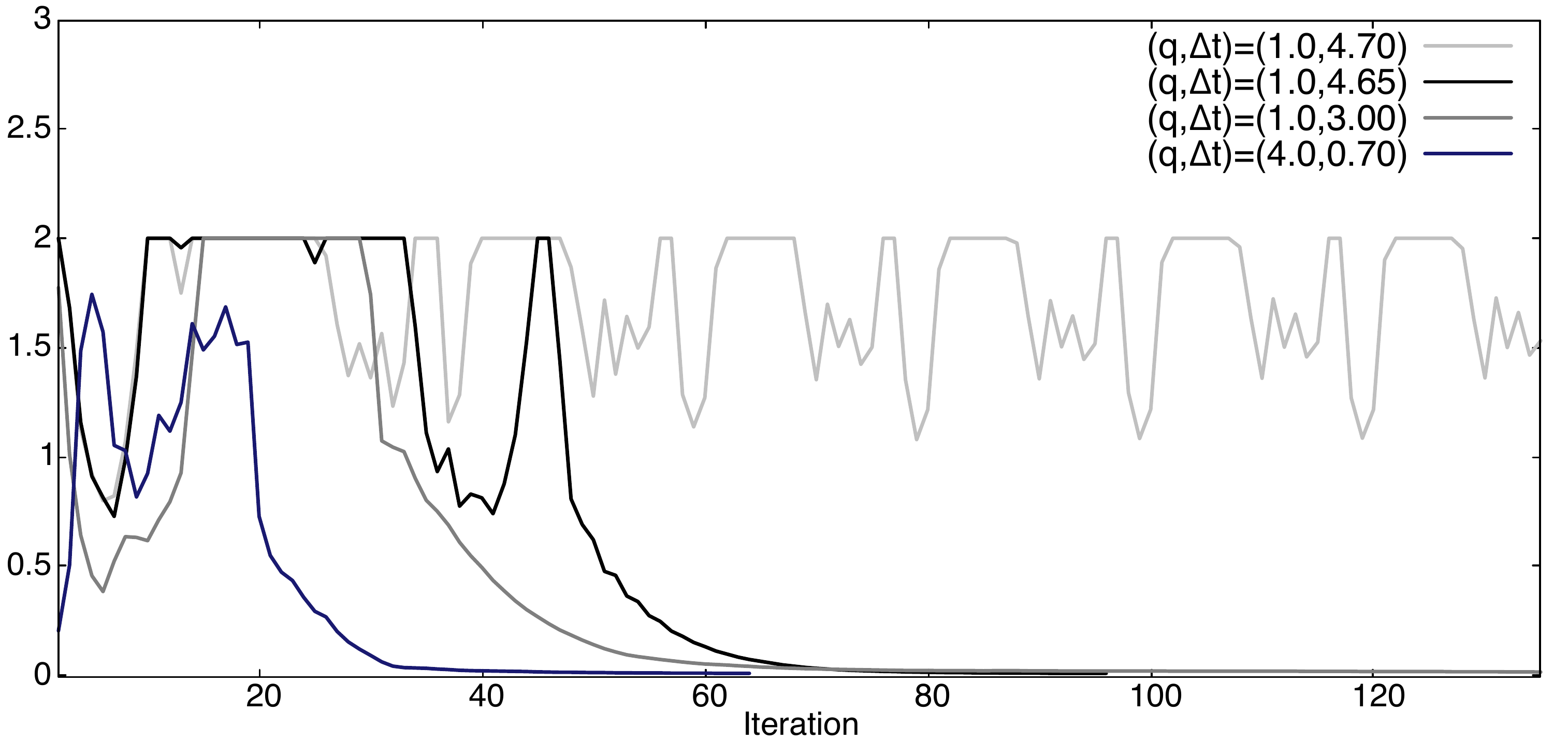}
   \caption{Convergence condition $\|\phi_{n+1}-\phi_n\|_{L^{\infty}(D)}$.}
\label{fig:caconv}
\end{figure}

\begin{rmk}[Non-monotonicity of convergence]
\rm 
We deduce from Figure \ref{fig:mcobj} that the monotonicity of 
convergence speed with respect to $q>1$ does not seem to be obtained.
\end{rmk}

  \begin{rmk}[Irrelevance of time step] \label{R:tinc}
  \rm
 Since $\varDelta t>0$ corresponds to the updated step width of the LSF, it is possible to improve the speed for convergence by changing to larger values; indeed, one can confirm the fact in Figure \ref{fig:caconv}. However, under this setting, we see that the results for $q=4$ cannot be improved by the method using reaction-diffusion even if so is $\varDelta t>0$. 

On the other hand, since both methods with reaction-diffusion and (doubly) nonlinear diffusion are based on the gradient descent method, if $\varDelta t>0$ is too large, the objective functional oscillates, and the convergence condition \eqref{eq:CC} is not satisfied. 
To avoid the oscillation of shapes, we restrict $0<\varDelta t< 1$ below. 
\end{rmk}

As for the corresponding three-dimensional case, the same results can be obtained (see Figure \ref{fig:3dmc}). Here we set $(\tau, G_{\rm max},\varDelta t,\xi)=(1.0\times 10^{-4}, 0.45, 0.5,0)$, and the boundary condition is the same as in \cite[Fig.12]{Y10}.
In particular, it is noteworthy that only $40$ steps are needed to optimize the topology in case of fast diffusion, and the sensitivity in the void domain is effective in Figure \ref{3d-b}-\ref{3d-c}. 

\begin{figure}[htbp]
   \hspace*{-5mm} 
    \begin{tabular}{cccc}
      \begin{minipage}[t]{0.24\hsize}
        \centering
        \includegraphics[keepaspectratio, scale=0.12,angle=0]{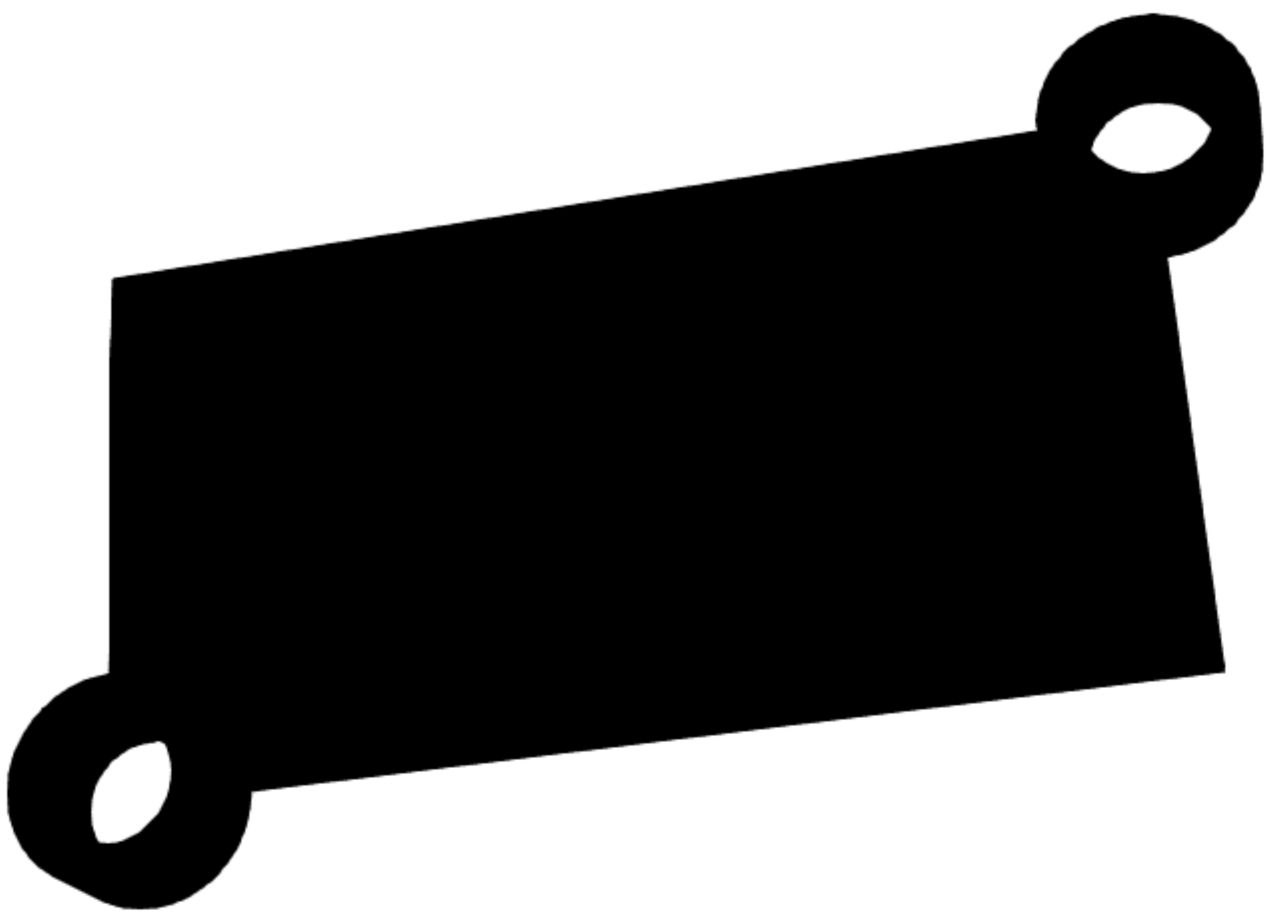}
        \subcaption{Step\,0}
        \label{3d-a}
      \end{minipage} 
      \begin{minipage}[t]{0.24\hsize}
        \centering
        \includegraphics[keepaspectratio, scale=0.12,angle=0]{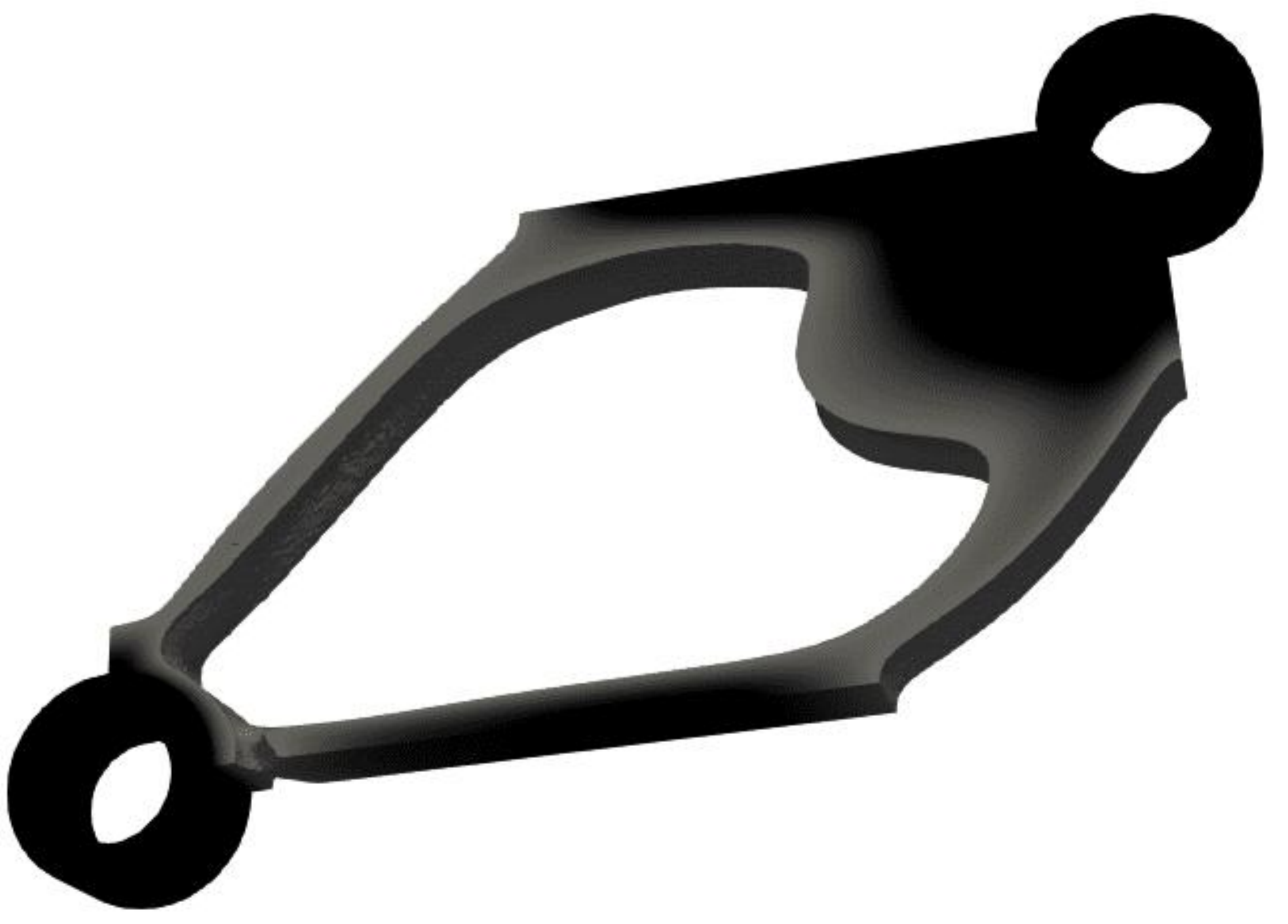}
        \subcaption{Step\,40}
        \label{3d-b}
      \end{minipage} 
         \begin{minipage}[t]{0.24\hsize}
        \centering
        \includegraphics[keepaspectratio, scale=0.12,angle=0]{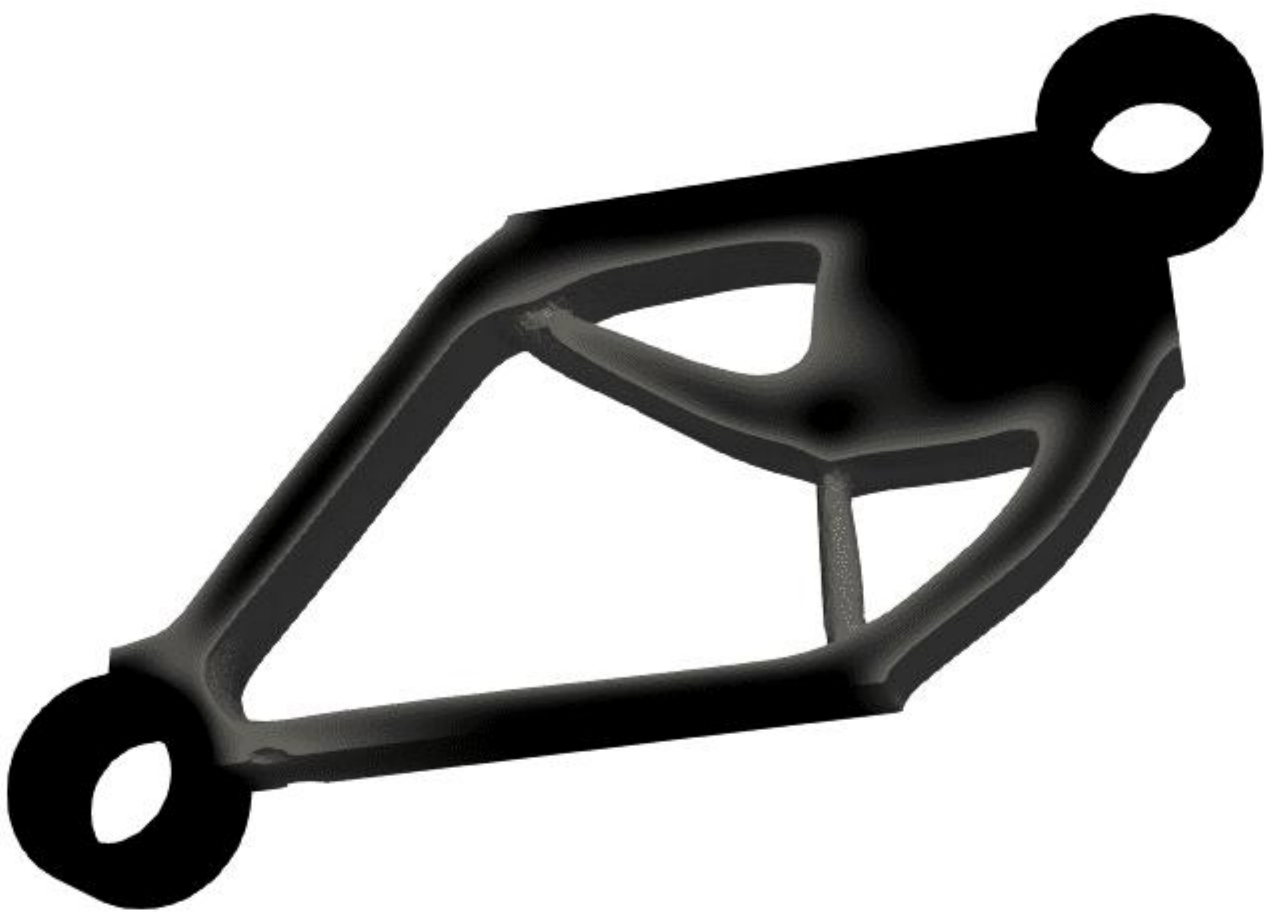}
        \subcaption{Step\,80}
        \label{3d-c}
      \end{minipage}
           \begin{minipage}[t]{0.24\hsize}
        \centering
        \includegraphics[keepaspectratio, scale=0.12,angle=0]{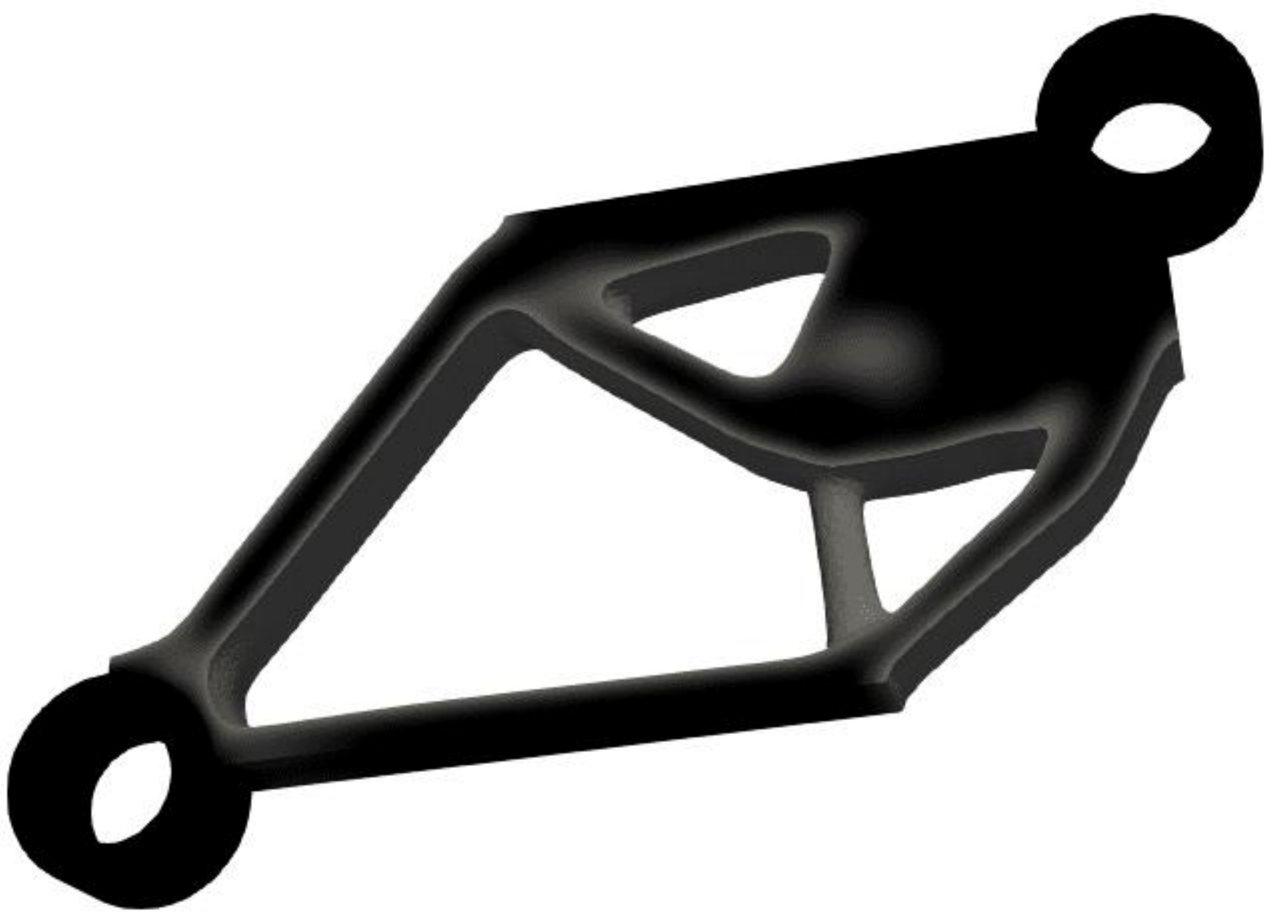}
        \subcaption{Step\,164$^{\#}$}
        \label{3d-d}
      \end{minipage}
         \\ 
    \begin{minipage}[t]{0.24\hsize}
        \centering
        \includegraphics[keepaspectratio, scale=0.12,angle=0]{3d0.pdf}
        \subcaption{Step\,0}
        \label{3d-e}
      \end{minipage} 
      \begin{minipage}[t]{0.24\hsize}
        \centering
        \includegraphics[keepaspectratio, scale=0.12,angle=0]{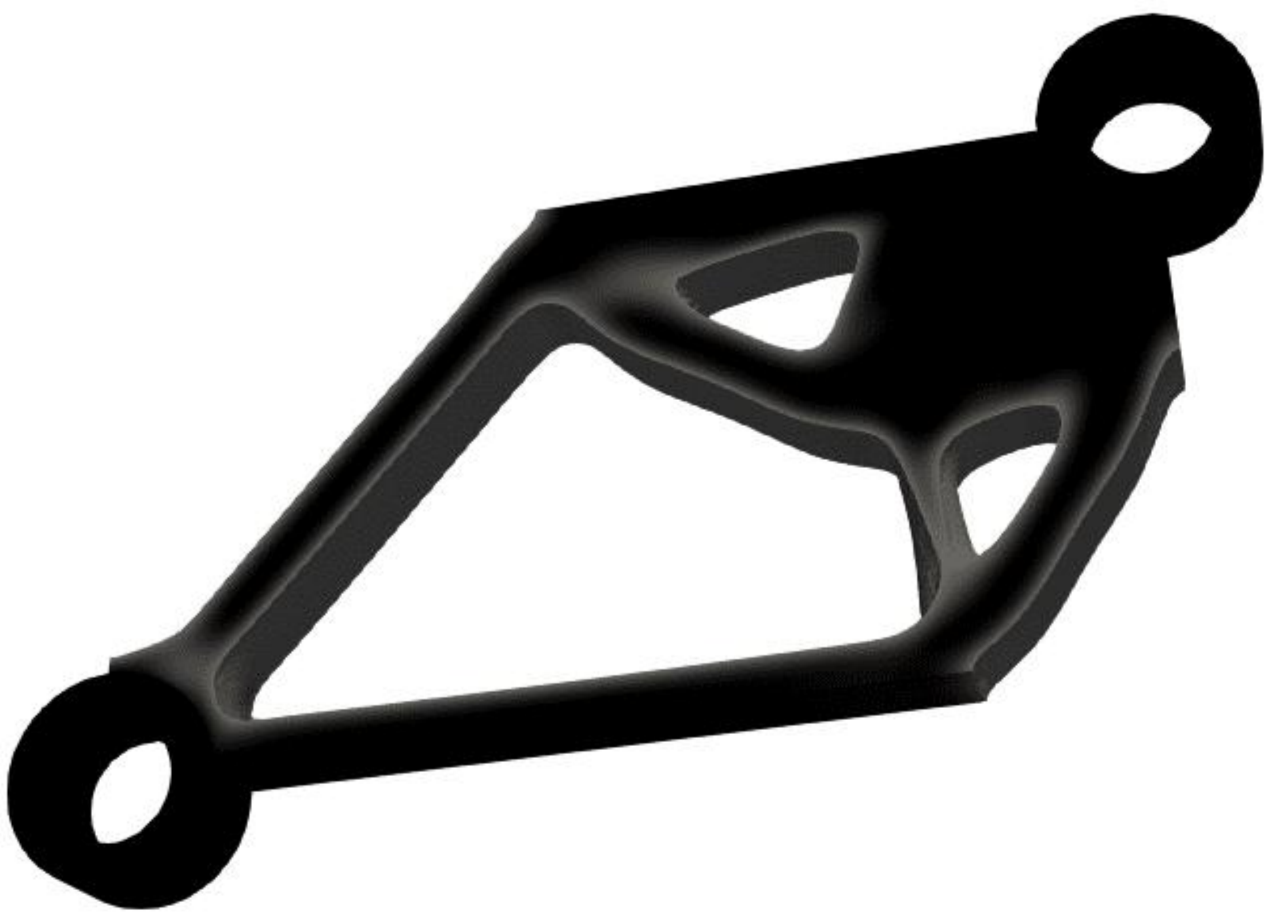}
        \subcaption{Step\,40 }
        \label{3d-f}
      \end{minipage} 
       \begin{minipage}[t]{0.24\hsize}
        \centering
        \includegraphics[keepaspectratio, scale=0.12,angle=0]{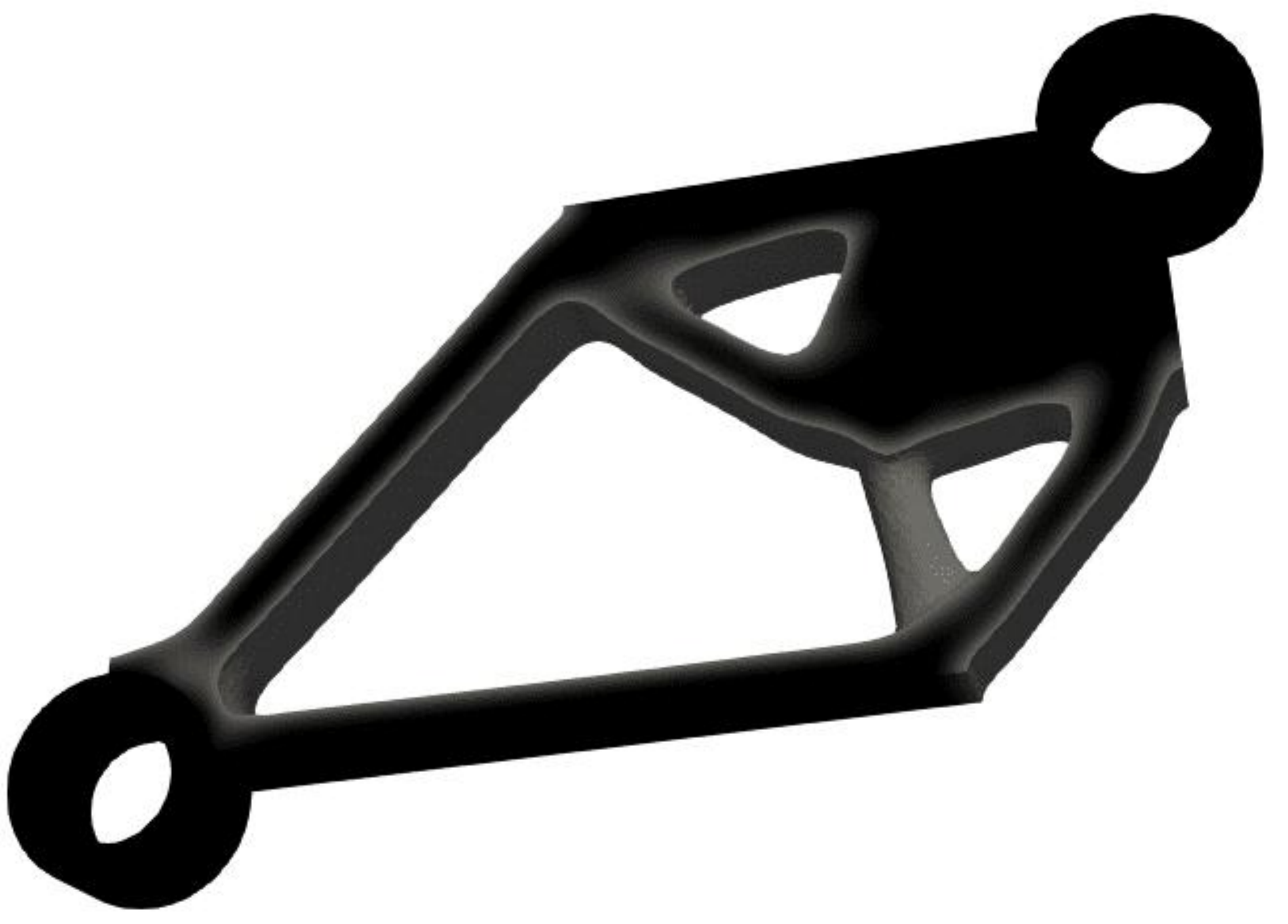}
        \subcaption{Step\,80} 
        \label{3d-g}
      \end{minipage} 
           \begin{minipage}[t]{0.24\hsize}
        \centering
        \includegraphics[keepaspectratio, scale=0.12,angle=0]{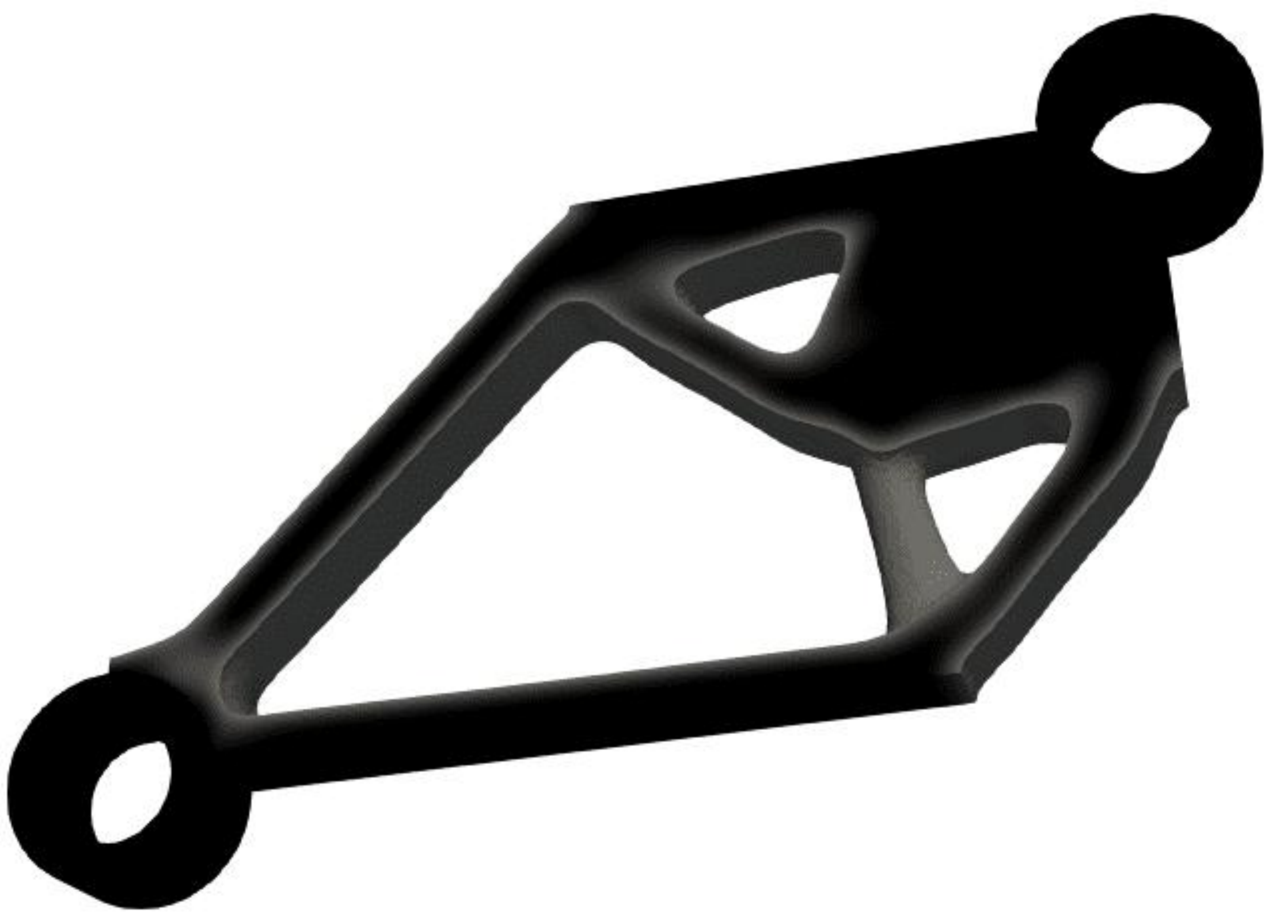}
        \subcaption{Step\,94$^{\#}$}
        \label{3d-h}
        \end{minipage} 
    \end{tabular}
     \caption{Configuration $\Omega_{\phi_n}\subset D\subset\R^3$ for the case where the initial configuration is the whole domain. 
Figures (a)--(d) and (e)--(h)  
represent $\Omega_{\phi_n}\subset D$ for $q=1$ and $q=5$ in \eqref{discNLD}, respectively. 
The symbol ${}^{\#}$ implies the final step.      
     }
     \label{fig:3dmc}
  \end{figure}

\subsubsection{MBB beam model}
In order to get the validity of the proposed method, we next consider another boundary condition (see Figure \ref{imbb}). Here we choose $(\tau, G_{\rm max},\varDelta t)=(6.0\times 10^{-5}, 0.4, 0.3)$ as the given parameters to make the setting different from that of the cantilever.
Then the proposed method optimizes the topology in $40$ steps (see Figure \ref{mbb-f}), whereas the method using reaction-diffusion does not, even in $160$ steps.
Comparing Figure \ref{mbb-d} with Figure \ref{mbb-h},  we see that convergence for configurations can be improved, which completes the confirmation of (i-FDE). 
\begin{figure}[htbp]
   \hspace*{-5mm} 
    \begin{tabular}{ccccc} 
             \begin{minipage}[t]{0.25\hsize}
        \centering
        \includegraphics[keepaspectratio, scale=0.09]{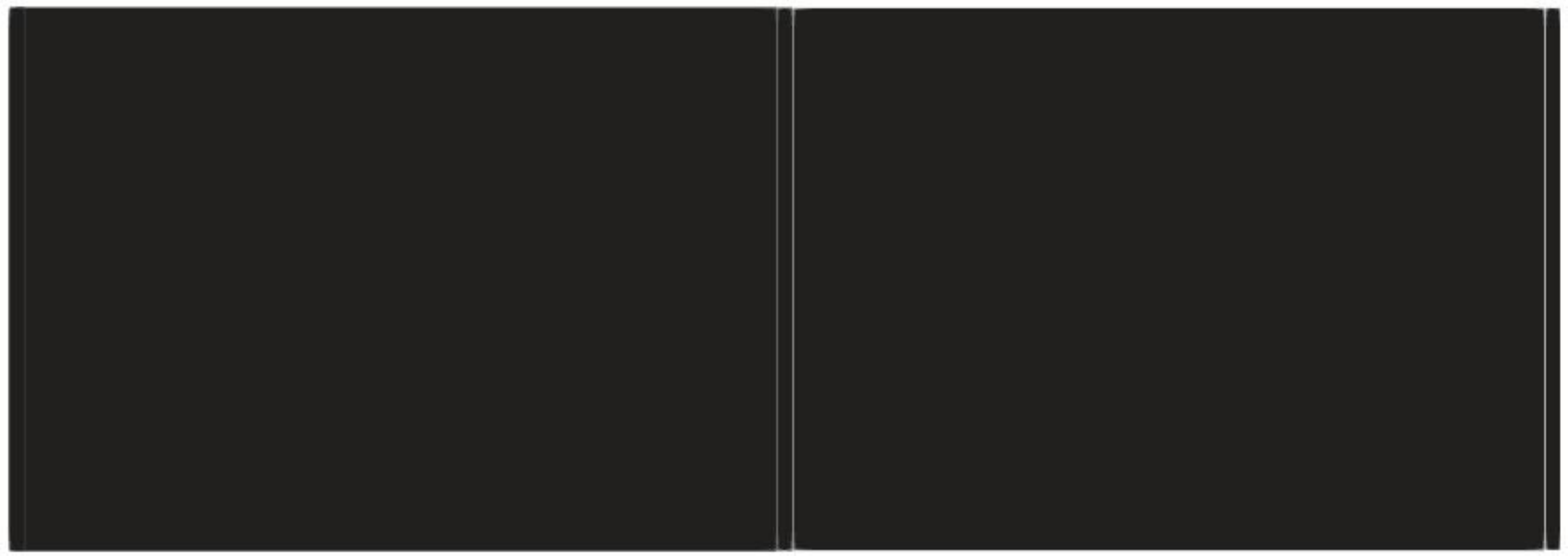}
        \subcaption{Step\,0}
        \label{mbb-a}
      \end{minipage} 
      \begin{minipage}[t]{0.25\hsize}
        \centering
        \includegraphics[keepaspectratio, scale=0.09]{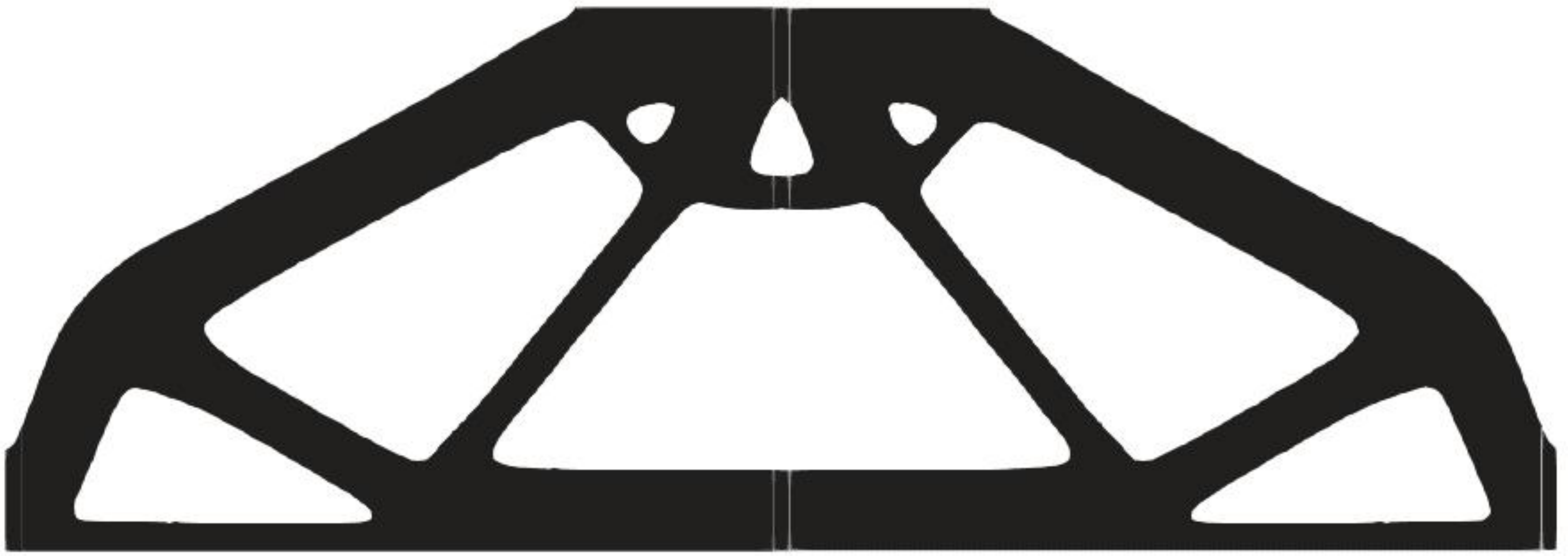}
        \subcaption{Step\,40}
        \label{mbb-b}
      \end{minipage} 
         \begin{minipage}[t]{0.25\hsize}
        \centering
        \includegraphics[keepaspectratio, scale=0.09]{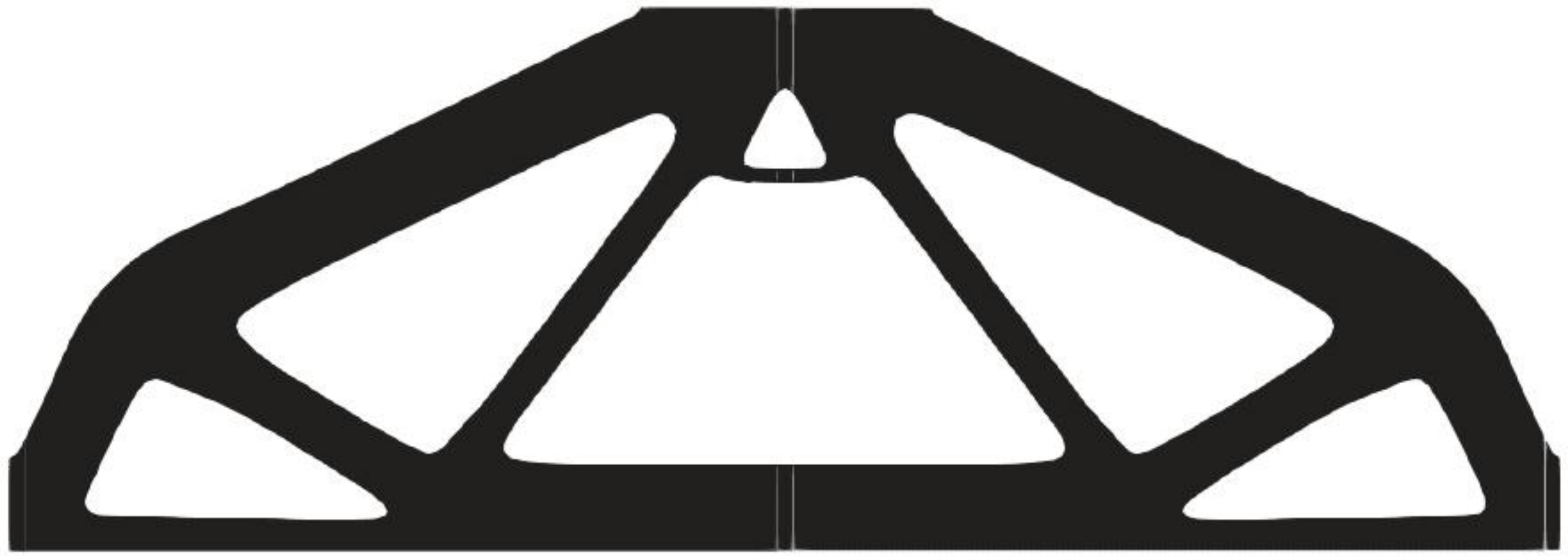}
        \subcaption{Step\,160}
        \label{mbb-c}
      \end{minipage}
      \begin{minipage}[t]{0.25\hsize}
        \centering
        \includegraphics[keepaspectratio, scale=0.09]{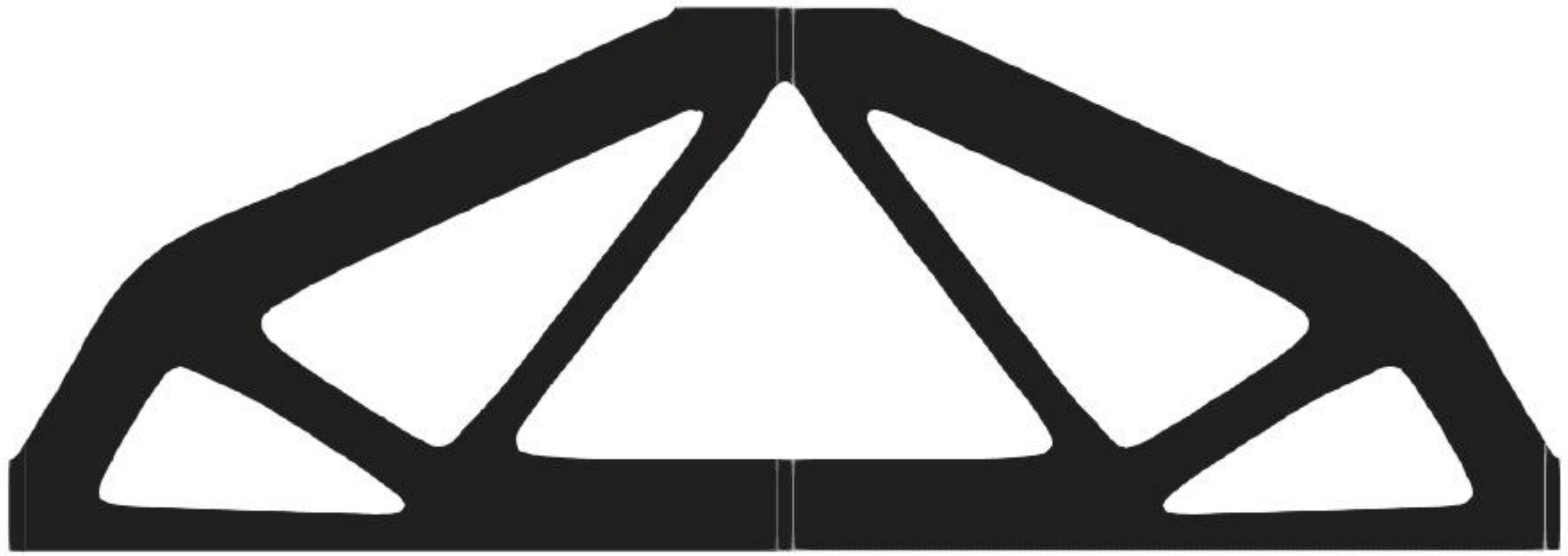}
        \subcaption{Step\,414$^{\#}$}
        \label{mbb-d}
      \end{minipage}
        \\
       \begin{minipage}[t]{0.25\hsize}
        \centering
        \includegraphics[keepaspectratio, scale=0.09]{mbbb0.pdf}
        \subcaption{Step\,0}
        \label{mbb-e}
      \end{minipage} 
      \begin{minipage}[t]{0.25\hsize}
        \centering
        \includegraphics[keepaspectratio, scale=0.09]{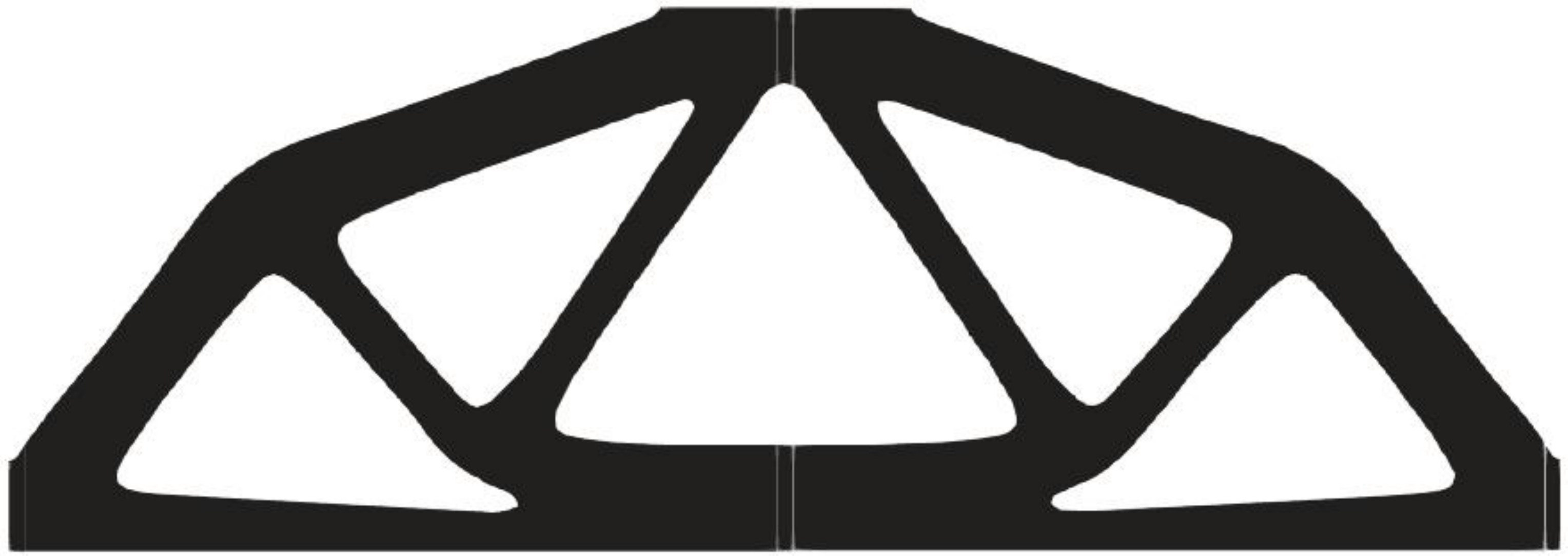}
        \subcaption{Step\,40}
        \label{mbb-f}
      \end{minipage} 
         \begin{minipage}[t]{0.25\hsize}
        \centering
        \includegraphics[keepaspectratio, scale=0.09]{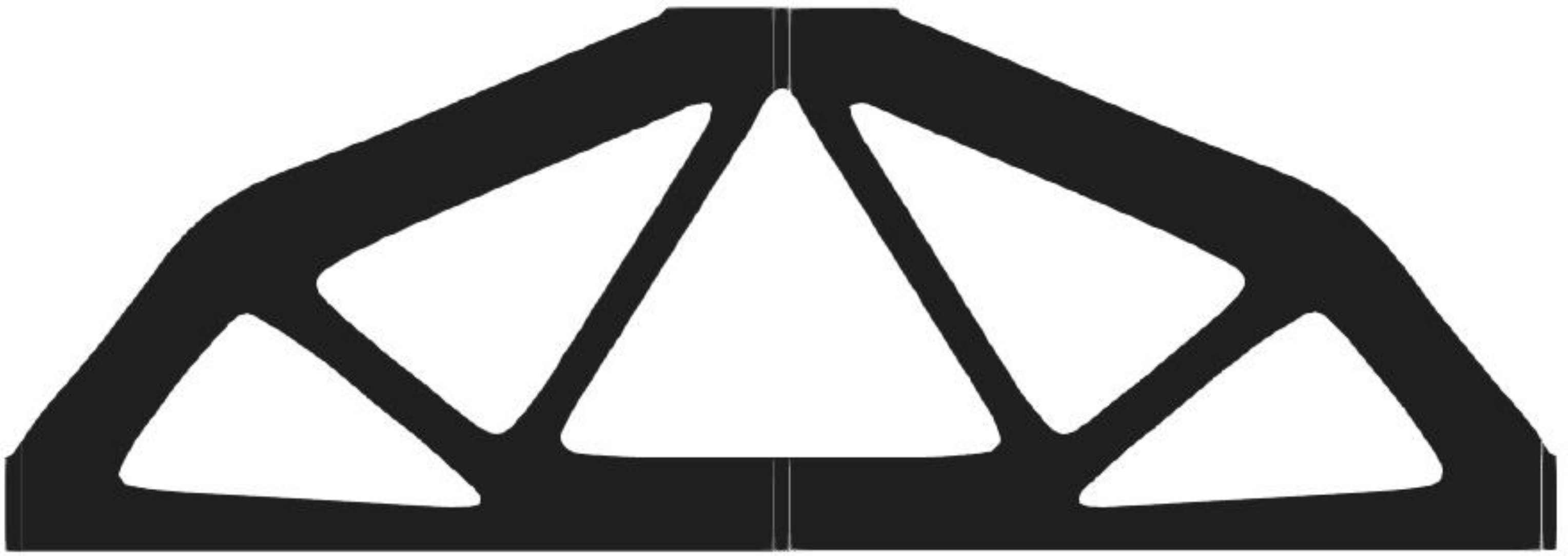}
        \subcaption{Step\,160}
        \label{mbb-g}
      \end{minipage}
      \begin{minipage}[t]{0.25\hsize}
        \centering
        \includegraphics[keepaspectratio, scale=0.09]{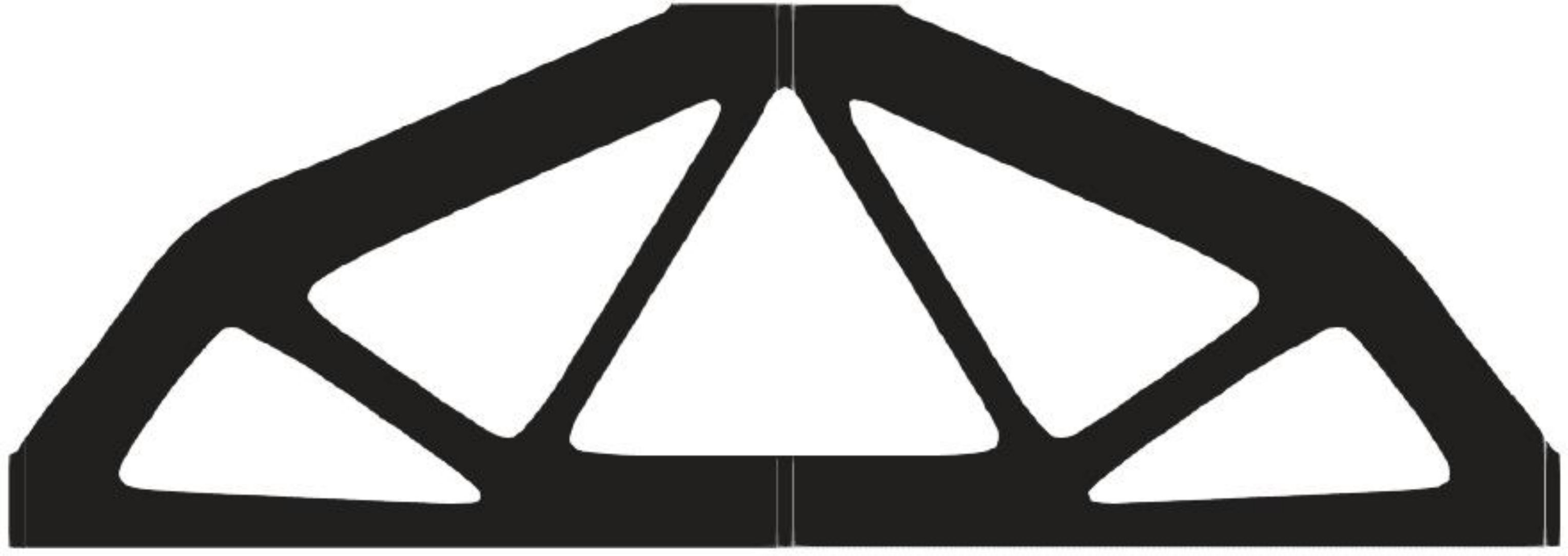}
        \subcaption{Step\,366$^\#$}
        \label{mbb-h}
      \end{minipage}
    \end{tabular}
     \caption{Configuration $\Omega_{\phi_n}\subset D$ for the case where the initial configuration is the whole domain. 
Figures (a)--(d) and (e)--(h)  
represent $\Omega_{\phi_n}\subset D$ for $q=1$ and $q=2$ in \eqref{discNLD}, respectively. 
{The symbol ${}^{\#}$ implies the final step.}     
     }
     \label{fig:cm}
  \end{figure}

\subsubsection{Bridge model}\label{SS:bri}
We consider the case where the boundary condition is shown in Figure \ref{ibri} as the last example for \eqref{op:MC}. 
Here we choose $(\tau, G_{\rm max},\varDelta t)=(1.0\times 10^{-4}, 0.35, 0.5)$.
In particular, let the initial configuration be a perforated domain to remove initial level set dependencies. Then Figure \ref{fig:bri} is obtained. 
Comparing Figures \ref{bri-a}--\ref{bri-e} with \ref{bri-k}--\ref{bri-o}, one can verify that
the proposed method yields the optimal topology in $30$ steps and drastically improves the convergence.  

As in Remark \ref{R:tinc}, since the improvement of convergence is expected by increasing $\varDelta t>0$, we next change $\varDelta t>0$ from $0.5$ to $0.9$ for $q=1$.
Then, we obtain Figure \ref{bri-f}--\ref{bri-j}, and the convergence for optimal configuration can be improved; however, the improvement is not as significant as that of the proposed method, which confirms the effectiveness of the proposed method. 

\begin{figure}[htbp]
   \hspace*{-5mm} 
    \begin{tabular}{ccccc} 
             \begin{minipage}[t]{0.2\hsize}
        \centering
        \includegraphics[keepaspectratio, scale=0.09]{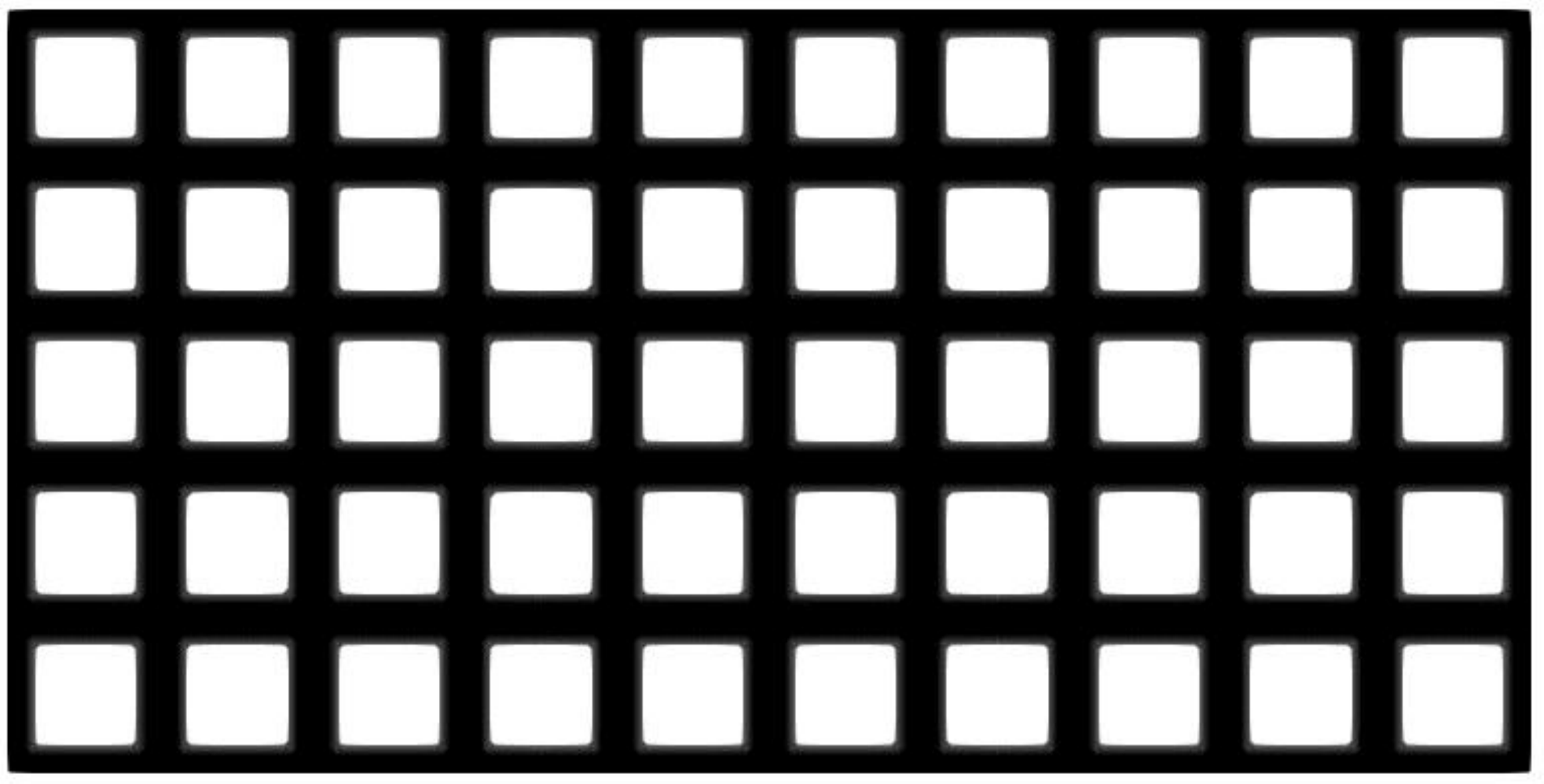}
        \subcaption{Step\,0}
        \label{bri-a}
      \end{minipage} 
      \begin{minipage}[t]{0.2\hsize}
        \centering
        \includegraphics[keepaspectratio, scale=0.09]{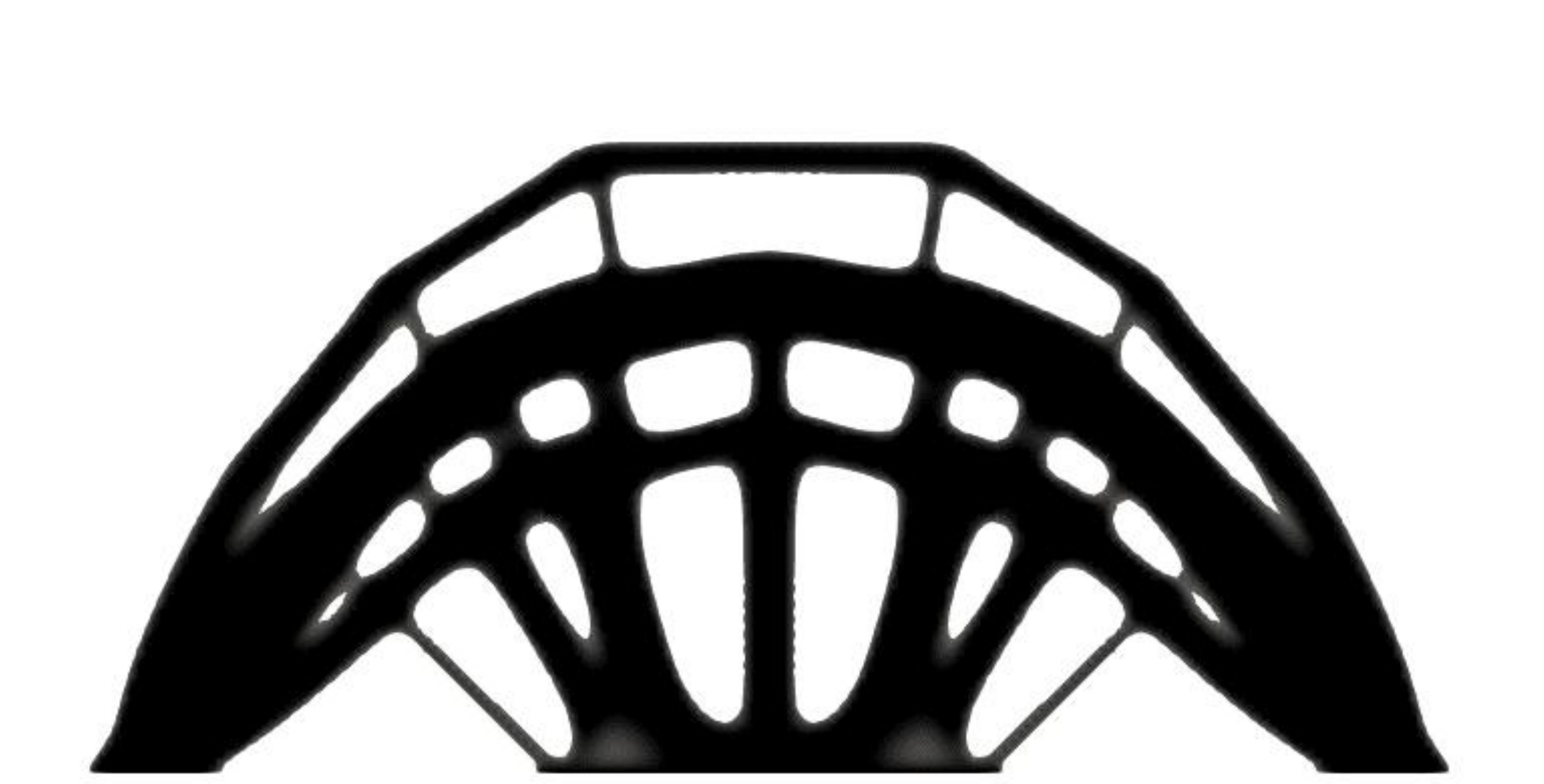}
        \subcaption{Step\,30}
        \label{bri-b}
      \end{minipage} 
         \begin{minipage}[t]{0.2\hsize}
        \centering
        \includegraphics[keepaspectratio, scale=0.09]{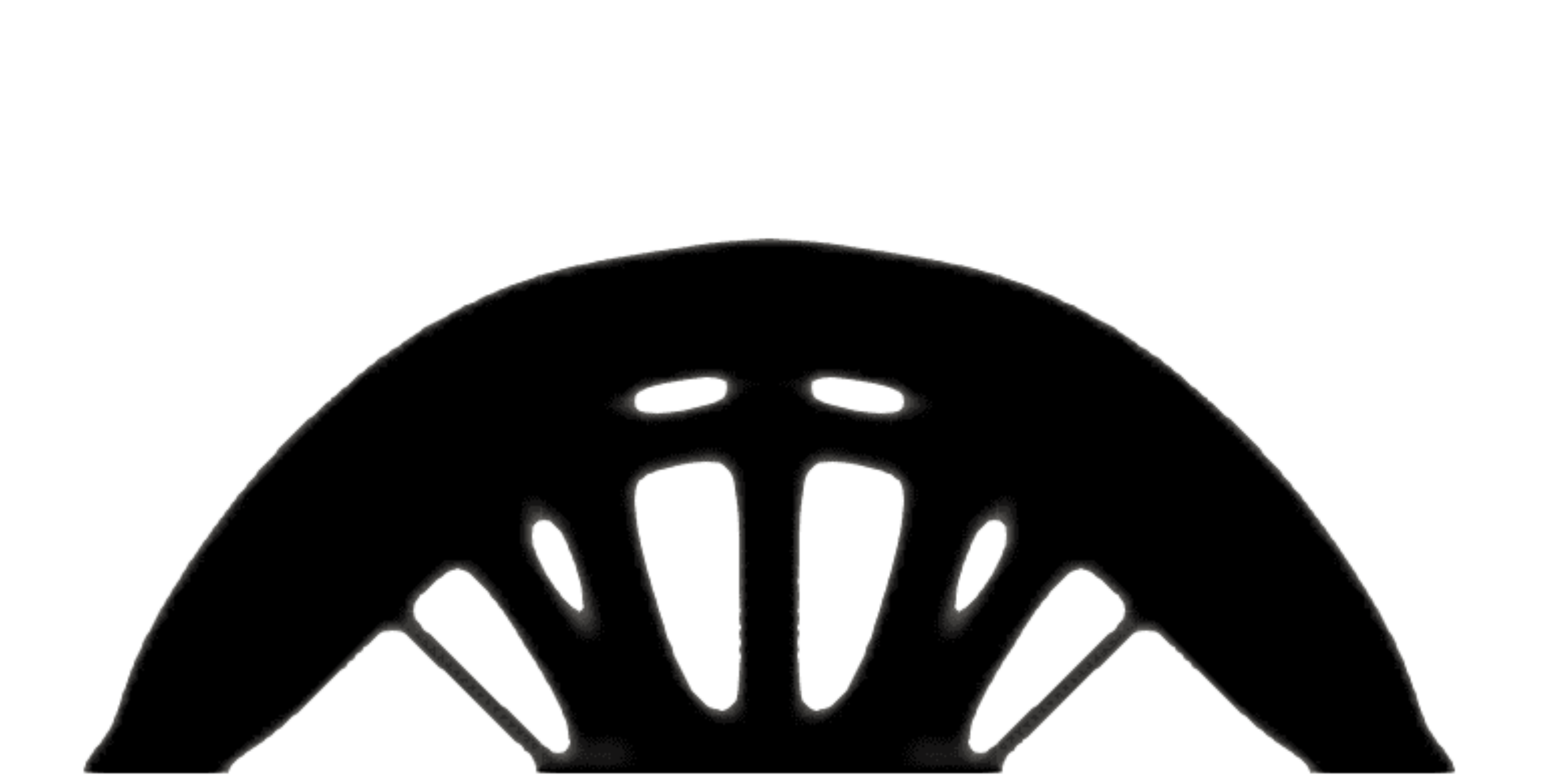}
        \subcaption{Step\,60}
        \label{bri-c}
      \end{minipage}
      \begin{minipage}[t]{0.2\hsize}
        \centering
        \includegraphics[keepaspectratio, scale=0.09]{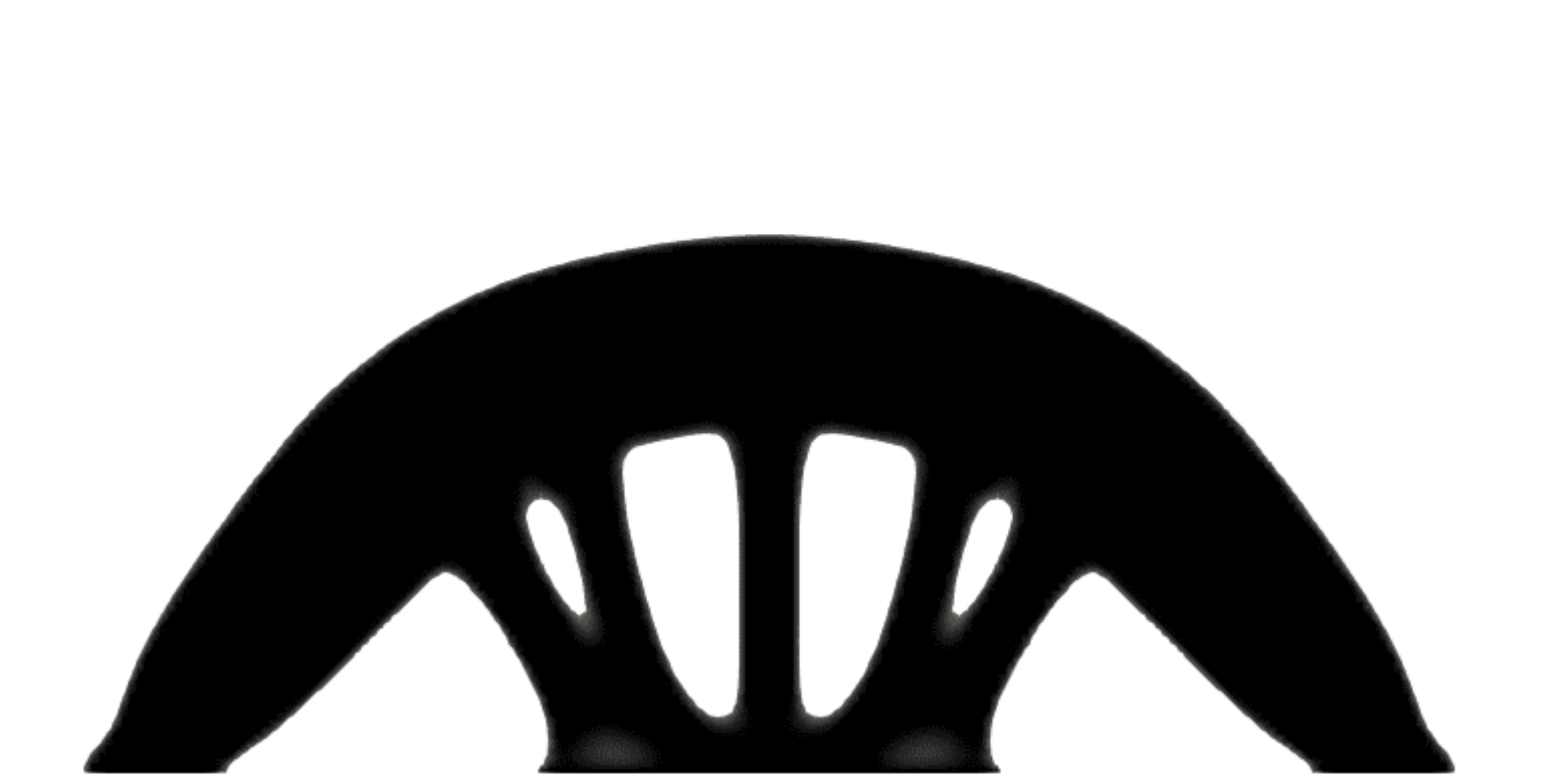}
        \subcaption{Step\,90}
        \label{bri-d}
      \end{minipage}
           \begin{minipage}[t]{0.2\hsize}
        \centering
        \includegraphics[keepaspectratio, scale=0.09]{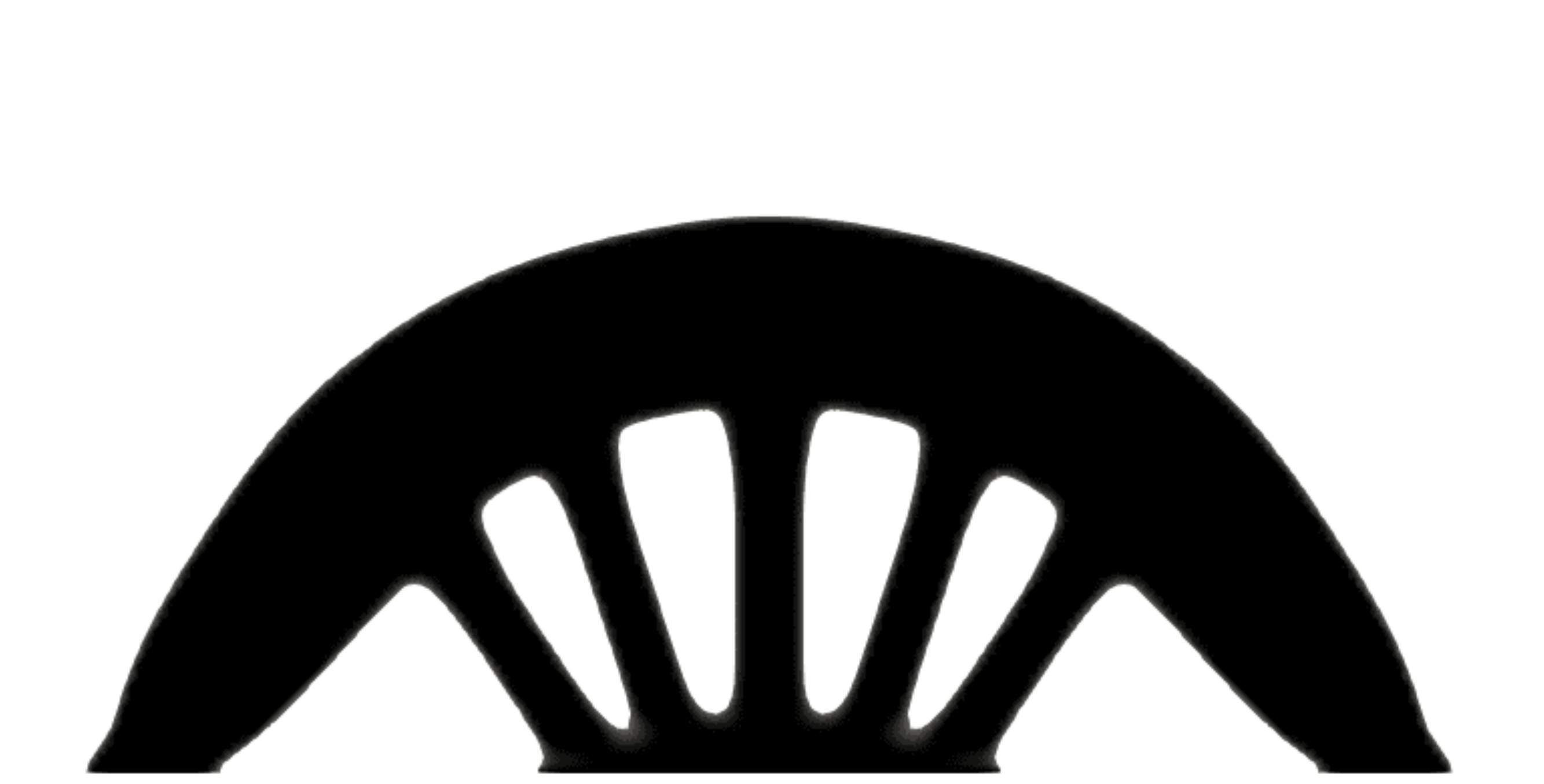}
        \subcaption{Step\,299$^{\#}$}
        \label{bri-e}
      \end{minipage}
      \\
                 \begin{minipage}[t]{0.2\hsize}
        \centering
        \includegraphics[keepaspectratio, scale=0.09]{bri0.pdf}
        \subcaption{Step\,0}
        \label{bri-f}
      \end{minipage} 
      \begin{minipage}[t]{0.2\hsize}
        \centering
        \includegraphics[keepaspectratio, scale=0.09]{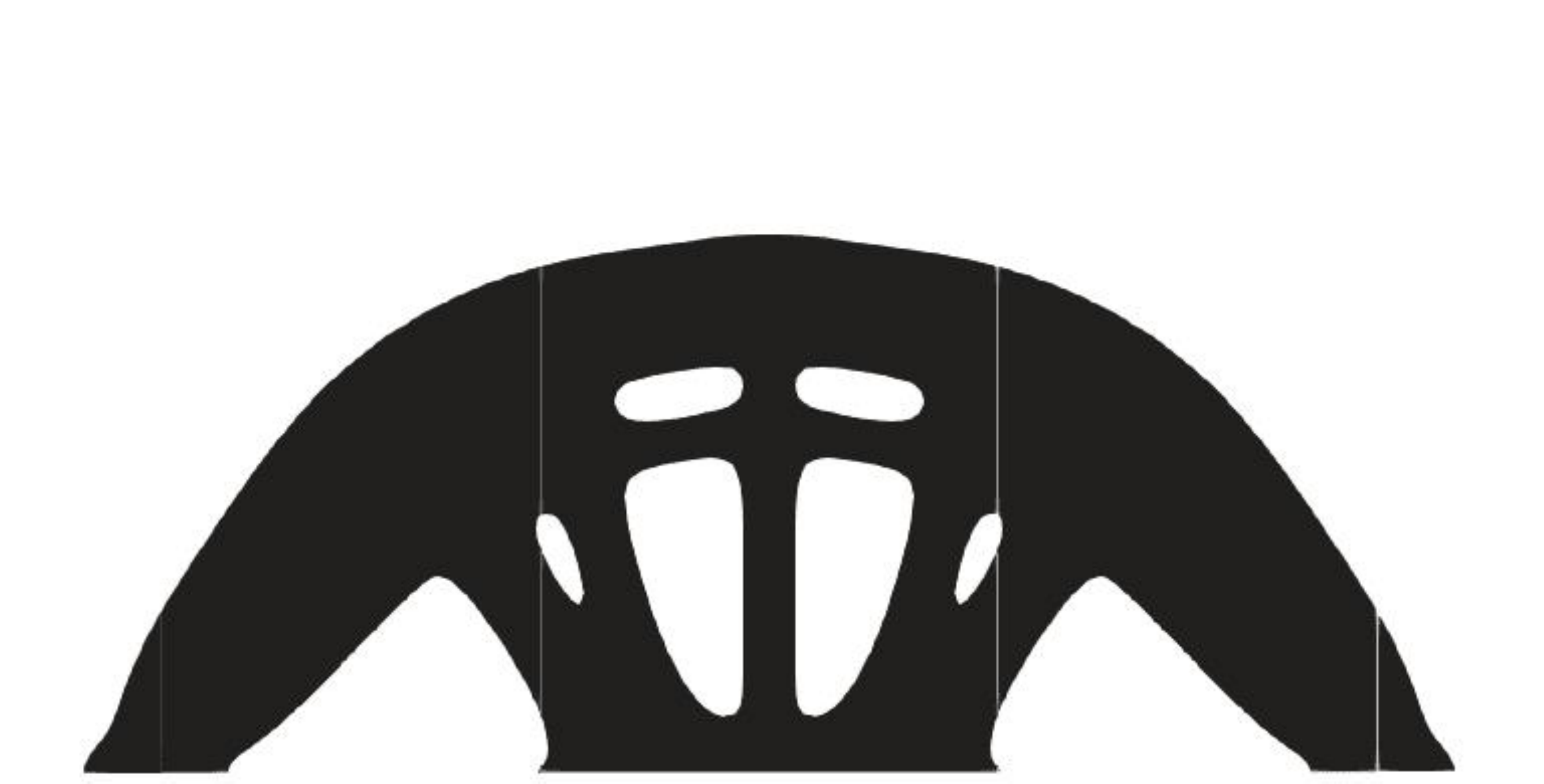}
        \subcaption{Step\,30}
        \label{bri-g}
      \end{minipage} 
         \begin{minipage}[t]{0.2\hsize}
        \centering
        \includegraphics[keepaspectratio, scale=0.09]{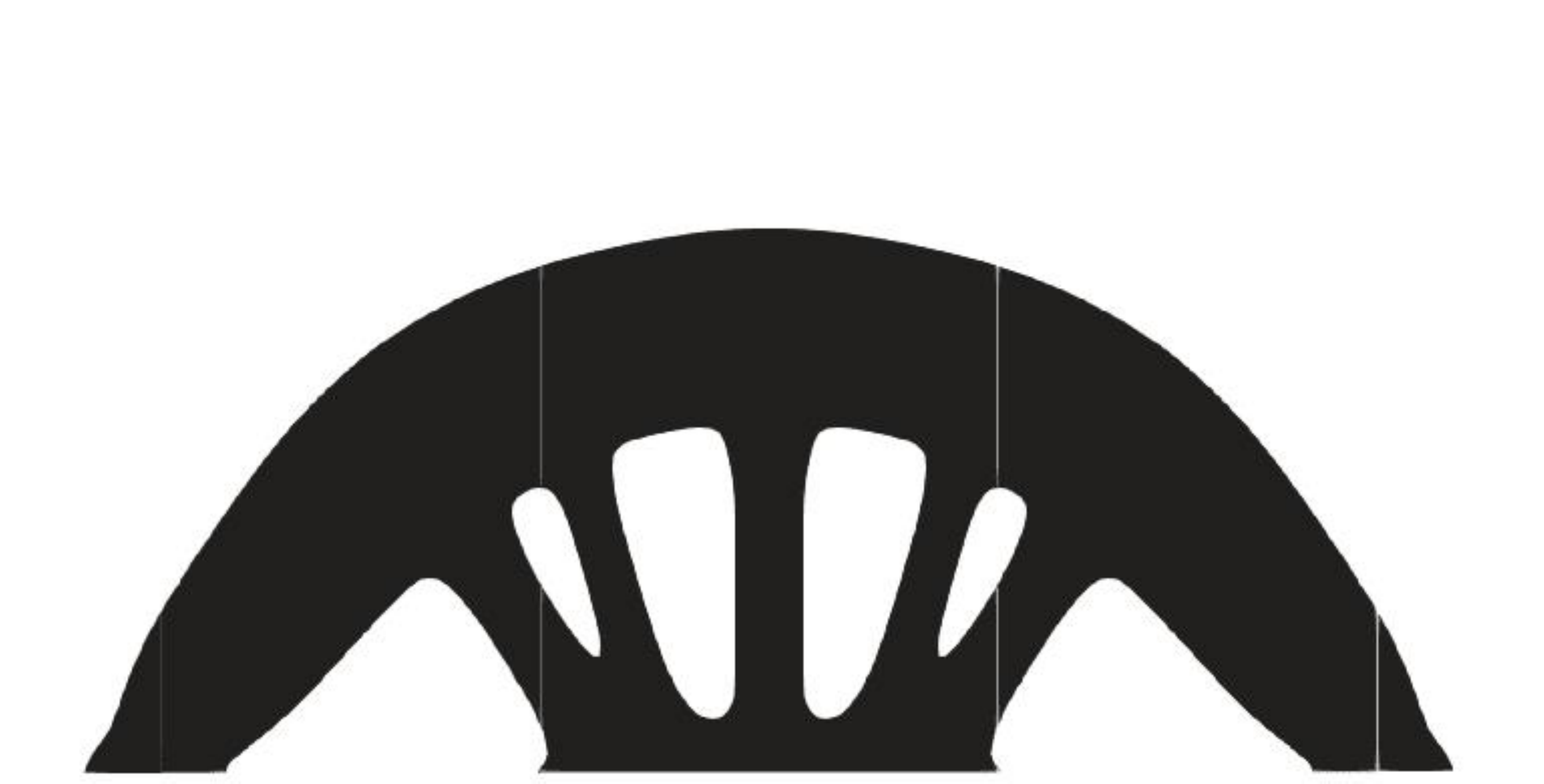}
        \subcaption{Step\,60}
        \label{bri-h}
      \end{minipage}
      \begin{minipage}[t]{0.2\hsize}
        \centering
        \includegraphics[keepaspectratio, scale=0.09]{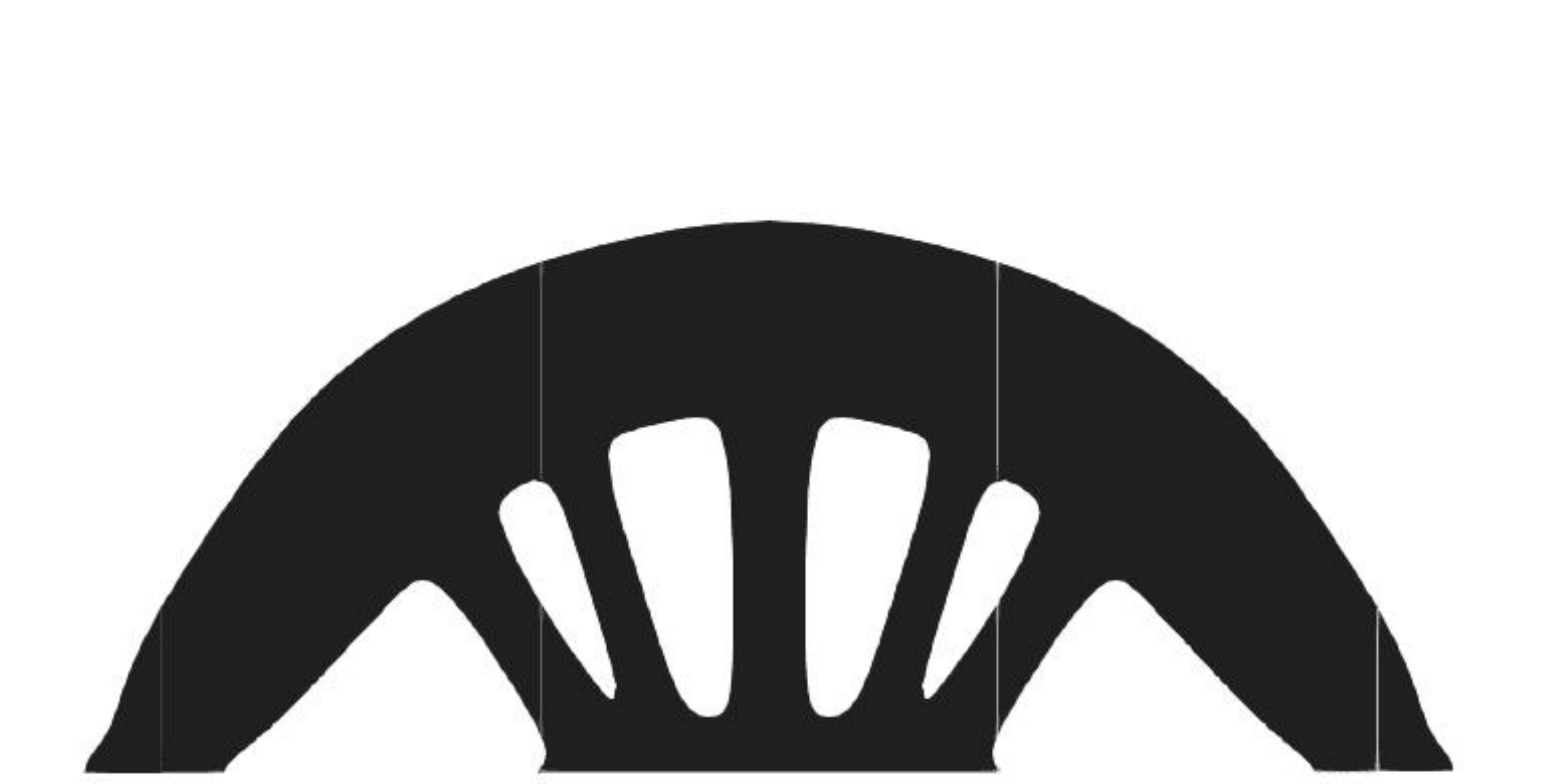}
        \subcaption{Step\,90}
        \label{bri-i}
      \end{minipage}
           \begin{minipage}[t]{0.2\hsize}
        \centering
        \includegraphics[keepaspectratio, scale=0.09]{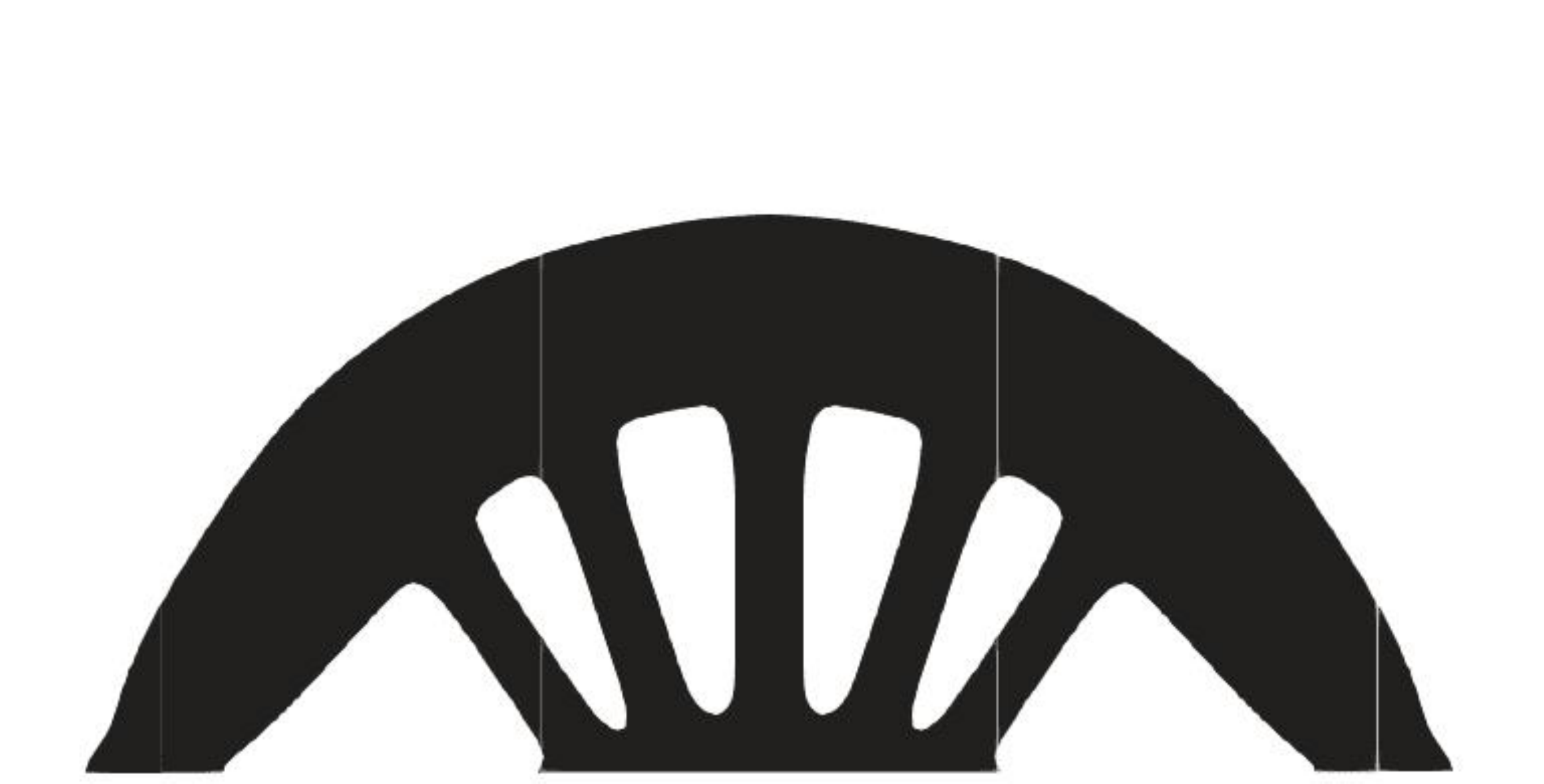}
        \subcaption{Step\,220$^{\#}$}
        \label{bri-j}
      \end{minipage}
\\
             \begin{minipage}[t]{0.2\hsize}
        \centering
        \includegraphics[keepaspectratio, scale=0.09]{bri0.pdf}
        \subcaption{Step\,0}
        \label{bri-k}
      \end{minipage} 
      \begin{minipage}[t]{0.2\hsize}
        \centering
        \includegraphics[keepaspectratio, scale=0.09]{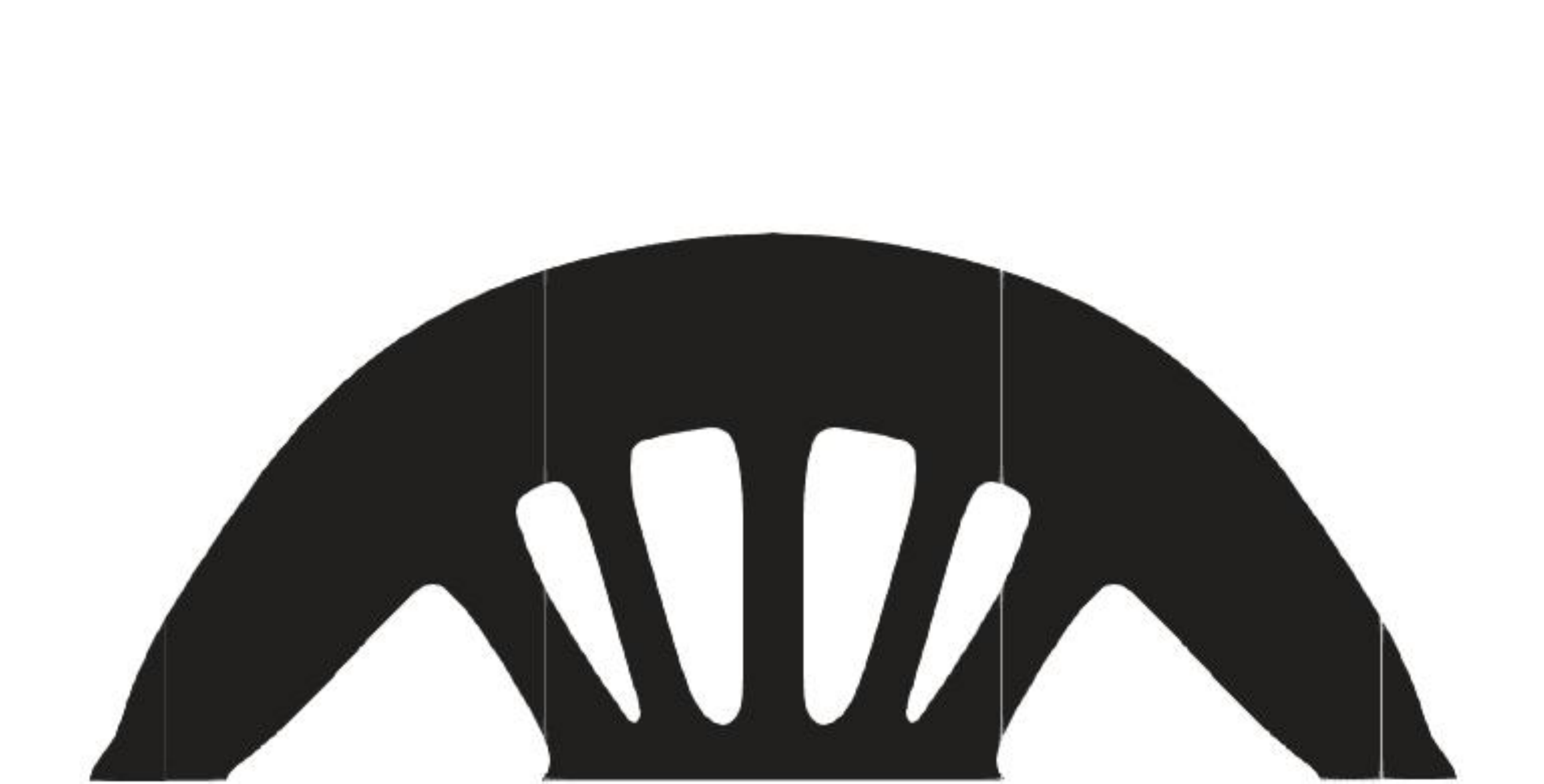}
        \subcaption{Step\,30}
        \label{bri-l}
      \end{minipage} 
         \begin{minipage}[t]{0.2\hsize}
        \centering
        \includegraphics[keepaspectratio, scale=0.09]{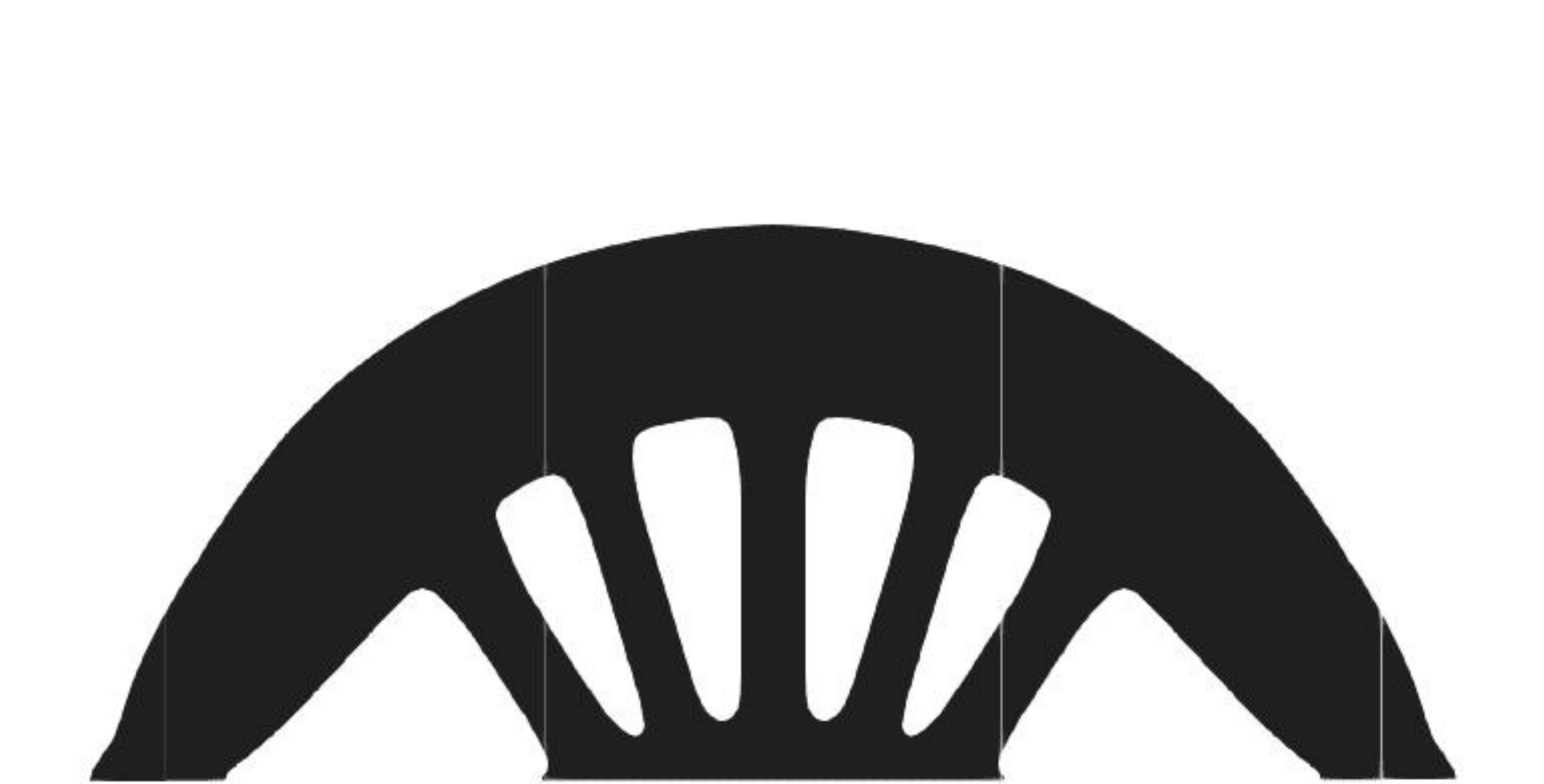}
        \subcaption{Step\,60}
        \label{bri-m}
      \end{minipage}
      \begin{minipage}[t]{0.2\hsize}
        \centering
        \includegraphics[keepaspectratio, scale=0.09]{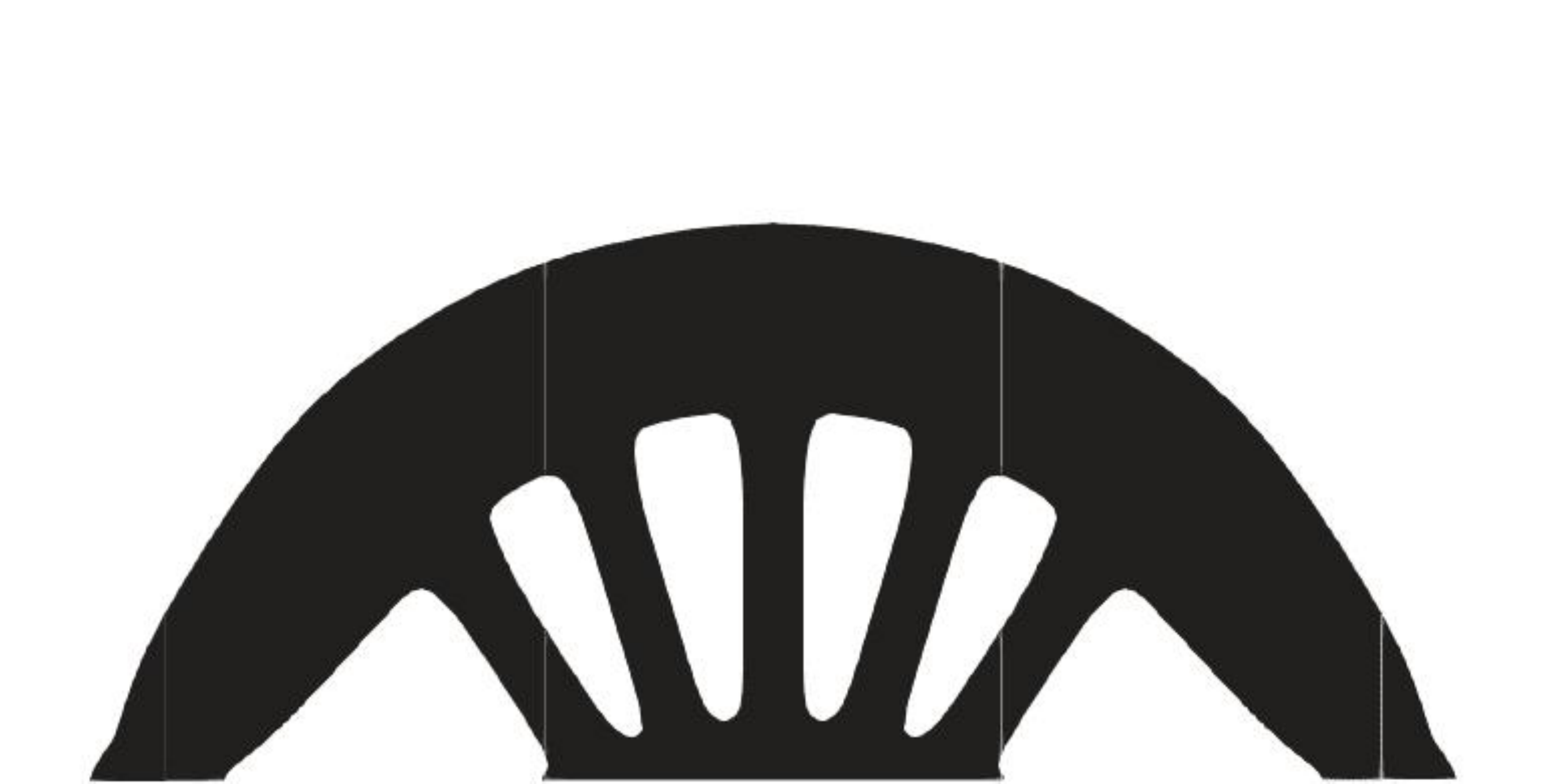}
        \subcaption{Step\,90}
        \label{bri-n}
      \end{minipage}
           \begin{minipage}[t]{0.2\hsize}
        \centering
        \includegraphics[keepaspectratio, scale=0.09]{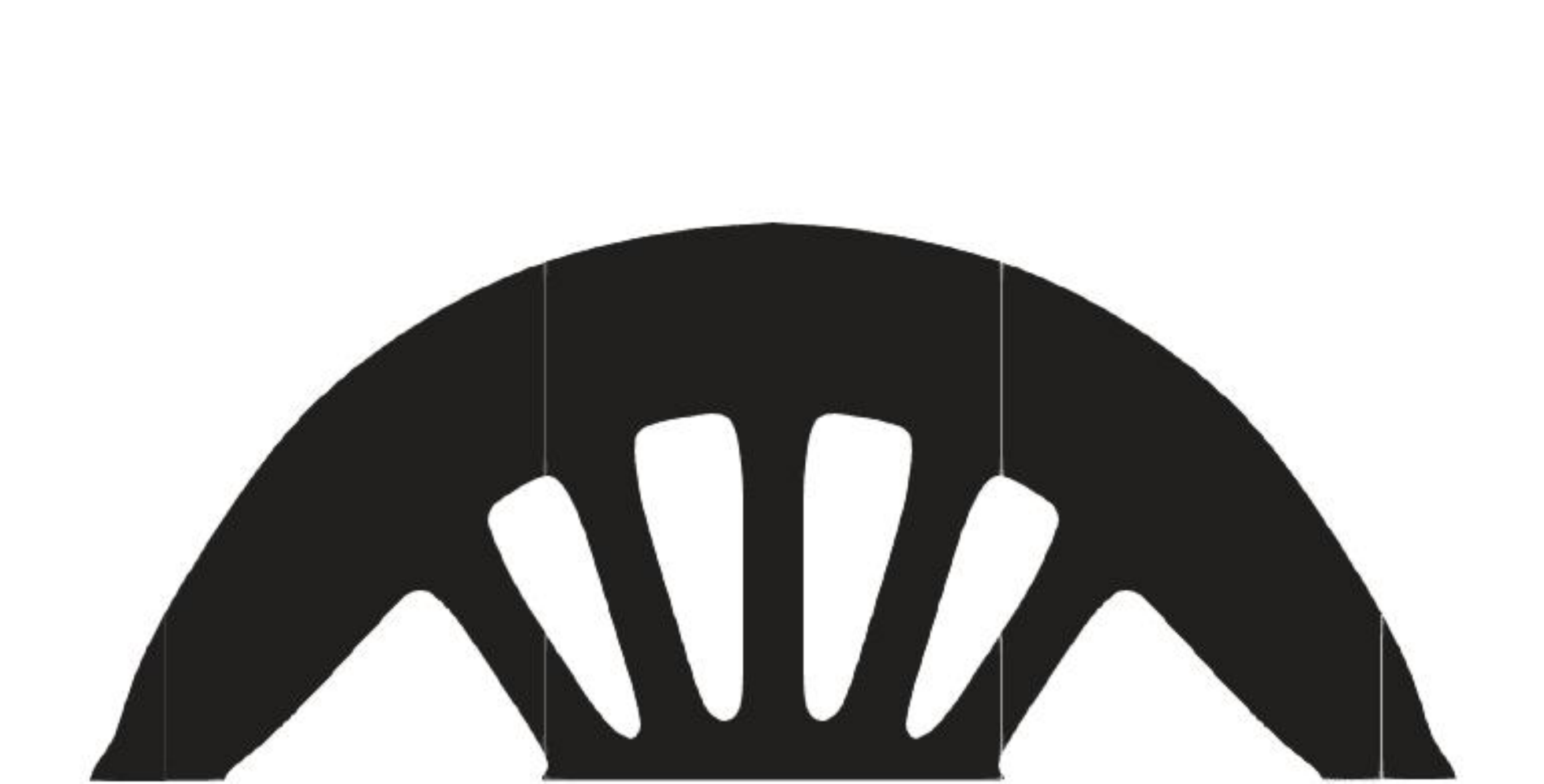}
        \subcaption{Step\,129$^{\#}$}
        \label{bri-o}
      \end{minipage}
    \end{tabular}
     \caption{Configuration $\Omega_{\phi_n}\subset D$ for the case where the initial configuration is the perforated domain. 
Figures (a)--(e), (f)--(j) and (k)--(o) 
represent $\Omega_{\phi_n}\subset D$ for $q=1$ , $q=1$ with $\varDelta t=0.9$ and $q=8$ in \eqref{discNLD}, respectively. The symbol ${}^{\#}$ implies the final step.    
     }
     \label{fig:bri}
  \end{figure}

\subsection{Optimal design problem for compliant mechanism} \label{SS:cm}
We next show that the proposed method is also valid for non-self-adjoint types. 
Let us consider the minimization problem \eqref{eq:opt-prob} under the following Lagrangian\/{\rm :}
\begin{align*}
\mathcal{L}(\phi,\lambda)=F(\phi)+\lambda G(\phi)
=
-\int_{\Gamma_{\rm out}}\boldsymbol{t}\cdot \boldsymbol{u}_{\phi}(x)\, \d \sigma
+\lambda\underbrace{\left(\int_D \chi_\phi(x)\, \d x- G_{\text{max}}|D|\right)}_{\le 0}, 
\end{align*}
where $G_{\text{max}}>0$,
$\boldsymbol{u}_\phi\in V^d$ denotes the state variable, which is a unique solution to
the following system\/{\rm :} 
\begin{align}
\lefteqn{\int_{D} \mathbb{D}\chi_\phi(x)\boldsymbol{\e} (\boldsymbol{u}_\phi)(x)\colon 
\boldsymbol{\e} (\boldsymbol{v})(x)\, \d  x}\nonumber\\
&\qquad=
\int_{\Gamma_{\rm in}} (\boldsymbol{t}+K_1 \boldsymbol{u}_\phi(x))\cdot \boldsymbol{v}(x)\, \d  \sigma
+
\int_{\Gamma_{\rm out}} K_2 \boldsymbol{u}_\phi(x)\cdot \boldsymbol{v}(x)\, \d  \sigma
\label{eq:state-cm}
\end{align}
for all $\boldsymbol{v}\in V^d$.
Here $K_{i}\in L^{\infty}(D;\R^{d\times d})$ ($i=1,2$). 
In particular, we set $\boldsymbol{t}=(1,0)$ on $\Gamma_{\rm in}$ and $\boldsymbol{t}=(0,\pm1)$ on $\Gamma_{\rm out}$ (see Figure \ref{icm} and \cite[Fig.~13]{O97} for boundary conditions).

The procedure of numerical analysis for the above minimization problem is similar to the previous subsection.
However, we note that one can not readily apply to Proposition \ref{prop}.
Here we employ the following
\begin{lem}[Sensitivity analysis for compliant mechanism]\label{lem}
Let $\boldsymbol{u}_\phi\in V^d$ be a unique solution to the state system \eqref{eq:state-cm} and let $\boldsymbol{\tilde{u}}_\phi\in V^d$ be 
a  unique solution to the following adjoint system\/{\rm :}
\begin{align*}
\lefteqn{\int_{D} \mathbb{D}\chi_\phi(x)\boldsymbol{\e} (\boldsymbol{\tilde{u}}_\phi)(x)\colon 
\boldsymbol{\e} (\boldsymbol{v})(x)\, \d  x}\nonumber\\
&\qquad=
\int_{\Gamma_{\rm in}} K_1 \boldsymbol{\tilde{u}}_\phi(x)\cdot \boldsymbol{v}(x)\, \d  \sigma
+
\int_{\Gamma_{\rm out}} (-\boldsymbol{t}+K_2 \boldsymbol{\tilde{u}}_\phi(x))\cdot \boldsymbol{v}(x)\, \d  \sigma 
\end{align*}
for all  $\boldsymbol{v}\in V^d$.
Then $F'(\phi)$ can be approximated by
\begin{align}
F_\eta'(\phi)=-\mathbb{D}\delta_\eta(\phi)\boldsymbol{\e} (\boldsymbol{u}_\phi)\colon\boldsymbol{\e} (\tilde{\boldsymbol{u}}_\phi)
\label{eq:appro-s}
\end{align}
for $\eta>0$ small enough. 
\end{lem}
\begin{proof}
We prove \eqref{eq:appro-s} by employing the adjoint method. Define $\tilde{\mathcal{L}}:H^1(D)\times V^d\times V^d\to\R$ by 
\begin{align*}
\tilde{\mathcal{L}}(\phi,\boldsymbol{u}_\phi, \boldsymbol{v})
&=F(\phi)+
\int_{D} \mathbb{D}\chi_\phi(x)\boldsymbol{\e} (\boldsymbol{u}_\phi)(x)\colon 
\boldsymbol{\e} (\boldsymbol{v})(x)\, \d  x\\
&\qquad -
\int_{\Gamma_{\rm in}} (\boldsymbol{t}+K_1 \boldsymbol{u}_\phi(x))\cdot \boldsymbol{v}(x)\, \d  \sigma
-
\int_{\Gamma_{\rm out}} K_2 \boldsymbol{u}(x)\cdot \boldsymbol{v}(x)\, \d  \sigma.
\end{align*}
Then we have 
$F(\phi)=\tilde{\mathcal{L}}(\phi,\boldsymbol{u}_\phi, \boldsymbol{v})
$ and $F'(\phi)=\nabla_\phi \tilde{\mathcal{L}}(\phi,\boldsymbol{u}_\phi,-\boldsymbol{\tilde{u}}_\phi)$
by noting that $\nabla_{\boldsymbol{u}_\phi}\tilde{\mathcal{L}}(\phi,\boldsymbol{u}_\phi,-\boldsymbol{\tilde{u}}_\phi)=0 $.
Hence, we formally deduce that
\begin{align*}
\langle F'(\phi), w\rangle_{H^1(D;[-1,1])}
=-\lim_{\eta\to 0_+}\int_{D}\mathbb{D}\frac{\chi_{\phi+\eta w}-\chi_{\phi}}{\eta}(x)\boldsymbol{\e} (\boldsymbol{u}_\phi)(x)\colon\boldsymbol{\e} (\boldsymbol{\tilde{u}}_\phi)(x)\, \d x
\end{align*}
for all $w\in H^1(D;[-1,1])$, which completes the proof.
\end{proof}

\begin{rmk}[Approximated sensitivity without restrictions for initial LSFs and domains]\label{R:GR}
\rm
We note that the assumption for initial LSFs in Proposition \ref{prop} is not required in Lemma \ref{lem}. The same argument mentioned above is also effective for the sensitivity in \S \ref{SS:ca} as $\boldsymbol{\tilde{u}}_\phi=\boldsymbol{u}_\phi$ (i.e.,~self-adjoint type). 
\end{rmk}

The numerical results are shown in Figure \ref{fig:cm}. Here we set  	
$(\tau,G_{\rm max},\varDelta t)=(1.5\times 10^{-4},0.4,0.2)$ in Figures \ref{cm-a}--\ref{cm-e} and \ref{cm-k}--\ref{cm-o}.
Then the proposed method optimizes the topology faster than the method using reaction-diffusion and consequently presents faster convergence. 
Here we switched from $q=3$ to $q=1$ after $50$ steps to avoid oscillation on boundary structures in Figure \ref{cm-a}--\ref{cm-e}. Thus (i-FDE) also holds for the non-self-adjoint type (see also Figure \ref{fig:3dcm} for the corresponding three-dimensional case).  

As mentioned in \S \ref{SS:bri}, increasing (suitably) $\varDelta t>0$ improves convergence. 
Therefore, we next change $\varDelta t>0$ from $0.2$ to $0.5$ for $q=1$. 
However, Figure \ref{fig:cmconv} shows that the configuration $\Omega_{\phi_n}\subset D$ does not converge due to the oscillation on boundary structures; indeed, 
it is well known that the gradient descent method occurs in oscillation and does not converge whenever the step width is (too) large. 
In particular, since $\|\phi_{n+1}-\phi_n\|_{L^{\infty}(D)}$ coincides with $\|\phi_{n+1}-\phi_n\|_{L^{\infty}(A)}$ in almost steps after $50$ steps (see Figure \ref{fig:cmconv}), we deduce that the oscillation on boundary structures is deeply related to the factor.

On the other hand, choosing $q=0.5$ (i.e.,~slow diffusion) under this setting,  
we see by Figures \ref{cm-f}--\ref{cm-j} and \ref{fig:cmconv} that $\Omega_{\phi_n}$ converges, which implies that the proposed method yields not only fast convergence but also damping boundary oscillation. This completes the check for (ii-SDE), and hence, one can deduce from
the above result that the proposed method is more applicable than the method using reaction-diffusion. 

\begin{figure}[htbp]
   \hspace*{-5mm} 
    \begin{tabular}{ccccc}
    
        \begin{minipage}[t]{0.2\hsize}
        \centering
        \includegraphics[keepaspectratio, scale=0.09]{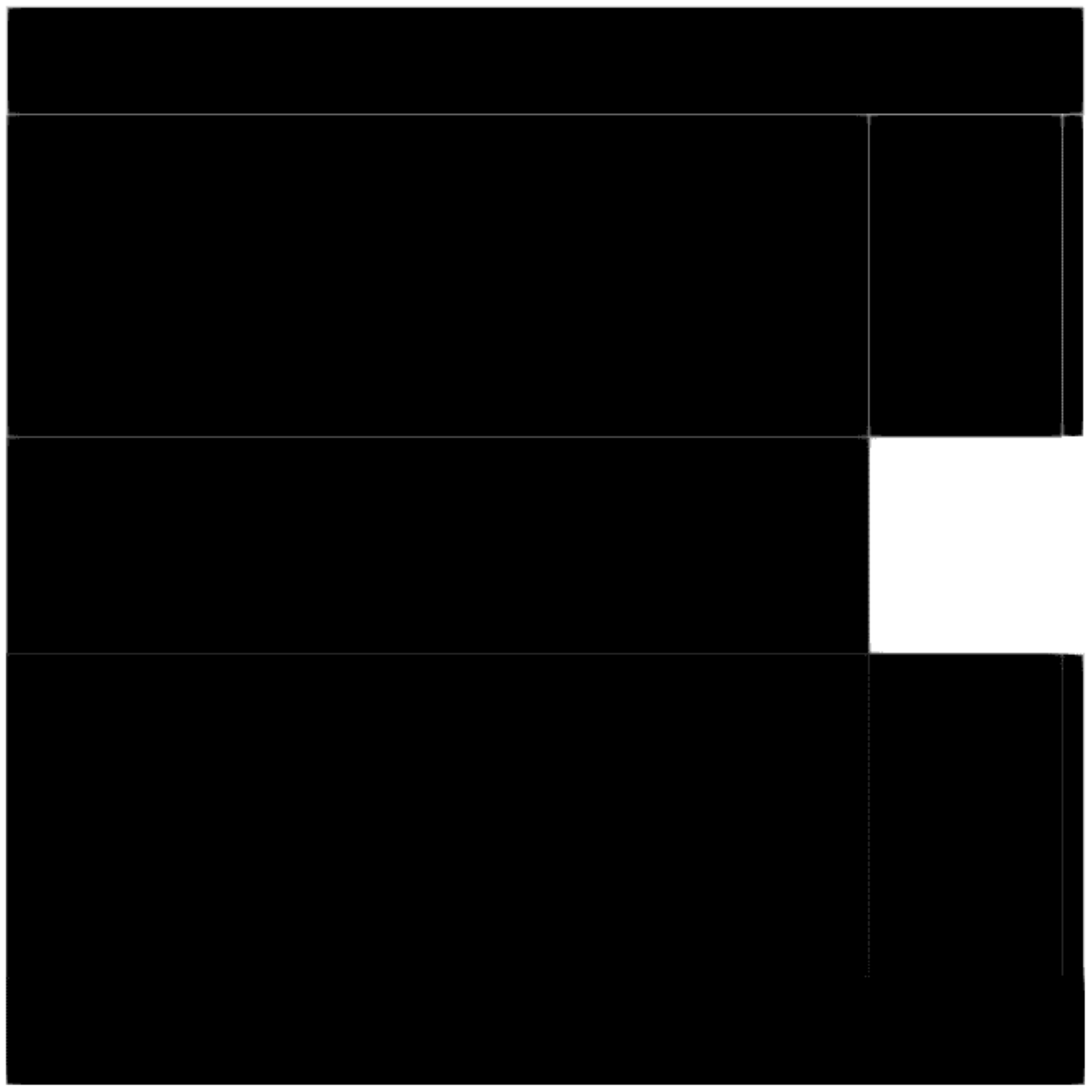}
        \subcaption{Step\,0}
        \label{cm-a}
      \end{minipage} 
      \begin{minipage}[t]{0.2\hsize}
        \centering
        \includegraphics[keepaspectratio, scale=0.09]{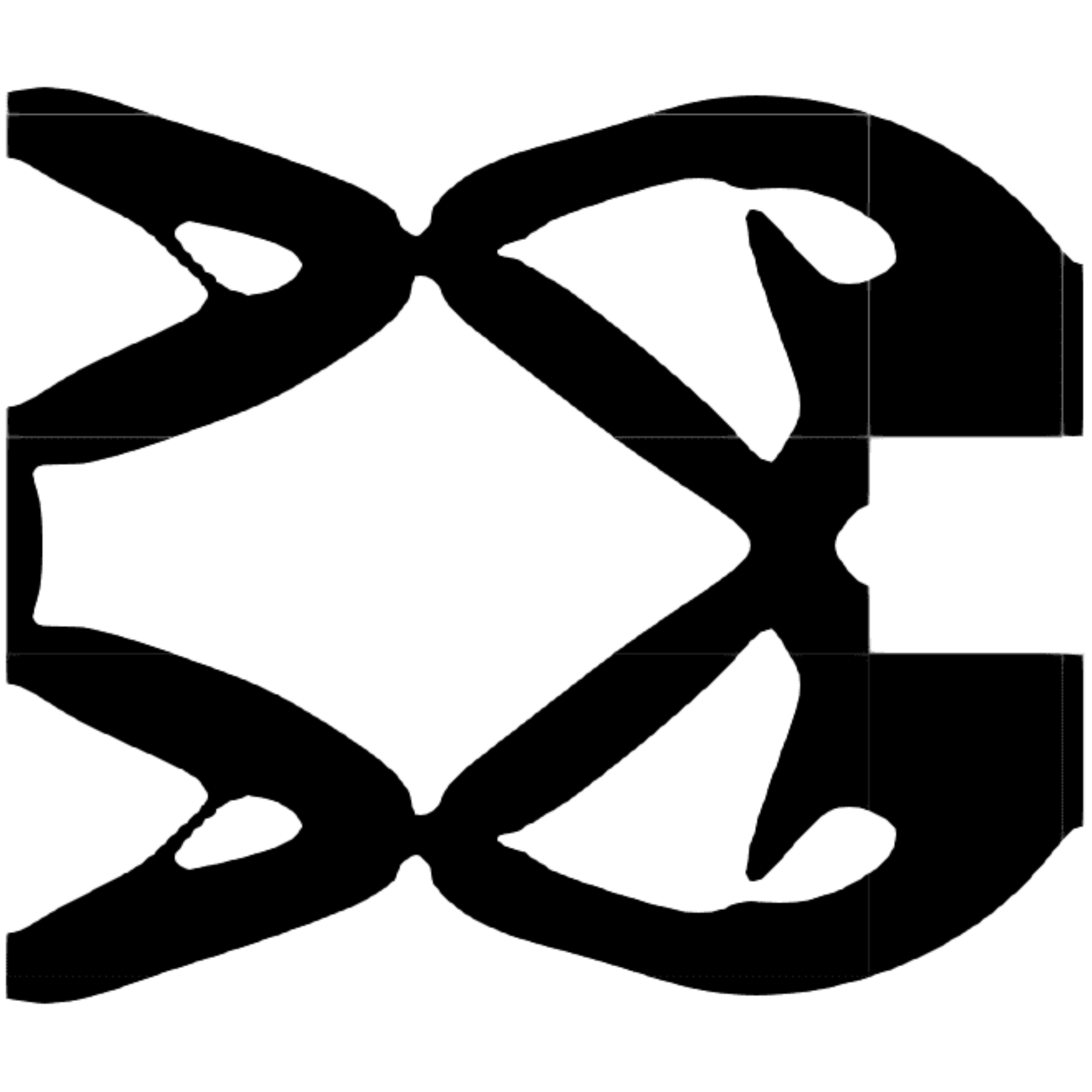}
        \subcaption{Step\,30}
        \label{cm-b}
      \end{minipage} 
         \begin{minipage}[t]{0.2\hsize}
        \centering
        \includegraphics[keepaspectratio, scale=0.09]{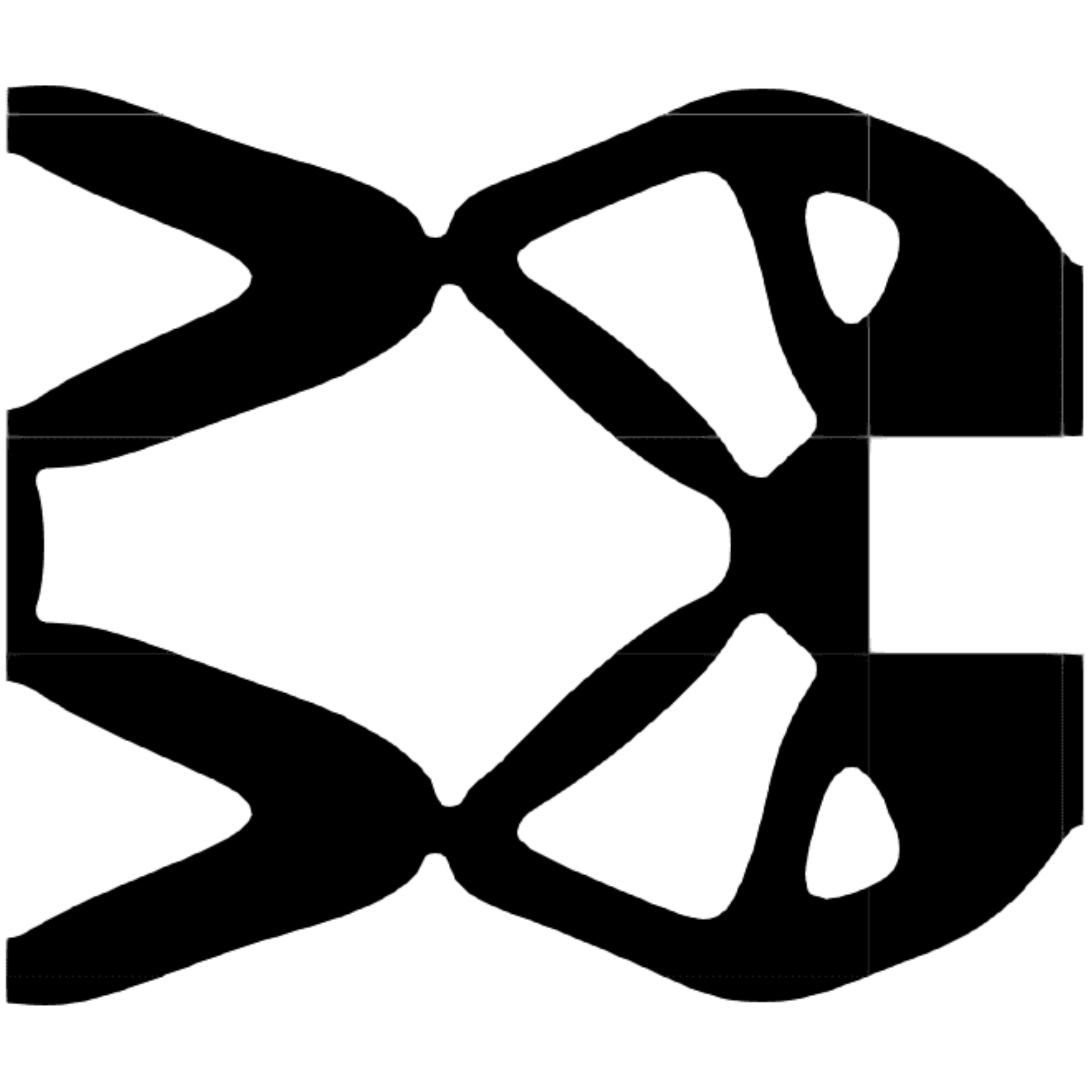}
        \subcaption{Step\,60}
        \label{cm-c}
      \end{minipage}
      \begin{minipage}[t]{0.2\hsize}
        \centering
        \includegraphics[keepaspectratio, scale=0.09]{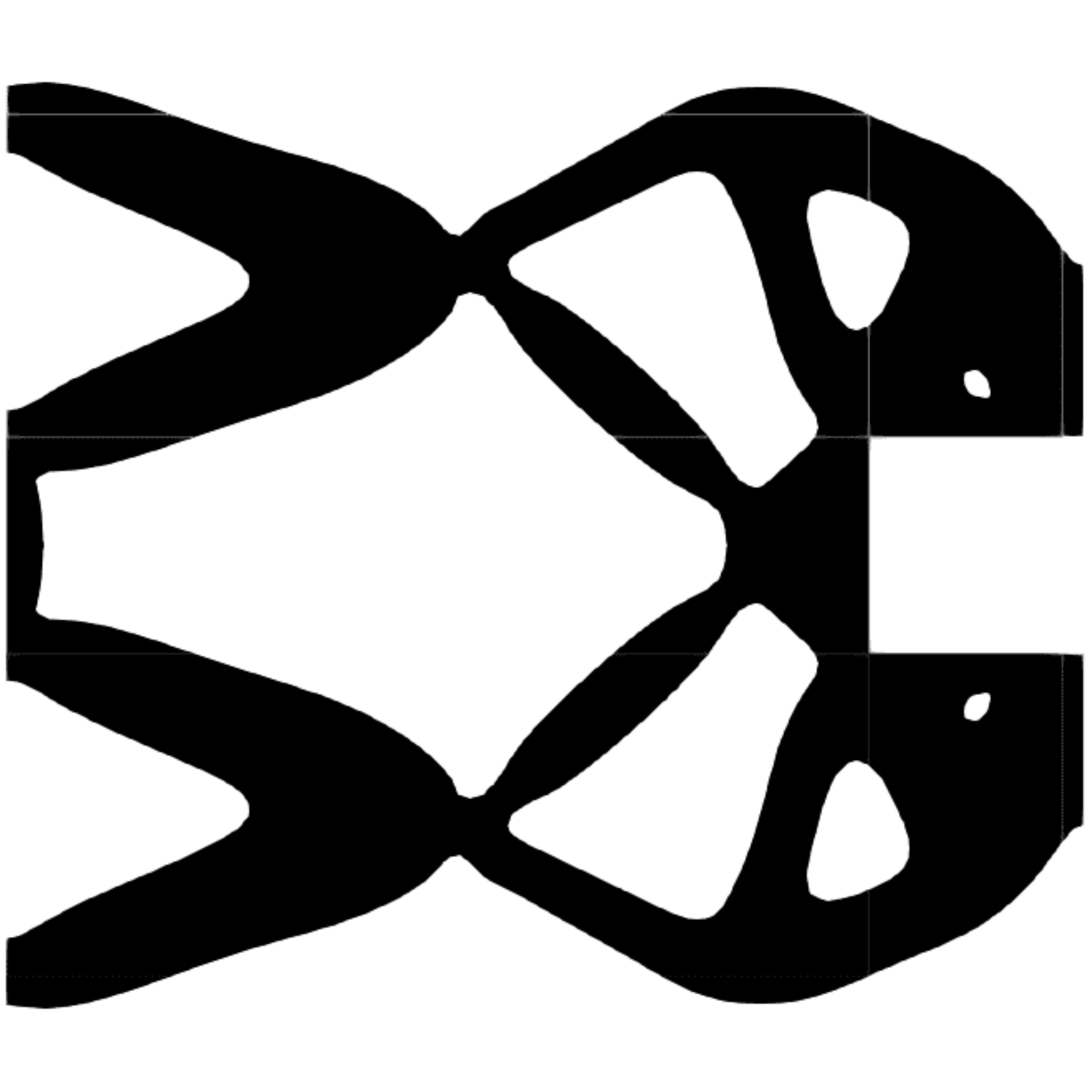}
        \subcaption{Step\,90}
        \label{cm-d}
      \end{minipage}
           \begin{minipage}[t]{0.2\hsize}
        \centering
        \includegraphics[keepaspectratio, scale=0.09]{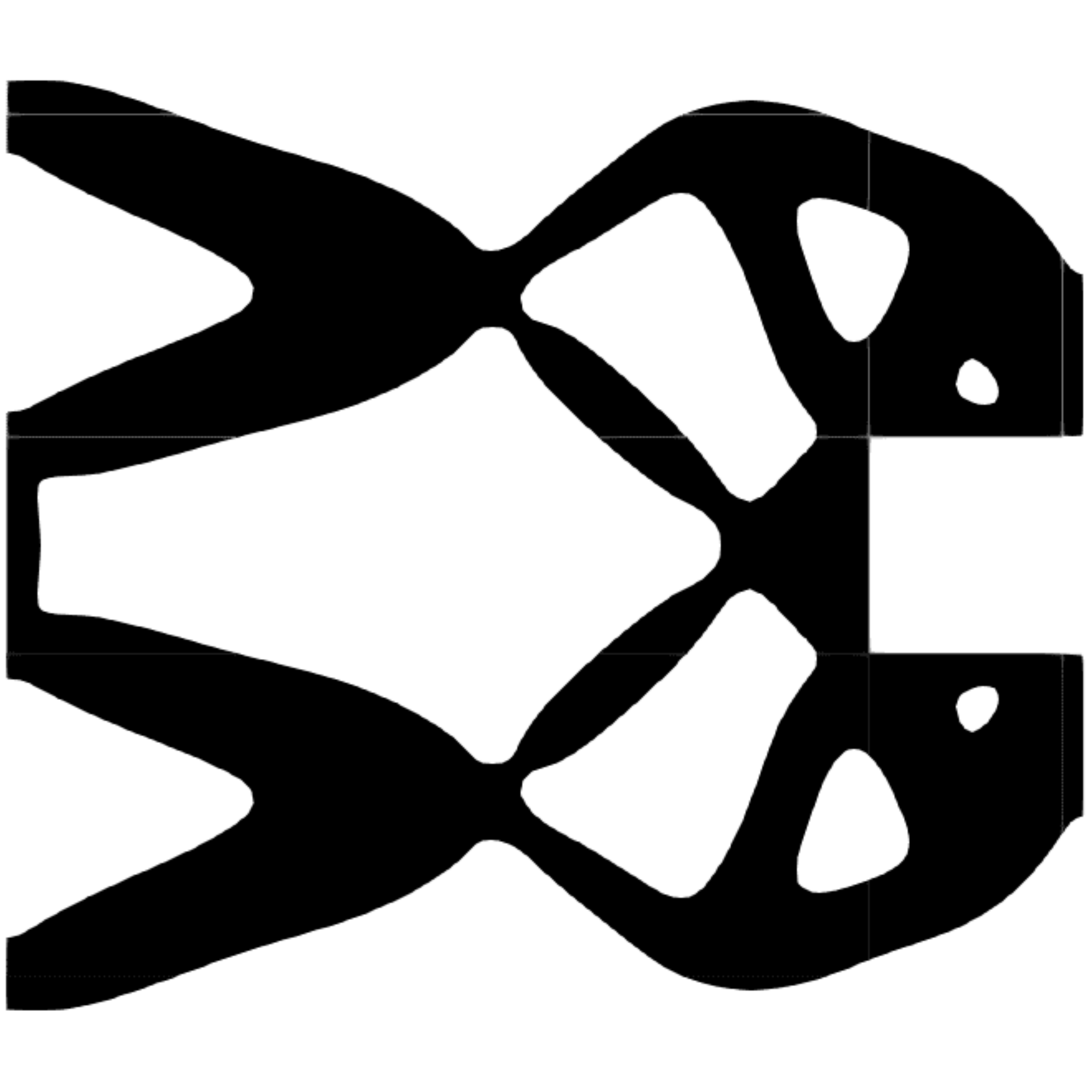}
        \subcaption{Step\,326$^{\#}$}
        \label{cm-e}
      \end{minipage}
      \\
      \begin{minipage}[t]{0.2\hsize}
        \centering
        \includegraphics[keepaspectratio, scale=0.09]{cm0.pdf}
        \subcaption{Step\,0}
        \label{cm-f}
      \end{minipage} 
      \begin{minipage}[t]{0.2\hsize}
        \centering
        \includegraphics[keepaspectratio, scale=0.09]{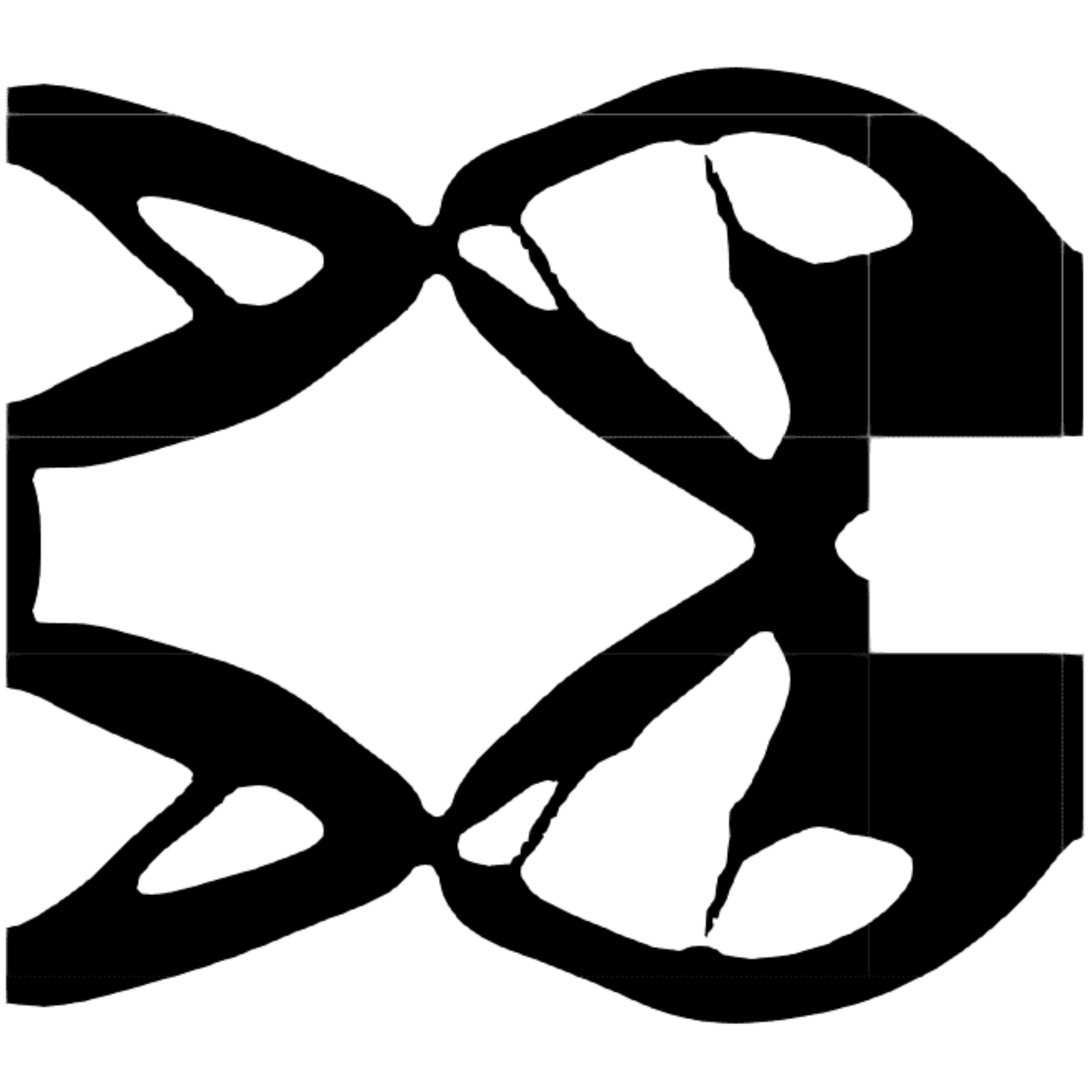}
        \subcaption{Step\,30}
        \label{cm-g}
      \end{minipage} 
         \begin{minipage}[t]{0.2\hsize}
        \centering
        \includegraphics[keepaspectratio, scale=0.09]{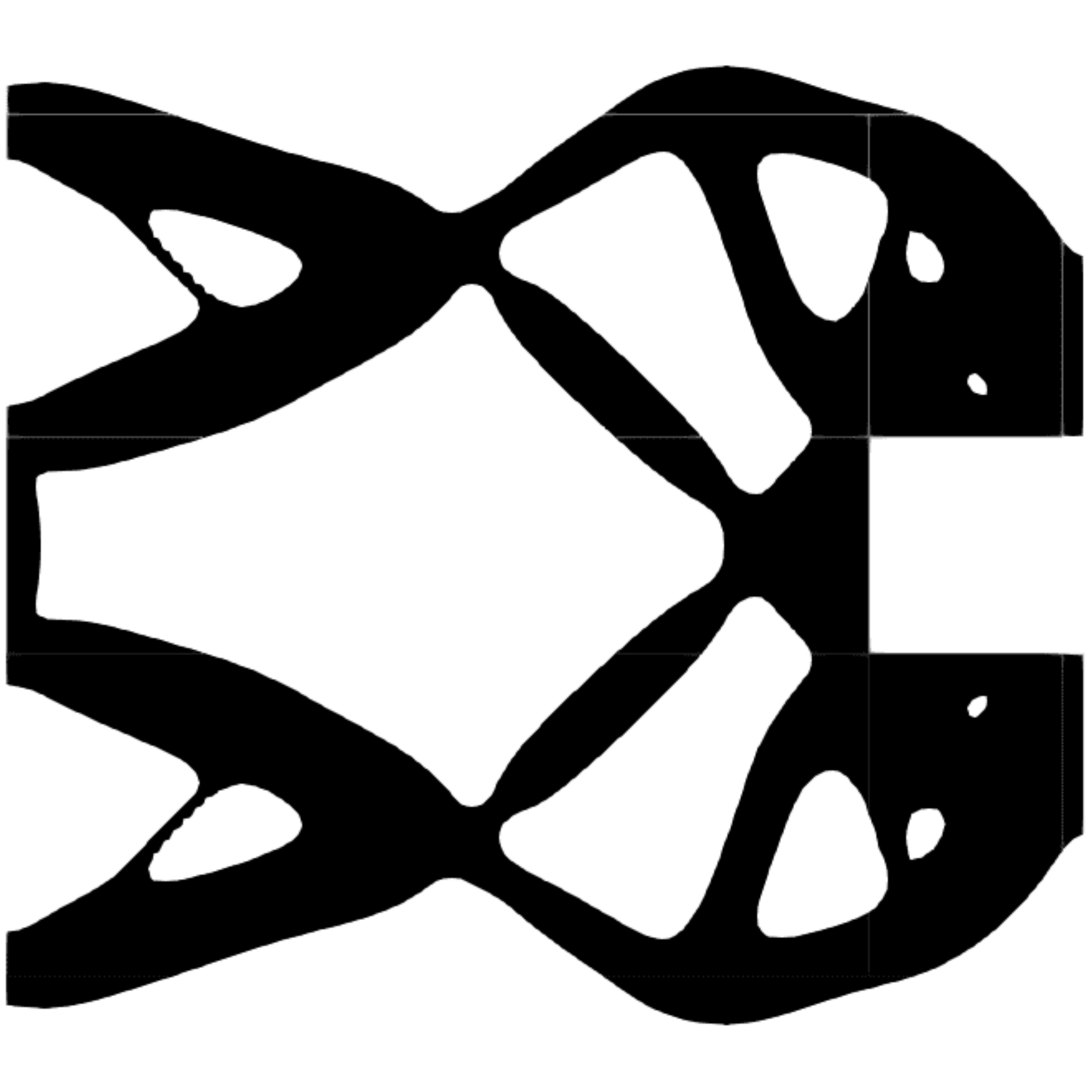}
        \subcaption{Step\,60}
        \label{cm-h}
      \end{minipage}
      \begin{minipage}[t]{0.2\hsize}
        \centering
        \includegraphics[keepaspectratio, scale=0.09]{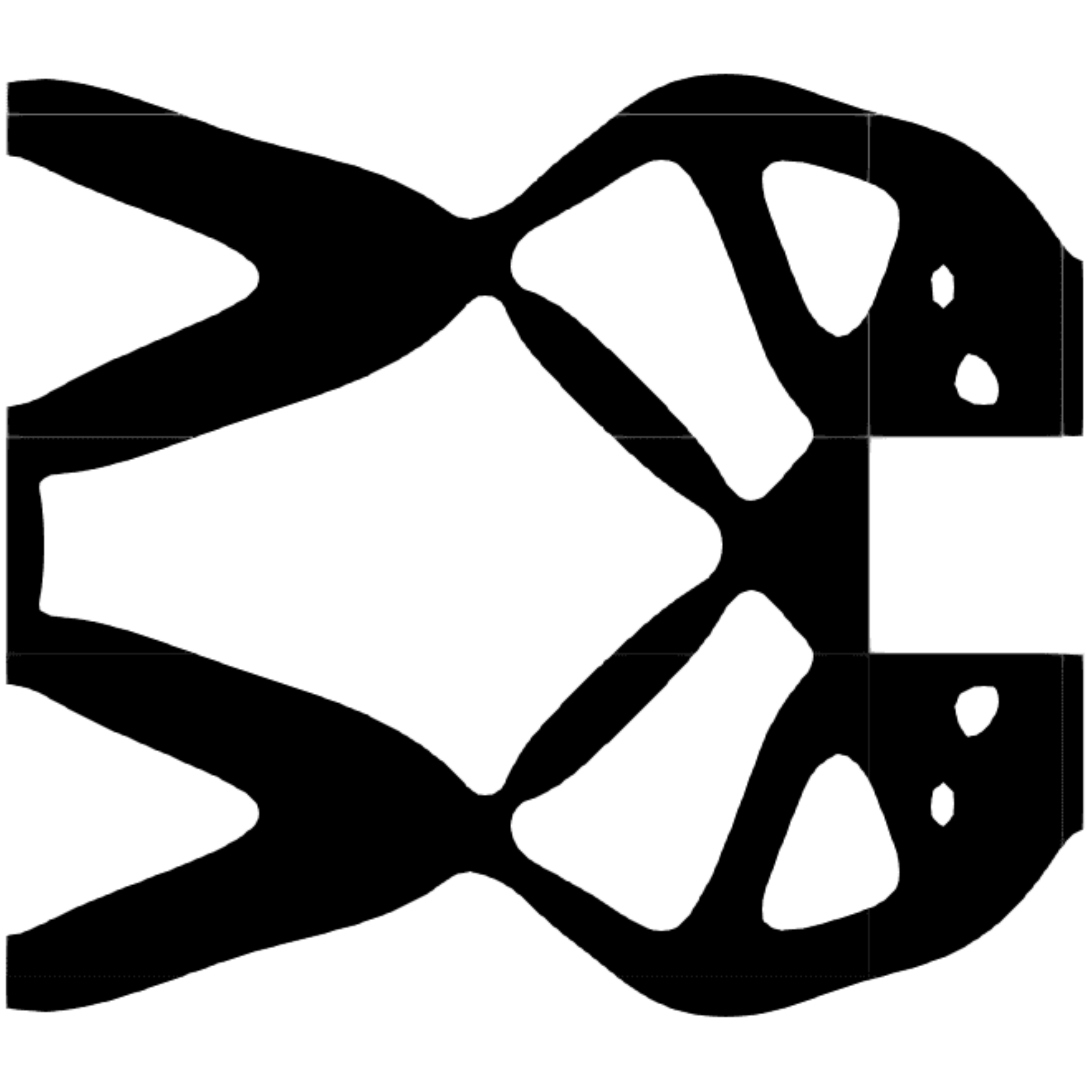}
        \subcaption{Step\,90}
        \label{cm-i}
      \end{minipage}
           \begin{minipage}[t]{0.2\hsize}
        \centering
        \includegraphics[keepaspectratio, scale=0.09]{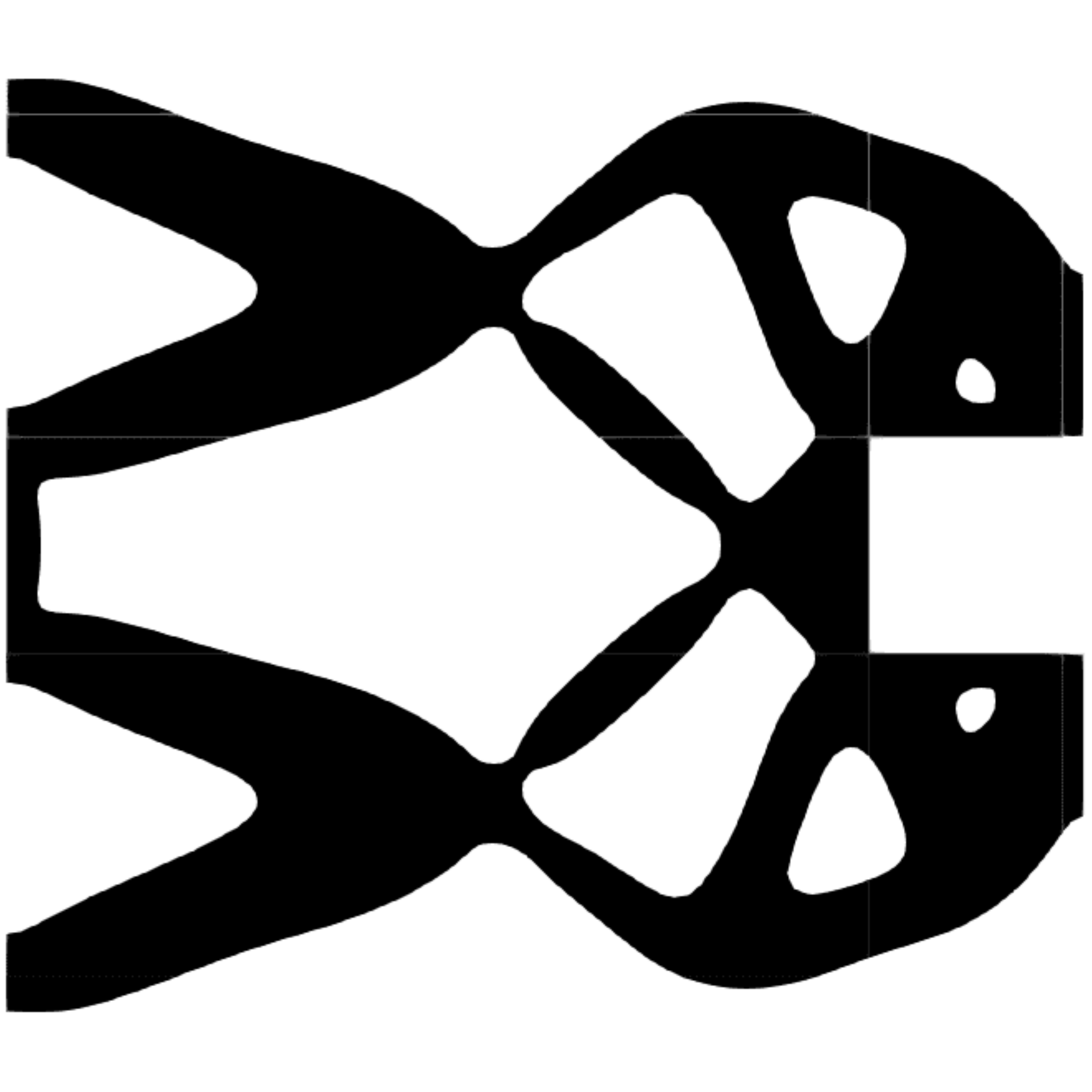}
        \subcaption{Step\,302$^{\#}$}
        \label{cm-j}
      \end{minipage}
      \\
        \begin{minipage}[t]{0.2\hsize}
        \centering
        \includegraphics[keepaspectratio, scale=0.09]{cm0.pdf}
        \subcaption{Step\,0}
        \label{cm-k}
      \end{minipage} 
      \begin{minipage}[t]{0.2\hsize}
        \centering
        \includegraphics[keepaspectratio, scale=0.09]{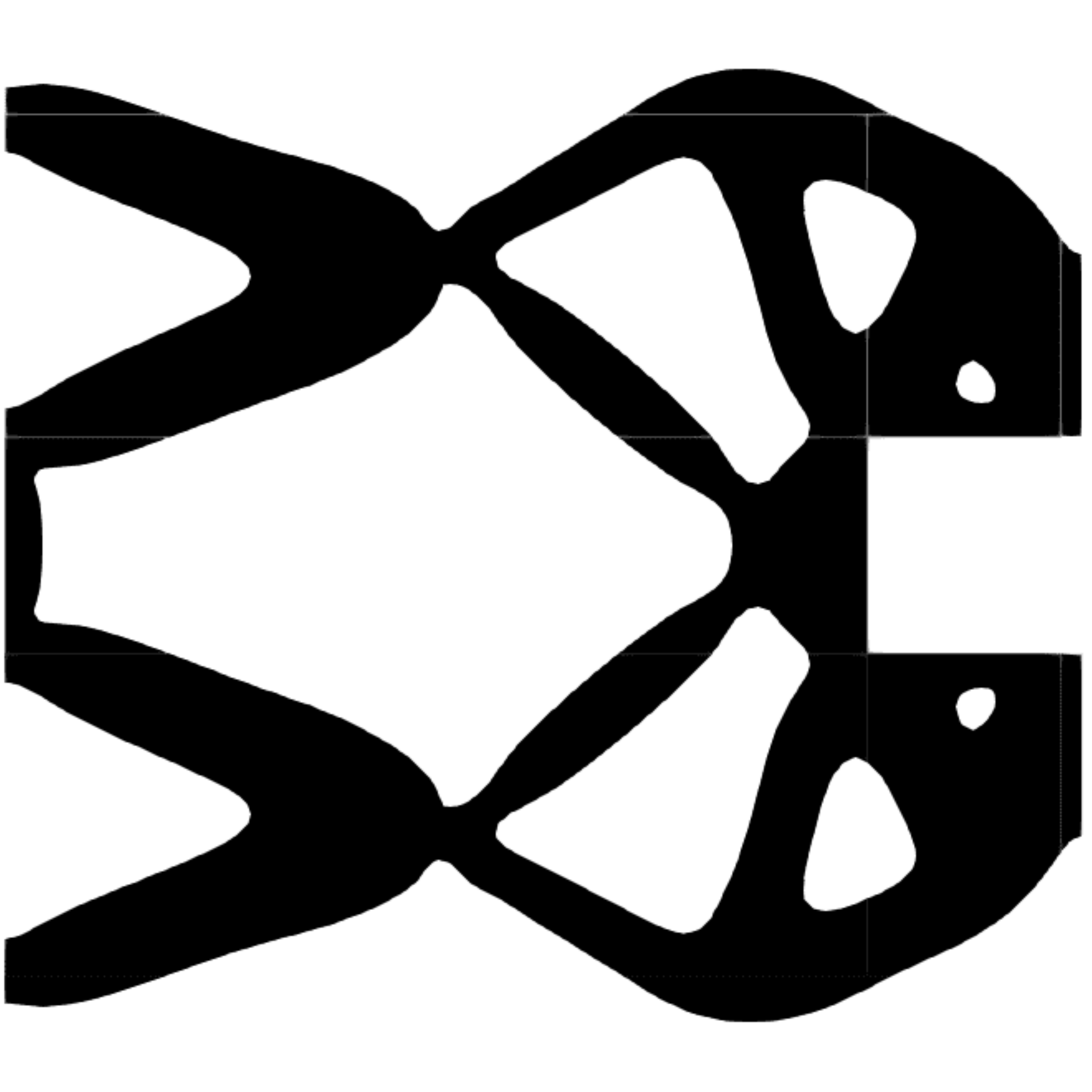}
        \subcaption{Step\,30}
        \label{cm-l}
      \end{minipage} 
         \begin{minipage}[t]{0.2\hsize}
        \centering
        \includegraphics[keepaspectratio, scale=0.09]{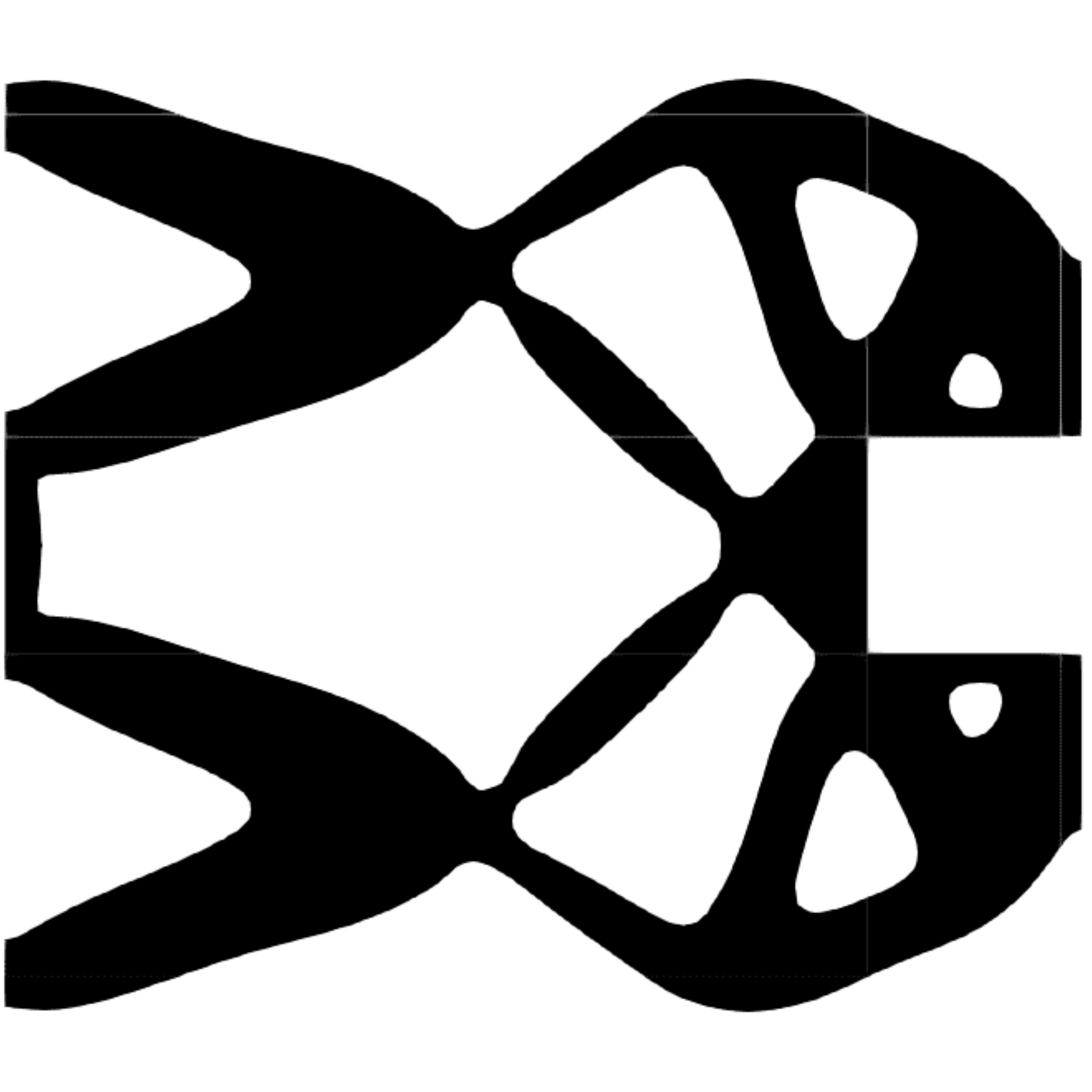}
        \subcaption{Step\,60}
        \label{cm-m}
      \end{minipage}
      \begin{minipage}[t]{0.2\hsize}
        \centering
        \includegraphics[keepaspectratio, scale=0.09]{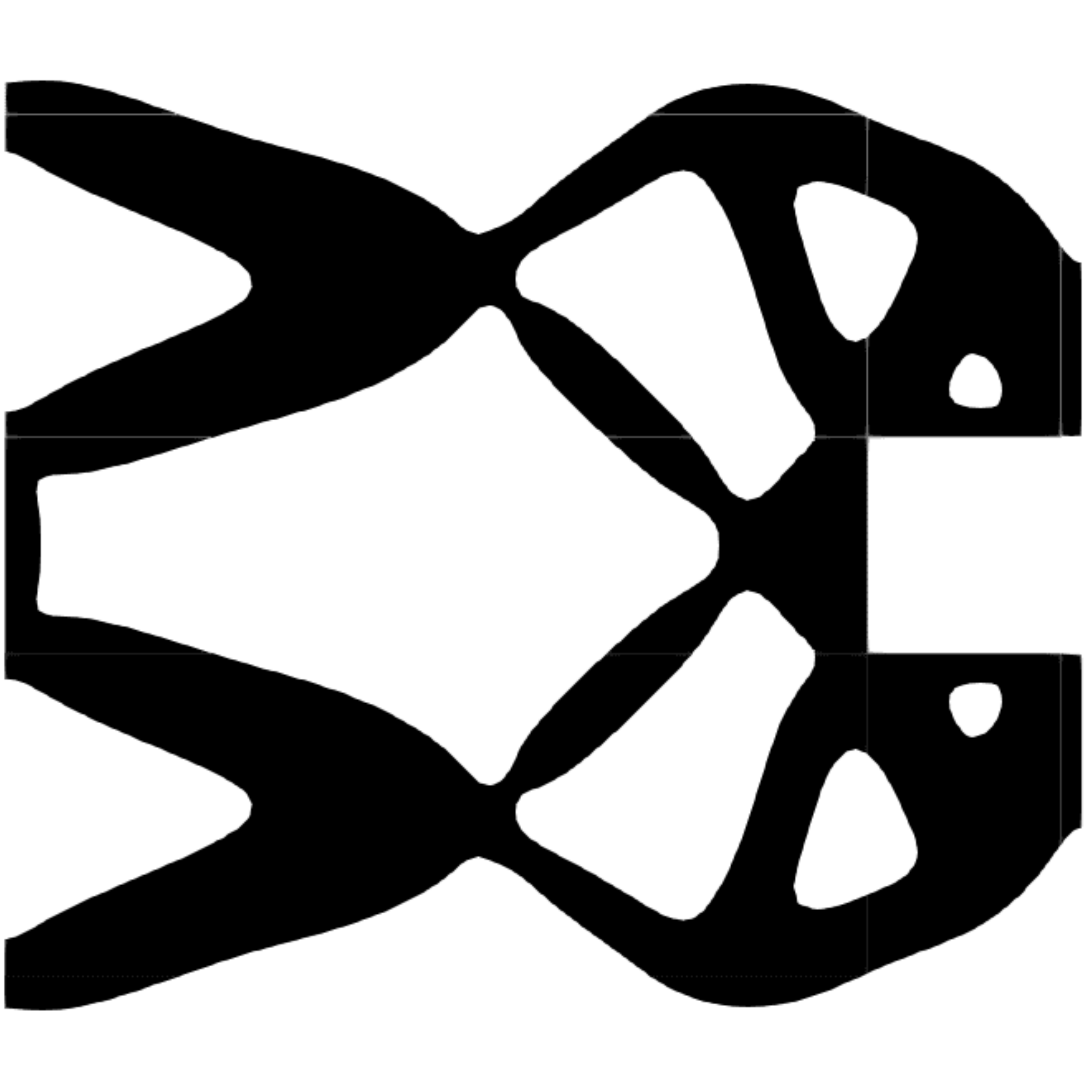}
        \subcaption{Step\,90}
        \label{cm-n}
      \end{minipage}
           \begin{minipage}[t]{0.2\hsize}
        \centering
        \includegraphics[keepaspectratio, scale=0.09]{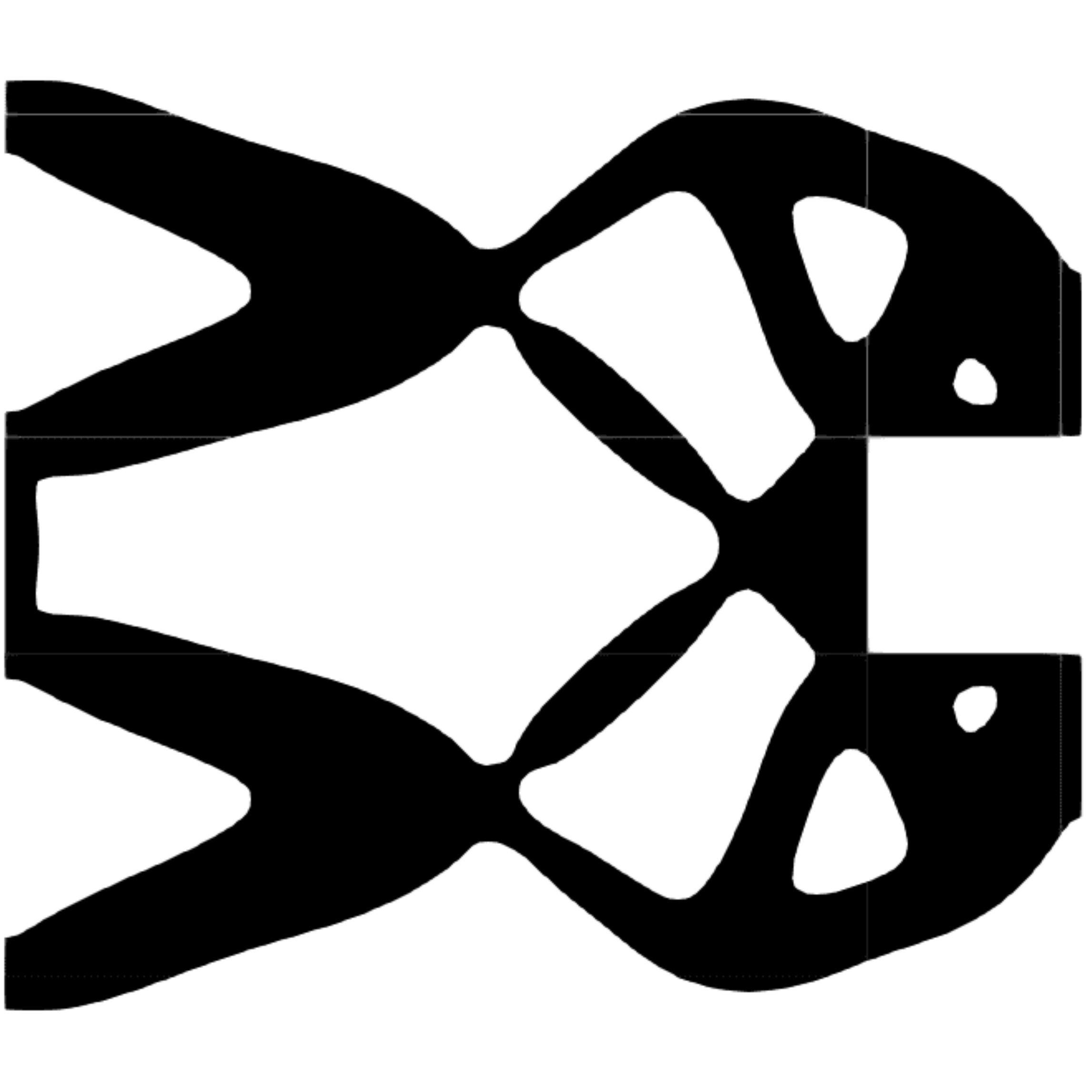}
        \subcaption{Step\,267$^{\#}$}
        \label{cm-o}
      \end{minipage}
    \end{tabular}
     \caption{Configuration $\Omega_{\phi_n}\subset D$ for the case where the initial configuration is the whole domain. 
Figures (a)--(e), (f)--(j) and (k)--(o) 
represent $\Omega_{\phi_n}\subset D$ for $(q,\varDelta t)=(1.0,0.2)$, $(q,\varDelta t)=(0.5,0.5)$ and $(q,\varDelta t)=(3.0,0.2)$ in \eqref{discNLD}, respectively. 
The symbol ${}^{ \#}$ implies the final step.    
     }
     \label{fig:cm}
  \end{figure}

\begin{figure}[htbp]
        \centering
        \includegraphics[keepaspectratio, scale=0.35]{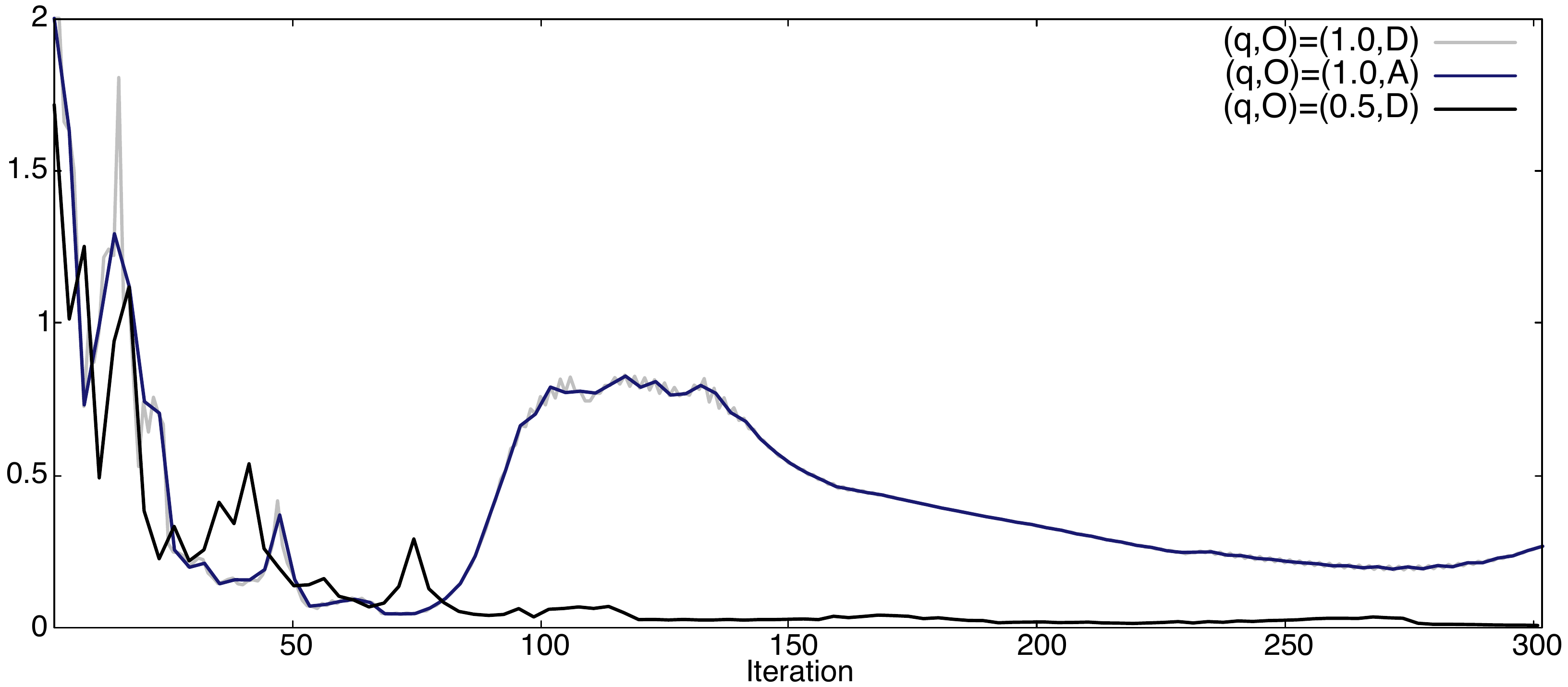}
   \caption{Convergence condition $\|\phi_{n+1}-\phi_{n}\|_{L^{\infty}(O)}$ with $\varDelta t=0.7$. Here $A\subset D$ denotes $[|\phi_n|\le 0.2]$.}
\label{fig:cmconv}
\end{figure}

  \begin{figure}[htbp]
   \hspace*{-5mm} 
    \begin{tabular}{ccccc}
    
        \begin{minipage}[t]{0.24\hsize}
        \centering
        \includegraphics[keepaspectratio, scale=0.11]{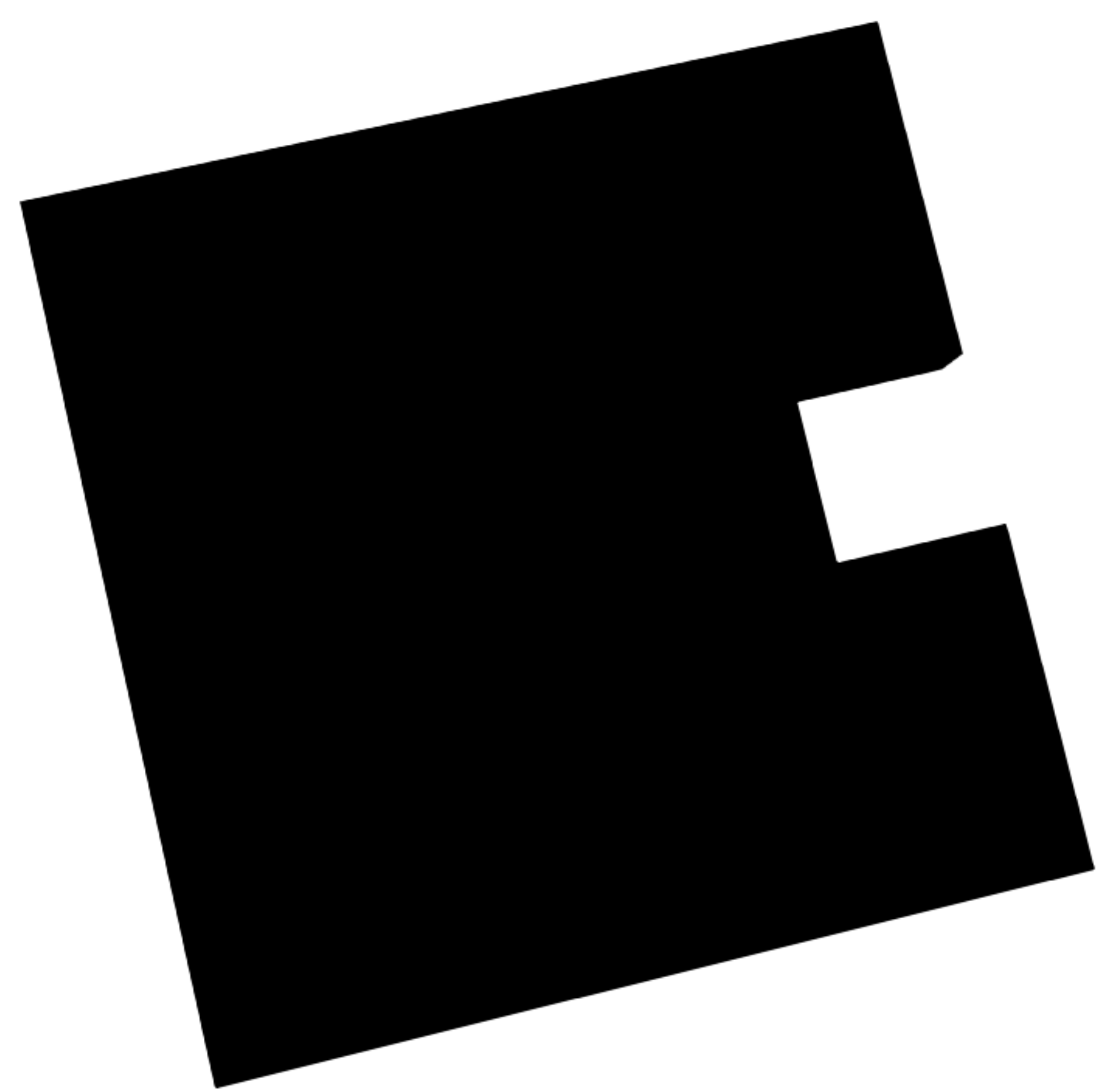}
        \subcaption{Step\,0}
        \label{3cm-a}
      \end{minipage} 
      \begin{minipage}[t]{0.24\hsize}
        \centering
        \includegraphics[keepaspectratio, scale=0.11]{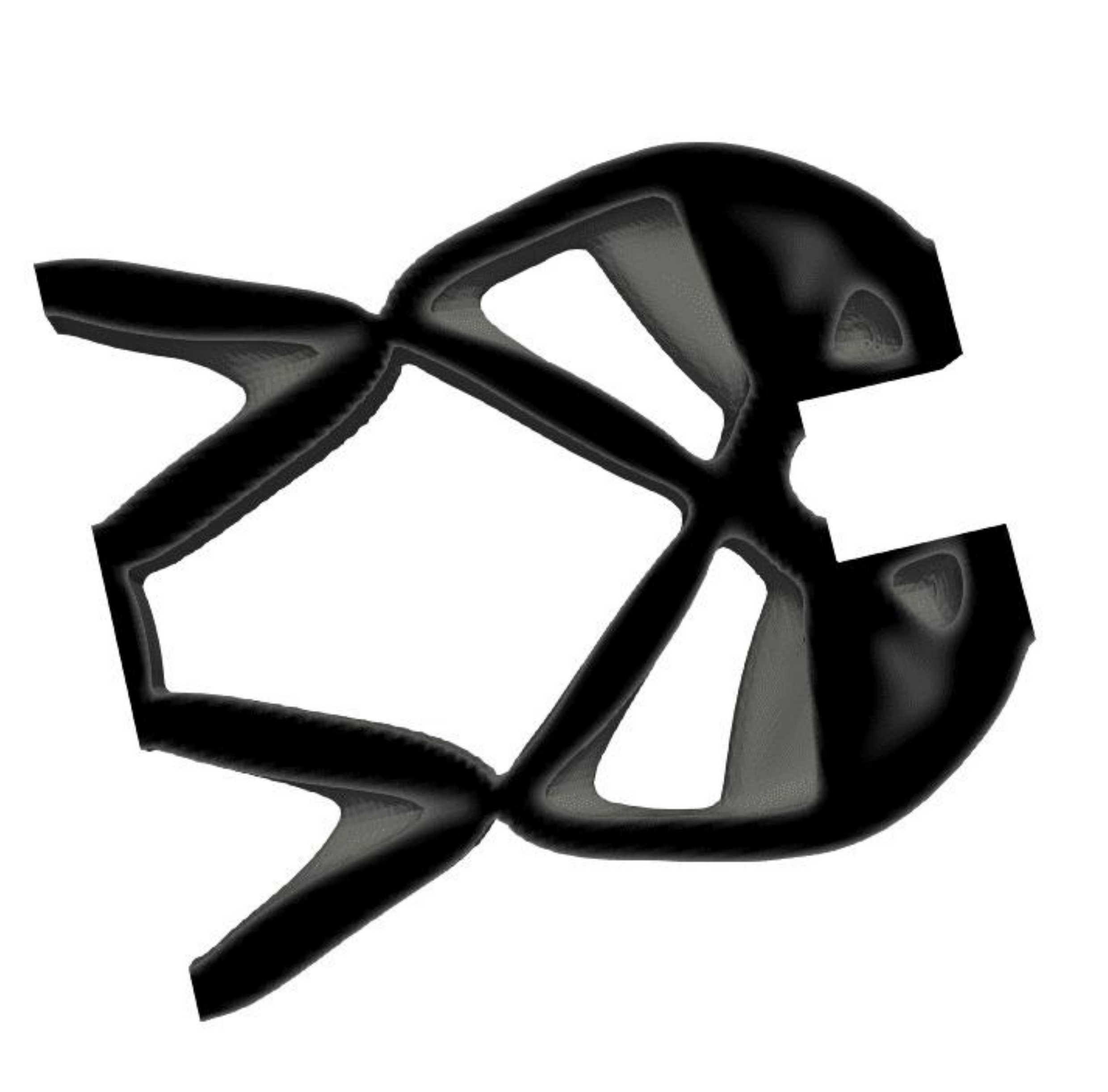}
        \subcaption{Step\,30}
        \label{3cm-b}
      \end{minipage} 
         \begin{minipage}[t]{0.24\hsize}
        \centering
        \includegraphics[keepaspectratio, scale=0.11]{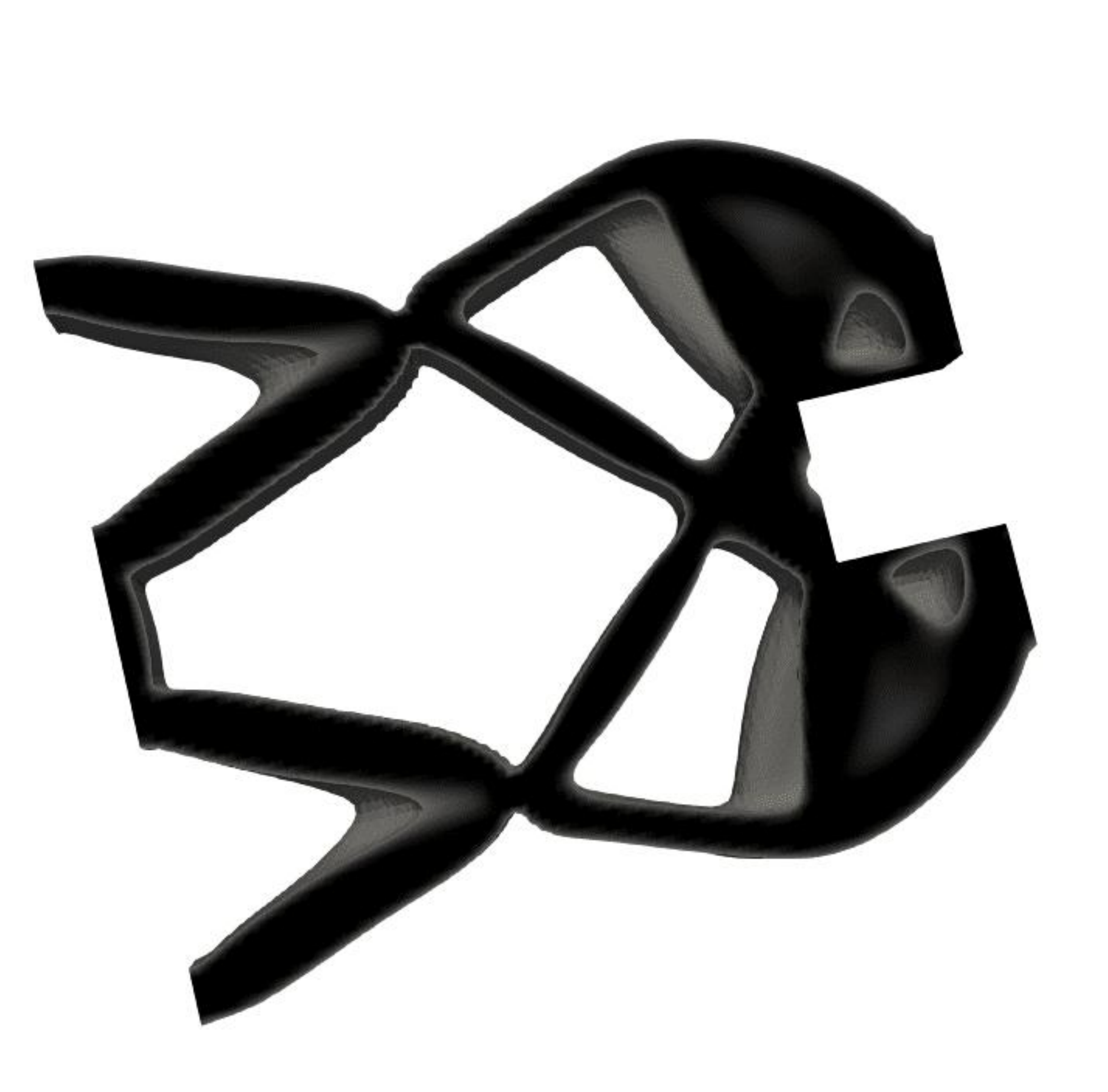}
        \subcaption{Step\,60}
        \label{3cm-c}
      \end{minipage}
           \begin{minipage}[t]{0.24\hsize}
        \centering
        \includegraphics[keepaspectratio, scale=0.11]{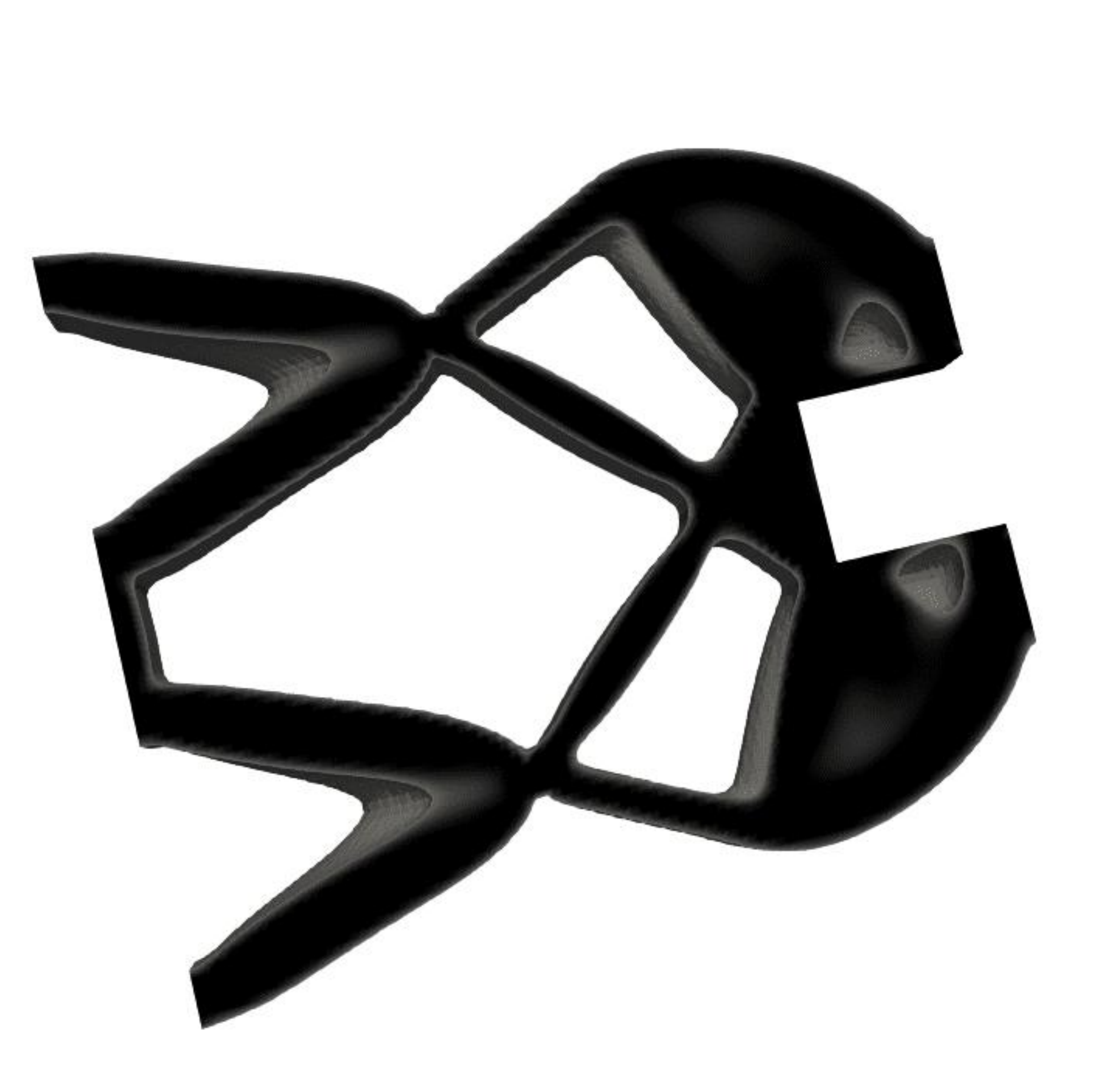}
        \subcaption{Step\,175$^{\#}$}
        \label{3cm-d}
      \end{minipage}
      \\
      \begin{minipage}[t]{0.24\hsize}
        \centering
        \includegraphics[keepaspectratio, scale=0.11]{3dcm0.pdf}
        \subcaption{Step\,0}
        \label{3cm-e}
      \end{minipage} 
      \begin{minipage}[t]{0.24\hsize}
        \centering
        \includegraphics[keepaspectratio, scale=0.11]{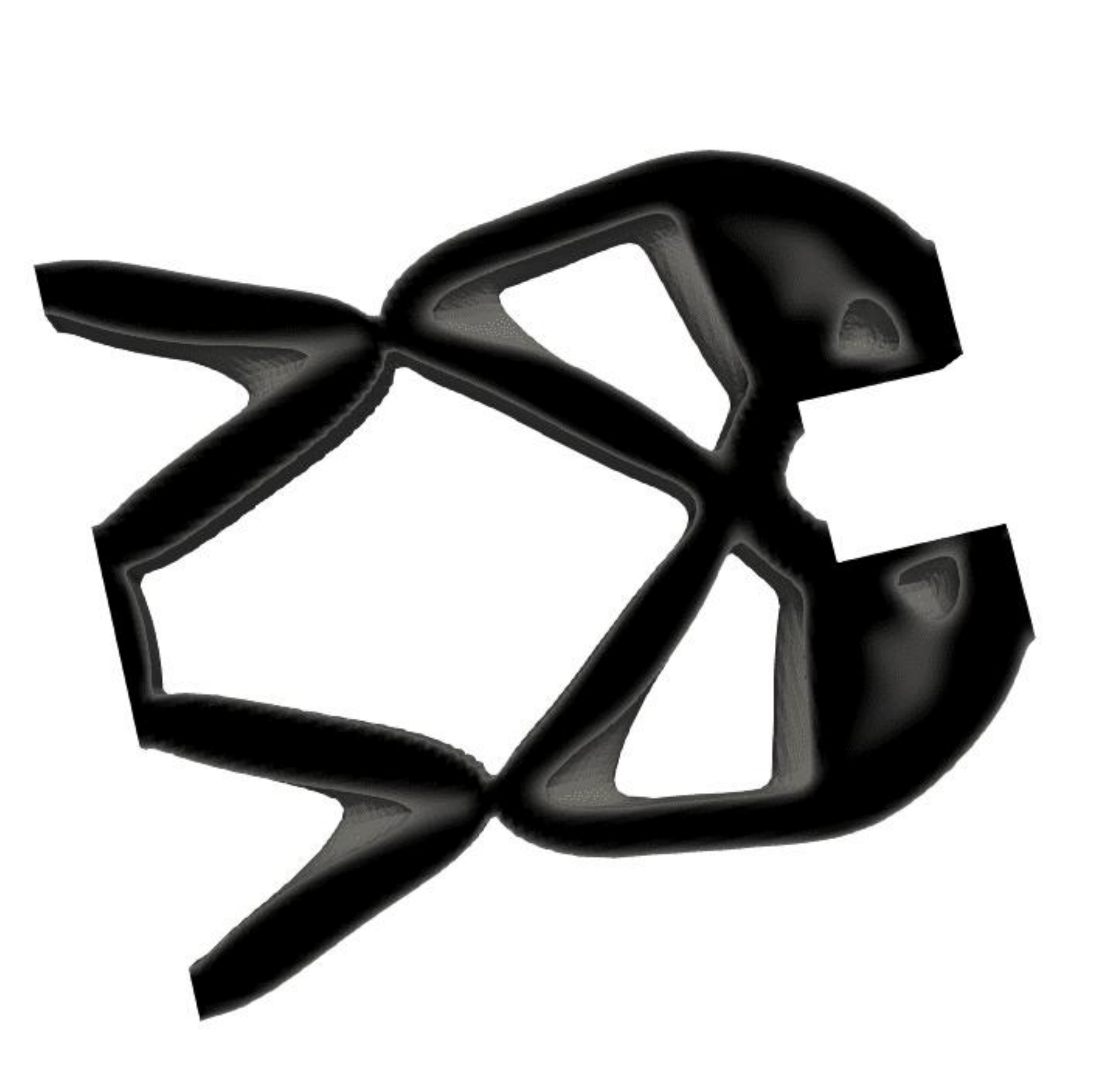}
        \subcaption{Step\,30}
        \label{3cm-f}
      \end{minipage} 
         \begin{minipage}[t]{0.24\hsize}
        \centering
        \includegraphics[keepaspectratio, scale=0.11]{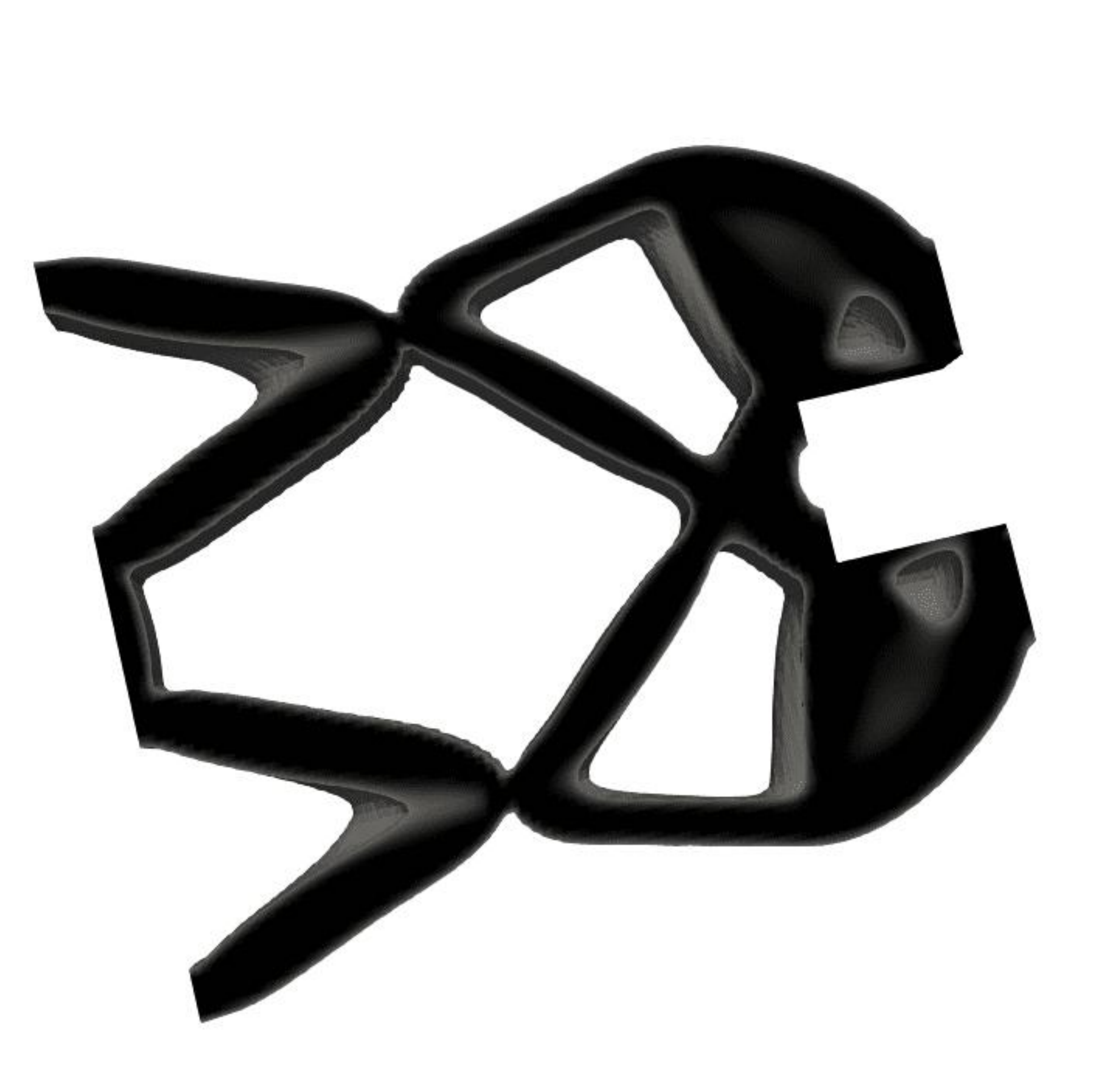}
        \subcaption{Step\,60}
        \label{3cm-g}
      \end{minipage}
           \begin{minipage}[t]{0.24\hsize}
        \centering
        \includegraphics[keepaspectratio, scale=0.11]{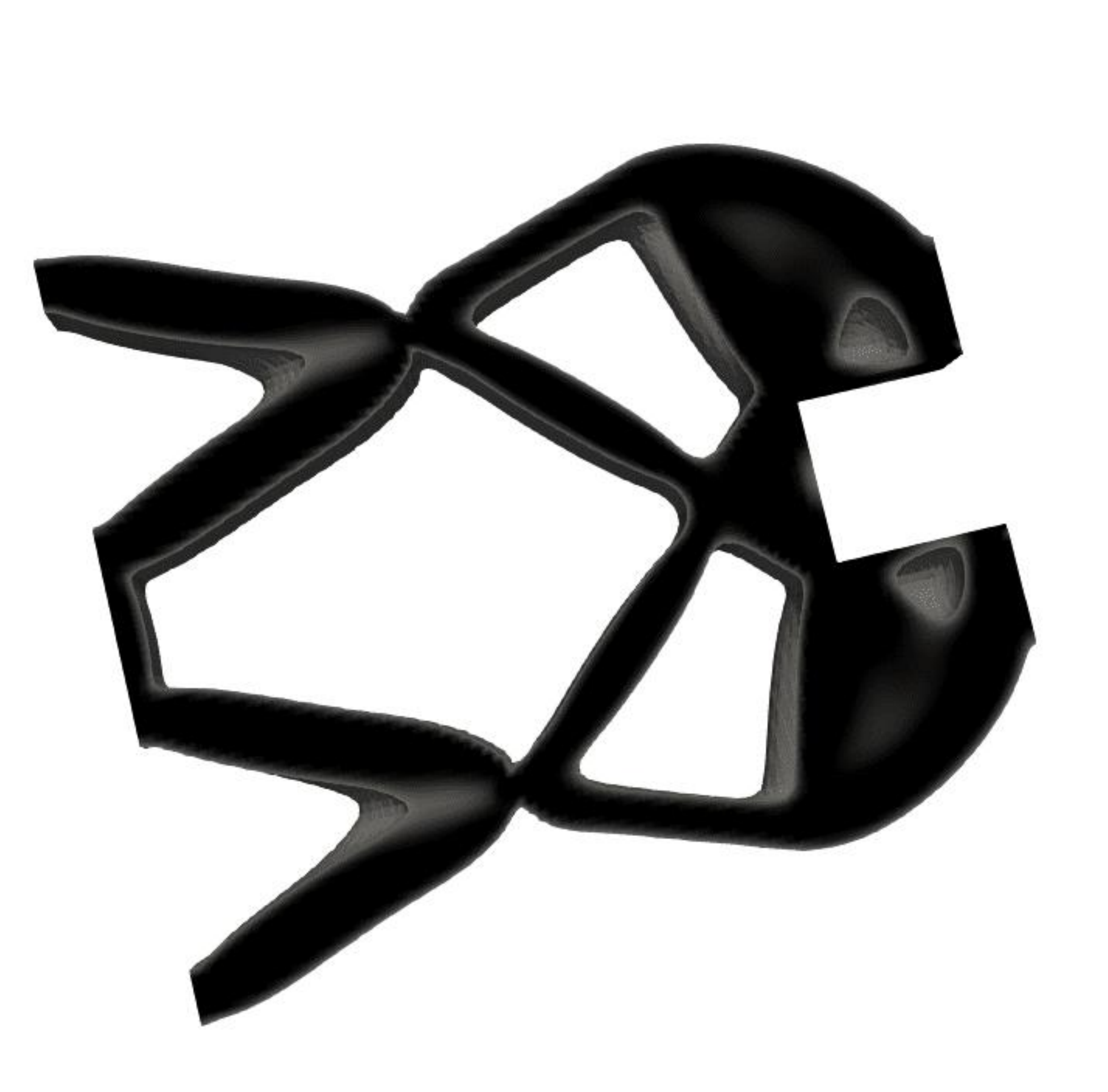}
        \subcaption{Step\,96$^{\#}$}
        \label{3cm-h}
      \end{minipage}
     \end{tabular}
     \caption{Configuration $\Omega_{\phi_n}\subset D\subset \R^3$ for the case where the initial configuration is the whole domain. 
Figures (a)--(d) and (e)--(h) represent $\Omega_{\phi_n}\subset D$ for $q=1$ and $q=2$ in \eqref{discNLD}, respectively. 
The symbol ${}^{ \#}$ implies the final step. 
Here the depth of $D$ is set to $0.1$ and $(\varDelta t, G_{\rm max})=(0.4,0.35)$.    
     }
     \label{fig:3dcm}
  \end{figure}

\subsection{Heat conduction problem}\label{SS:ex2}
We finally show (i-FDE) and (ii-SDE) for the so-called heat conduction problem with volume-constraint (see, e.g.,~\cite{HBS06, Y11, ZLWS15}).
Let us consider \eqref{eq:opt-prob}, i.e.,  
$$
\inf_{\phi\in H^1(D;[-1,1])} 
F(\phi)\quad \text{ subject to }\ G(\phi)\le 0.
$$
Here the Lagrangian of \eqref{eq:opt-prob} is given by 
\begin{align*}
\mathcal{L}(\phi,\lambda)=F(\phi)+\lambda G(\phi)
=
\langle f, u_\phi \rangle_{V}+\lambda\underbrace{\left(\int_D \chi_\phi(x)\, \d x- G_{\text{max}}|D|\right)}_{\le 0}, 
\end{align*}
where $G_{\text{max}}>0$, $u_\phi\in V$ is a unique solution to the steady-state heat equation,
\begin{align}
\int_D
\kappa_\phi(x) \nabla u_\phi(x)\cdot \nabla v(x)\, \d x
=
\langle f, v \rangle_{V}
\quad \text{ for all }\  v\in V
\label{eq:HC}
\end{align}
and $\kappa_\phi\in L^{\infty}(D)$ is the (two-phase) heat conductivity given by 
$\kappa_\phi=\alpha\chi_{\phi}+\beta(1-\chi_\phi)$, i.e.,~
$$
\kappa_\phi(x)=
\begin{cases}
\alpha\quad &\text{ if }\ \phi(x)\ge 0,\\  
\beta\quad &\text{ if }\ \phi(x)< 0.
\end{cases}
$$
Here $\alpha>0$ and $\beta>0$ are such that $\alpha\neq \beta$.

Based on \S \ref{S:algo}, Proposition \ref{prop} and Remark \ref{R:GR}, the numerical analysis is performed as in \S \ref{SS:ca}. Here we set $f\equiv 1$ in \eqref{eq:HC}.  

We first show (i-FDE) under $\partial D=\Gamma_D$ (i.e.,~$V=H^1_0(D)$), $(\alpha,\beta)=(1.0\times 10^{-2},1.0)$ (i.e.,~$\alpha<\beta$) and $(\tau,G_{\rm max},\varDelta t)=(1.0\times 10^{-5},0.5,0.5)$. 
The numerical results are shown in Figures \ref{fig:hb} and \ref{fig:hbb}.
From Figure \ref{fig:hb}, we infer that the method using fast diffusion (i.e.,~$q>1$) yields faster convergence to the optimal configuration than that of reaction-diffusion (i.e.,~$q=1$). 
In particular, the same assertion holds for different initial configurations (see Figure \ref{fig:hbb}). 
This completes the check for (i-FDE).  

\begin{figure}[htbp]
   \hspace*{-5mm} 
    \begin{tabular}{ccccc}
        \begin{minipage}[t]{0.2\hsize}
        \centering
        \includegraphics[keepaspectratio, scale=0.09]{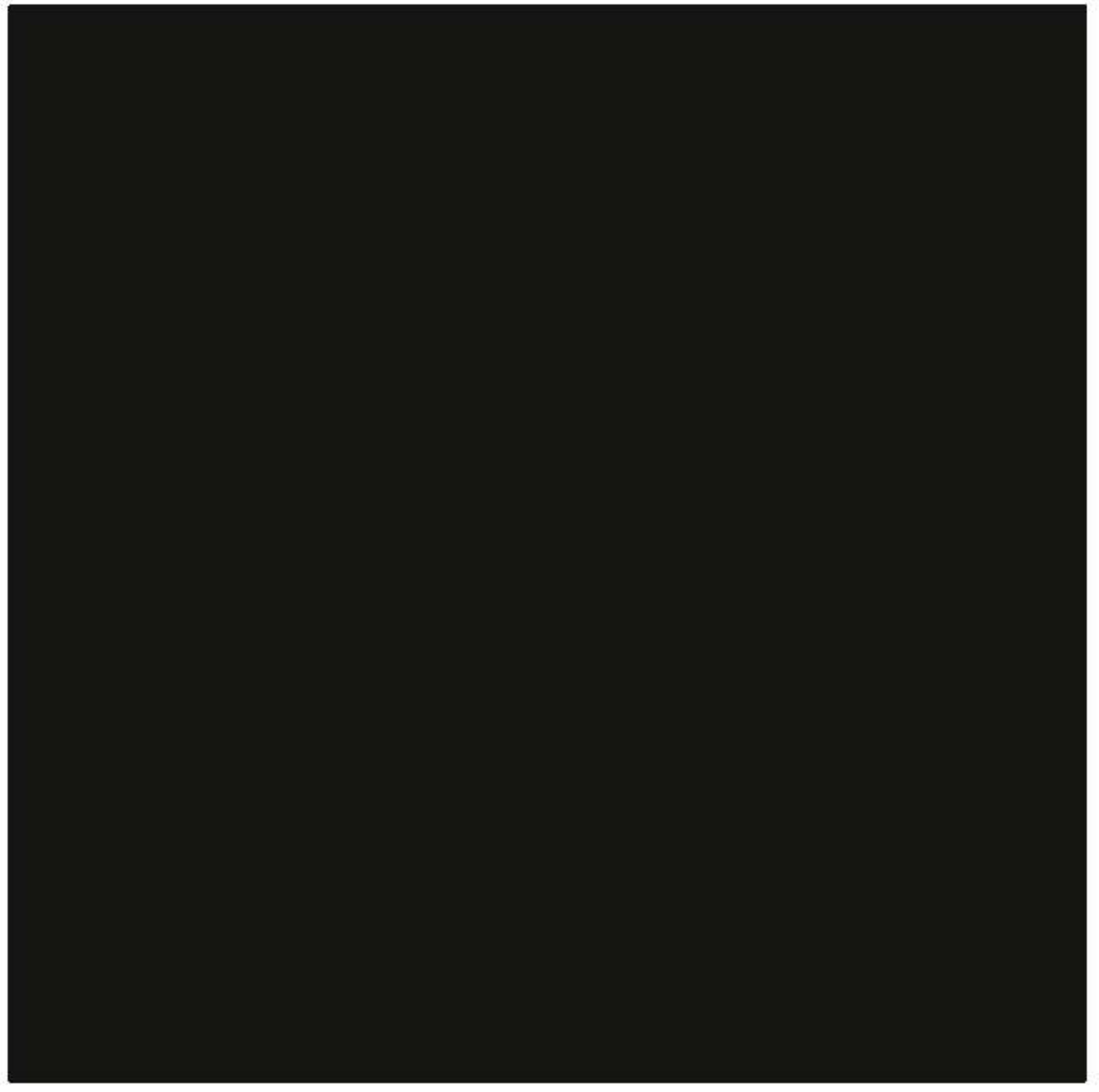}
        \subcaption{Step\,0}
        \label{hb-a}
      \end{minipage} 
      \begin{minipage}[t]{0.2\hsize}
        \centering
        \includegraphics[keepaspectratio, scale=0.09]{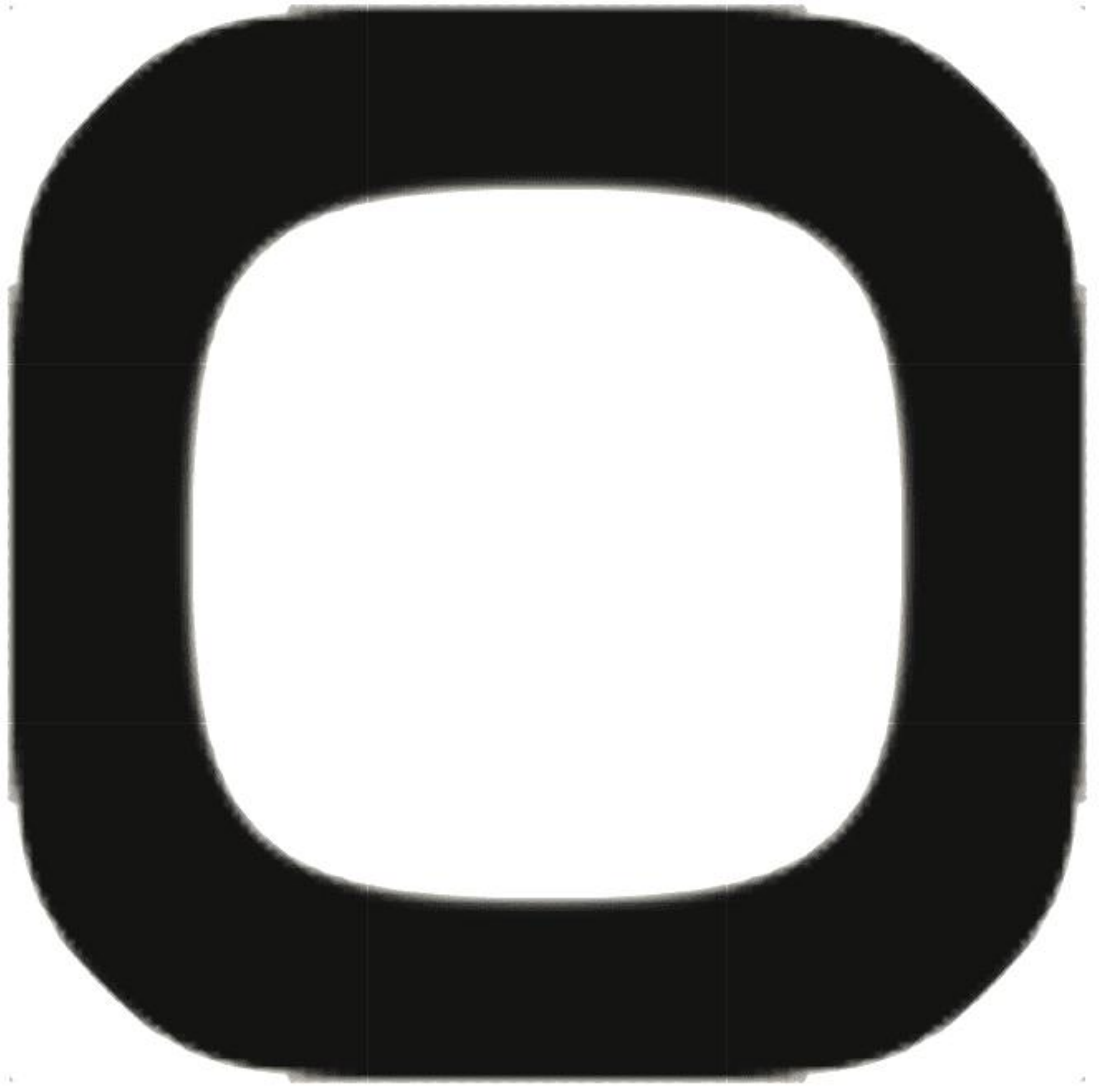}
        \subcaption{Step\,20}
        \label{hb-b}
      \end{minipage} 
         \begin{minipage}[t]{0.2\hsize}
        \centering
        \includegraphics[keepaspectratio, scale=0.09]{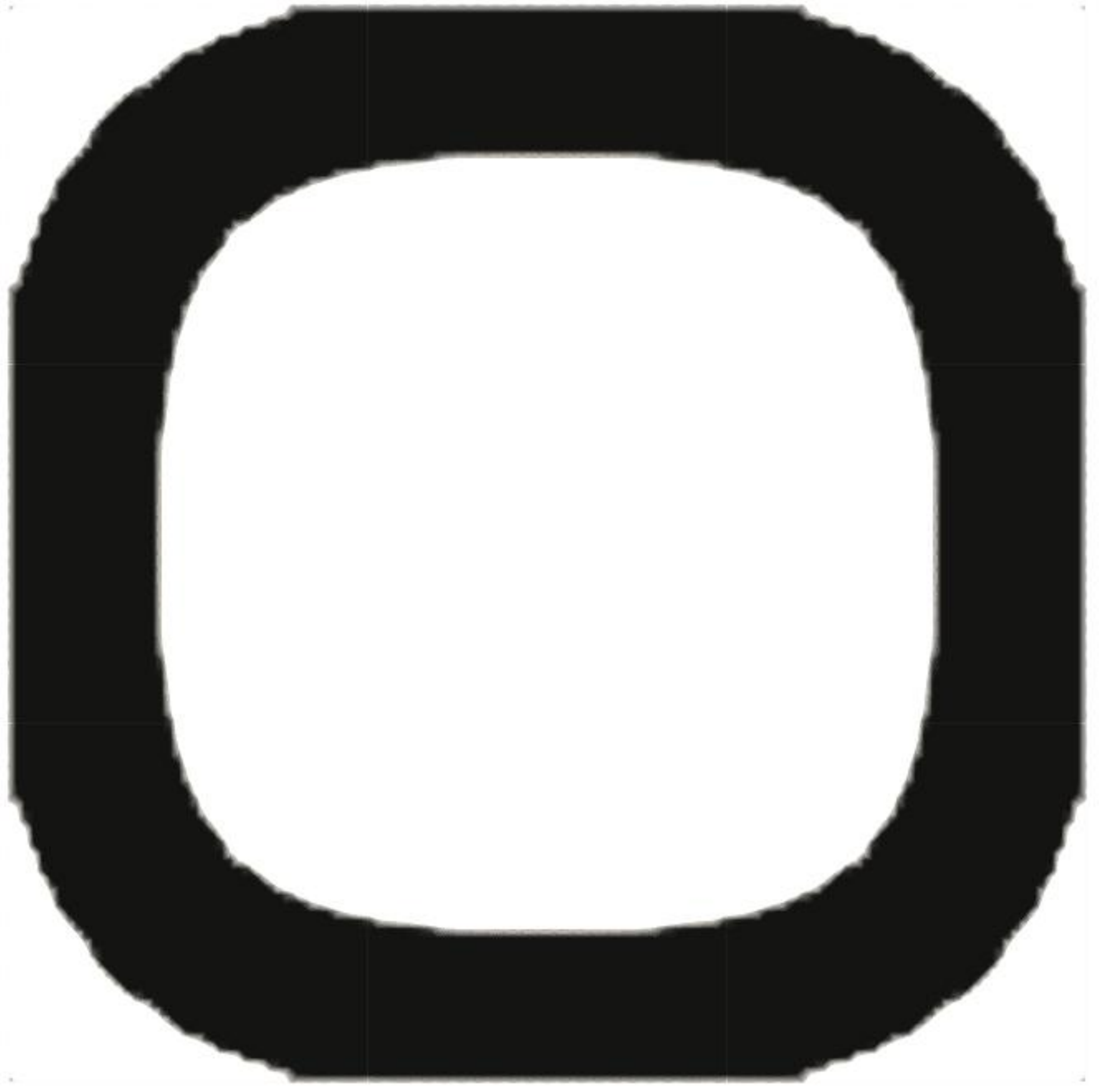}
        \subcaption{Step\,40}
        \label{hb-c}
      \end{minipage}
      \begin{minipage}[t]{0.2\hsize}
        \centering
        \includegraphics[keepaspectratio, scale=0.09]{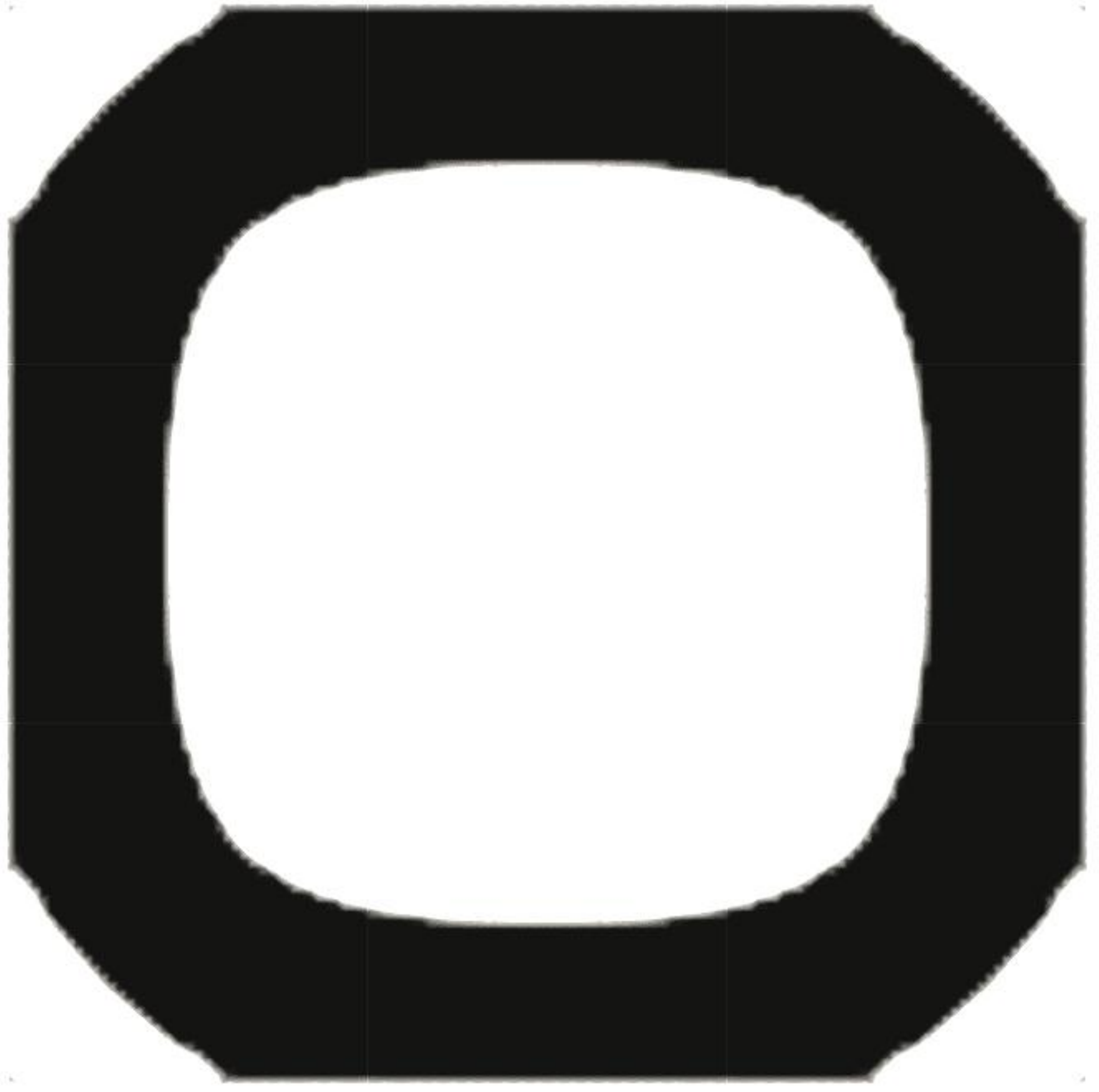}
        \subcaption{Step\,60}
        \label{hb-d}
      \end{minipage}
           \begin{minipage}[t]{0.2\hsize}
        \centering
        \includegraphics[keepaspectratio, scale=0.09]{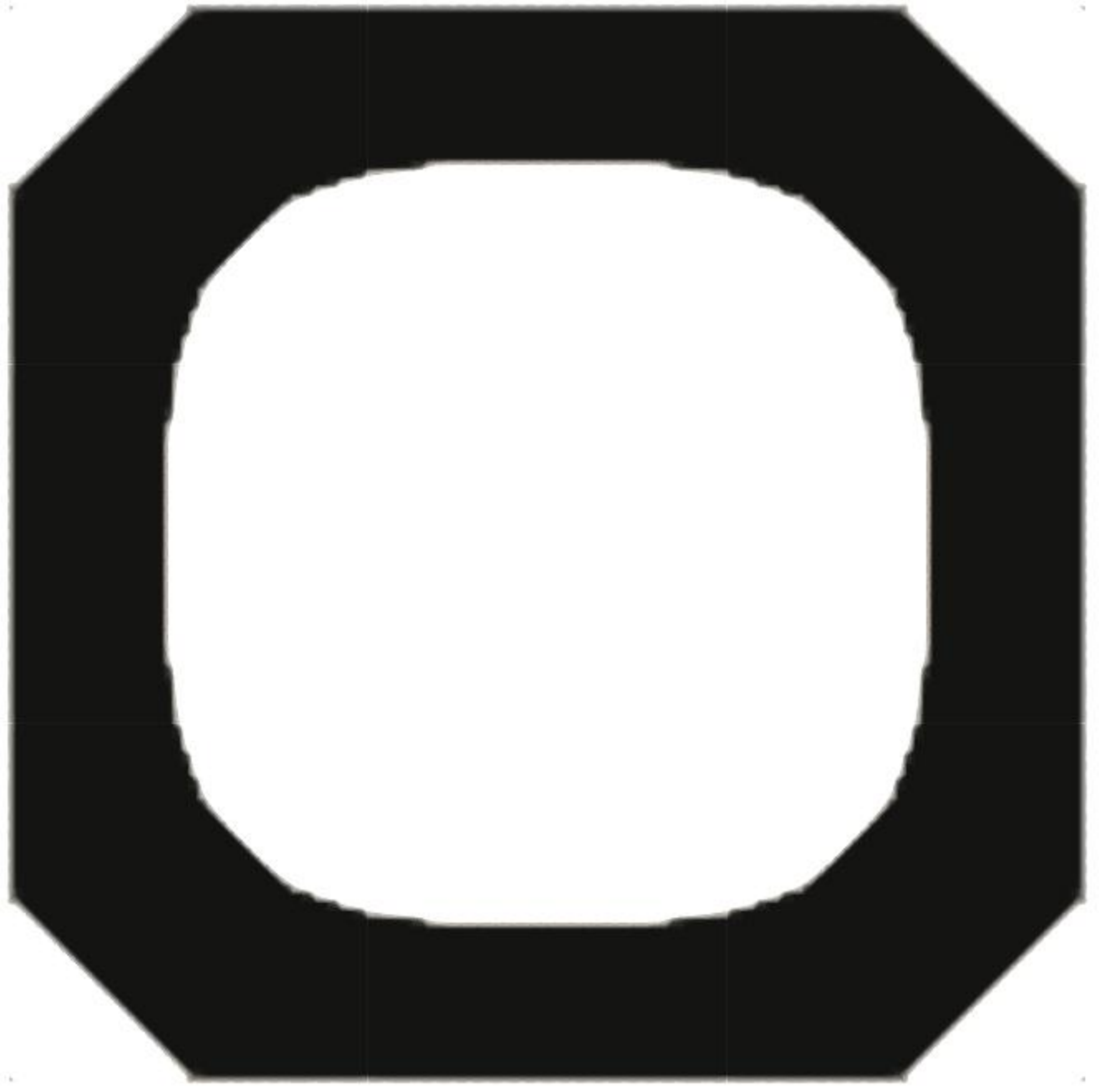}
        \subcaption{Step\,169$^{\#}$}
        \label{hb-e}
      \end{minipage}
      \\
              \begin{minipage}[t]{0.2\hsize}
        \centering
        \includegraphics[keepaspectratio, scale=0.09]{hb0.pdf}
        \subcaption{Step\,0}
        \label{hb-f}
      \end{minipage} 
      \begin{minipage}[t]{0.2\hsize}
        \centering
        \includegraphics[keepaspectratio, scale=0.09]{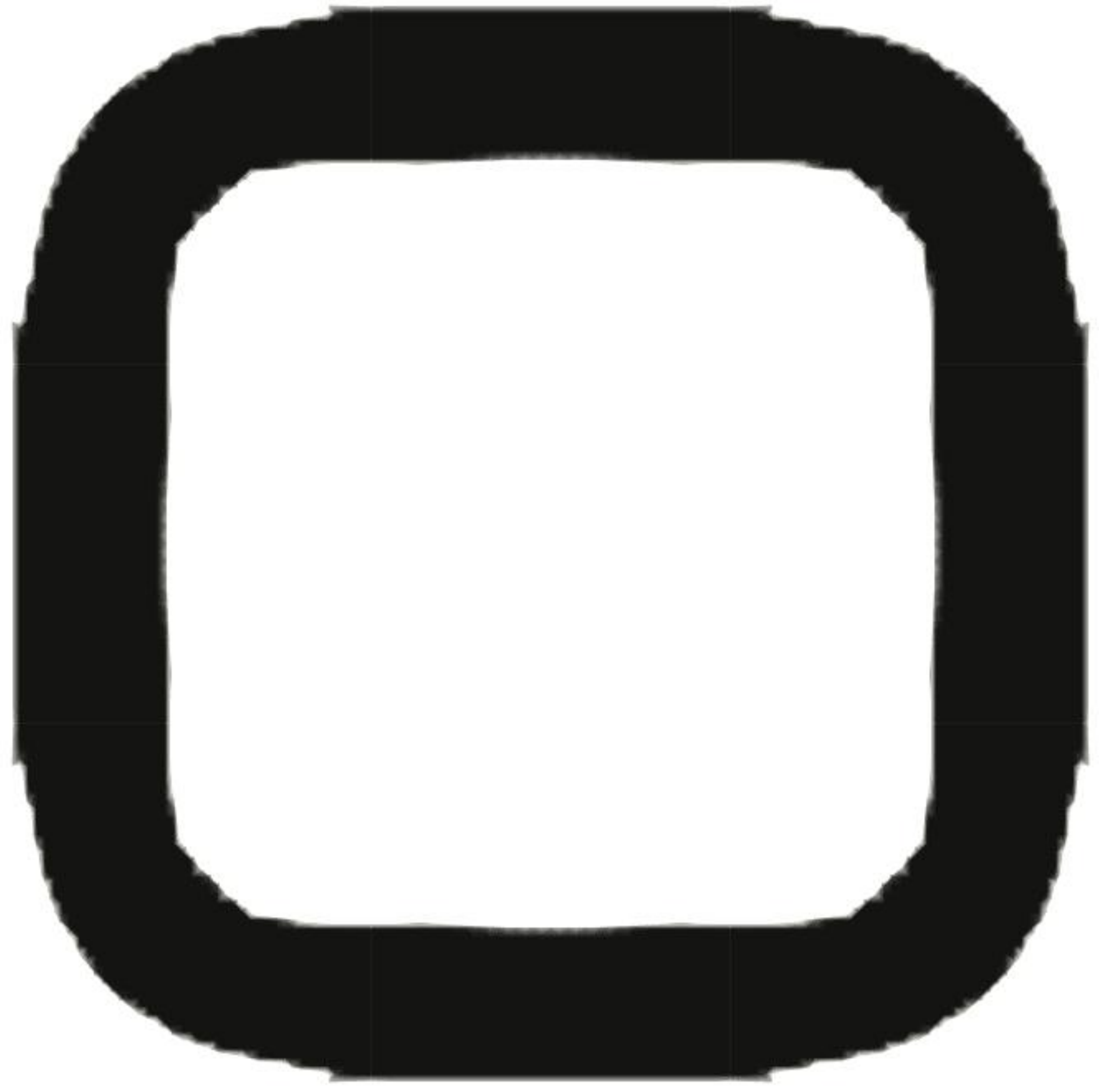}
        \subcaption{Step\,20}
        \label{hb-g}
      \end{minipage} 
         \begin{minipage}[t]{0.2\hsize}
        \centering
        \includegraphics[keepaspectratio, scale=0.09]{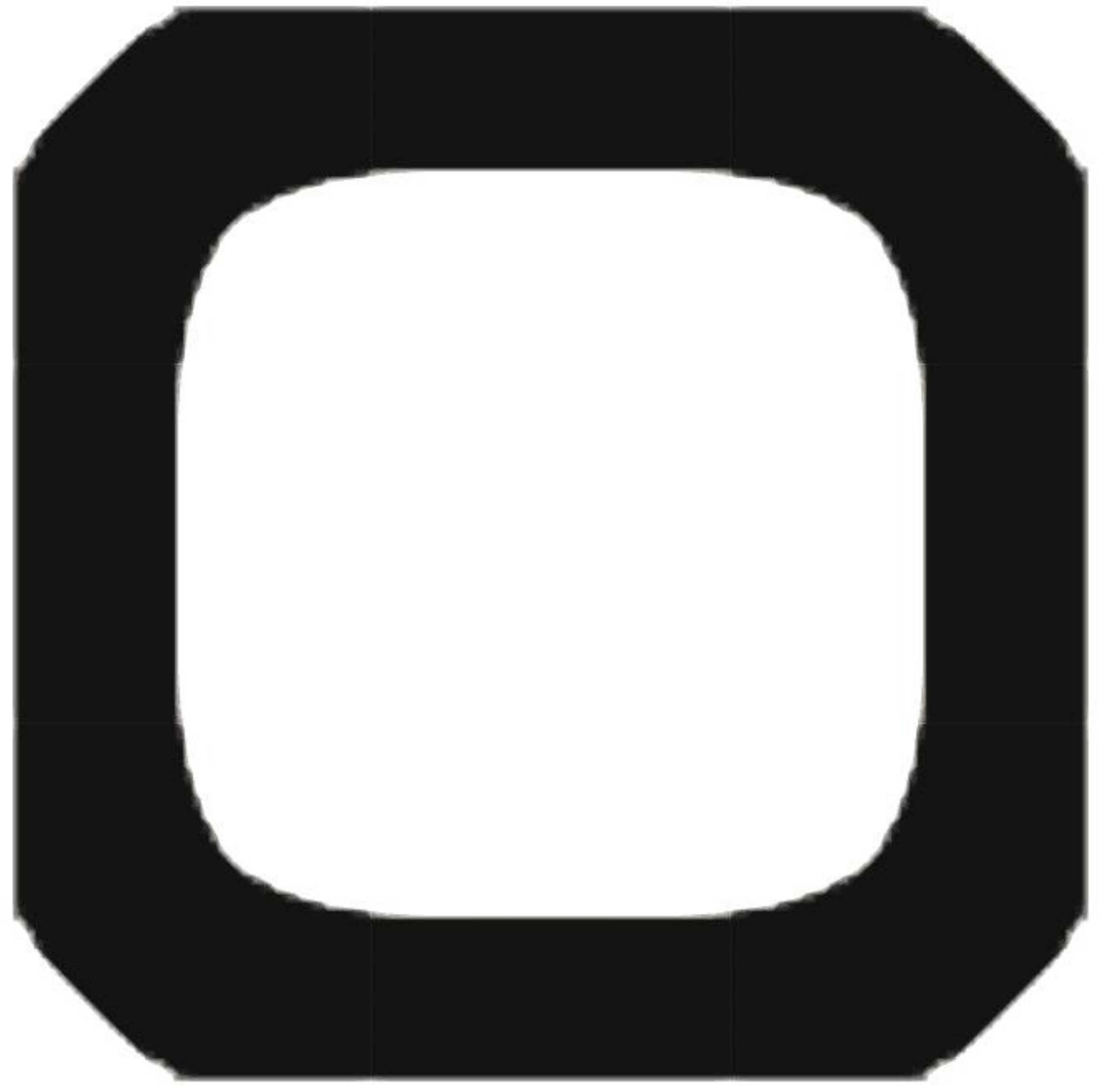}
        \subcaption{Step\,40}
        \label{hb-h}
      \end{minipage}
      \begin{minipage}[t]{0.2\hsize}
        \centering
        \includegraphics[keepaspectratio, scale=0.09]{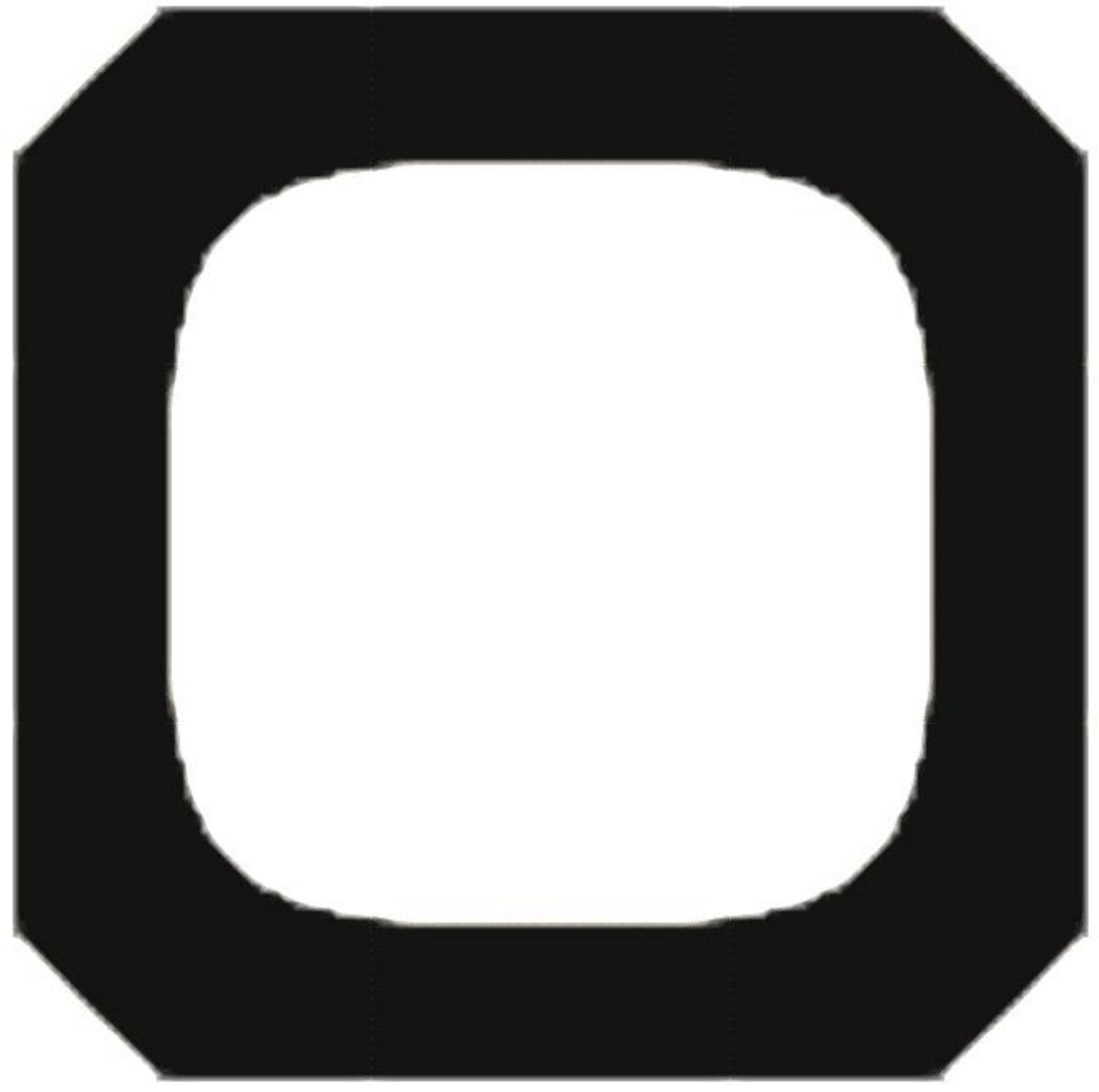}
        \subcaption{Step\,60}
        \label{hb-i}
      \end{minipage}
           \begin{minipage}[t]{0.2\hsize}
        \centering
        \includegraphics[keepaspectratio, scale=0.09]{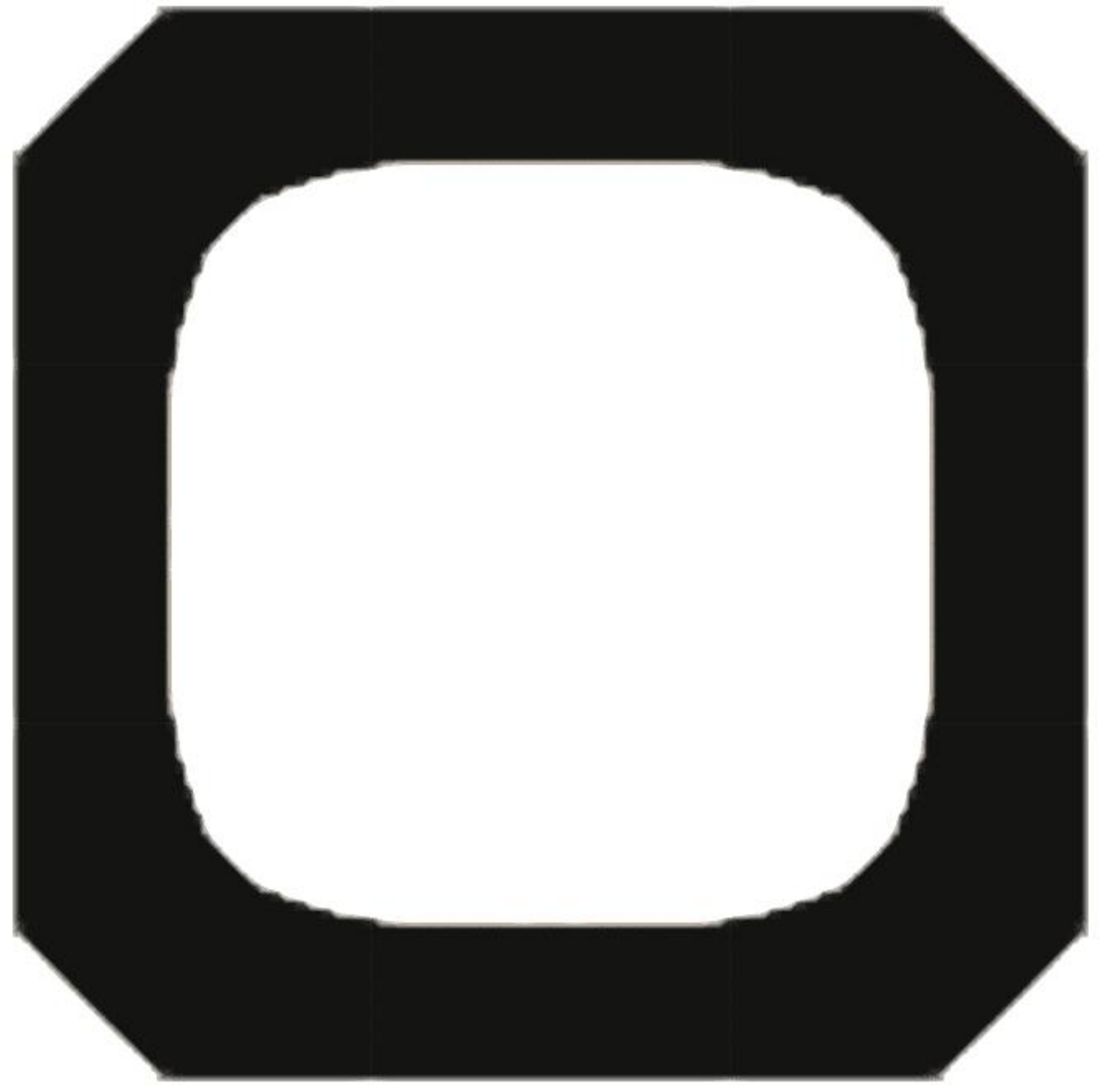}
        \subcaption{Step\,93$^{\#}$}
        \label{hb-j}
      \end{minipage}
       \end{tabular}
     \caption{Configuration $\Omega_{\phi_n}\subset D$ for the case where the initial configuration is the whole domain. 
Figures (a)--(e) and (f)--(j) 
represent $\Omega_{\phi_n}\subset D$ for $q=1$ and $q=6$ in \eqref{discNLD}, respectively. 
The symbol ${}^{ \#}$ implies the final step.    
     }
     \label{fig:hb}
  \end{figure}

\begin{figure}[htbp]
   \hspace*{-5mm} 
    \begin{tabular}{ccccc}
        \begin{minipage}[t]{0.2\hsize}
        \centering
        \includegraphics[keepaspectratio, scale=0.09]{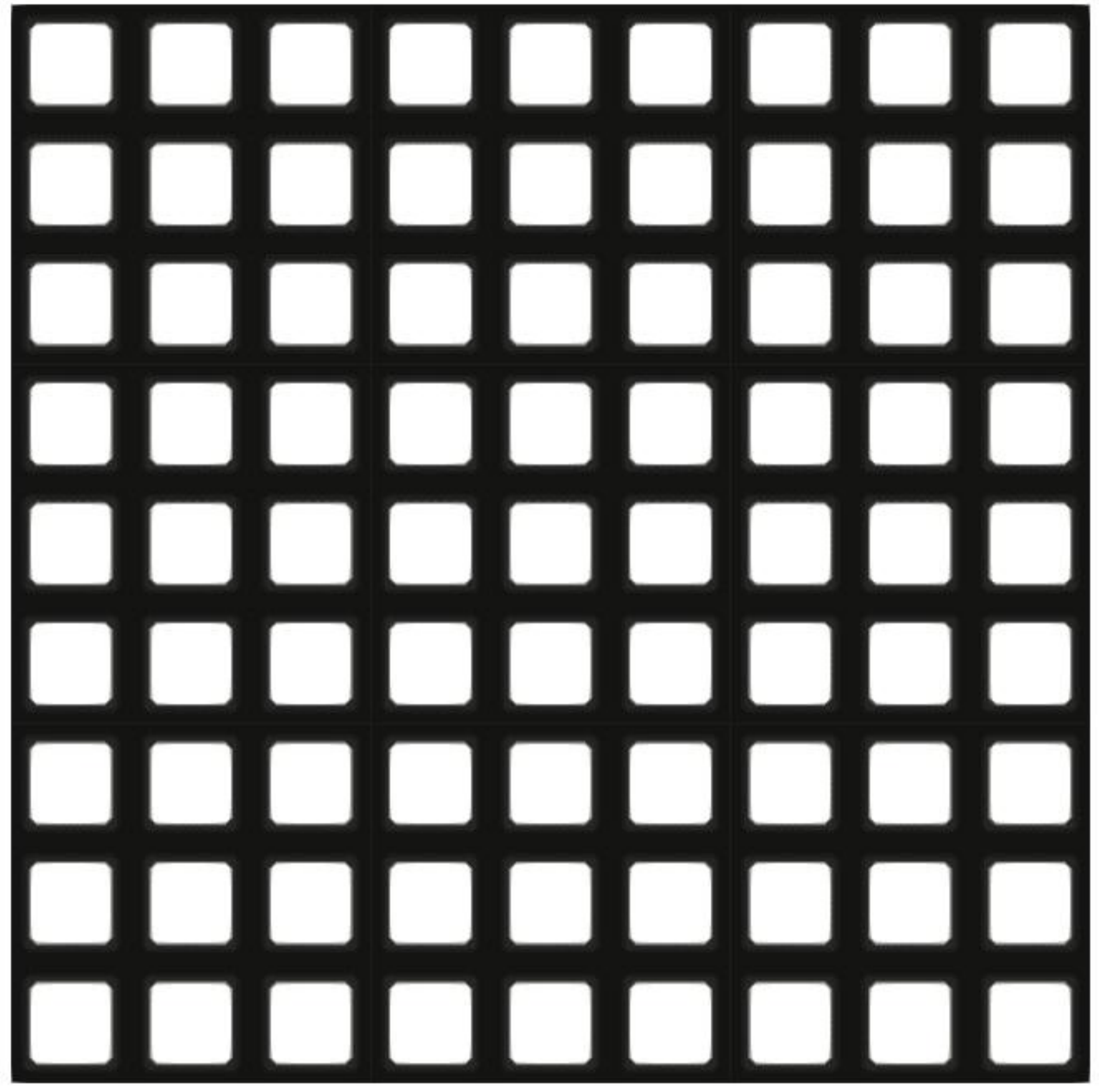}
        \subcaption{Step\,0}
        \label{hbb-a}
      \end{minipage} 
      \begin{minipage}[t]{0.2\hsize}
        \centering
        \includegraphics[keepaspectratio, scale=0.09]{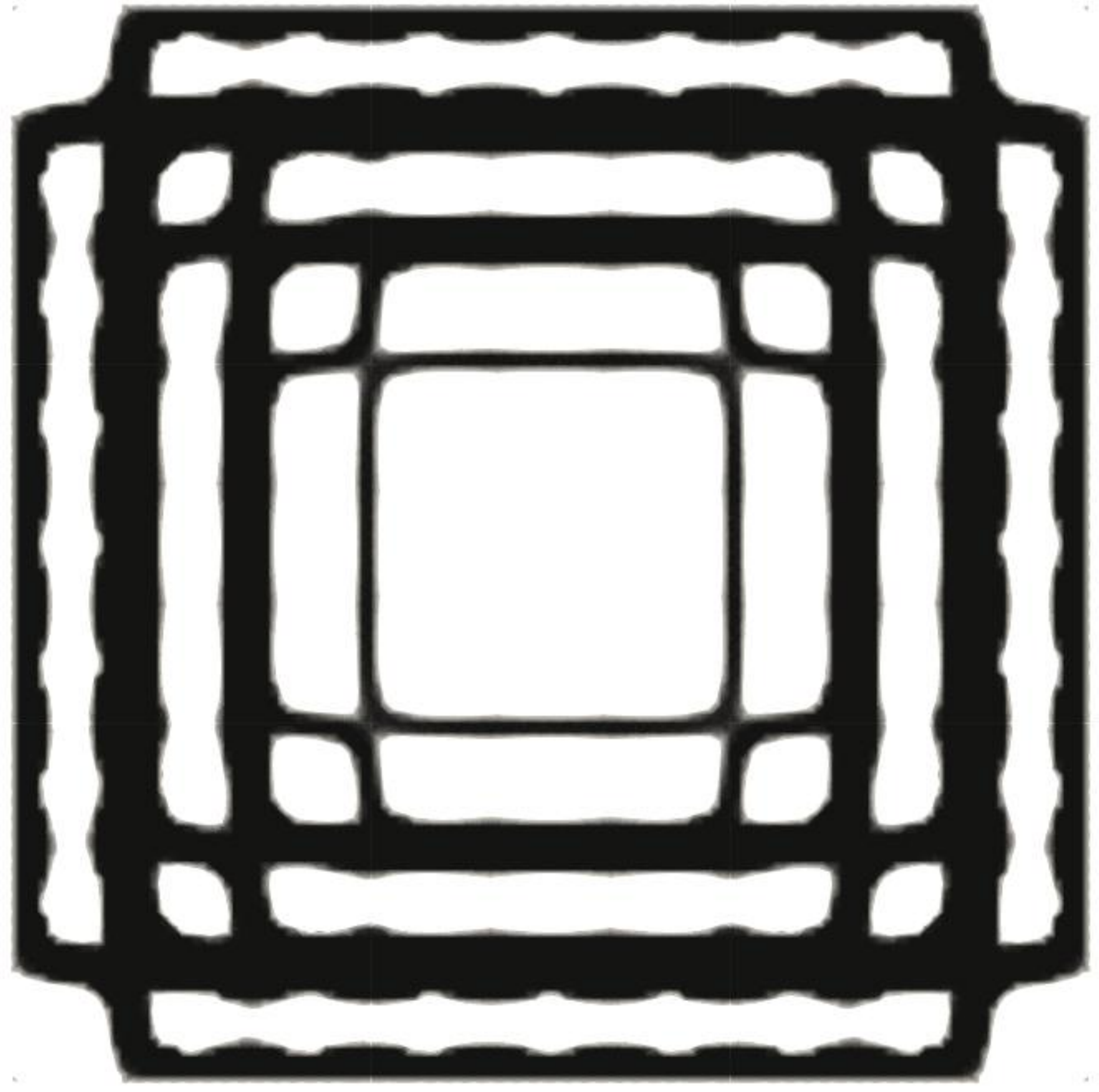}
        \subcaption{Step\,20}
        \label{hbb-b}
      \end{minipage} 
         \begin{minipage}[t]{0.2\hsize}
        \centering
        \includegraphics[keepaspectratio, scale=0.09]{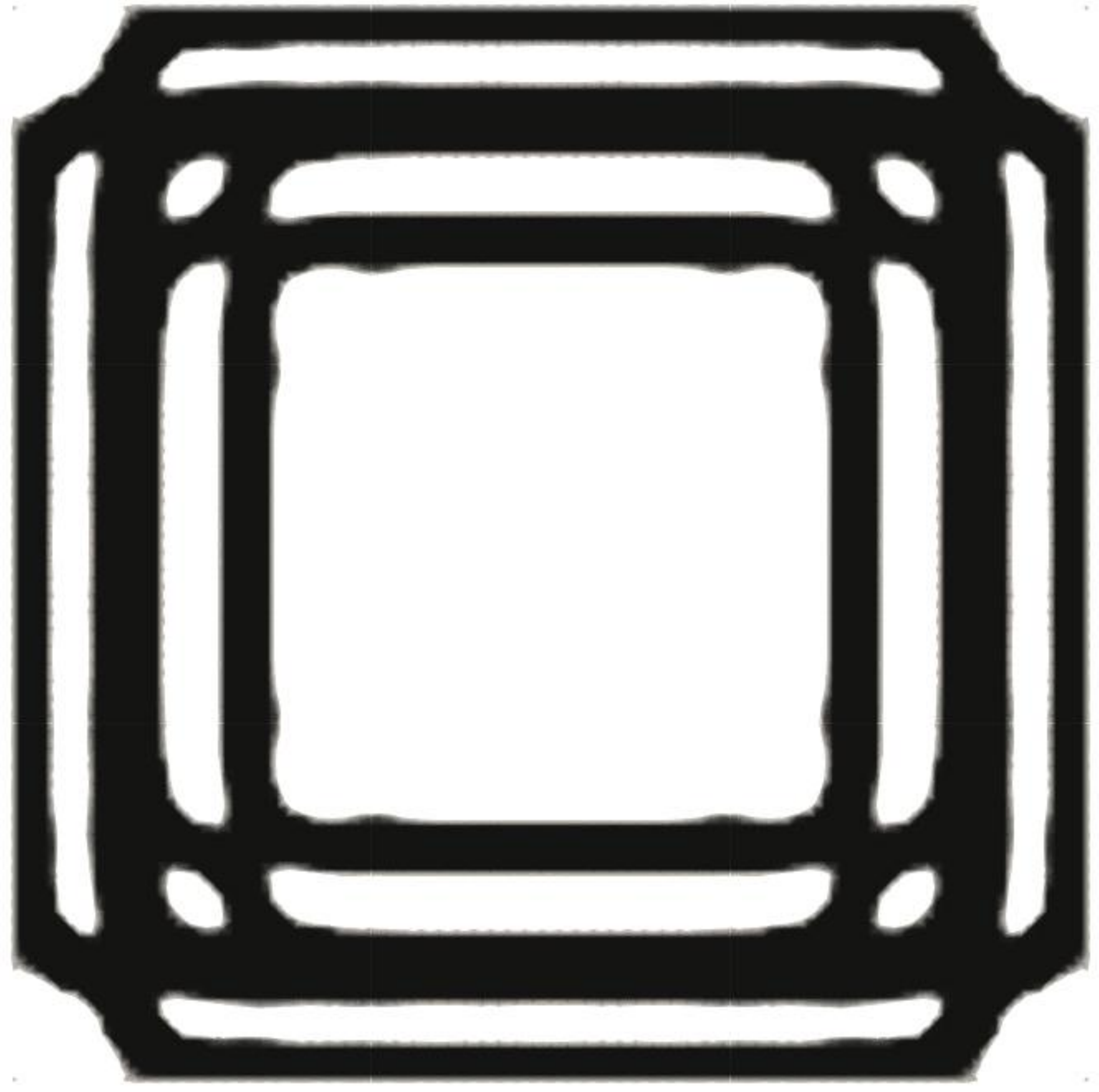}
        \subcaption{Step\,40}
        \label{hbb-c}
      \end{minipage}
      \begin{minipage}[t]{0.2\hsize}
        \centering
        \includegraphics[keepaspectratio, scale=0.09]{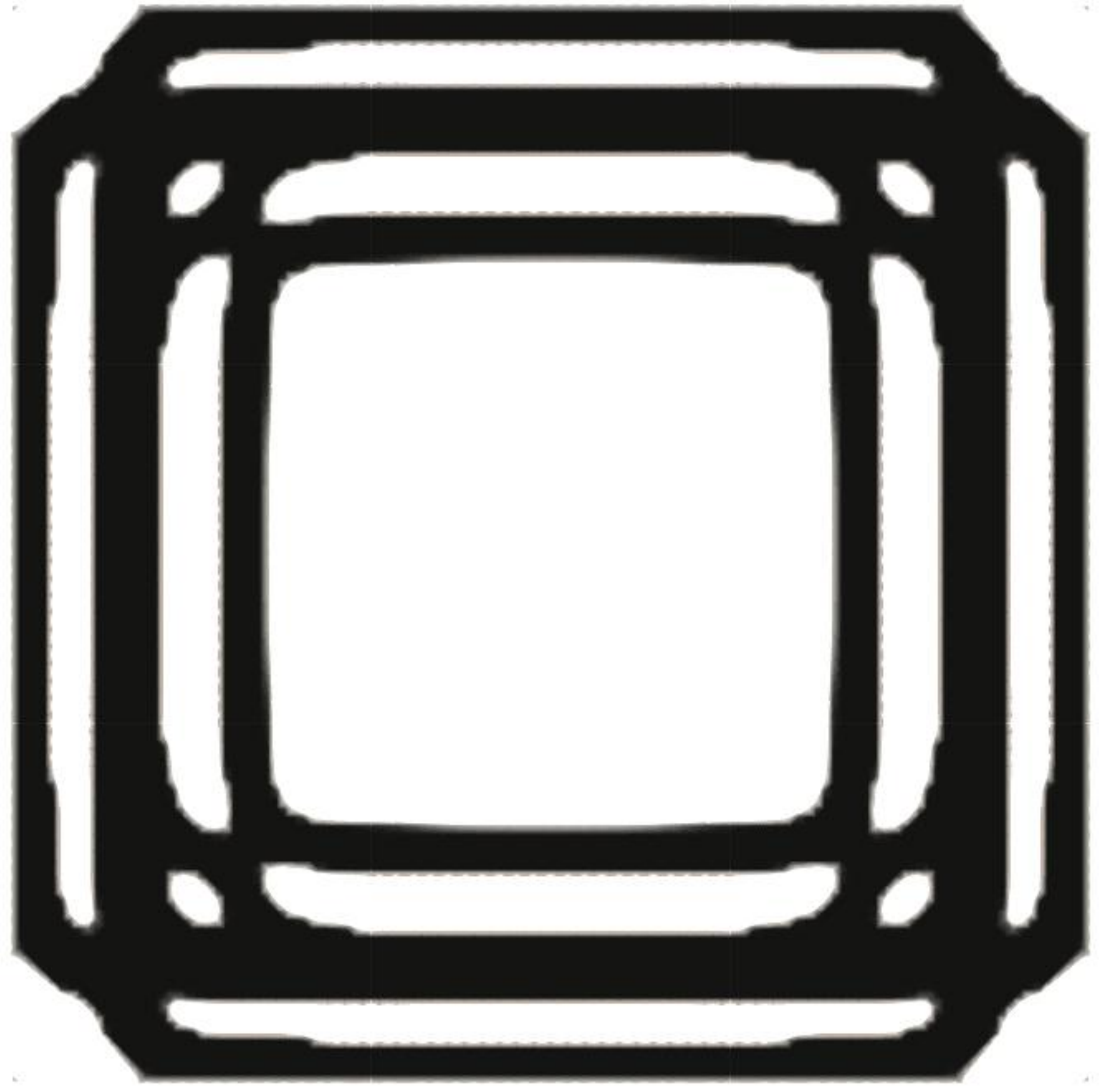}
        \subcaption{Step\,60}
        \label{hbb-d}
      \end{minipage}
           \begin{minipage}[t]{0.2\hsize}
        \centering
        \includegraphics[keepaspectratio, scale=0.09]{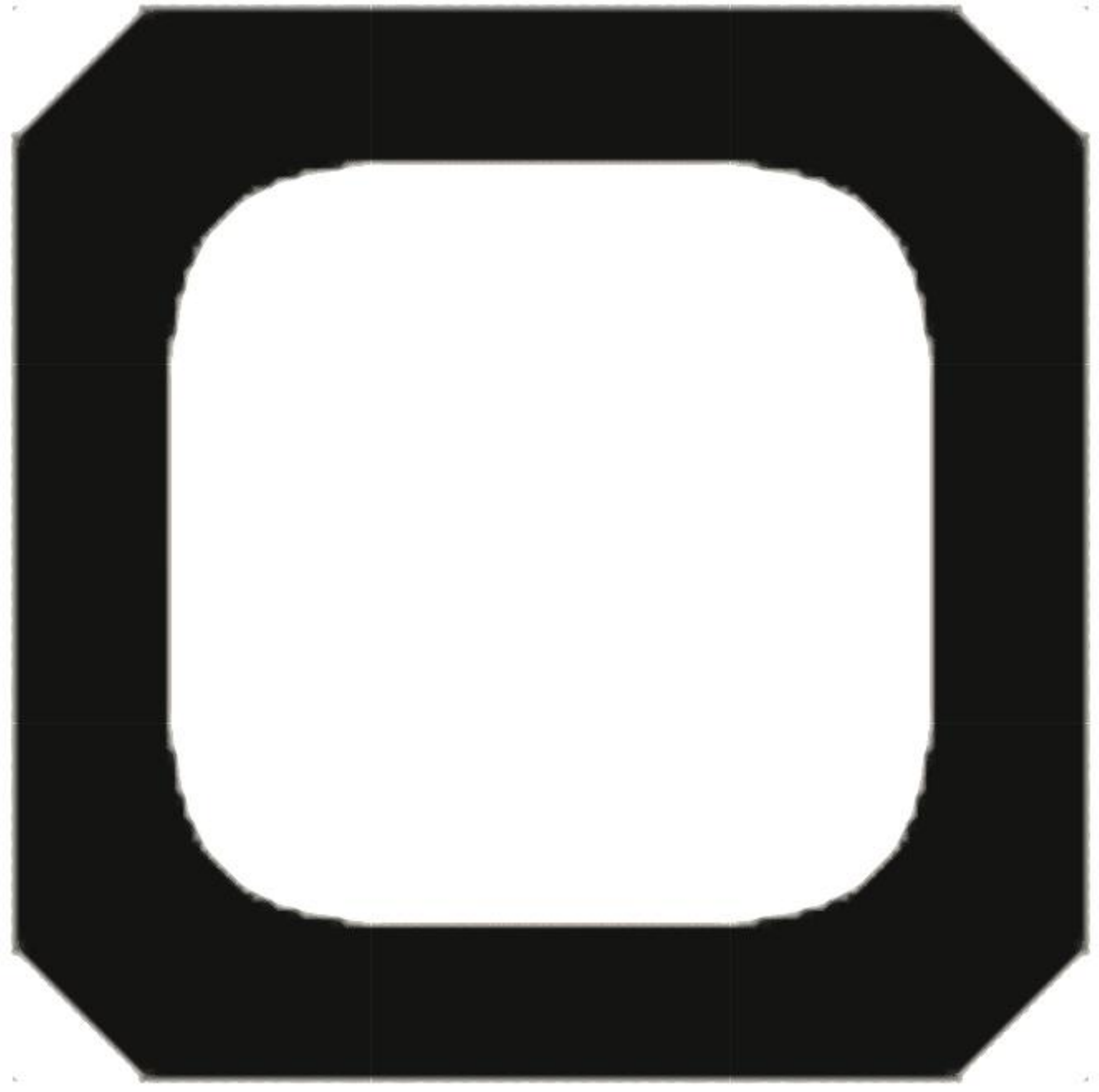}
        \subcaption{Step\,325$^{\#}$}
        \label{hbb-e}
      \end{minipage}
      \\
              \begin{minipage}[t]{0.2\hsize}
        \centering
        \includegraphics[keepaspectratio, scale=0.09]{hbb0.pdf}
        \subcaption{Step\,0}
        \label{hbb-f}
      \end{minipage} 
      \begin{minipage}[t]{0.2\hsize}
        \centering
        \includegraphics[keepaspectratio, scale=0.09]{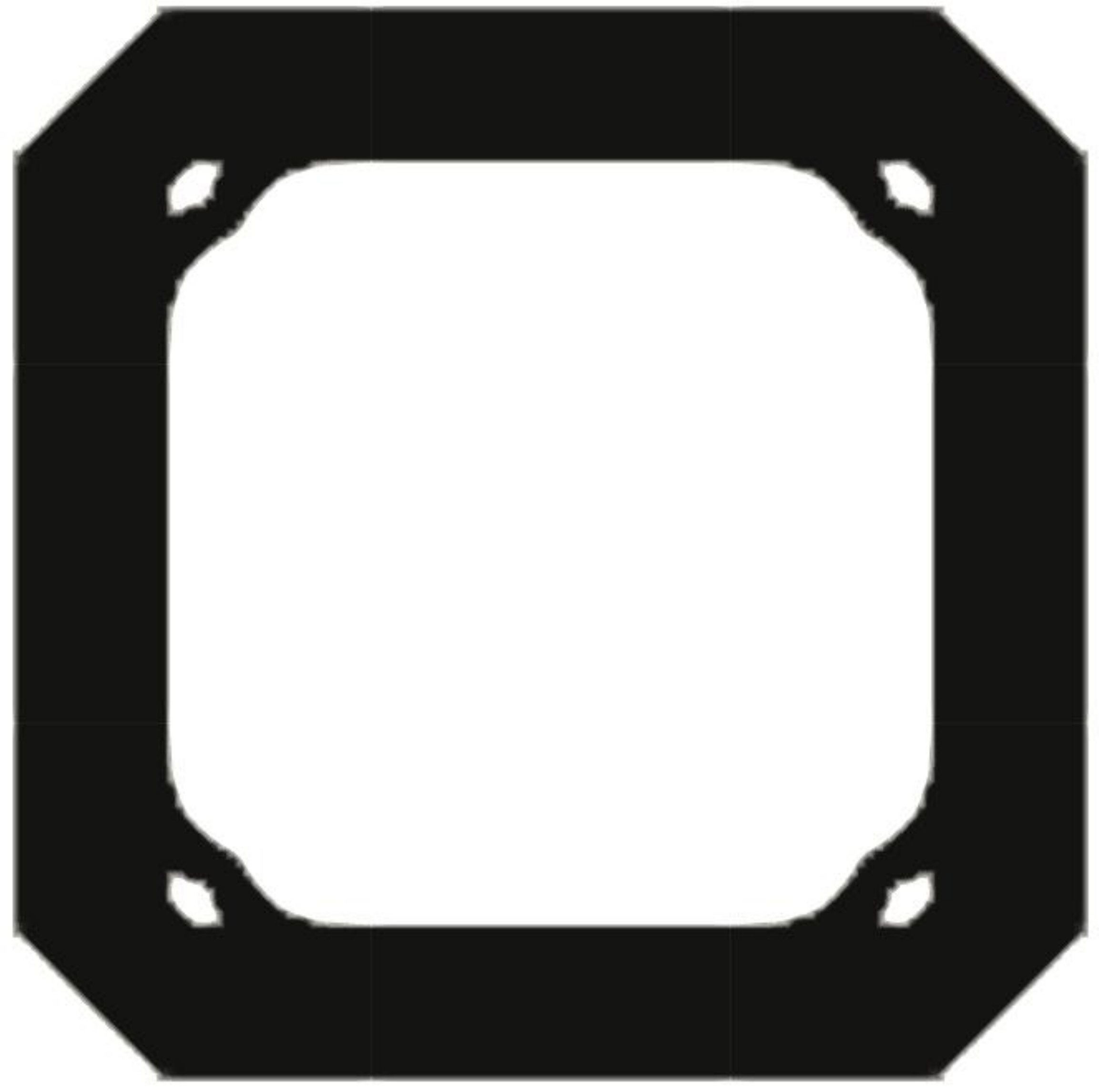}
        \subcaption{Step\,20}
        \label{hbb-g}
      \end{minipage} 
         \begin{minipage}[t]{0.2\hsize}
        \centering
        \includegraphics[keepaspectratio, scale=0.09]{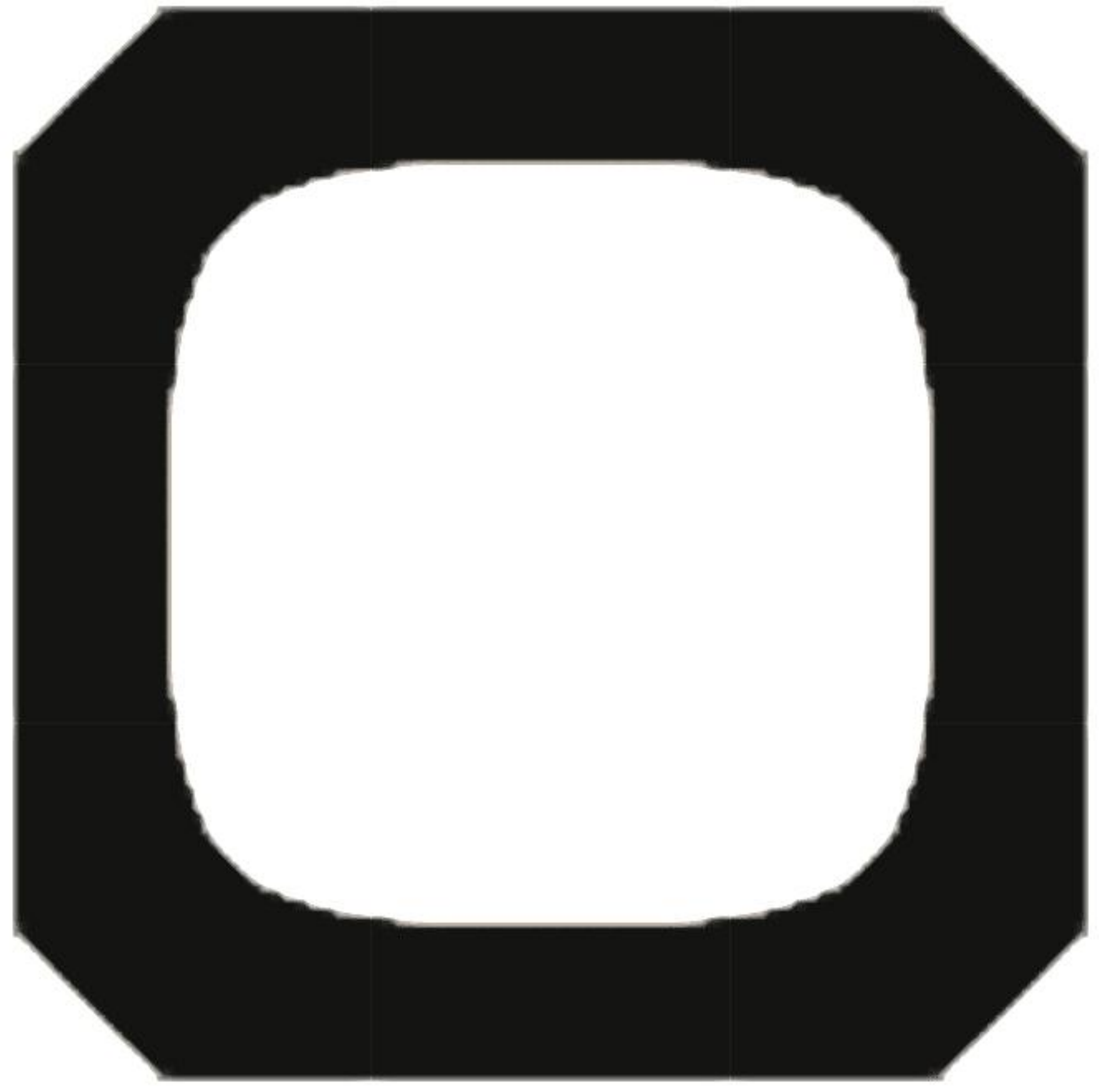}
        \subcaption{Step\,40}
        \label{hbb-h}
      \end{minipage}
      \begin{minipage}[t]{0.2\hsize}
        \centering
        \includegraphics[keepaspectratio, scale=0.09]{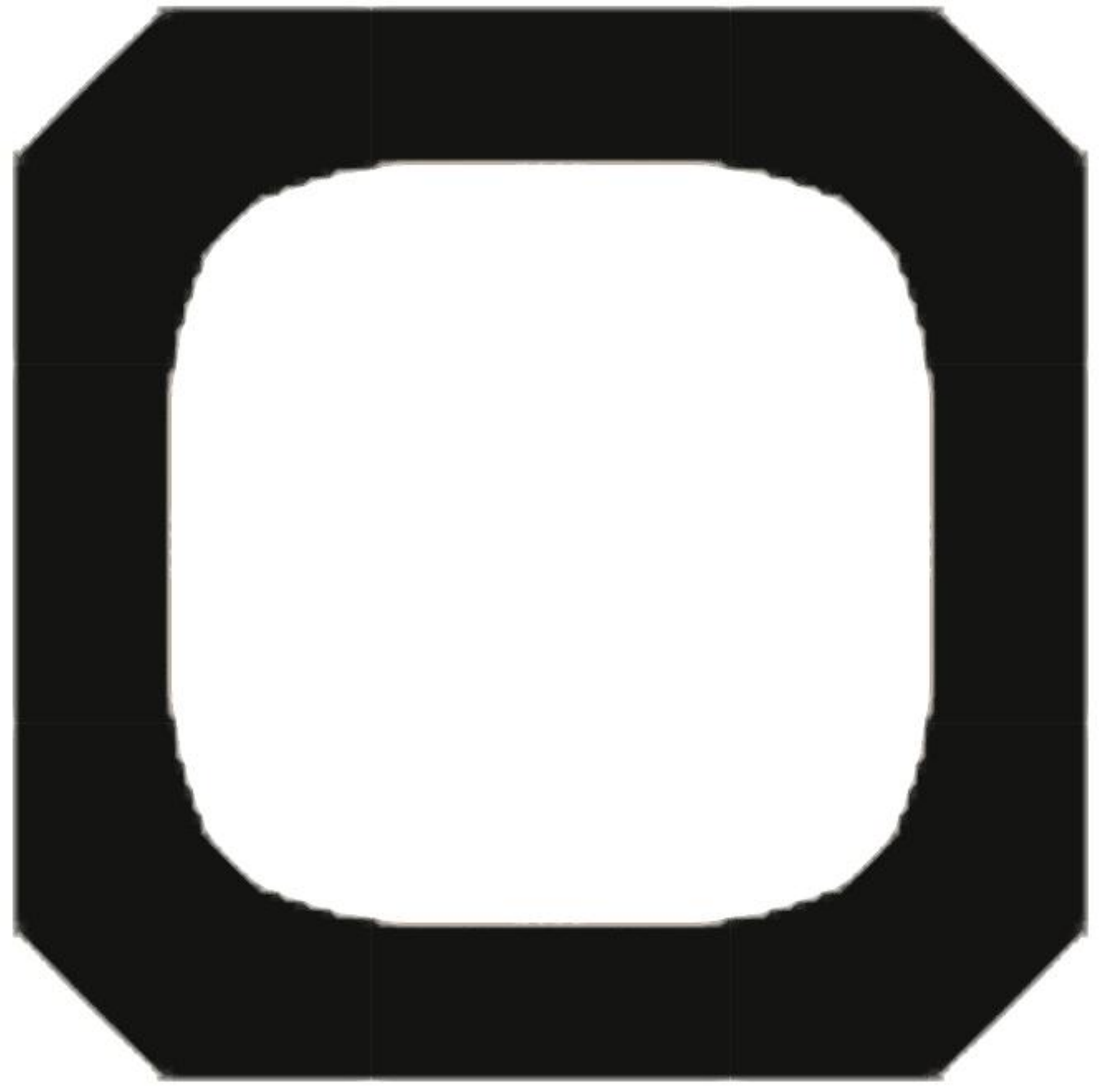}
        \subcaption{Step\,60}
        \label{hbb-i}
      \end{minipage}
           \begin{minipage}[t]{0.2\hsize}
        \centering
        \includegraphics[keepaspectratio, scale=0.09]{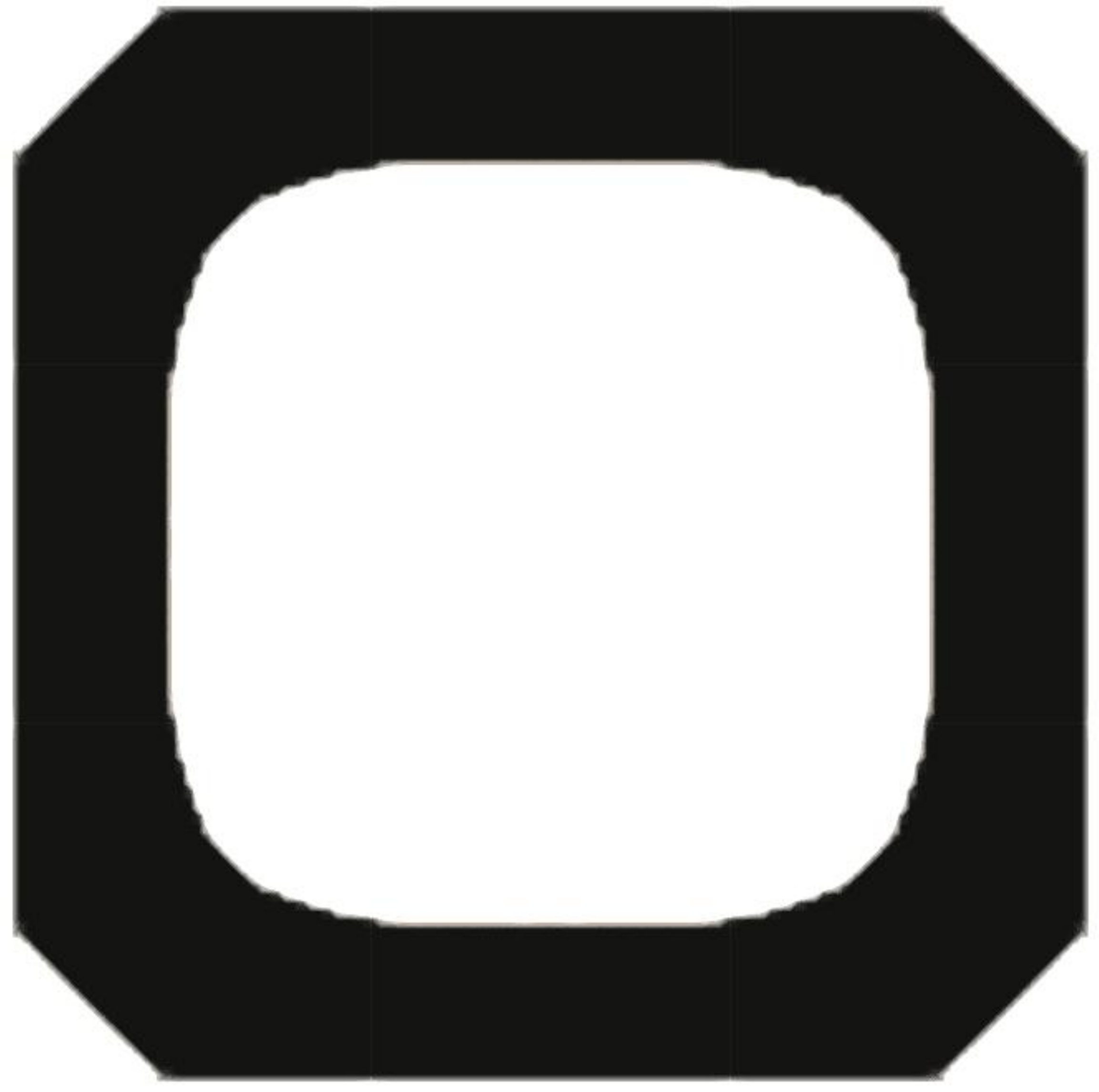}
        \subcaption{Step\,67$^{\#}$}
        \label{hbb-j}
      \end{minipage}
       \end{tabular}
     \caption{Configuration $\Omega_{\phi_n}\subset D$ for the case where the initial configuration is the perforated domain. 
Figures (a)--(e) and (f)--(j) 
represent $\Omega_{\phi_n}\subset D$ for $q=1$ and $q=6$ in \eqref{discNLD}, respectively. 
The symbol ${}^{ \#}$ implies the final step.    
     }
     \label{fig:hbb}
  \end{figure}

\begin{rmk}[Optimality]
\rm
We recall that a minimum value for the generalized problem is derived in terms of homogenization theory.
Actually, $H$-convergence theory ensures that
\begin{equation*}
\inf_{\phi\in H^1(D)}F(\phi)
= \min_{(\theta,\kappa_\theta)\in L^{\infty}(D;[-1,1]\times \R^{d\times d})} \left\{F^\ast(\theta,\kappa_\theta)
:=\int_{D} f(x)u_\theta(x)\, \d x\right\},
\end{equation*}
where $u_\theta \in H^1_0(D)$ is a unique weak solution to the following equation\/{\rm:}
$$
-\dv (\kappa_\theta\nabla u_\theta)=f\quad \text{ in } H^{-1}(D).
$$
Since $\nabla_{\kappa_\theta}F^\ast(\theta,\kappa_\theta)=-|\nabla u_\theta|^2\le 0$,
$\kappa_\theta\in L^{\infty}(D;\R^{d\times d})$ can be treated as the upper bound 
$(\alpha\theta+\beta(1-\theta))I$, and hence,~$F^\ast(\theta,\kappa_\theta)=F^\ast(\theta)$. 
Furthermore, with the aid of the so-called {\it dual energy method} (see, e.g.,~\cite[\S 2.4]{ACMOY19} for details), setting $\theta_{n}:=\theta_{n-1}-k|\nabla u_{\theta_{n-1}}|^2$ for some step with $k>0$ and $\theta^\ast=\lim_{n\to +\infty}\theta_{n}$, we have
\begin{equation*}
\inf_{\phi\in H^1(D)}F(\phi)=F^\ast(\theta^\ast).
\end{equation*}
Noting that $(F^{\ast})'(\theta)=-(\alpha-\beta)|\nabla u_\theta|^2\ge 0$ and Figures \ref{hb-e}, \ref{hb-j}, \ref{hbb-e} and \ref{hbb-j} are similar to \cite[Fig.~23]{ACMOY19}, one can conclude that the proposed method yields an optimal configuration by verifying $F^{\ast}(\theta^{\ast})\approx F(\phi_n)$.
\end{rmk}

On the other hand, let us next consider the opposite heat conductivity, i.e.,~$\alpha=1.0$ and $\beta=1.0\times 10^{-2}$. 
Here we set $\Gamma_D\subset \partial D$ as in Figure \ref{ih} and $(\tau,G_{\rm max},\varDelta t)=(1.0\times 10^{-4},0.4,0.4)$.
Then, from Figures \ref{hs-a}-\ref{hs-e}, \ref{hs-f}-\ref{hs-j} and \ref{fig:h2}, it can be seen that the convergence for configurations is improved by the method using slow diffusion. 
Although the shapes of the optimized configurations are different, Figure \ref{fig:h1} shows that the objective functionals converge to almost the same values, which means that the same assertion as in (i-FDE) can also be obtained for the method using slow diffusion.
This is due to the contribution of sensitivity in the void domains and the finite propagation property based on the degeneracy of the diffusion coefficient for nonlinear diffusion near the optimal configuration;
indeed, from Figure \ref{hs-g}-\ref{hs-h}, one can confirm that the sensitivity of the void domains plays a crucial role since the degeneracy of the diffusion coefficient prevents the propagation of the interface. Furthermore, due to the degeneracy, it is suggested that preventing propagation for the interface near the optimal configuration makes it easier to satisfy the convergence condition \eqref{eq:CC}.

As for the high-contrast case $(\alpha,\beta)=(1.0, 1.0\times 10^{-3})$, 
we note that the so-called fin appears remarkably near the boundary structures, which occurs oscillation (see Figures \ref{fig:hs} and \ref{fig:h3}), and moreover, Figure \ref{fig:h3} shows that the configuration $\Omega_{\phi_n}\subset D$ for $q=1$
does not converge due to the oscillation near the boundary structures. 
On the other hand, Figure \ref{fig:h3} also shows that the method using slow diffusion 
enables optimization of the configuration $\Omega_{\phi_n}\subset D$ (see also Figure \ref{hs-k}-\ref{hs-o}), which completes the confirmation of (ii-SDE).

\begin{figure}[htbp]
   \hspace*{-5mm} 
    \begin{tabular}{ccccc}
        \begin{minipage}[t]{0.2\hsize}
        \centering
        \includegraphics[keepaspectratio, scale=0.11]{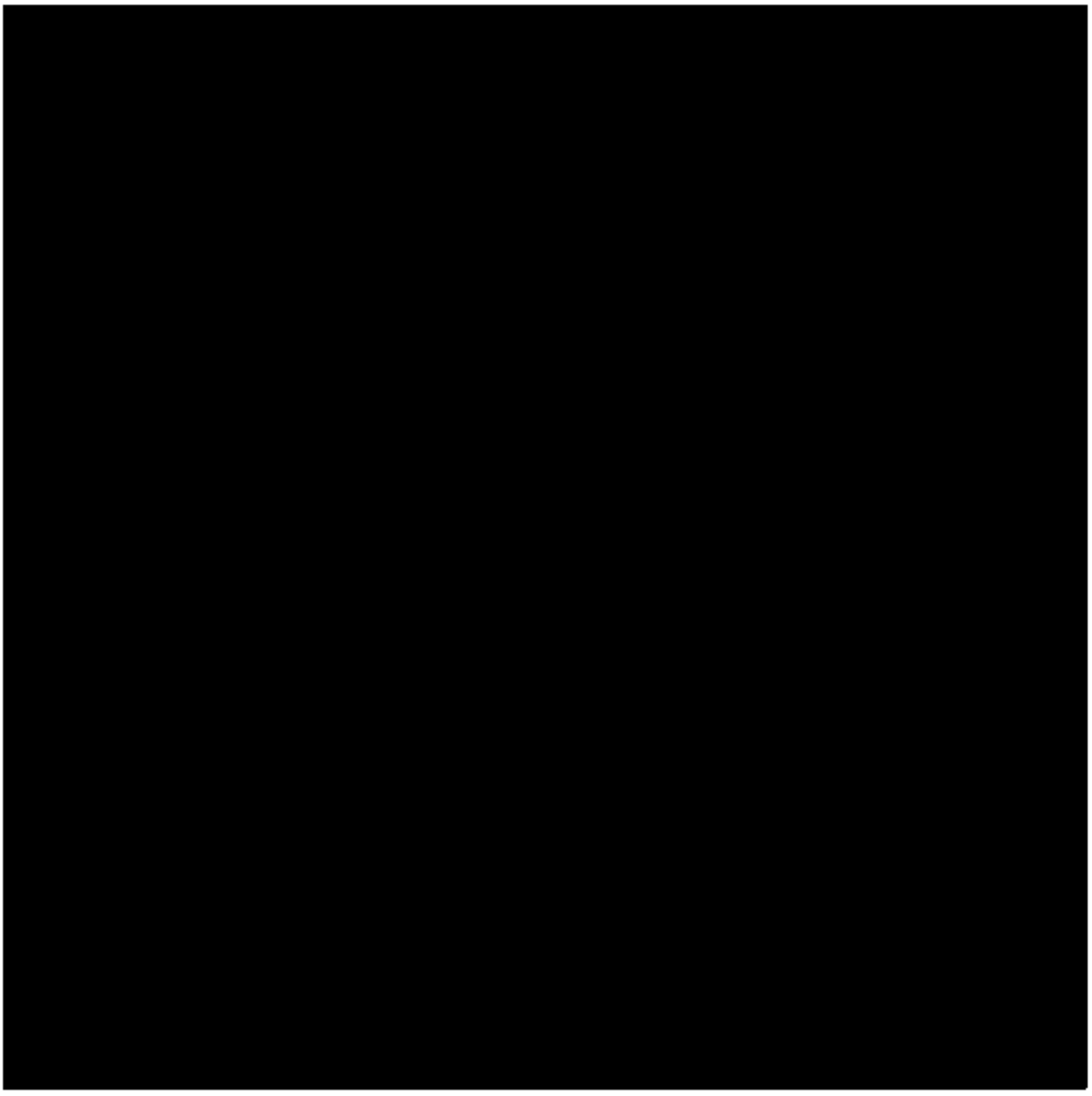}
        \subcaption{Step\,0}
        \label{hs-a}
      \end{minipage} 
      \begin{minipage}[t]{0.2\hsize}
        \centering
        \includegraphics[keepaspectratio, scale=0.11]{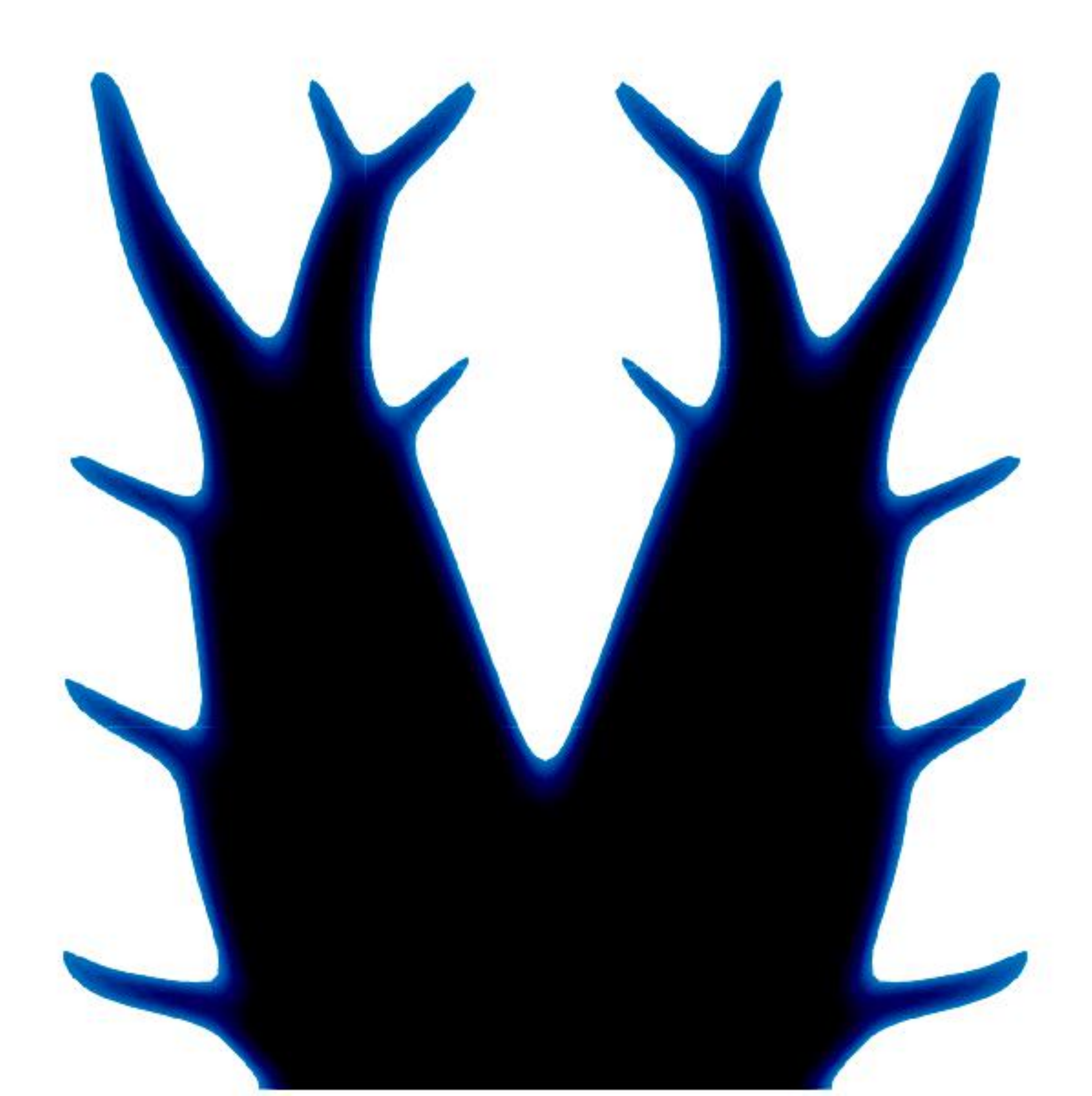}
        \subcaption{Step\,100}
        \label{hs-b}
      \end{minipage} 
         \begin{minipage}[t]{0.2\hsize}
        \centering
        \includegraphics[keepaspectratio, scale=0.11]{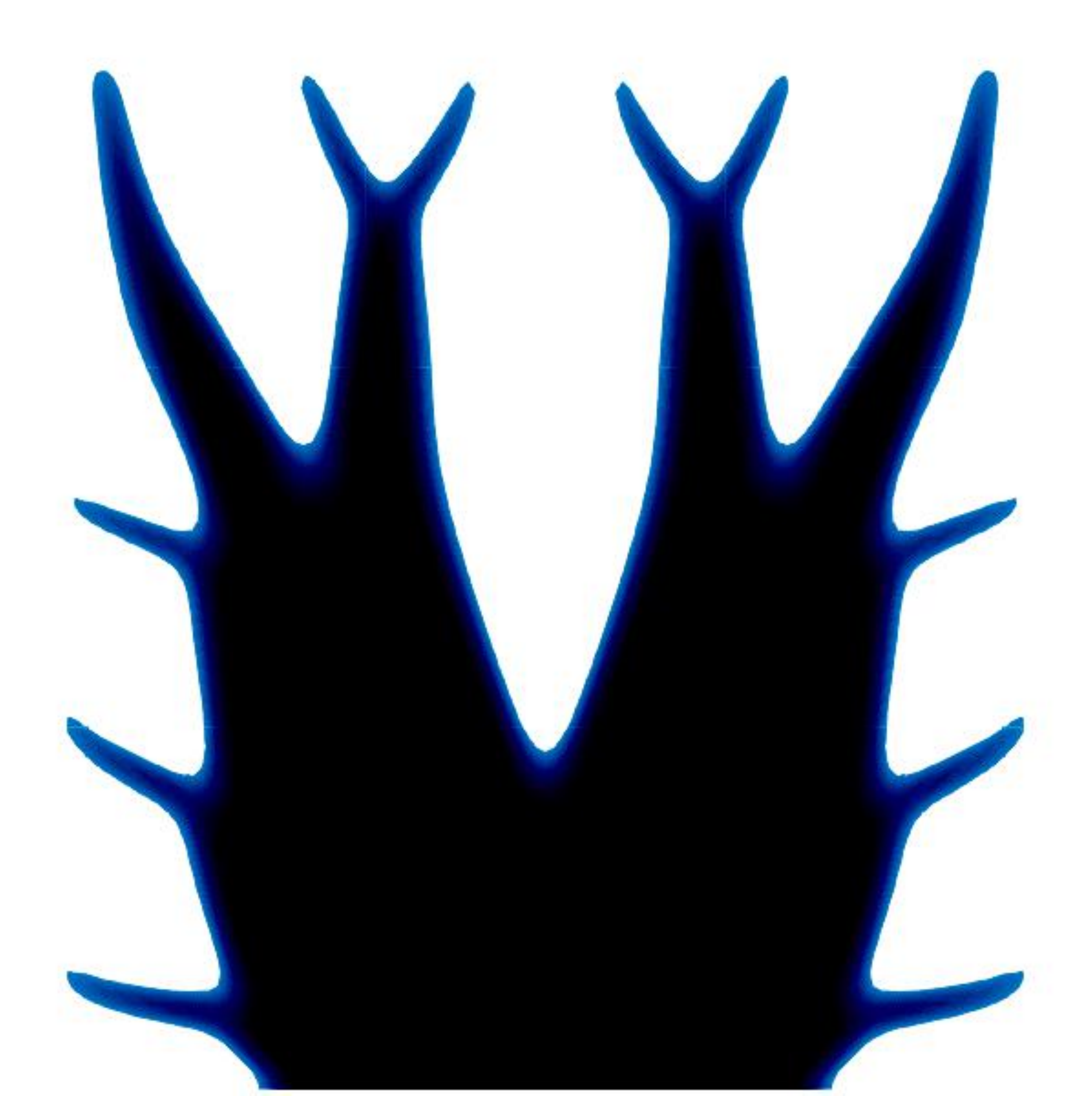}
        \subcaption{Step\,200}
        \label{hs-c}
      \end{minipage}
      \begin{minipage}[t]{0.2\hsize}
        \centering
        \includegraphics[keepaspectratio, scale=0.11]{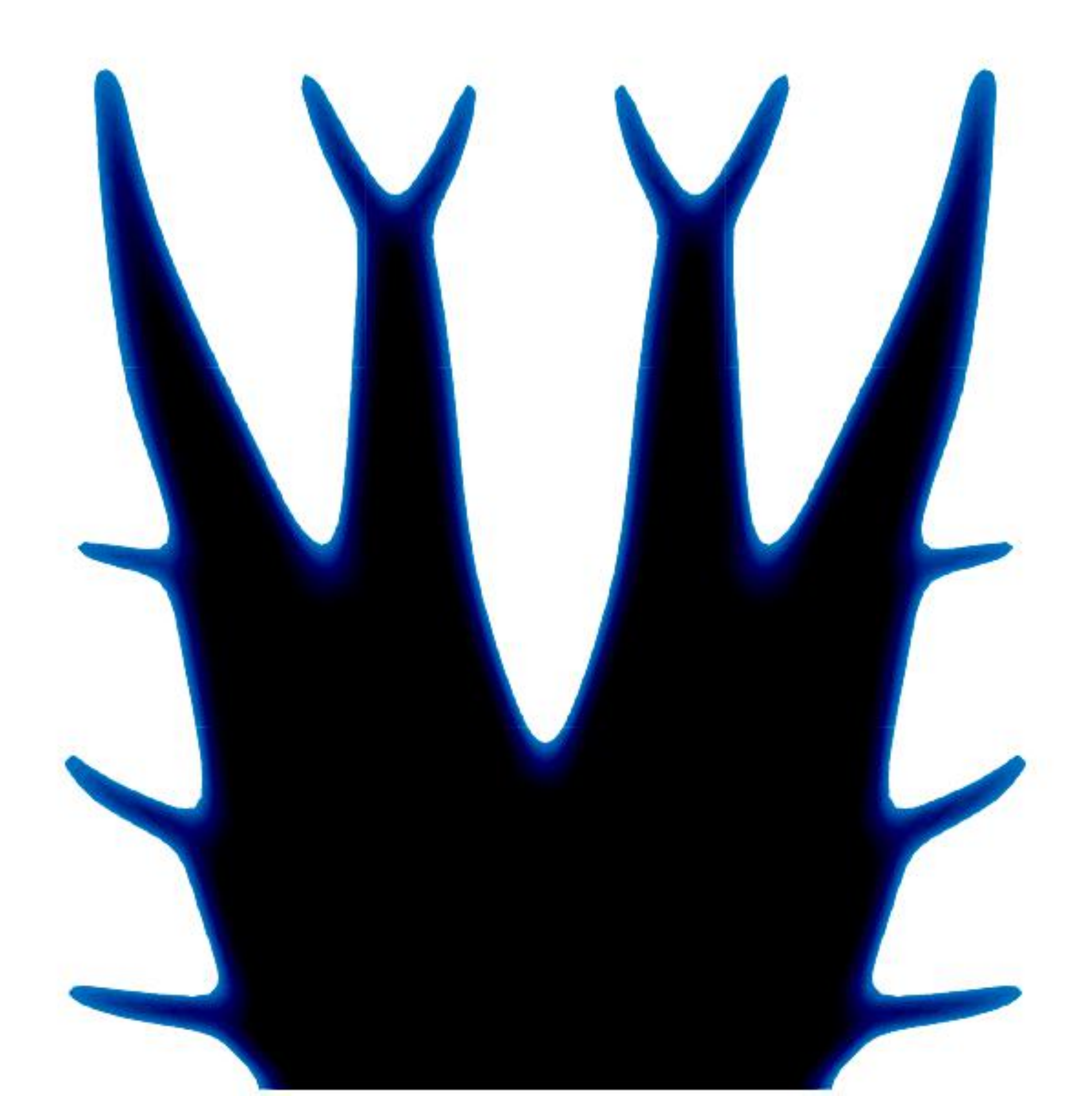}
        \subcaption{Step\,300}
        \label{hs-d}
      \end{minipage}
           \begin{minipage}[t]{0.2\hsize}
        \centering
        \includegraphics[keepaspectratio, scale=0.11]{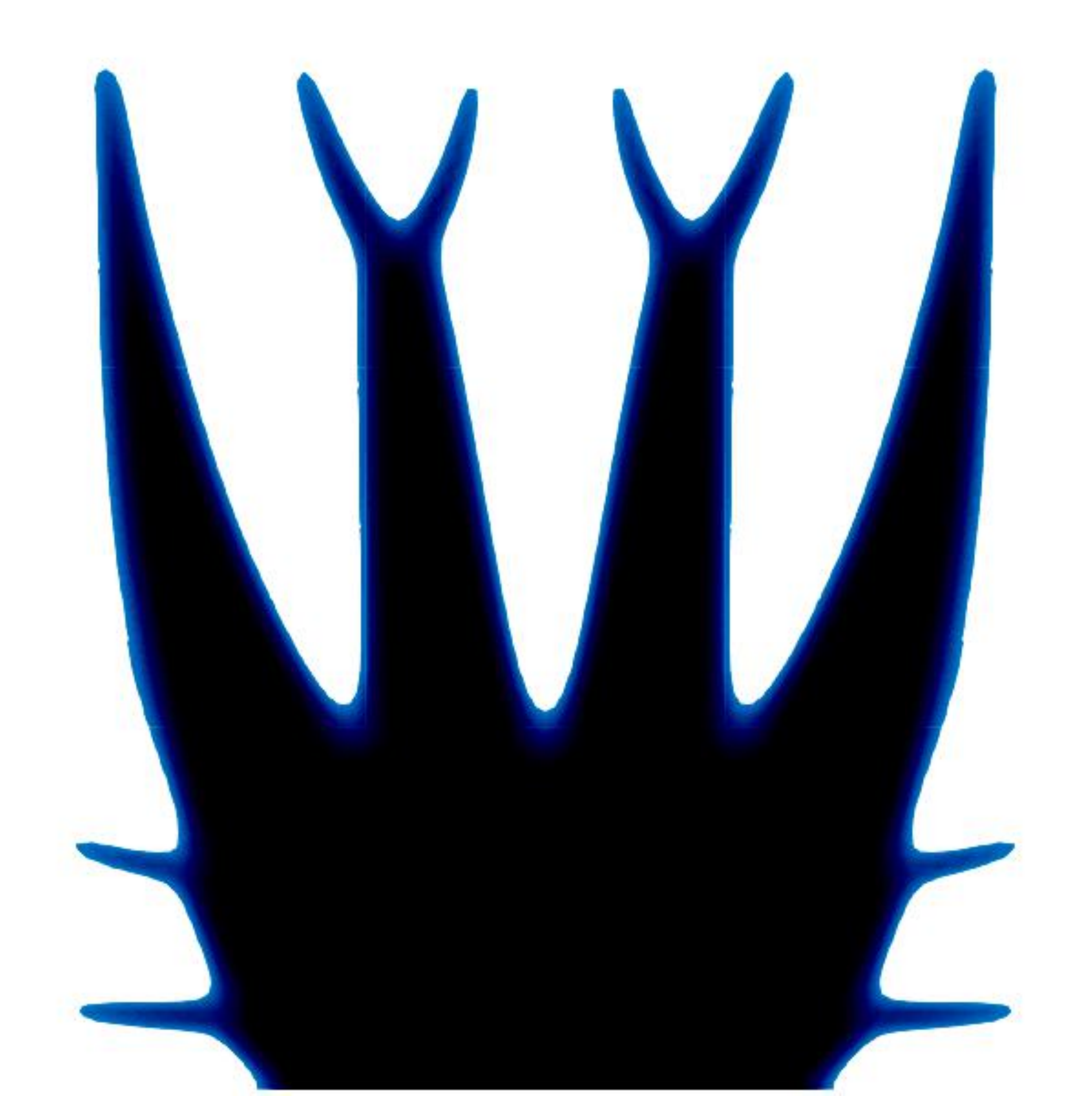}
        \subcaption{Step\,594$^{\#}$}
        \label{hs-e}
      \end{minipage}
      \\
          \begin{minipage}[t]{0.2\hsize}
        \centering
        \includegraphics[keepaspectratio, scale=0.11]{hs0.pdf}
        \subcaption{Step\,0}
        \label{hs-f}
      \end{minipage} 
      \begin{minipage}[t]{0.2\hsize}
        \centering
        \includegraphics[keepaspectratio, scale=0.11]{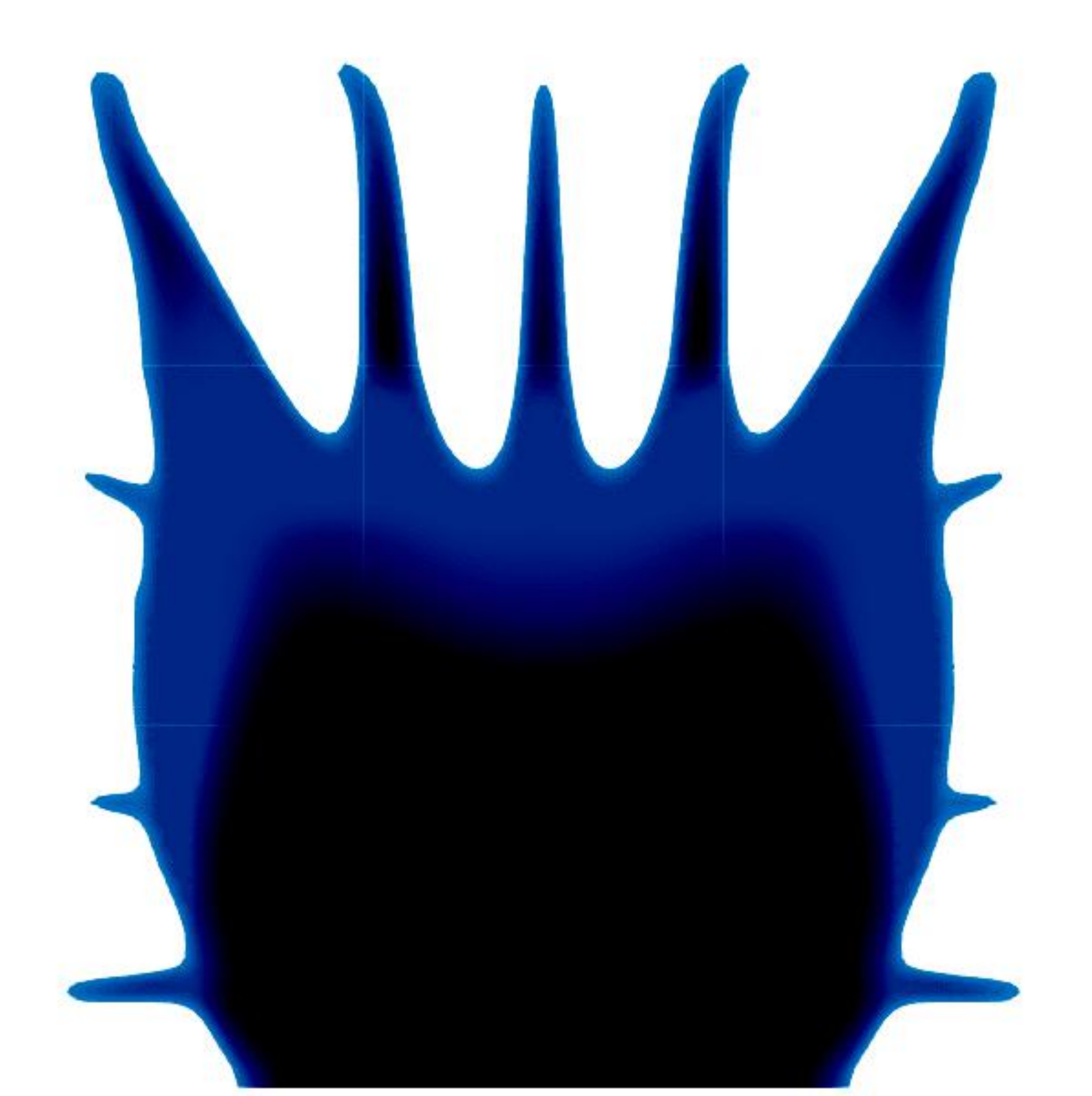}
        \subcaption{Step\,100}
        \label{hs-g}
      \end{minipage} 
         \begin{minipage}[t]{0.2\hsize}
        \centering
        \includegraphics[keepaspectratio, scale=0.11]{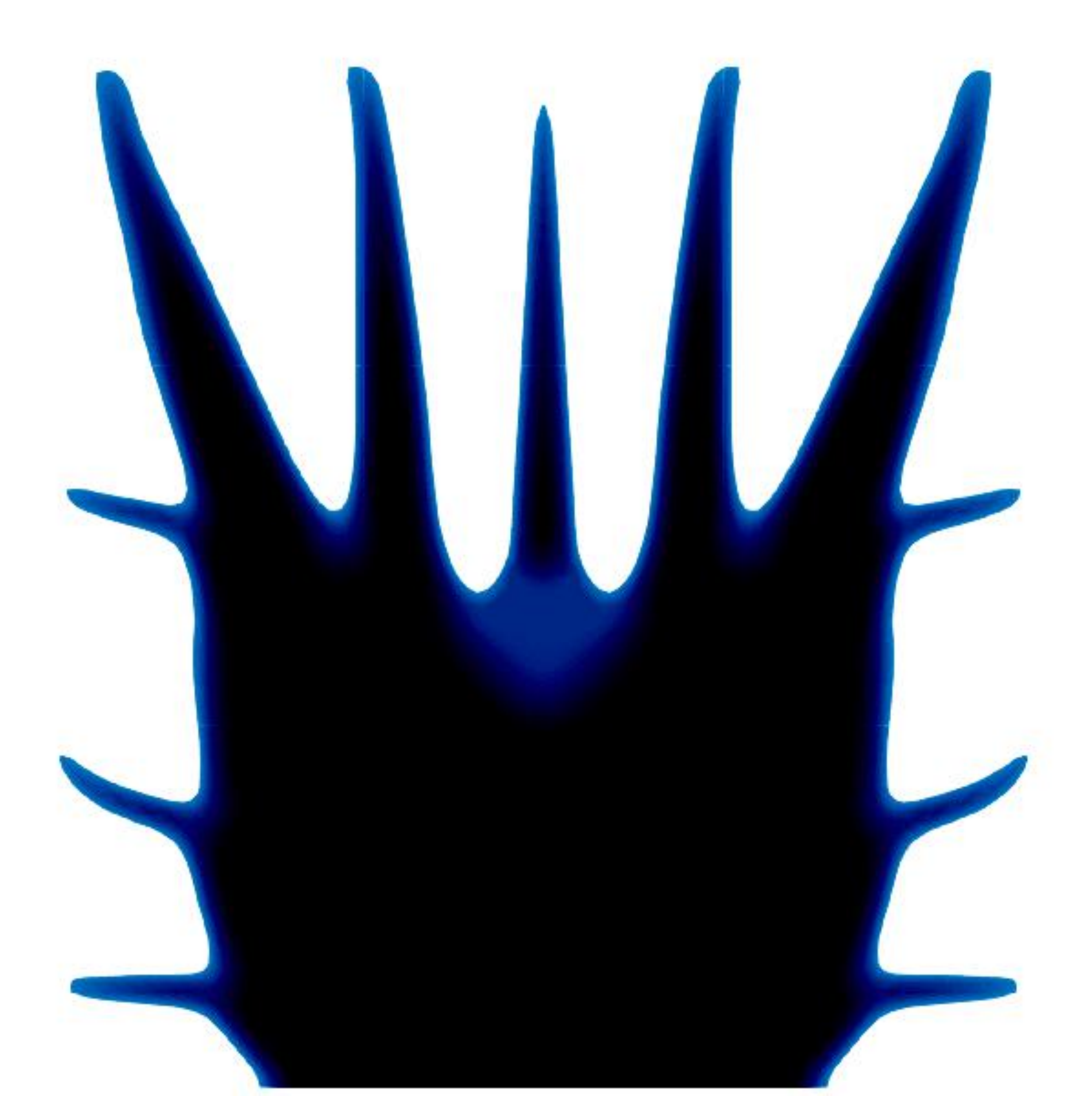}
        \subcaption{Step\,200}
        \label{hs-h}
      \end{minipage}
      \begin{minipage}[t]{0.2\hsize}
        \centering
        \includegraphics[keepaspectratio, scale=0.11]{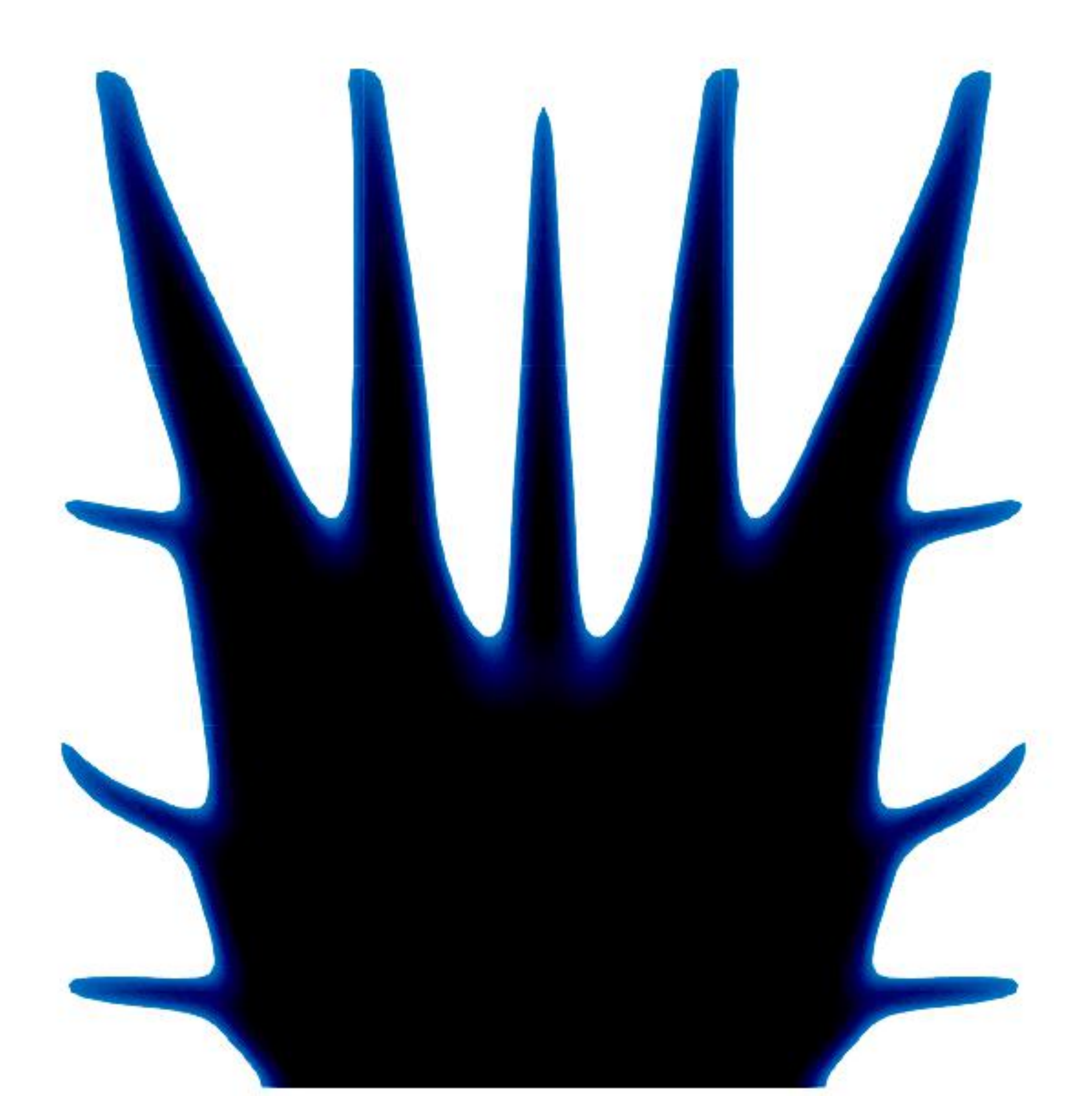}
        \subcaption{Step\,300}
        \label{hs-i}
      \end{minipage}
           \begin{minipage}[t]{0.2\hsize}
        \centering
        \includegraphics[keepaspectratio, scale=0.11]{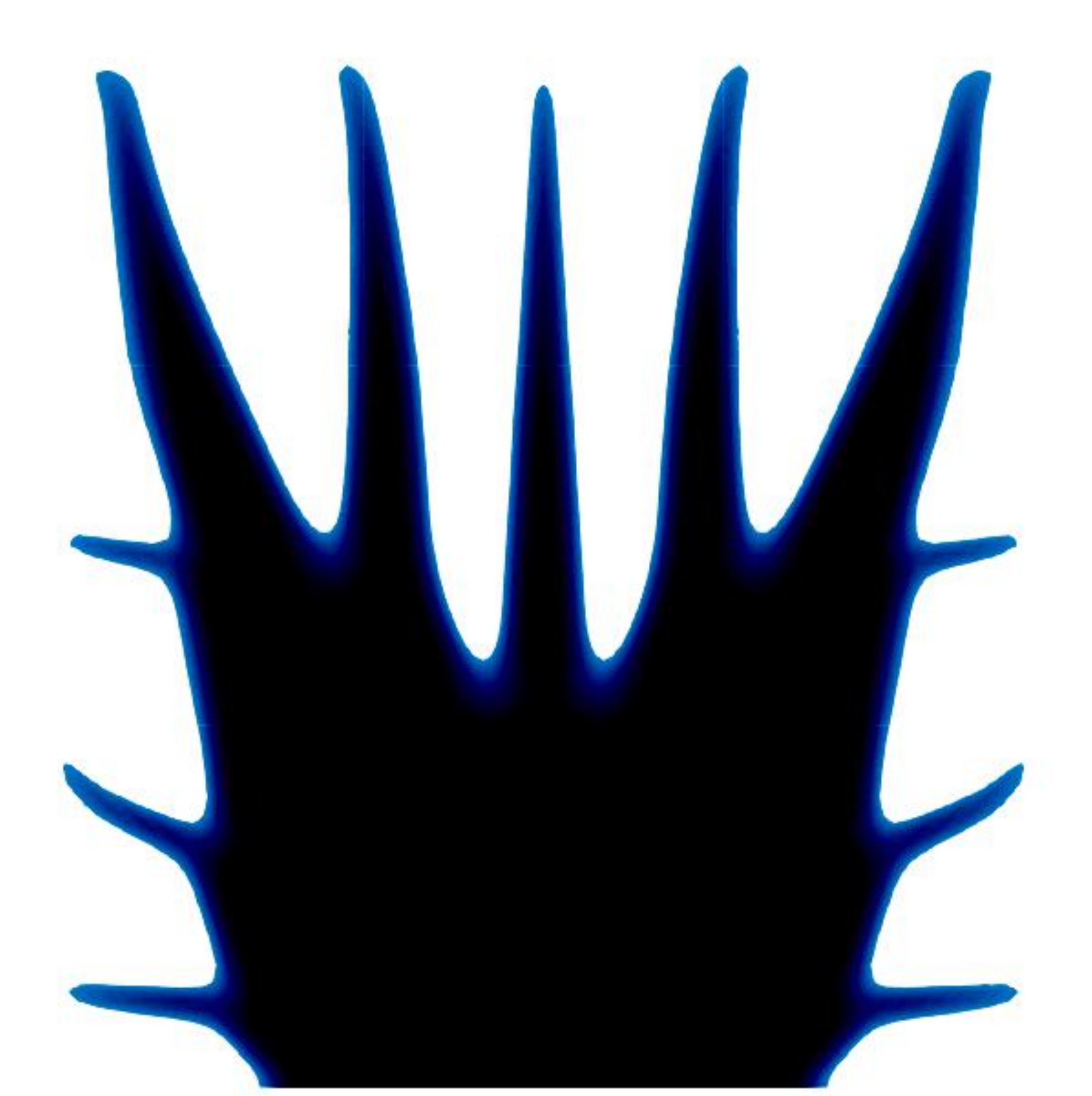}
        \subcaption{Step\,324$^{\#}$}
        \label{hs-j}
      \end{minipage}
      \\
   \begin{minipage}[t]{0.2\hsize}
        \centering
        \includegraphics[keepaspectratio, scale=0.11]{hs0.pdf}
        \subcaption{Step\,0}
        \label{hs-k}
      \end{minipage} 
      \begin{minipage}[t]{0.2\hsize}
        \centering
        \includegraphics[keepaspectratio, scale=0.11]{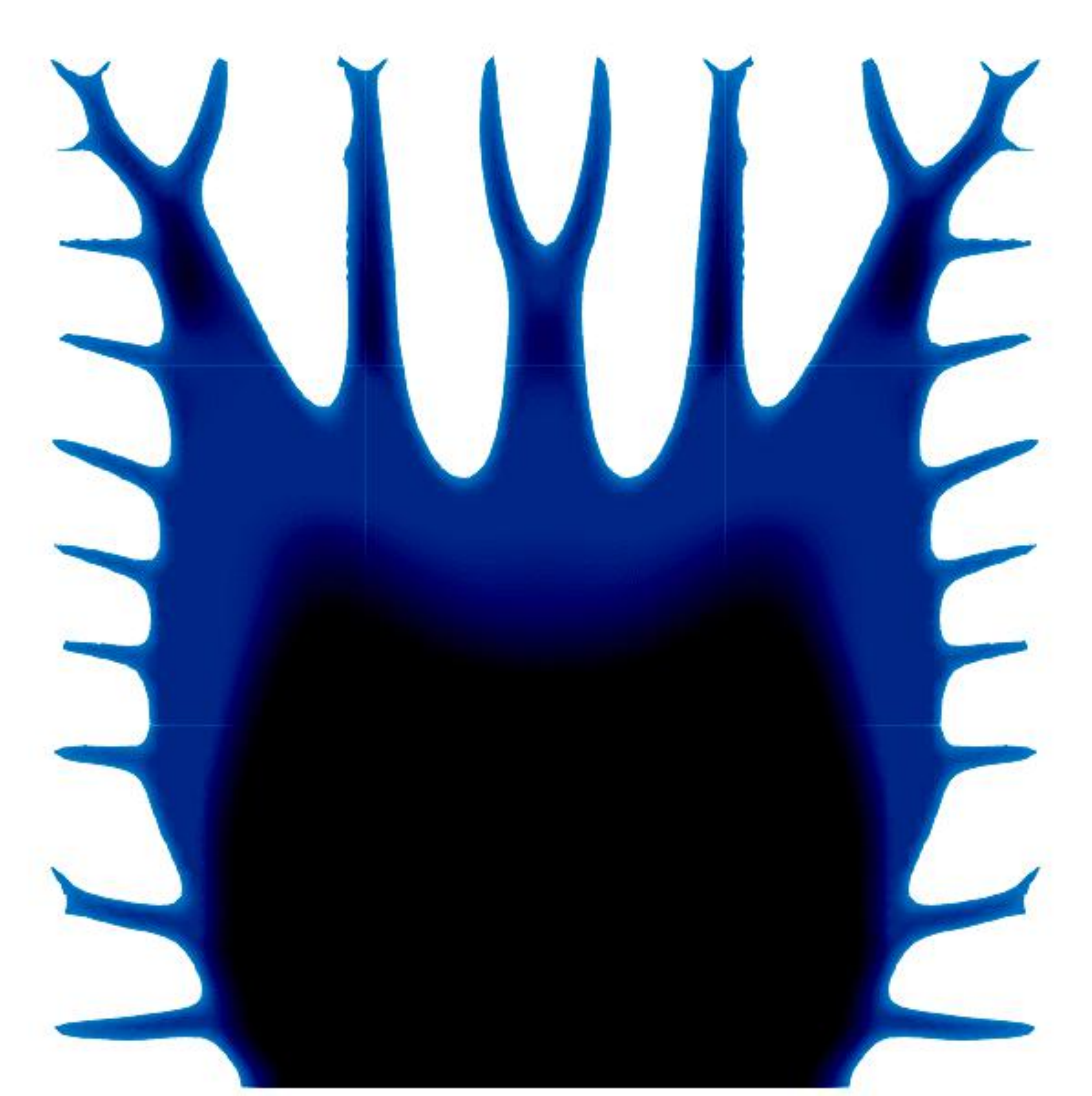}
        \subcaption{Step\,100}
        \label{hs-l}
      \end{minipage} 
         \begin{minipage}[t]{0.2\hsize}
        \centering
        \includegraphics[keepaspectratio, scale=0.11]{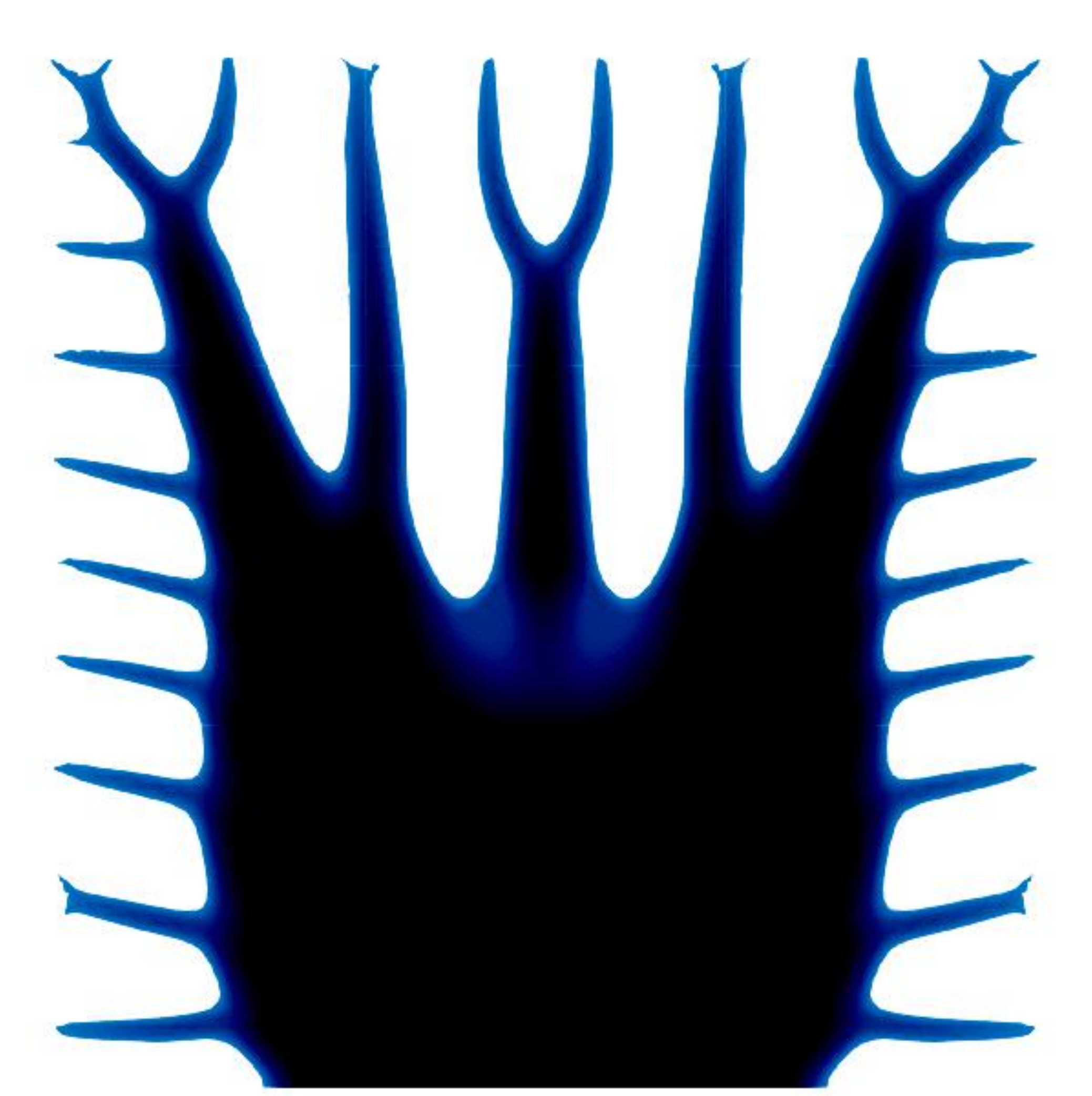}
        \subcaption{Step\,200}
        \label{hs-m}
      \end{minipage}
      \begin{minipage}[t]{0.2\hsize}
        \centering
        \includegraphics[keepaspectratio, scale=0.11]{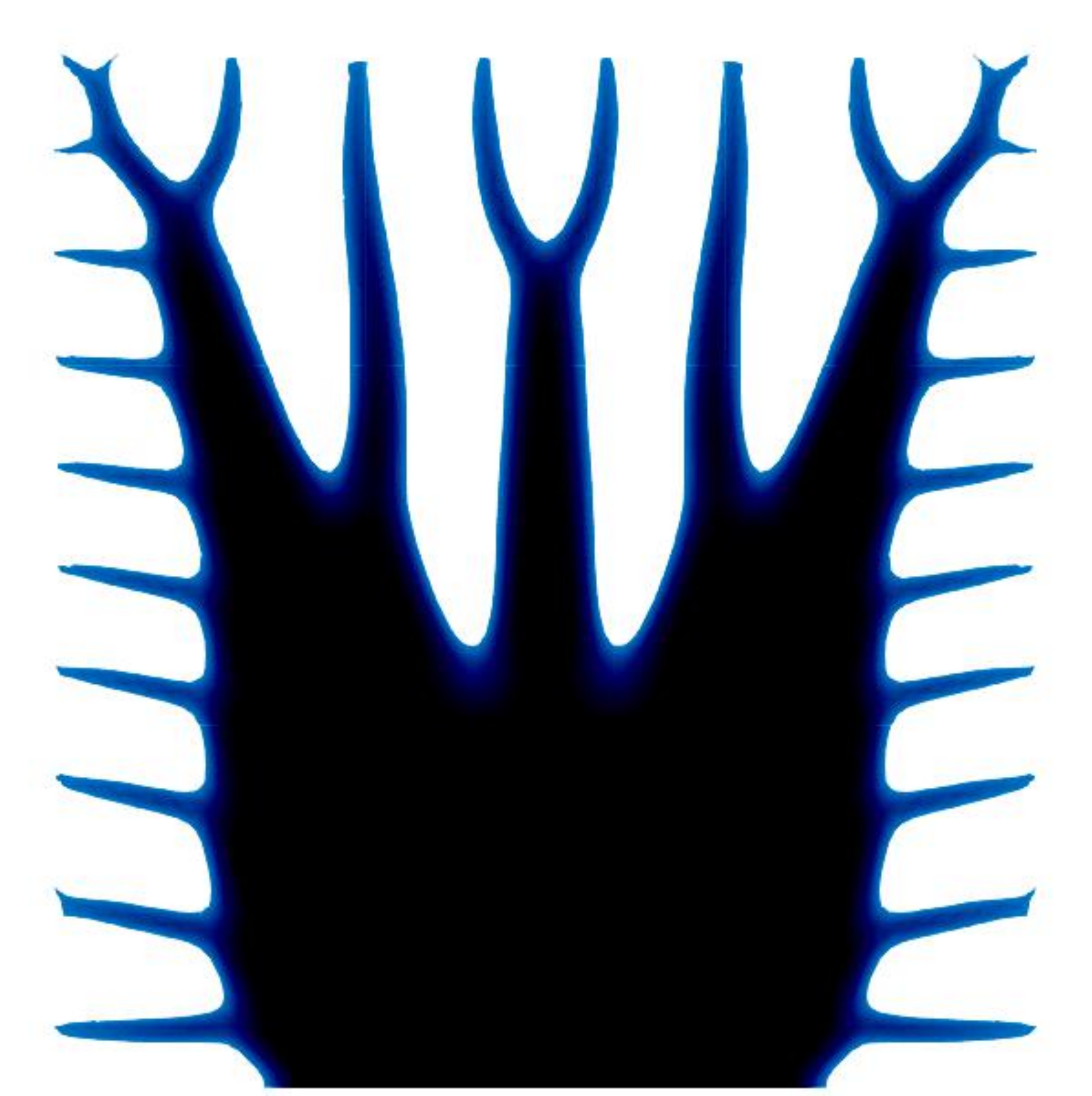}
        \subcaption{Step\,300}
        \label{hs-n}
      \end{minipage}
           \begin{minipage}[t]{0.2\hsize}
        \centering
        \includegraphics[keepaspectratio, scale=0.11]{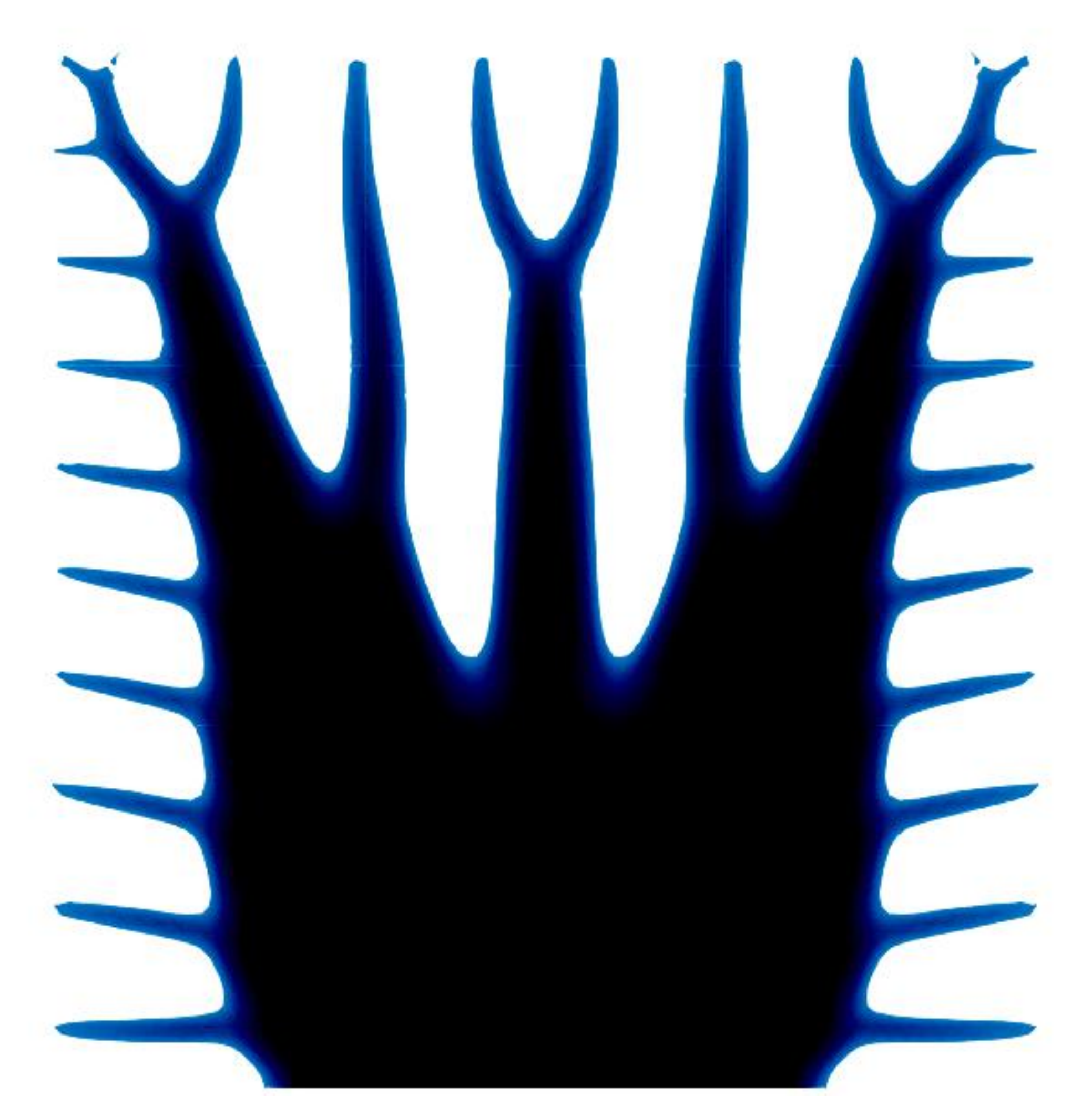}
        \subcaption{Step\,387$^{\#}$}
        \label{hs-o}
      \end{minipage}
              \end{tabular}
     \caption{ 
     Black and blue domains indicate material domain $[\phi_n\ge 0]=\Omega_{\phi_n}\subset D$ and  partial void domain $[-0.5\le \phi_n< 0]\subset D$, respectively. 
Figures (a)--(e), (f)--(j) and (k)--(o) represent domains for $(q,r)=(1.0,2.0)$, $(q,r)=(0.1,2.0)$  and $(q,r)=(0.1,3.0)$ in \eqref{discNLD}, respectively. Here $r>0$ is a power of $(\alpha,\beta)=(1.0, 1.0\times 10^{-r})$.
The symbol ${}^{ \#}$ implies the final step.    
     }
     \label{fig:hs}
  \end{figure}

\begin{figure}[htbp]
        \centering
        \includegraphics[keepaspectratio, scale=0.33]{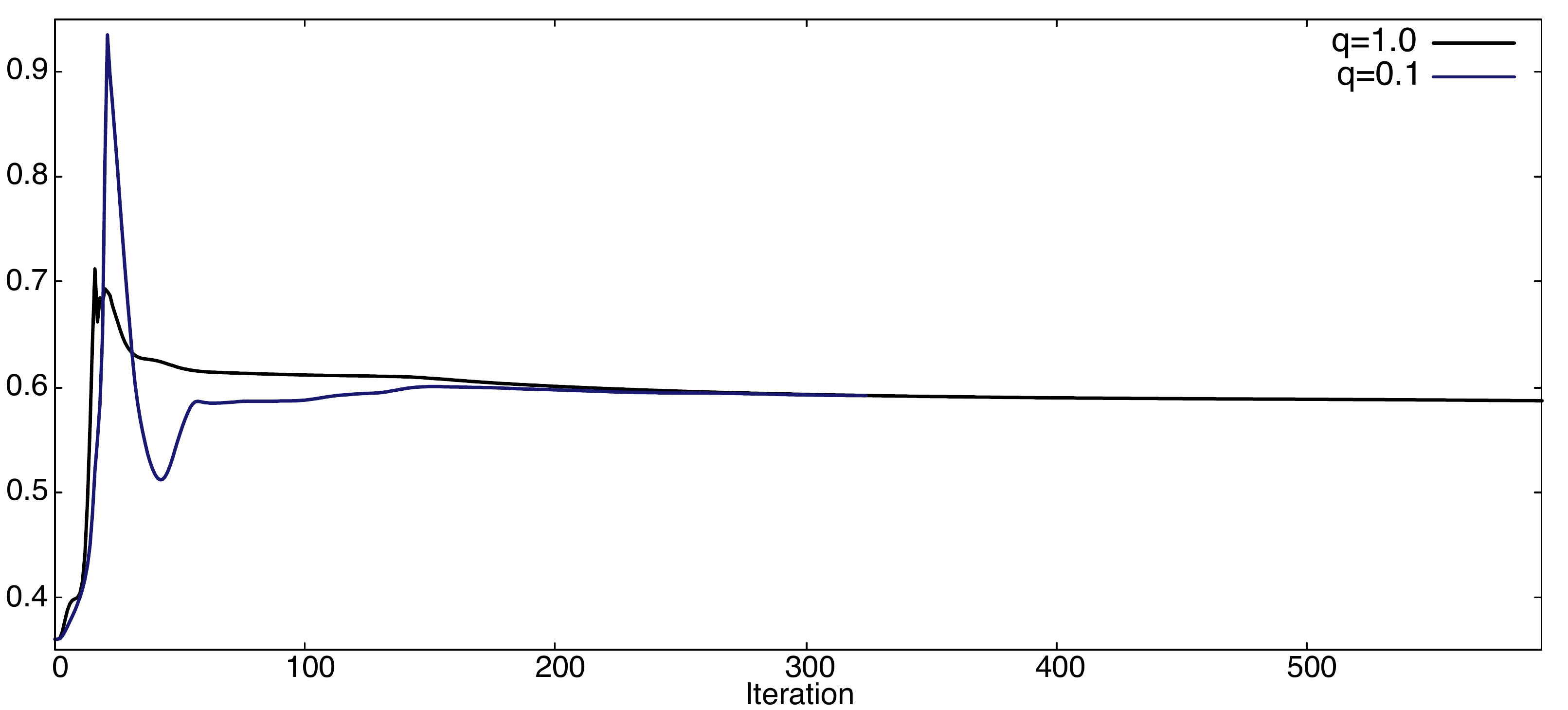}
   \caption{Objective functional $F(\phi_n)$ with $(\alpha,\beta)=(1.0,1.0\times 10^{-2})$.}
\label{fig:h1}
\end{figure}       

\begin{figure}[htbp]
        \centering
        \includegraphics[keepaspectratio, scale=0.33]{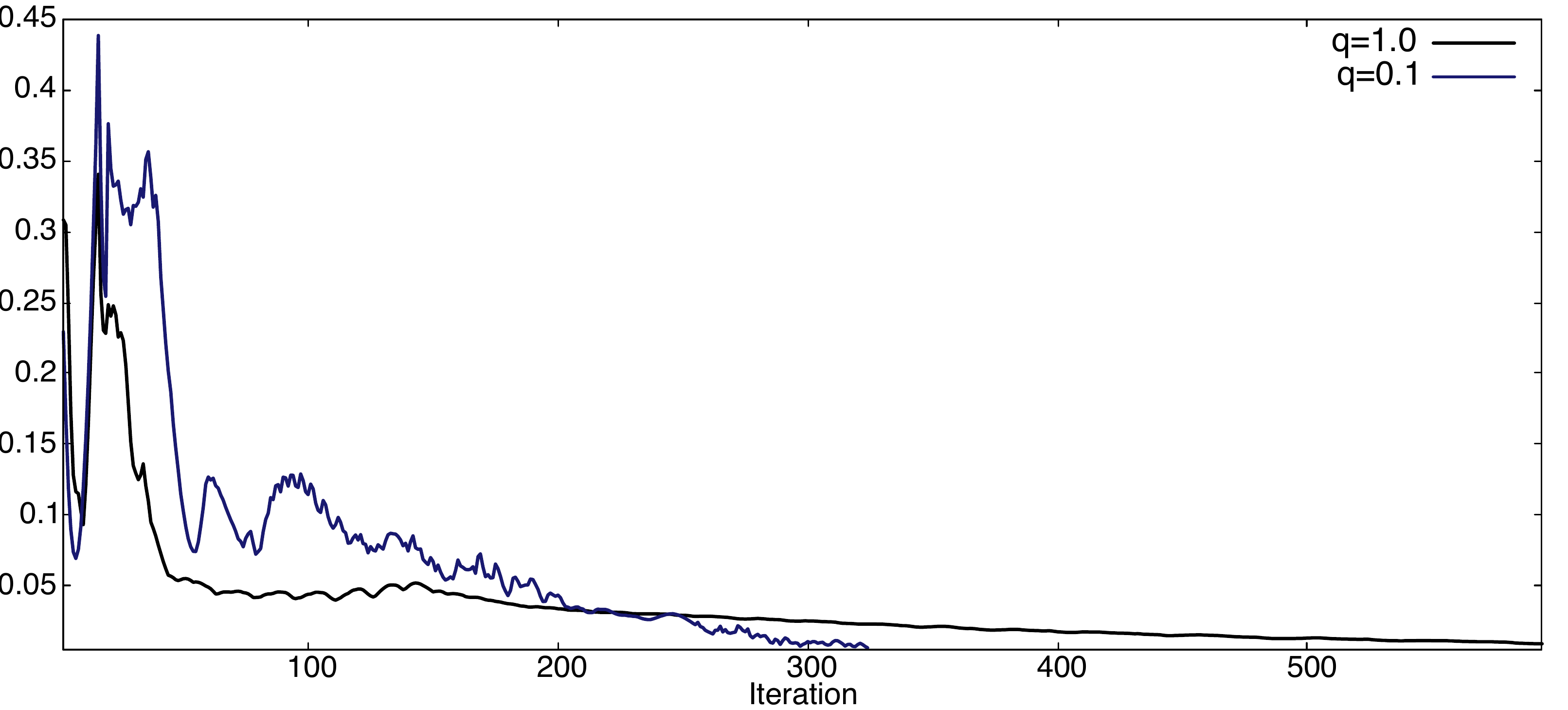}
   \caption{Convergence condition $\|\phi_{n+1}-\phi_{n}\|_{L^{\infty}(D)}$ with $(\alpha,\beta)=(1.0,1.0\times 10^{-2})$.}
\label{fig:h2}
\end{figure}       

\begin{figure}[htbp]
        \centering
        \includegraphics[keepaspectratio, scale=0.33]{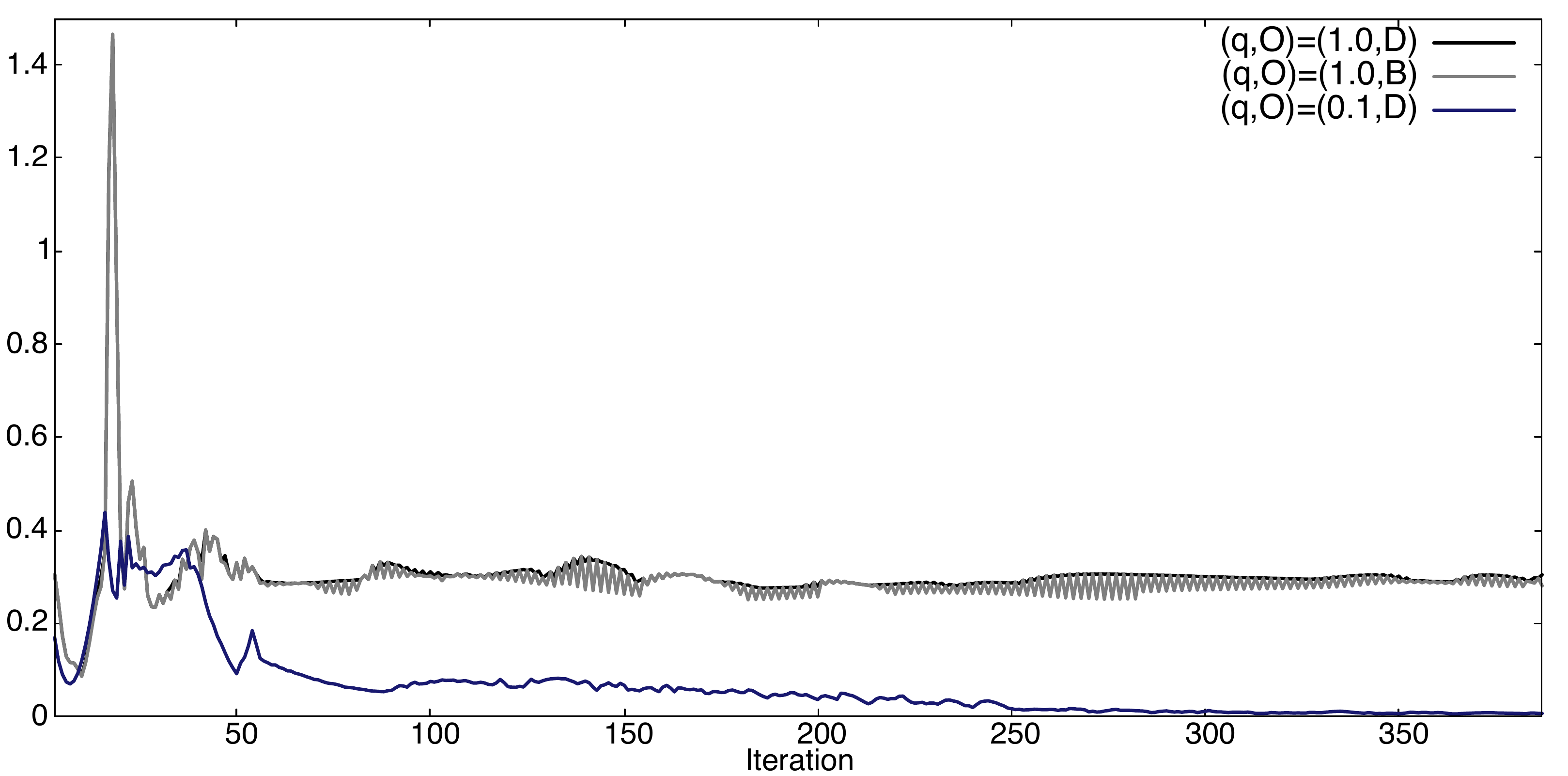}
   \caption{Convergence condition $\|\phi_{n+1}-\phi_{n}\|_{L^{\infty}(O)}$ with $(\alpha,\beta)=(1.0,1.0\times 10^{-3})$. Here $B\subset D$ denotes $[-0.5\le \phi_n <0]$.}
\label{fig:h3}
\end{figure}

Furthermore, (i-FDE) and (ii-SDE) can also be confirmed for the three-dimensional case; indeed,
it can be seen by Figures \ref{3dh-a}-\ref{3dh-d} and \ref{3dh-i}-\ref{3dh-l} that
convergence to the optimal configuration is slightly improved (see Figure \ref{fig:3dh1}), and moreover, Figure \ref{fig:3dh2} shows that, by increasing $\varDelta t>0 $, the method for $q=1$ (i.e.,~reaction-diffusion) does not converge due to oscillation near the boundary structure, but that for $q=0.5$ (i.e.,~slow diffusion) enables the configuration $\Omega_{\phi_n}\subset D$ to be optimized (see also Figure \ref{3dh-e}-\ref{3dh-h}). 
Here $D\subset \R^3$ and $\Gamma_D\subset \partial D$ are the cube  with a side length of $0.3$ and the square with a side length of $0.1$, respectively, and we put $G_{\rm max}=0.2$ and $(\alpha,\beta)=(1.0,1.0\times 10^{-3})$.

 \begin{figure}[htbp]
   \hspace*{-5mm} 
    \begin{tabular}{ccccc}
    
        \begin{minipage}[t]{0.24\hsize}
        \centering
        \includegraphics[keepaspectratio, scale=0.12]{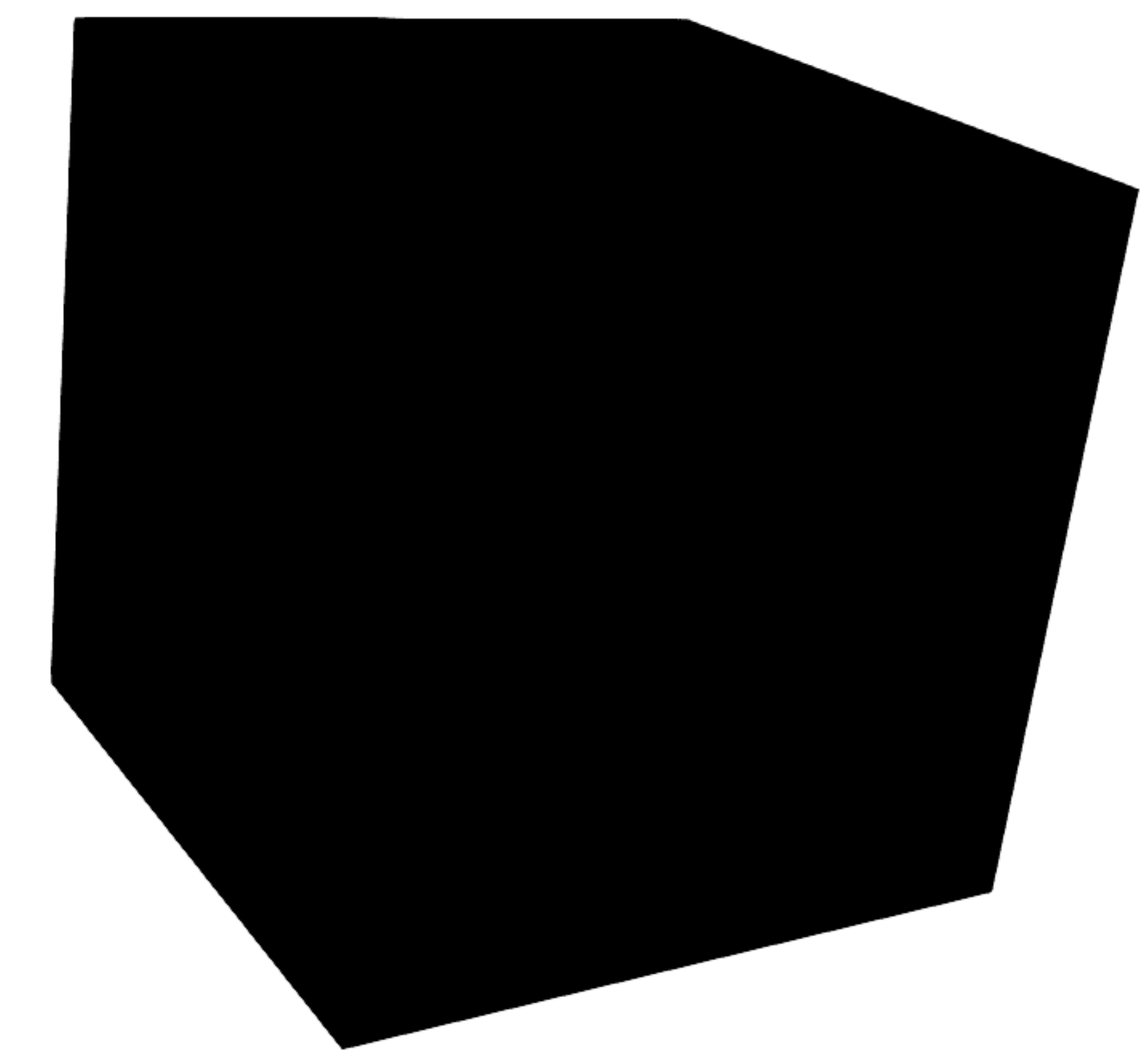}
        \subcaption{Step\,0}
        \label{3dh-a}
      \end{minipage} 
      \begin{minipage}[t]{0.24\hsize}
        \centering
        \includegraphics[keepaspectratio, scale=0.12]{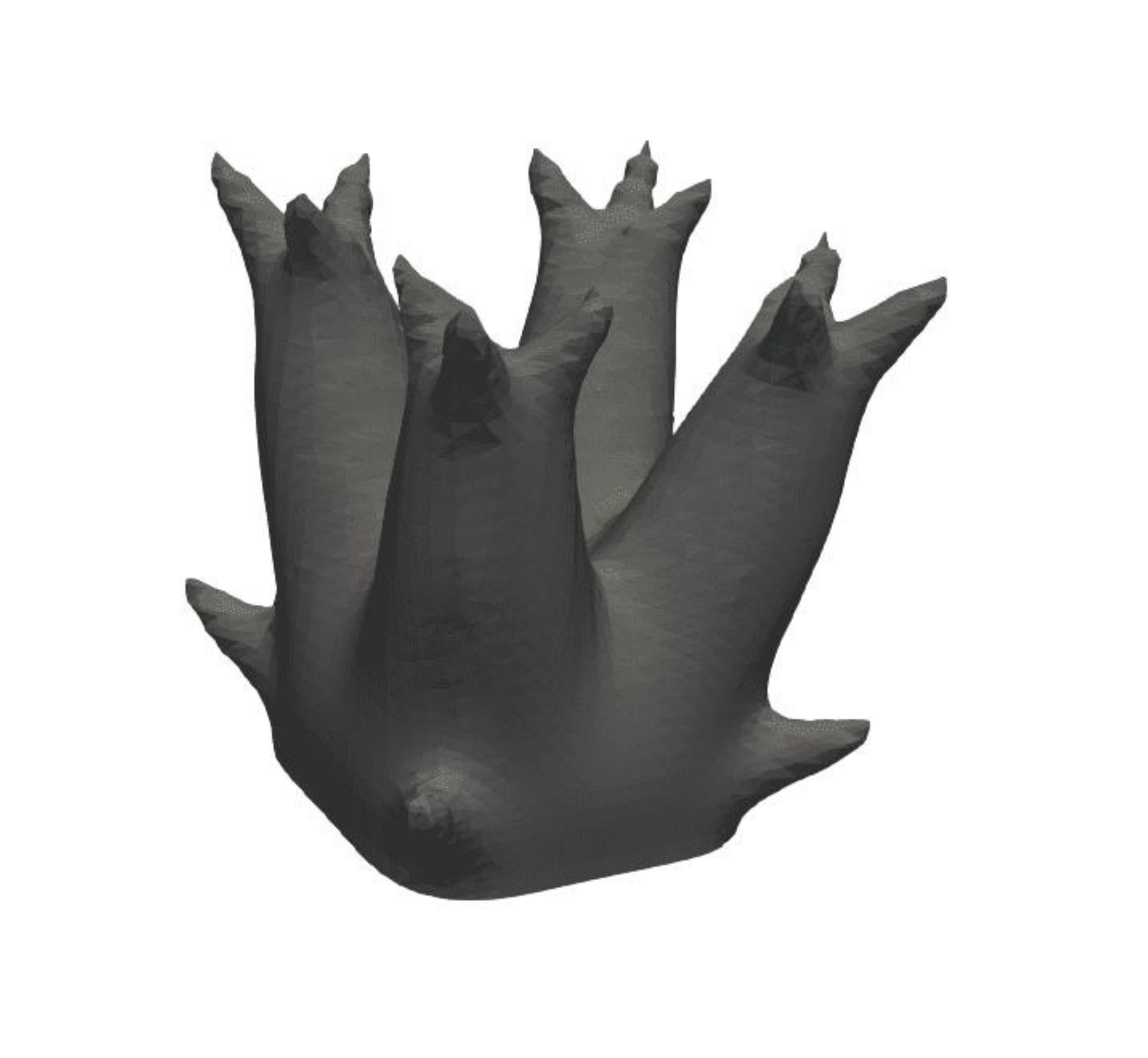}
        \subcaption{Step\,30}
        \label{3dh-b}
      \end{minipage} 
         \begin{minipage}[t]{0.24\hsize}
        \centering
        \includegraphics[keepaspectratio, scale=0.12]{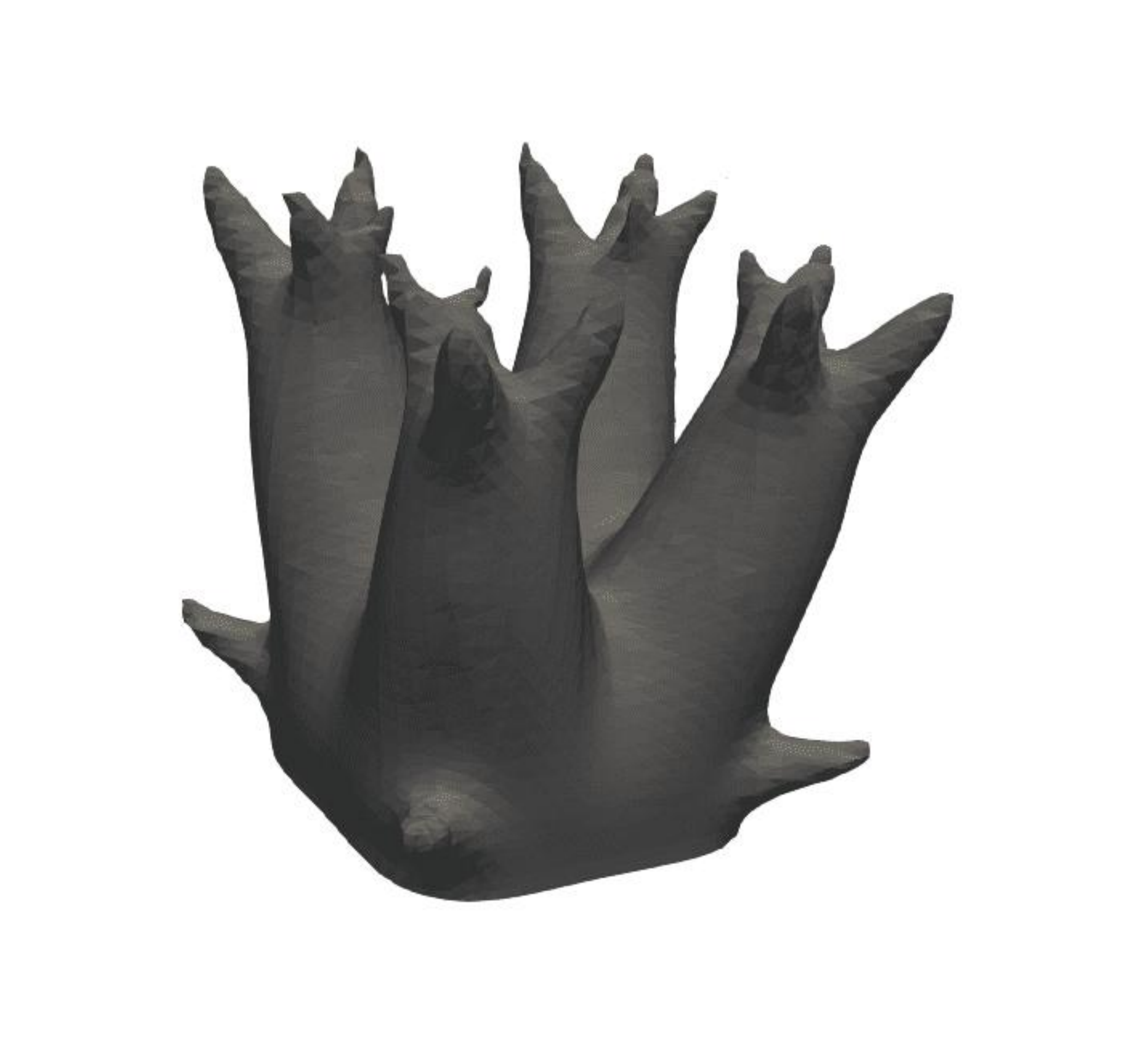}
        \subcaption{Step\,60}
        \label{3dh-c}
      \end{minipage}
           \begin{minipage}[t]{0.24\hsize}
        \centering
        \includegraphics[keepaspectratio, scale=0.12]{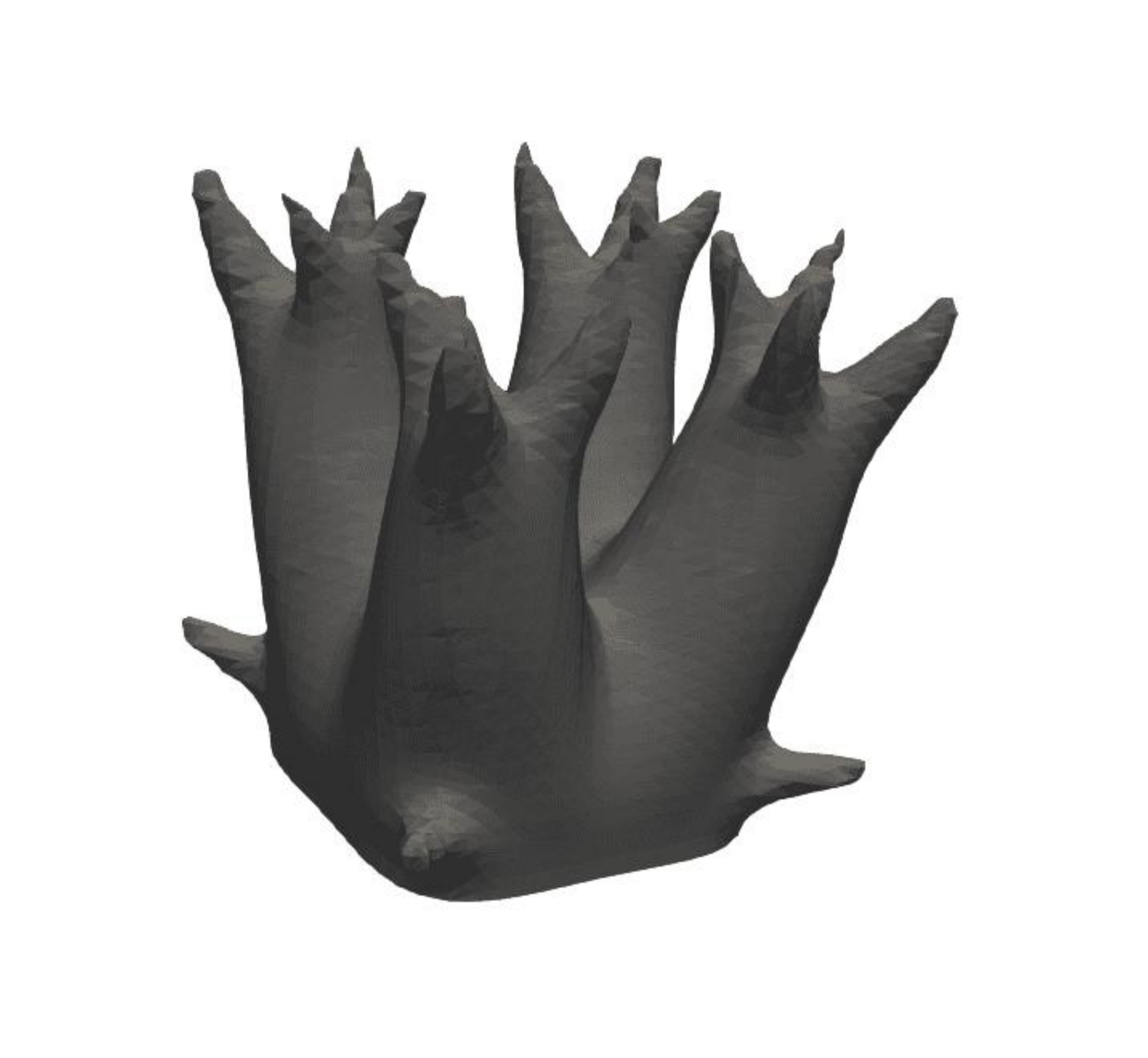}
        \subcaption{Step\,105$^{\#}$}
        \label{3dh-d}
      \end{minipage}
      \\
      \begin{minipage}[t]{0.24\hsize}
        \centering
        \includegraphics[keepaspectratio, scale=0.12]{3dh0.pdf}
        \subcaption{Step\,0}
        \label{3dh-e}
      \end{minipage} 
      \begin{minipage}[t]{0.24\hsize}
        \centering
        \includegraphics[keepaspectratio, scale=0.12]{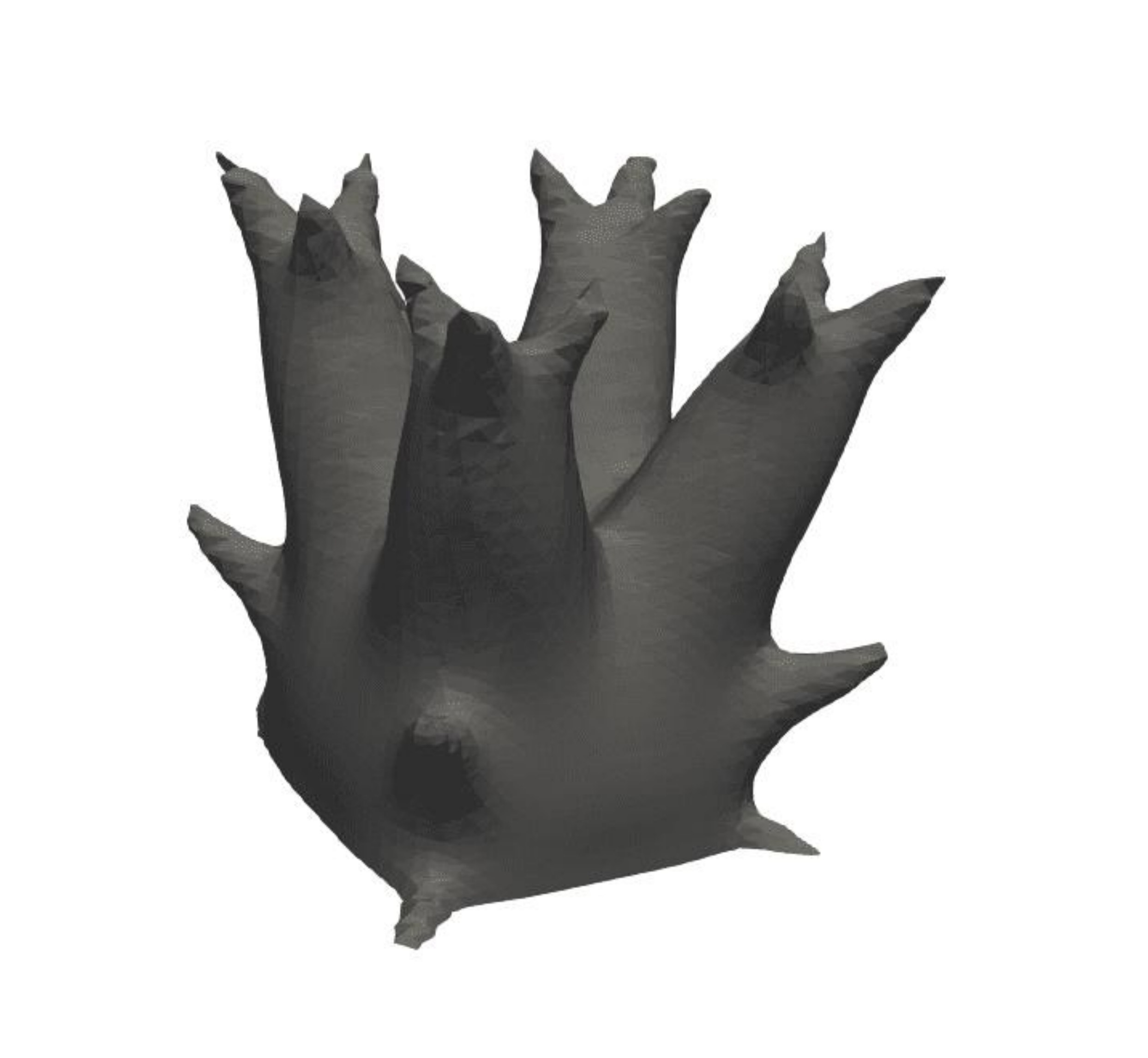}
        \subcaption{Step\,30}
        \label{3dh-f}
      \end{minipage} 
         \begin{minipage}[t]{0.24\hsize}
        \centering
        \includegraphics[keepaspectratio, scale=0.12]{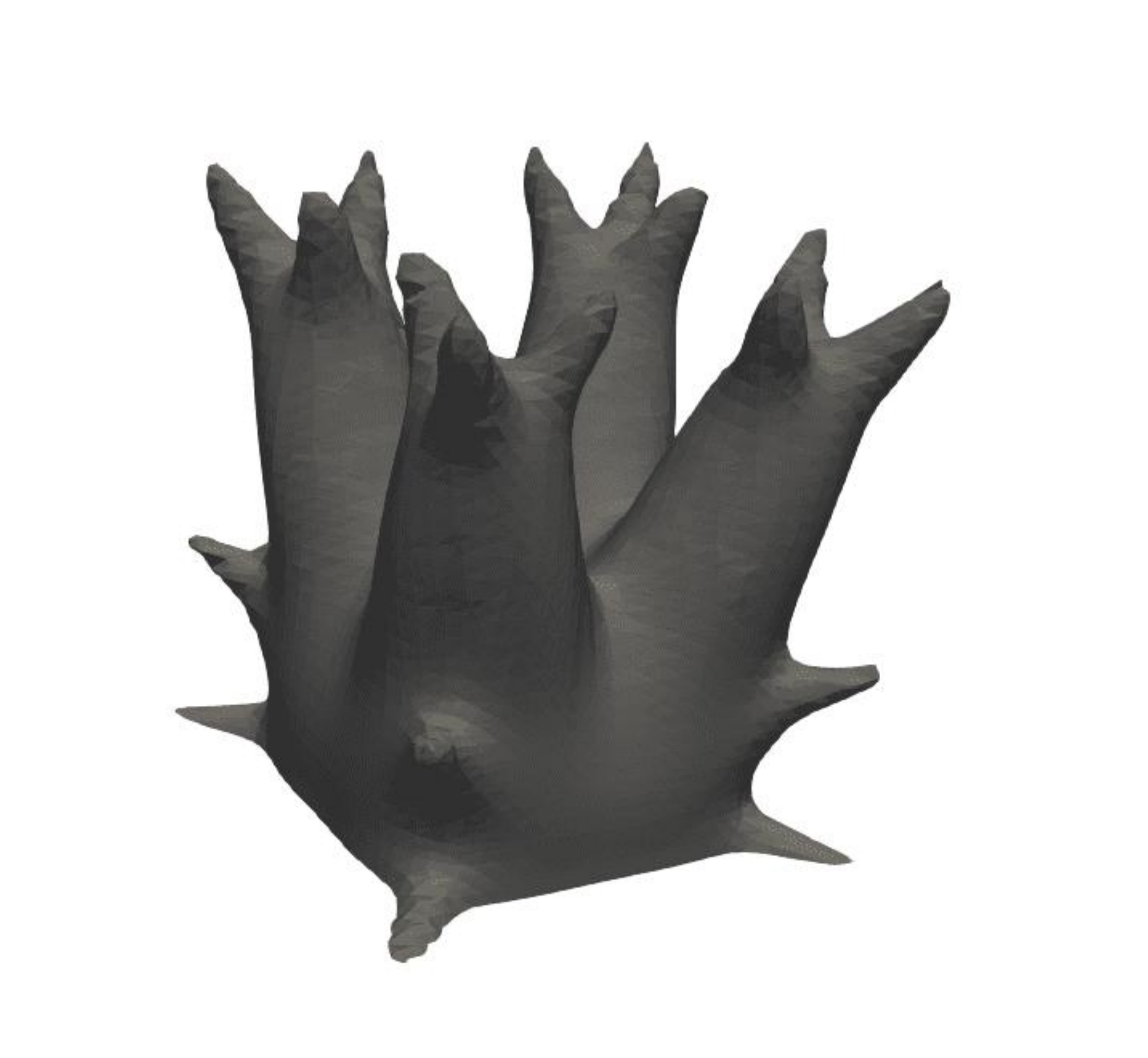}
        \subcaption{Step\,60}
        \label{3dh-g}
      \end{minipage}
           \begin{minipage}[t]{0.24\hsize}
        \centering
        \includegraphics[keepaspectratio, scale=0.12]{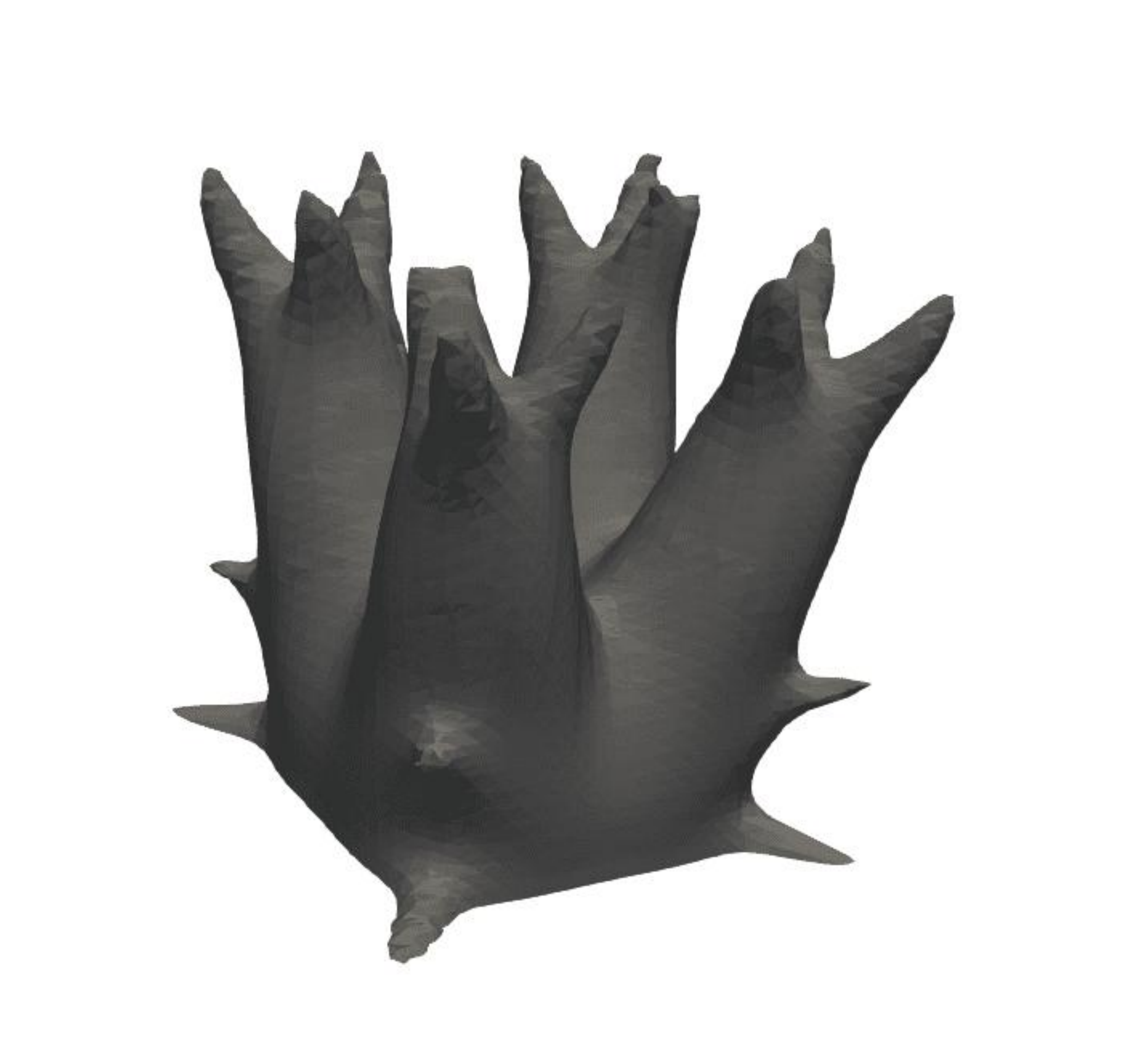}
        \subcaption{Step\,105$^{\#}$}
        \label{3dh-h}
      \end{minipage}
      \\
      \begin{minipage}[t]{0.24\hsize}
        \centering
        \includegraphics[keepaspectratio, scale=0.12]{3dh0.pdf}
        \subcaption{Step\,0}
        \label{3dh-i}
      \end{minipage} 
      \begin{minipage}[t]{0.24\hsize}
        \centering
        \includegraphics[keepaspectratio, scale=0.12]{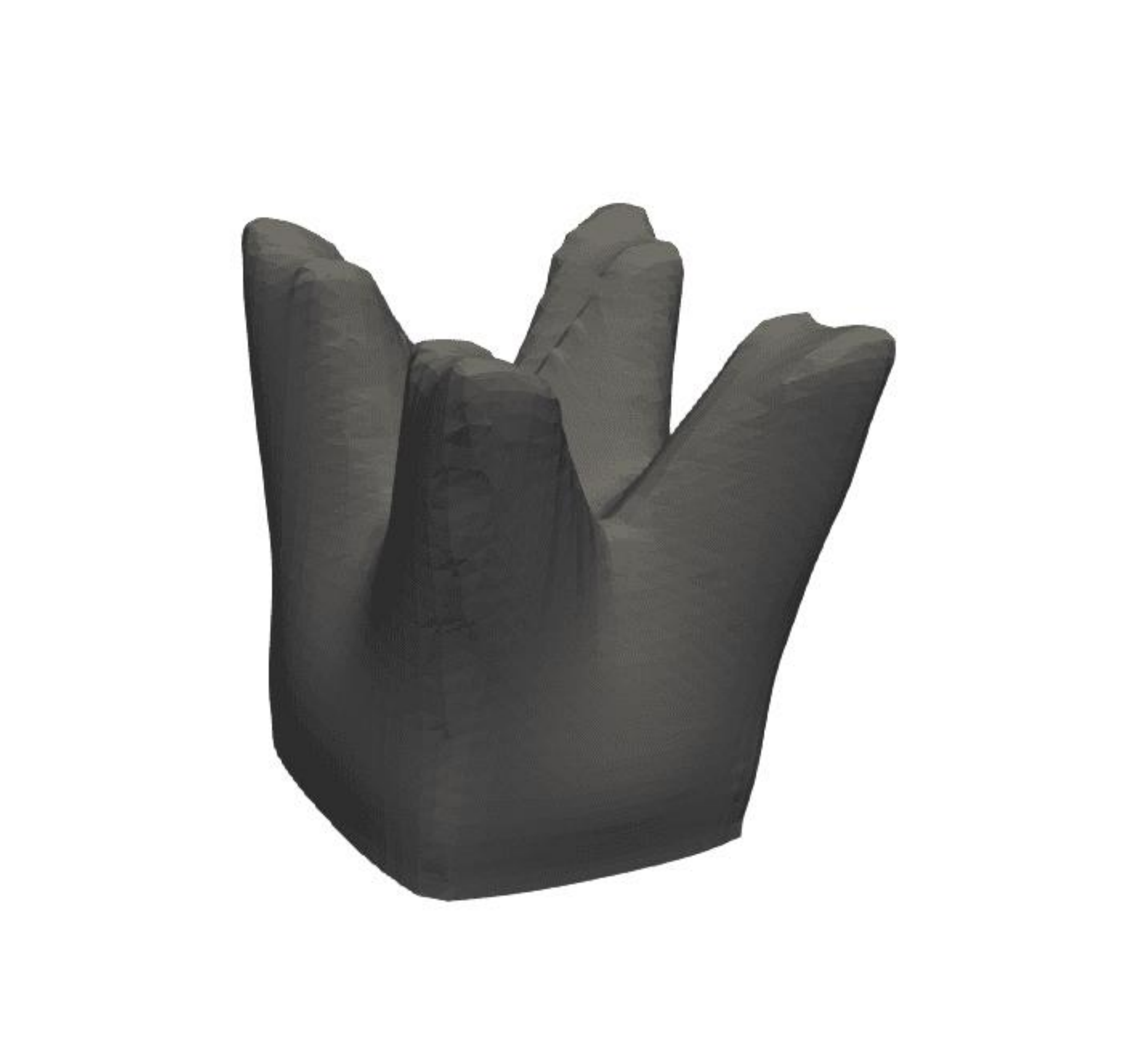}
        \subcaption{Step\,30}
        \label{3dh-j}
      \end{minipage} 
         \begin{minipage}[t]{0.24\hsize}
        \centering
        \includegraphics[keepaspectratio, scale=0.12]{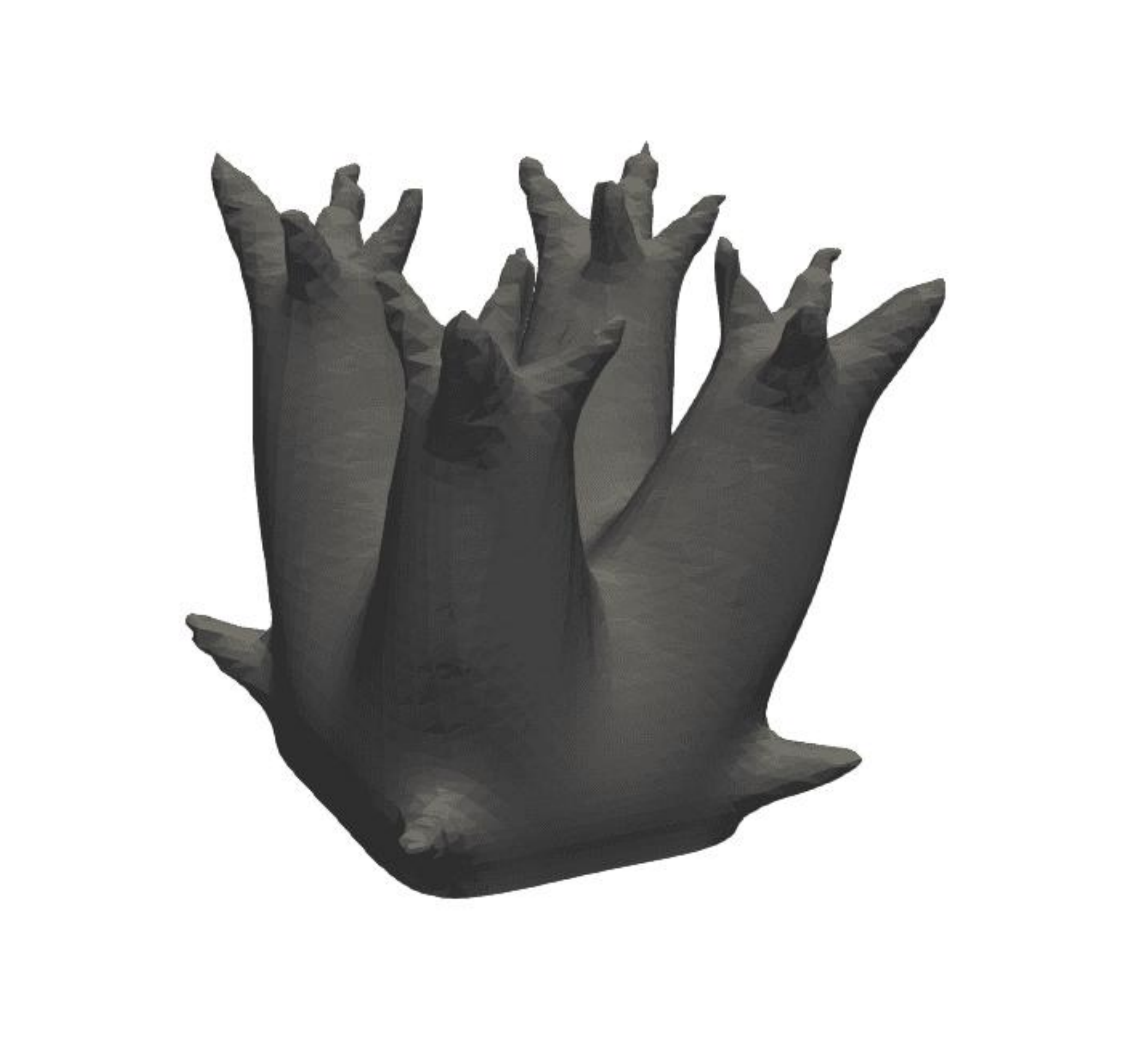}
        \subcaption{Step\,60}
        \label{3dh-k}
      \end{minipage}
           \begin{minipage}[t]{0.24\hsize}
        \centering
        \includegraphics[keepaspectratio, scale=0.12]{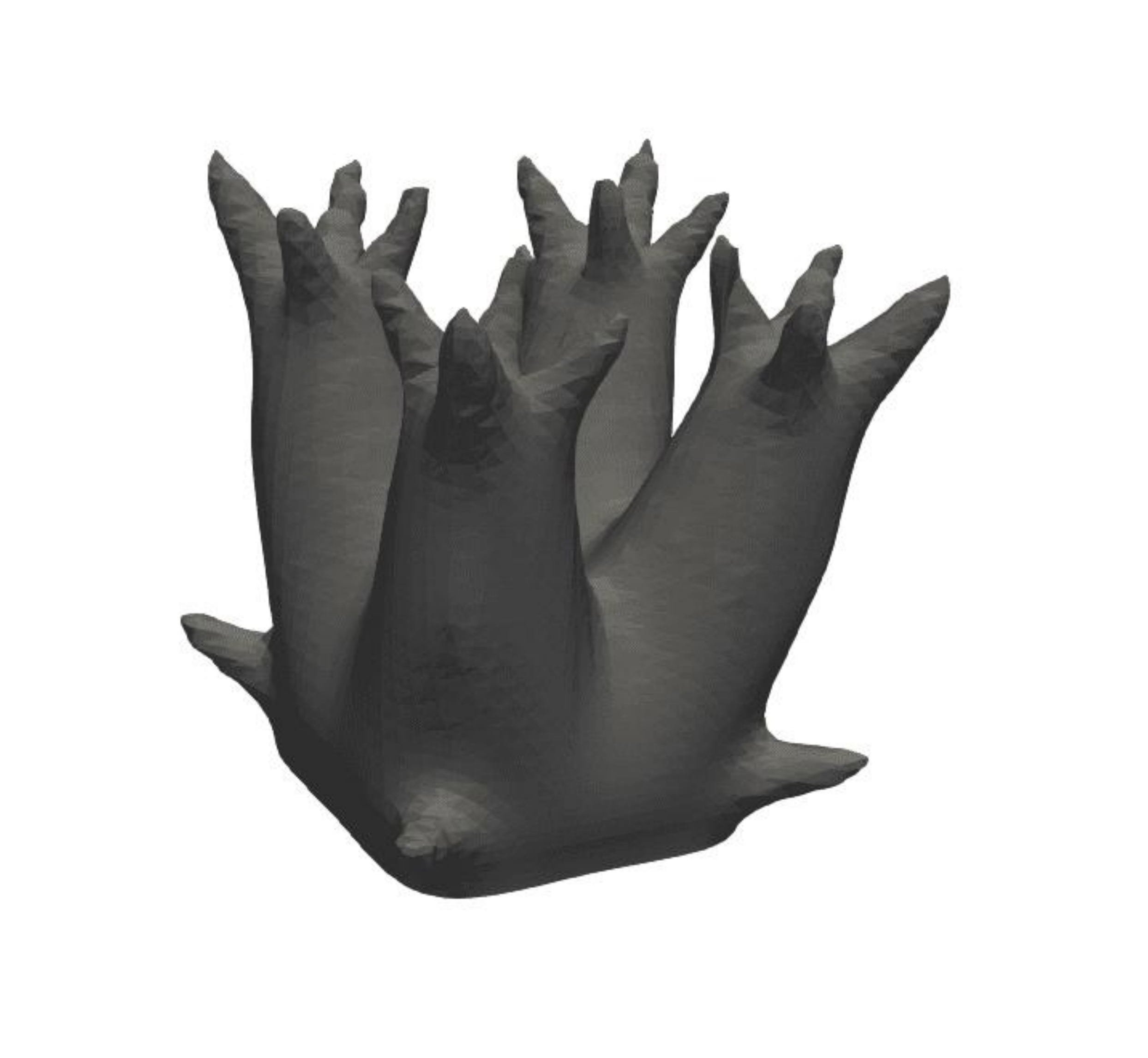}
        \subcaption{Step\,80$^{\#}$}
        \label{3dh-l}
      \end{minipage}
     \end{tabular}
     \caption{Configuration $\Omega_{\phi_n}\subset D\subset \R^3$ for the case where the initial configuration is the whole domain. 
Figures (a)--(d), (e)--(h) and (i)--(l) represent $\Omega_{\phi_n}\subset D$ for $(q,\varDelta t)=(1.0,0.4)$, $(q,\varDelta t)=(0.5,0.8)$ and $(q,\varDelta t)=(4.0,0.4)$ in \eqref{discNLD}, respectively. 
The symbol ${}^{ \#}$ implies the final step. 
     }
     \label{fig:3dh}
  \end{figure}

\begin{figure}[htbp]
        \centering
        \includegraphics[keepaspectratio, scale=0.33]{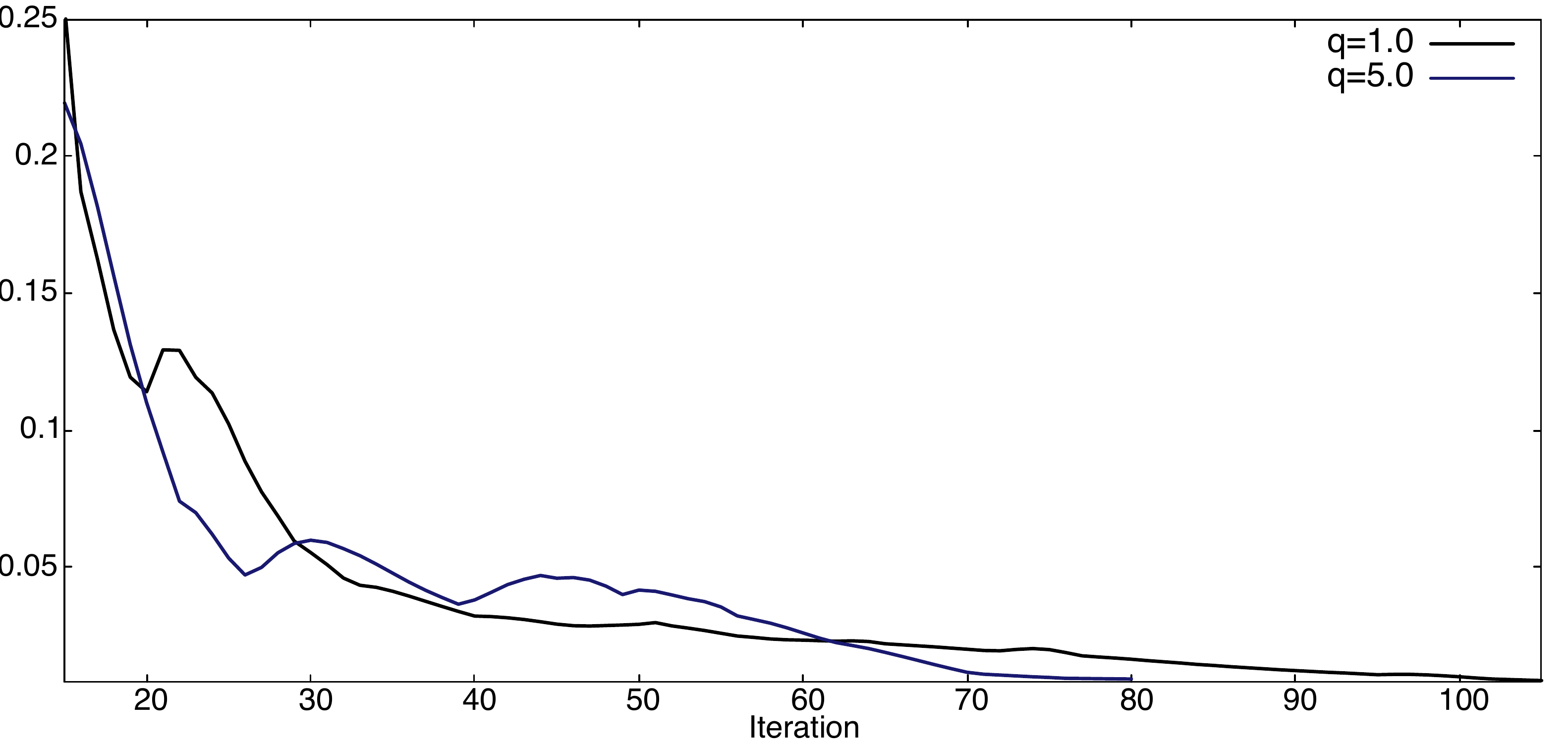}
   \caption{Convergence condition $\|\phi_{n+1}-\phi_{n}\|_{L^{\infty}(D)}$ with $\varDelta t=0.5$.}
\label{fig:3dh1}
\end{figure}       

\begin{figure}[htbp]
        \centering
        \includegraphics[keepaspectratio, scale=0.33]{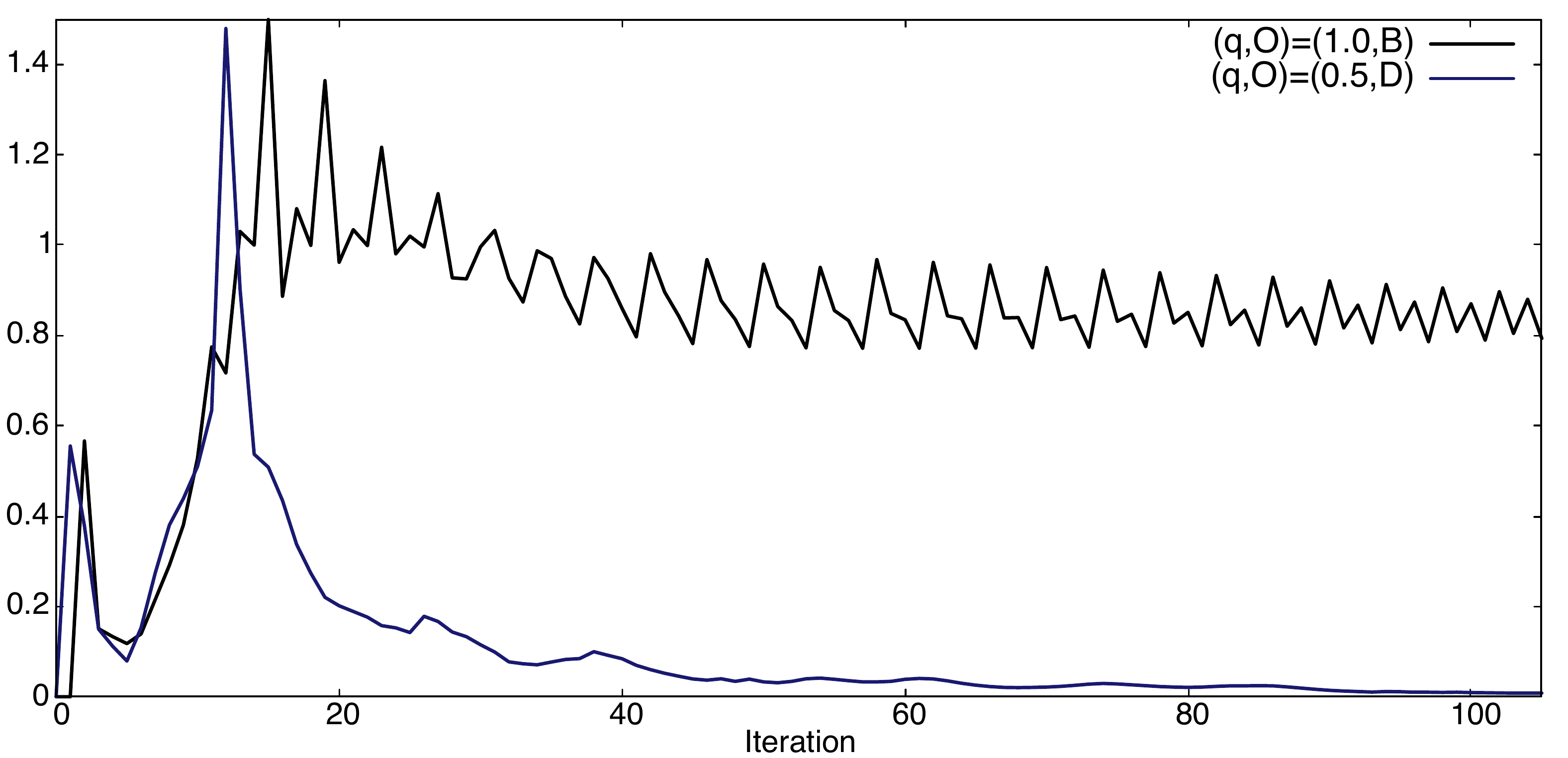}
   \caption{Convergence condition $\|\phi_{n+1}-\phi_{n}\|_{L^{\infty}(O)}$ with $\varDelta t=0.8$. Here $B\subset D$ denotes $[-0.5\le \phi_n <0]$.}
\label{fig:3dh2}
\end{figure}

\section{Applications}\label{S:app}

In this section, under the same assumption as in \S \ref{SS:ca}, 
recalling that the gradient descent method \eqref{GM} forms the basis in the derivation of \eqref{NLD}, 
we discuss methods to obtain even faster convergence to optimal configurations using other gradient methods. 
As already mentioned in a previous study \cite{O22}, Nesterov's accelerated gradient method in \cite{N83} is used to achieve this objective, and it is given by 
\begin{align}
\label{AGM}
\psi_{n+1}(x)=\phi_{n}(x)-K(\phi_n)F'(\phi_n),
\end{align}
where
\begin{align}\label{eq:inertia}
\phi_n(x)=\psi_{n}(x)+\frac{n-1}{n+2}(\psi_n-\psi_{n-1}).
\end{align}
Thus, the effect of inertia is expected by the second term of the right-hand side of \eqref{eq:inertia}, and moreover, by the results in \cite{SBC14} along with \eqref{AGM}--\eqref{eq:inertia}, we have obtained some nonlinear damped wave equation (see \cite{O22} for details).
Hence, combining the argument in \S \ref{S:nld} with Nesterov's accelerated gradient method together with \cite{SBC14}, one can formally get the following quasilinear hyperbolic equation\/{\rm :} 
\begin{align}\label{NLQH}
\begin{cases}
\partial_{t}^2\phi^q+(3/t)\partial_{t}\phi^q-\tau \Delta \phi=\rho F_\eta'(\phi) \ \text{ in } D\times (0,+\infty),\\
\phi|_{\partial D}=0,\quad \phi|_{t=0}=\phi_1,\quad \partial_t\phi|_{t=0}=\phi_2.
\end{cases}
\end{align}
Here $\phi_1$ and $\phi_2$ are initial data constructed by solving \eqref{discNLD}.
Based on \eqref{NLQH}, we devise a way to employ the following equation instead of \eqref{discNLD} in Step\,4\/{\rm :} 
\begin{align}
&
\int_D \tilde{q}(|\phi_{n}(x)|+\xi)^{q-1}\frac{\phi_{n+1}-(1+h)\phi_{n}+h\phi_{n-1}}{\varDelta t}(x){\psi}(x) \, \d x \nonumber \\
&\quad+
\int_D \frac{3h\tilde{q}}{n+\xi}(|\phi_{n}(x)|+\xi)^{q-1}\frac{\phi_{n+1}-\phi_{n}}{\varDelta t}(x){\psi}(x) \, \d x\nonumber \\
&\quad+\int_D\tau \nabla \phi_{n+1}(x)\cdot \nabla {\psi}(x)\, \d x
=
\int_D \rho \mathcal{L}_\eta'(\phi_{n},\lambda_n) \psi(x)\, \d x 
\quad \text{ for all } V, \label{discNNLD}
\end{align}
where $\phi_{n+i}=\phi(x,(n+i)\sqrt{\varDelta t})$ ($i=-1,0,1$) and $h\in\{0,1\}$. 
Thus, \eqref{discNLD} coincides with \eqref{discNNLD} for $h=0$ and $\phi_{n+i}=\phi(x,(n+i)\varDelta t)$ ($i=0,1)$. 
Then, by employing \eqref{discNNLD} to the optimization problem in \S \ref{SS:ca}, 
the numerical results shown in Figure \ref{fig:nmc} are obtained. 
 \begin{figure}[htbp]
   \hspace*{-5mm} 
    \begin{tabular}{cccc}
       \begin{minipage}[t]{0.24\hsize}
        \centering
        \includegraphics[keepaspectratio, scale=0.202]{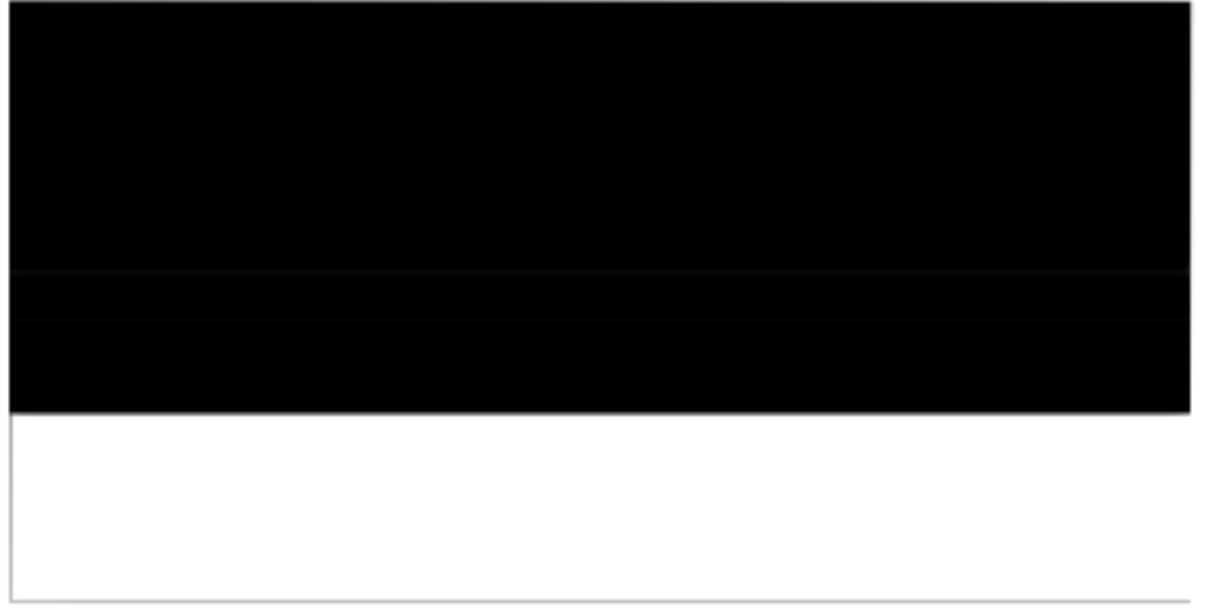}	
        \subcaption{Step\,0}
        \label{5-a}
      \end{minipage} 
      \begin{minipage}[t]{0.24\hsize}
        \centering
        \includegraphics[keepaspectratio, scale=0.1]{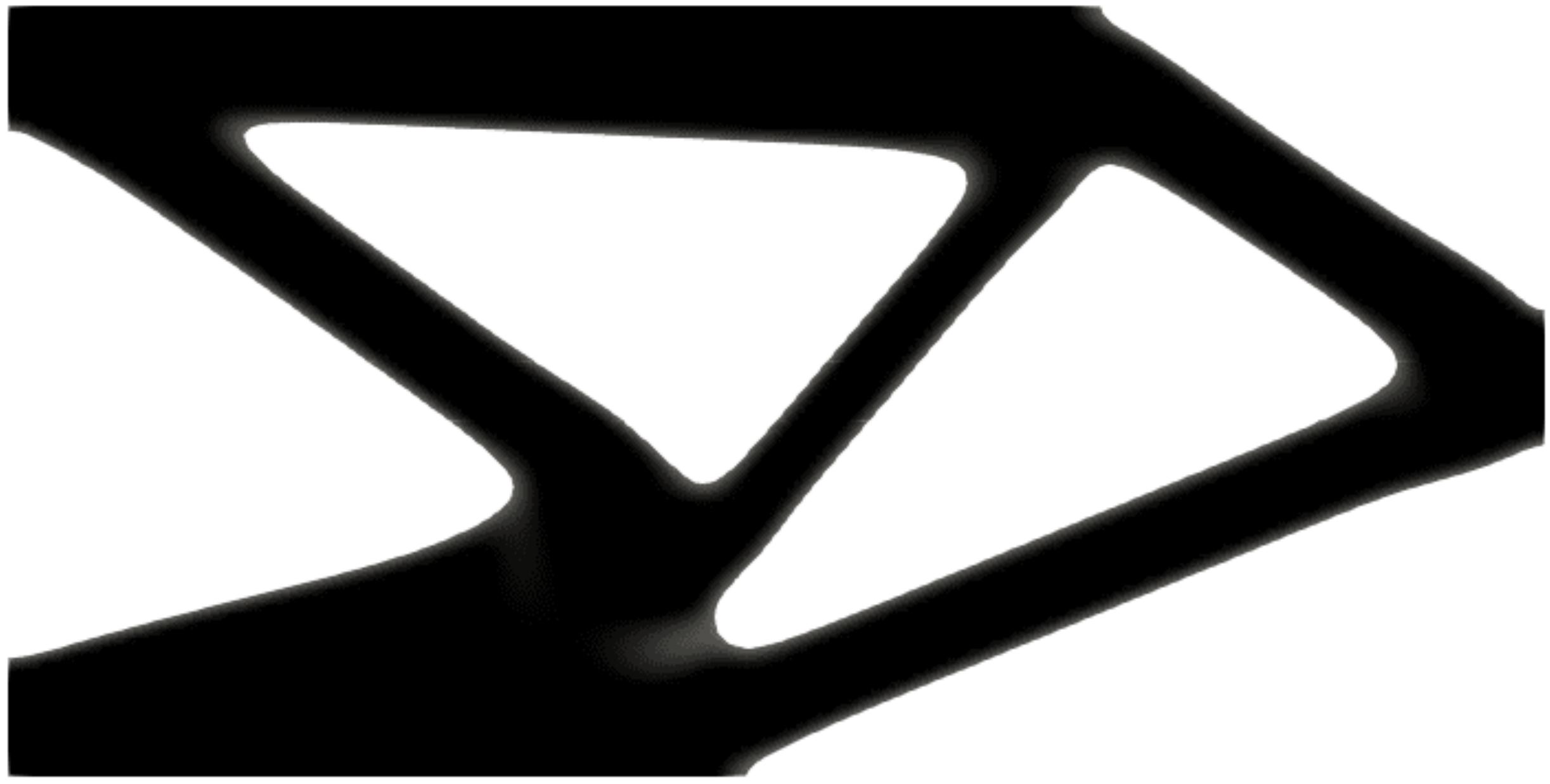}
        \subcaption{Step\,100}
        \label{5-b}
      \end{minipage} 
       \begin{minipage}[t]{0.24\hsize}
        \centering
        \includegraphics[keepaspectratio, scale=0.1]{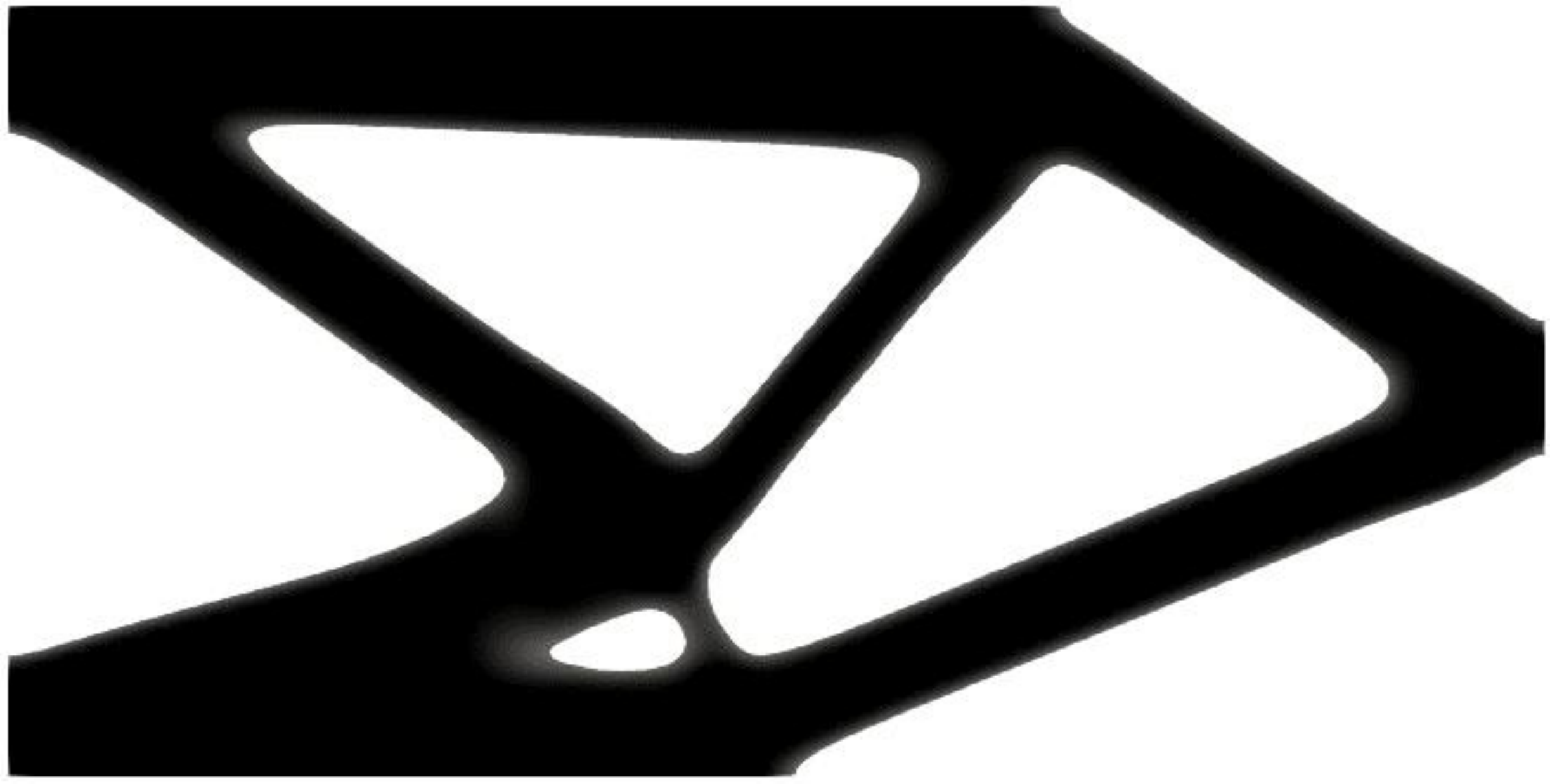}
        \subcaption{Step\,200}
        \label{5-c}
      \end{minipage} 
           \begin{minipage}[t]{0.24\hsize}
        \centering
        \includegraphics[keepaspectratio, scale=0.1]{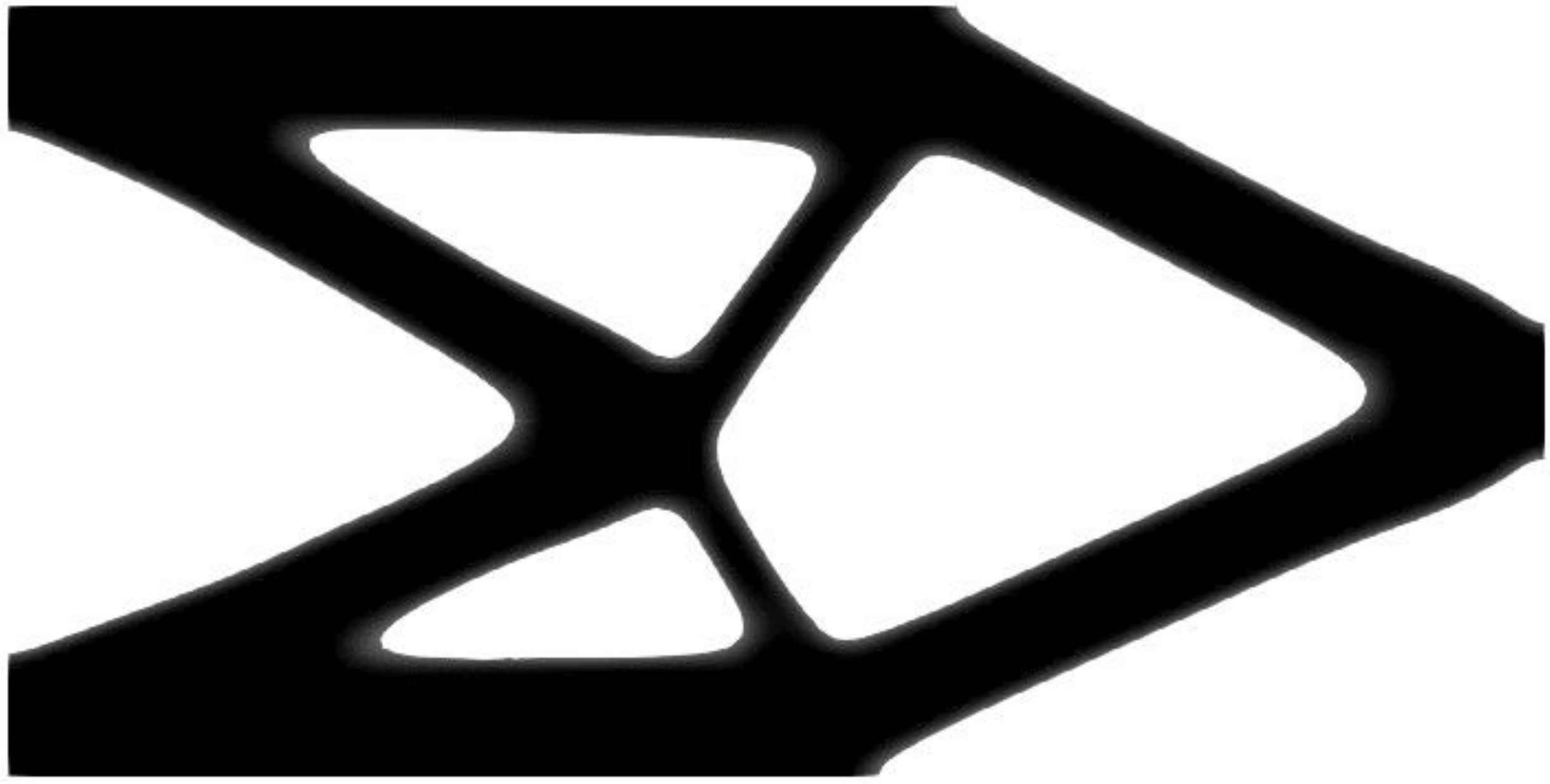}
        \subcaption{Step\,641${}^{ \#}$}
        \label{5-d}
        \end{minipage} \\ 
    \begin{minipage}[t]{0.24\hsize}
        \centering
        \includegraphics[keepaspectratio, scale=0.202]{n0.pdf}
        \subcaption{Step\,0}
        \label{5-e}
      \end{minipage} 
      \begin{minipage}[t]{0.24\hsize}
        \centering
        \includegraphics[keepaspectratio, scale=0.1]{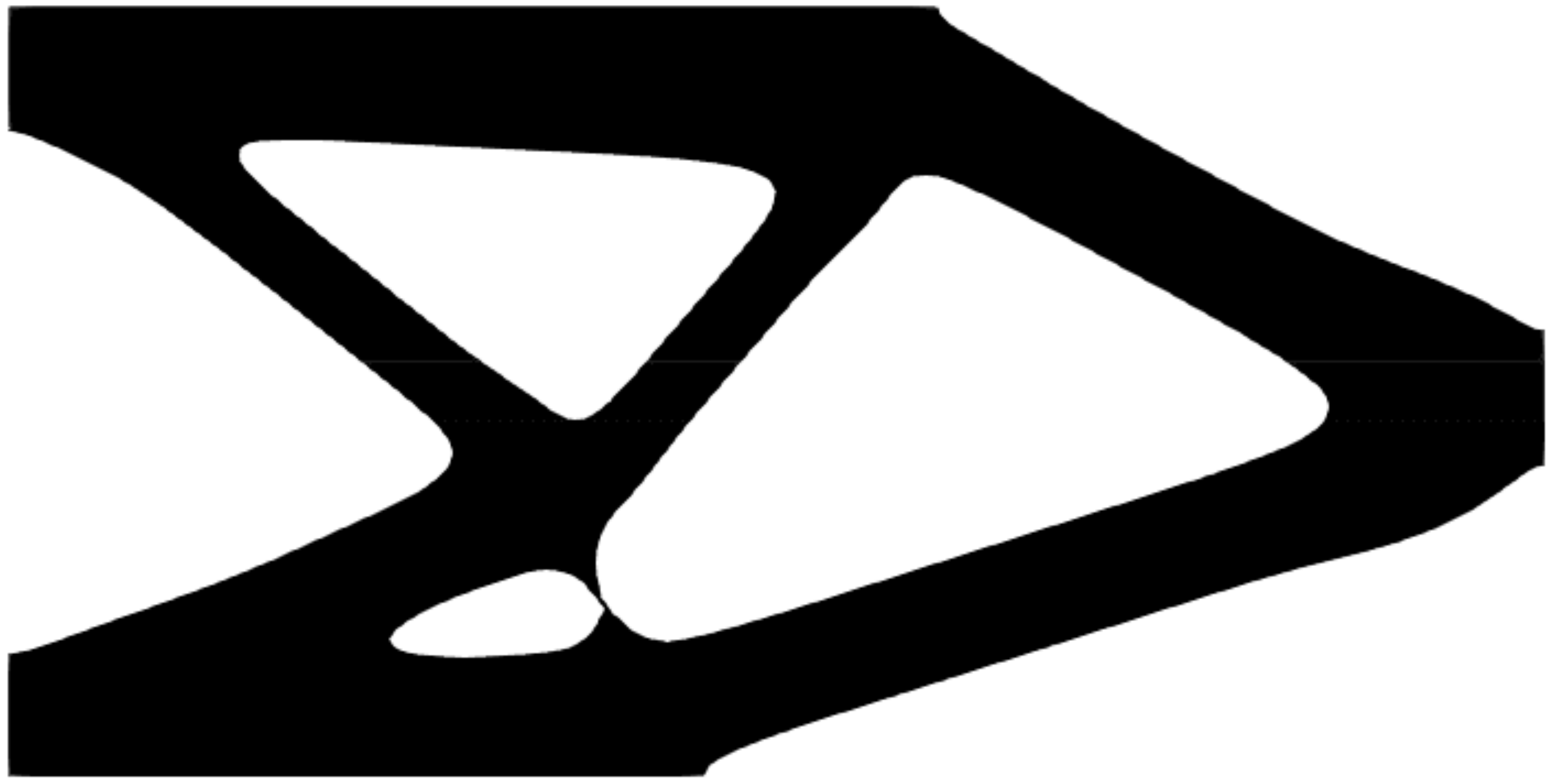}
        \subcaption{Step\,100}
        \label{5-f}
      \end{minipage} 
       \begin{minipage}[t]{0.24\hsize}
        \centering
        \includegraphics[keepaspectratio, scale=0.1]{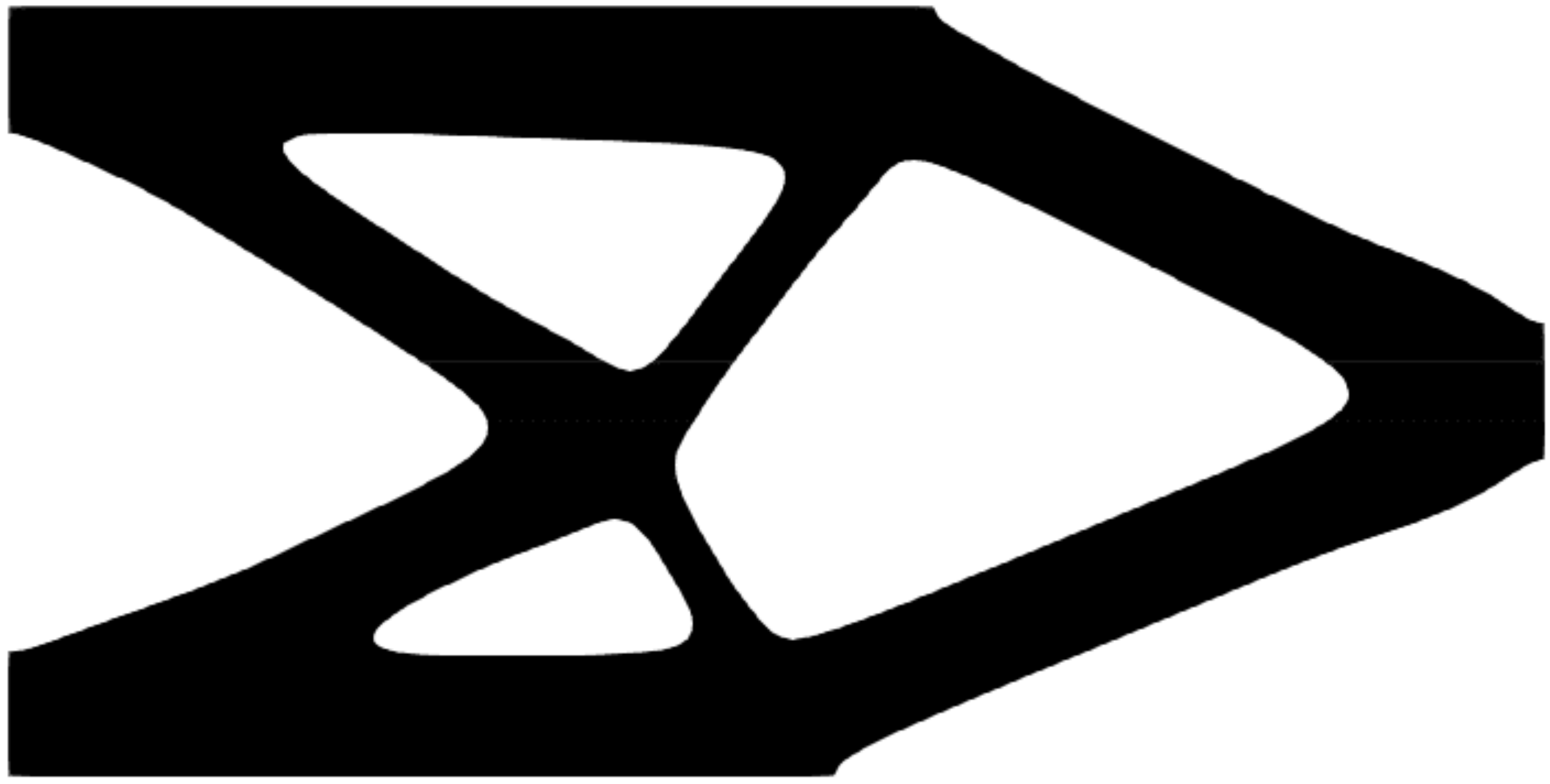}
        \subcaption{Step\,200}
        \label{5-g}
      \end{minipage} 
           \begin{minipage}[t]{0.24\hsize}
        \centering
        \includegraphics[keepaspectratio, scale=0.1]{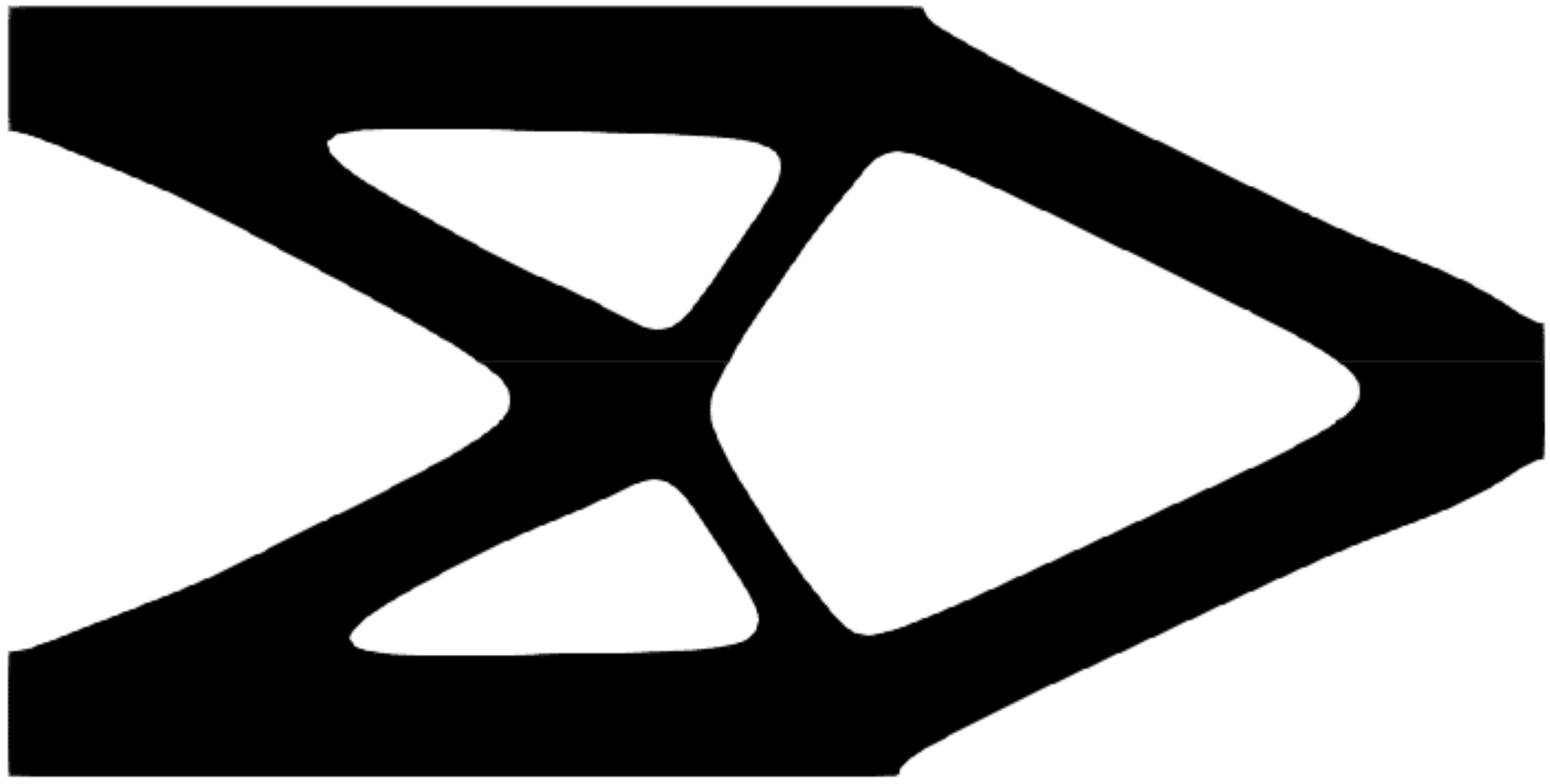}
        \subcaption{Step\,422${}^{ \#}$}
        \label{5-h}
        \end{minipage} \\
         \begin{minipage}[t]{0.24\hsize}
        \centering
        \includegraphics[keepaspectratio, scale=0.202]{n0.pdf}
        \subcaption{Step\,0}
        \label{5-i}
      \end{minipage} 
      \begin{minipage}[t]{0.24\hsize}
        \centering
        \includegraphics[keepaspectratio, scale=0.1]{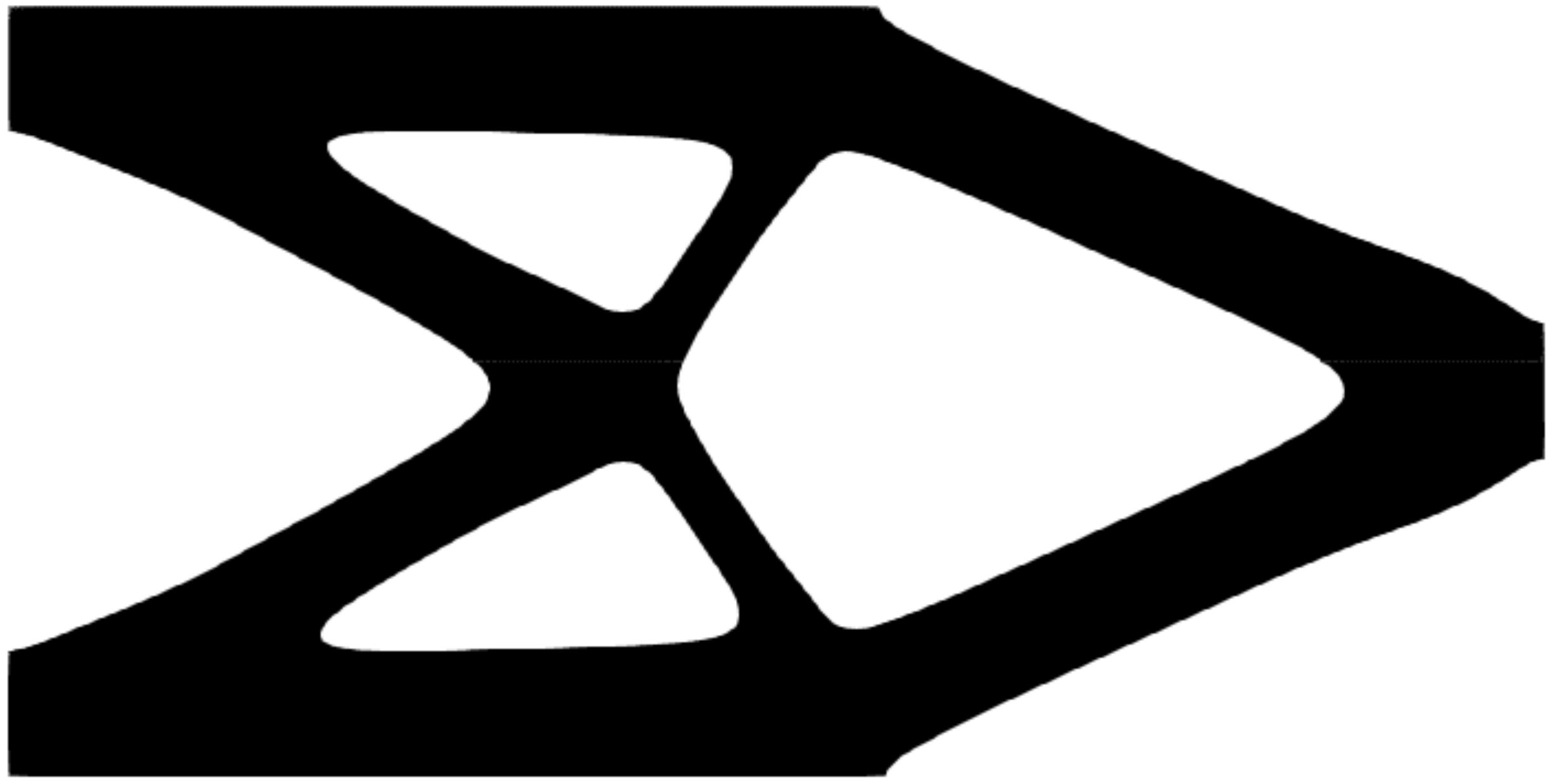}
        \subcaption{Step\,100}
        \label{5-j}
      \end{minipage} 
         \begin{minipage}[t]{0.24\hsize}
        \centering
        \includegraphics[keepaspectratio, scale=0.1]{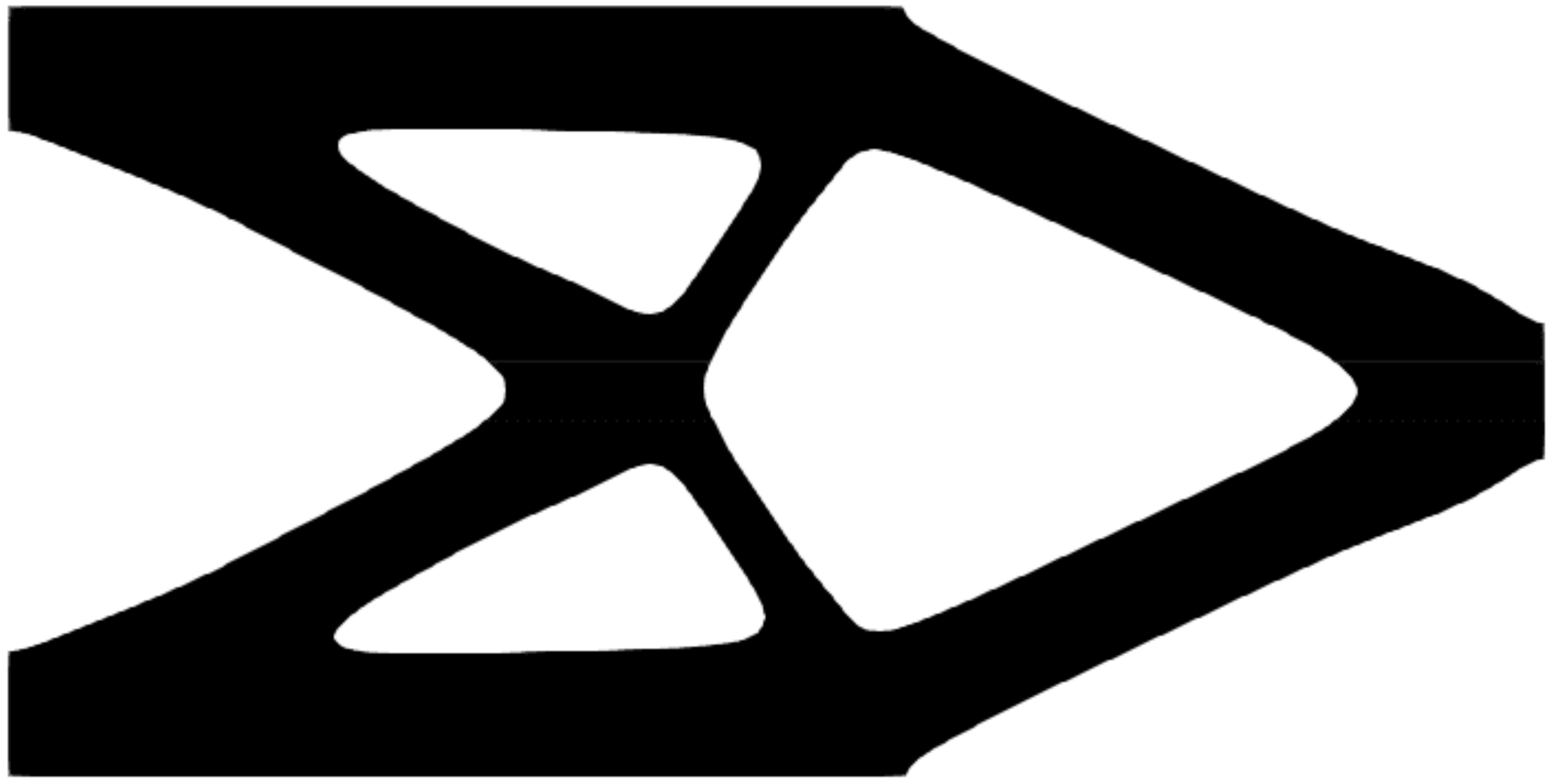}
        \subcaption{Step\,200}
        \label{5-k}
      \end{minipage}
           \begin{minipage}[t]{0.24\hsize}
        \centering
        \includegraphics[keepaspectratio, scale=0.1]{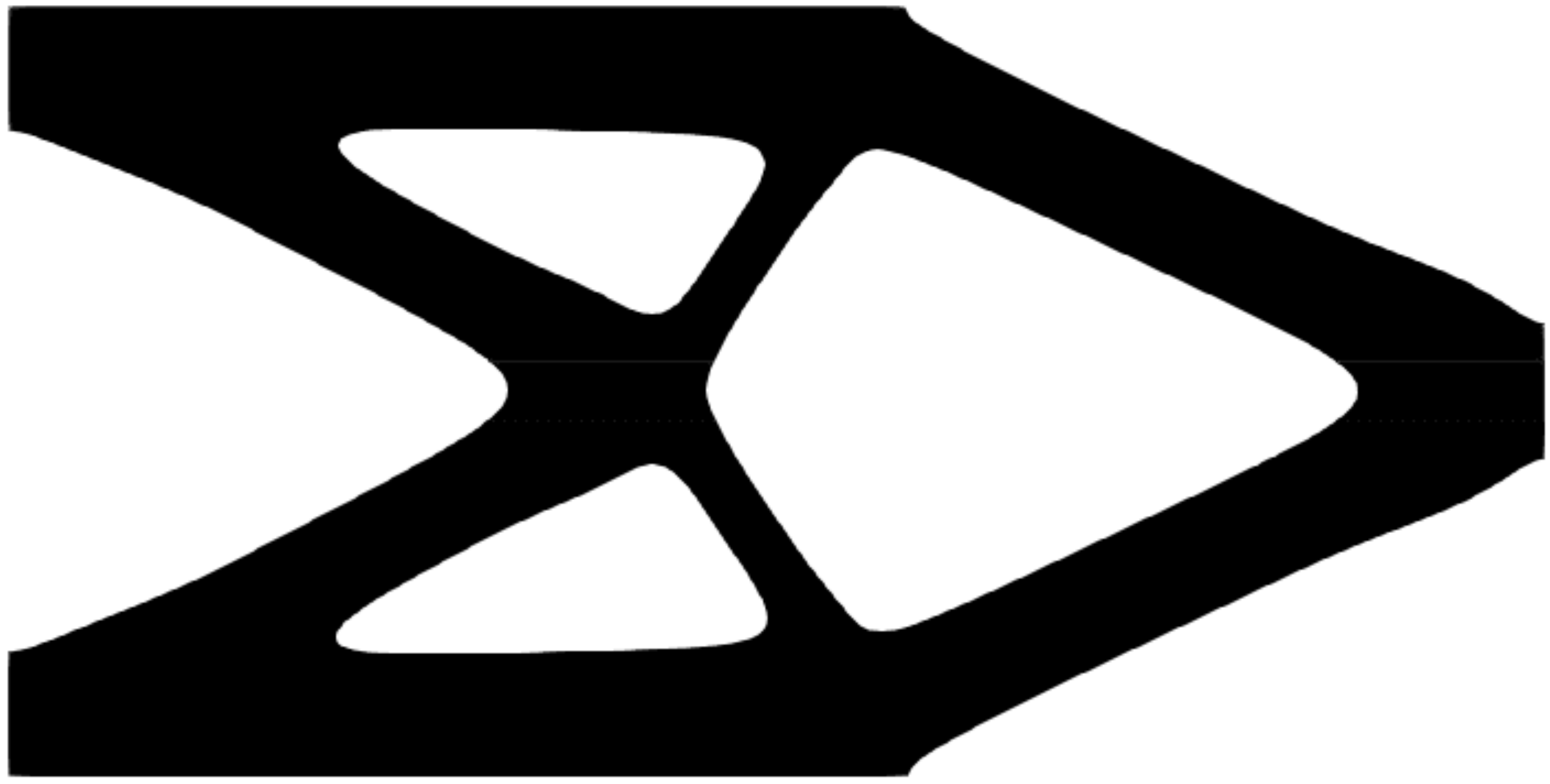}
        \subcaption{Step\,223${}^{ \#}$}
        \label{5-l}
      \end{minipage}
           \end{tabular}
     \caption{
     Configuration $\Omega_{\phi_n}\subset D$ for the case where the initial configuration is the upper domain. 
Figures (a)--(d), (e)--(h) and (i)--(l)
represent $\Omega_{\phi_n}\subset D$ using  \eqref{discNLD} for $q=1$, $q=5$ and \eqref{discNNLD} for $q=5$, respectively.
The symbol ${}^{ \#}$ implies the final step.   }
     \label{fig:nmc}
  \end{figure}

\begin{figure}[htbp]
   \hspace*{-5mm} 
    \begin{tabular}{cccc}
      \begin{minipage}[t]{0.24\hsize}
        \centering
        \includegraphics[keepaspectratio, scale=0.202]{n0.pdf}	
        \subcaption{Step\,0}
        \label{6-a}
      \end{minipage} 
      \begin{minipage}[t]{0.24\hsize}
        \centering
        \includegraphics[keepaspectratio, scale=0.1]{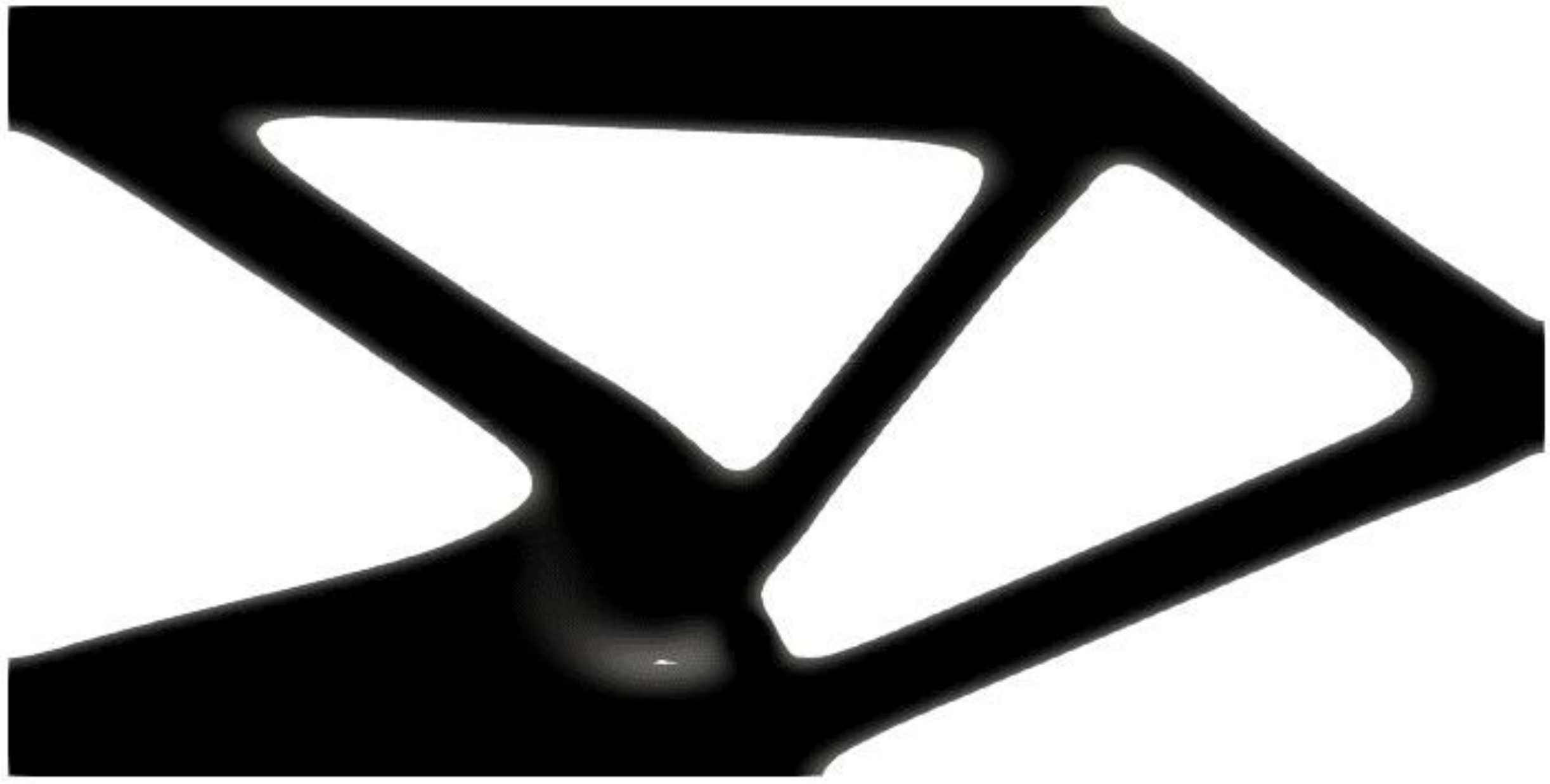}
        \subcaption{Step\,100}
        \label{6-b}
      \end{minipage} 
       \begin{minipage}[t]{0.24\hsize}
        \centering
        \includegraphics[keepaspectratio, scale=0.1]{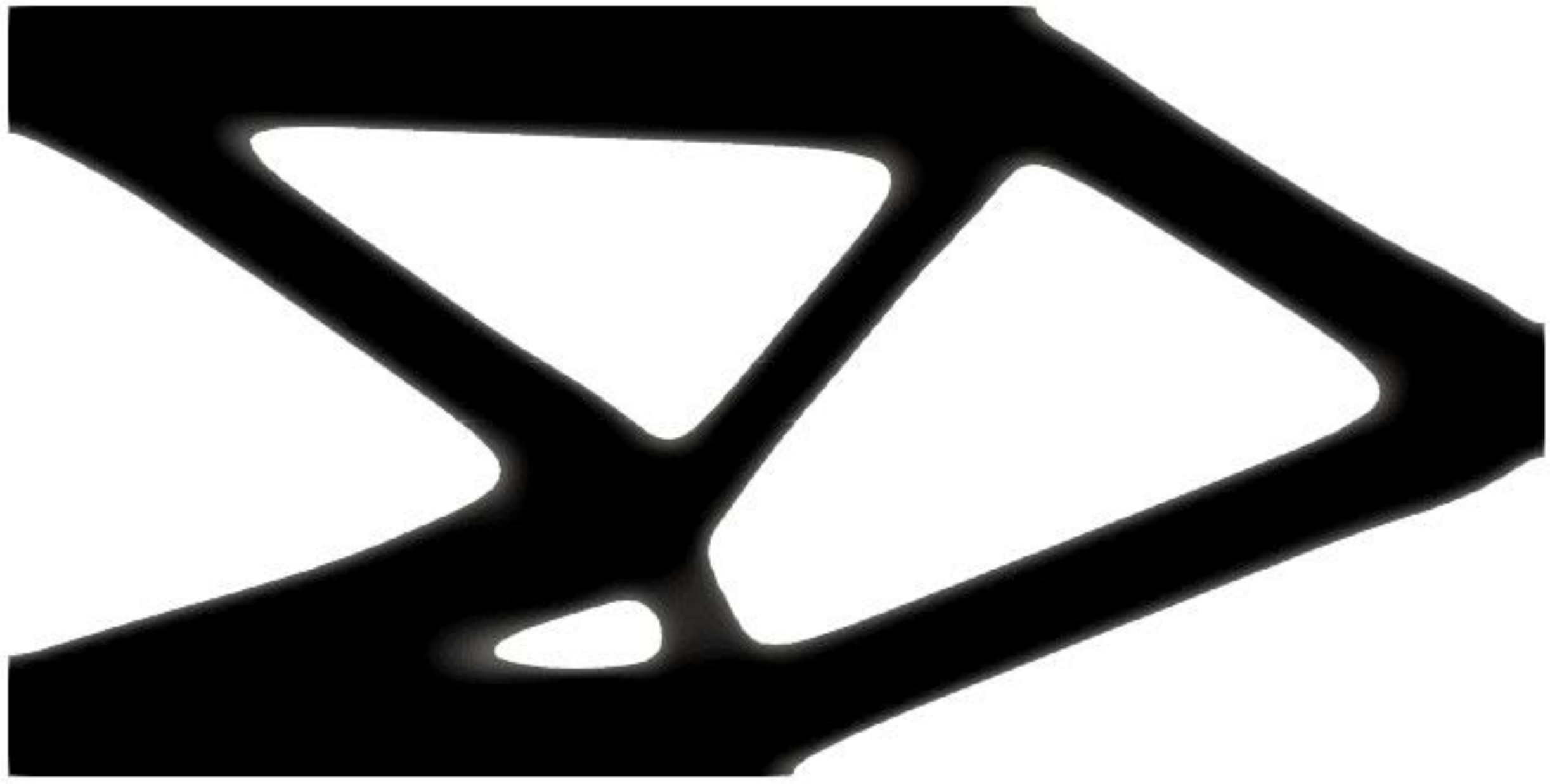}
        \subcaption{Step\,200}
        \label{6-c}
      \end{minipage} 
           \begin{minipage}[t]{0.24\hsize}
        \centering
        \includegraphics[keepaspectratio, scale=0.1]{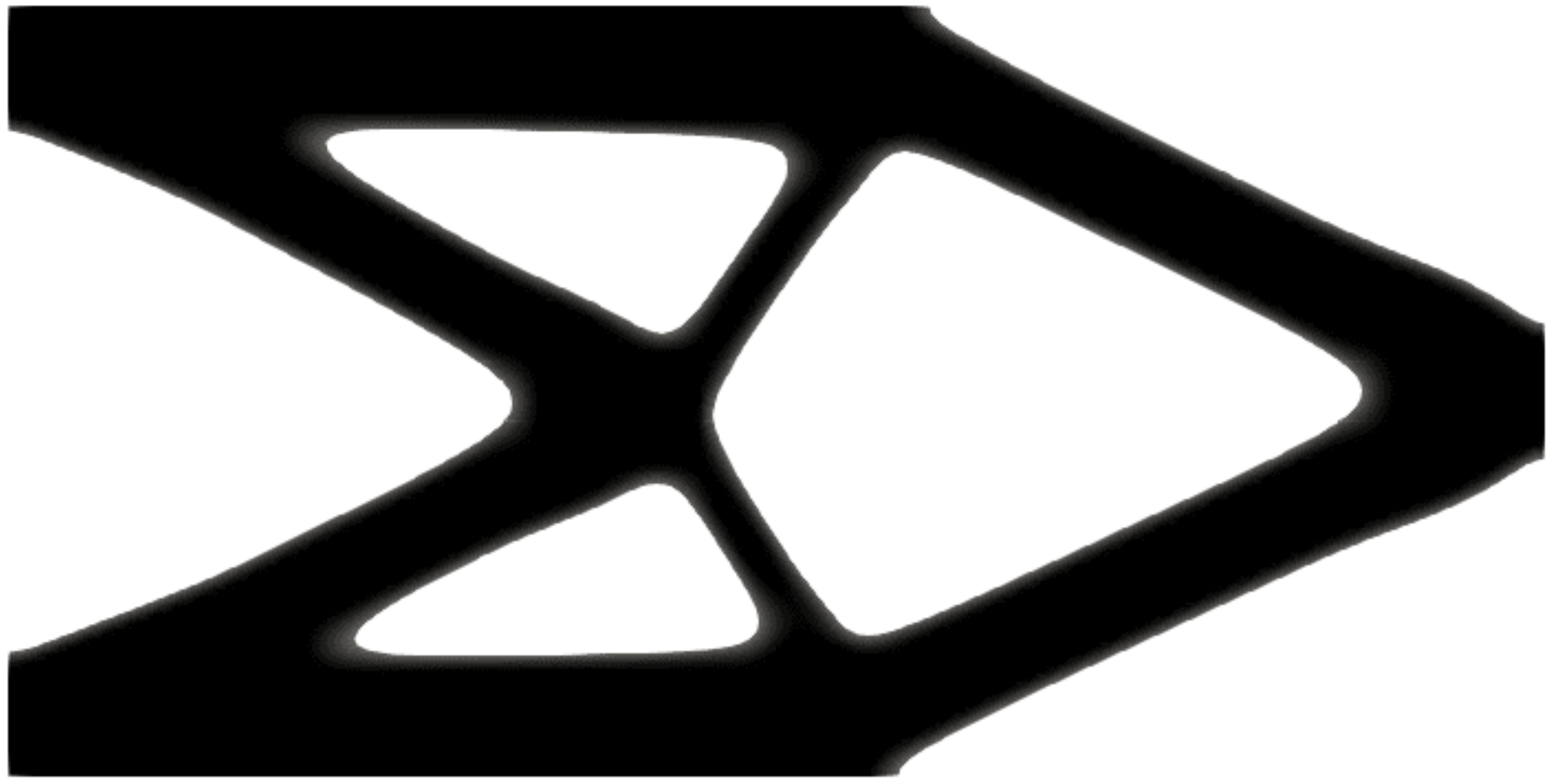}
        \subcaption{Step\,527$^{\#}$}
        \label{6-d}      
        \end{minipage} 
         \\ 
    \begin{minipage}[t]{0.24\hsize}
        \centering
        \includegraphics[keepaspectratio, scale=0.202]{n0.pdf}
        \subcaption{Step\,0}
        \label{6-e}
      \end{minipage} 
      \begin{minipage}[t]{0.24\hsize}
        \centering
        \includegraphics[keepaspectratio, scale=0.1]{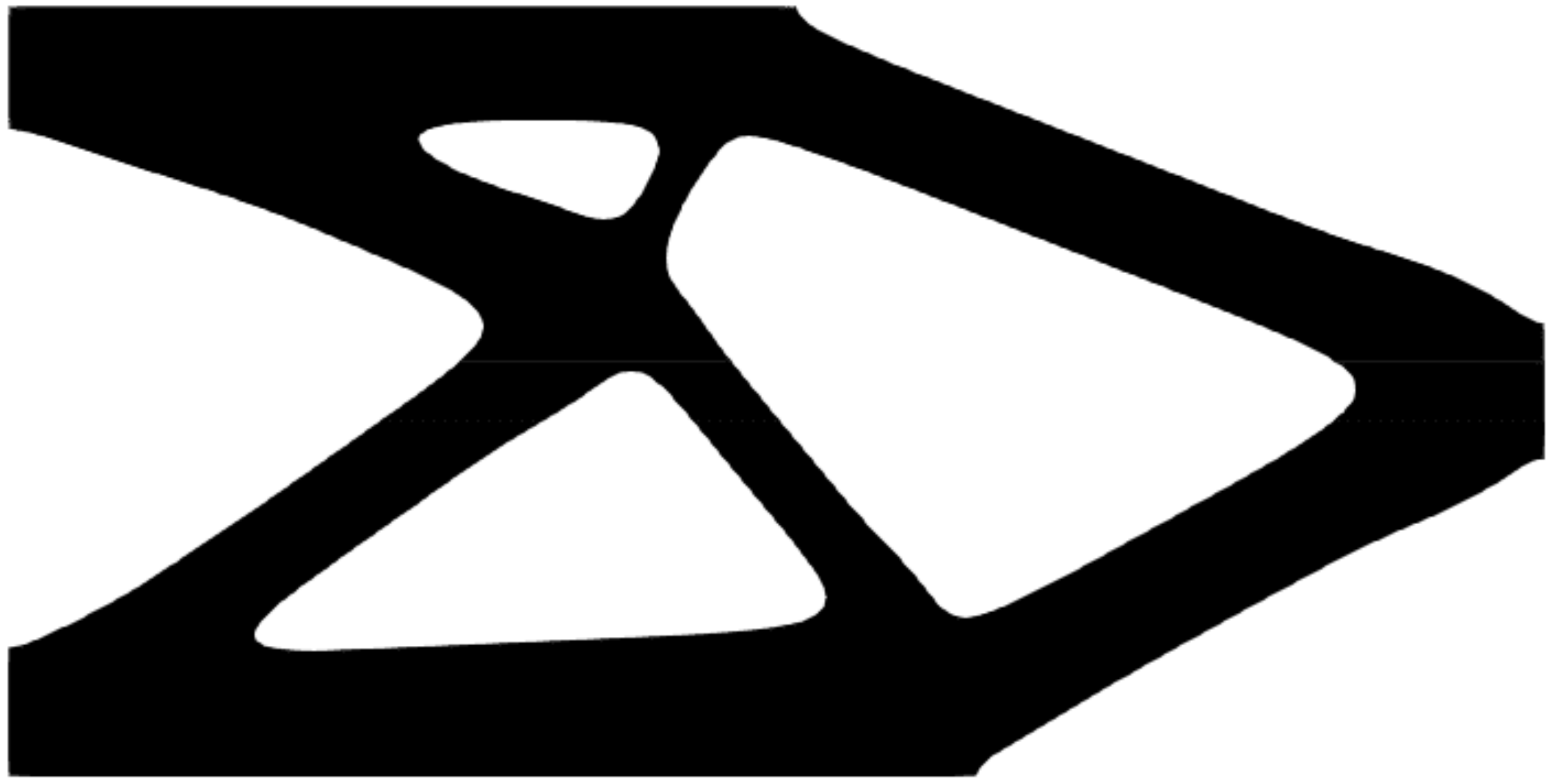}
        \subcaption{Step\,100}
        \label{6-f}
      \end{minipage} 
       \begin{minipage}[t]{0.24\hsize}
        \centering
        \includegraphics[keepaspectratio, scale=0.1]{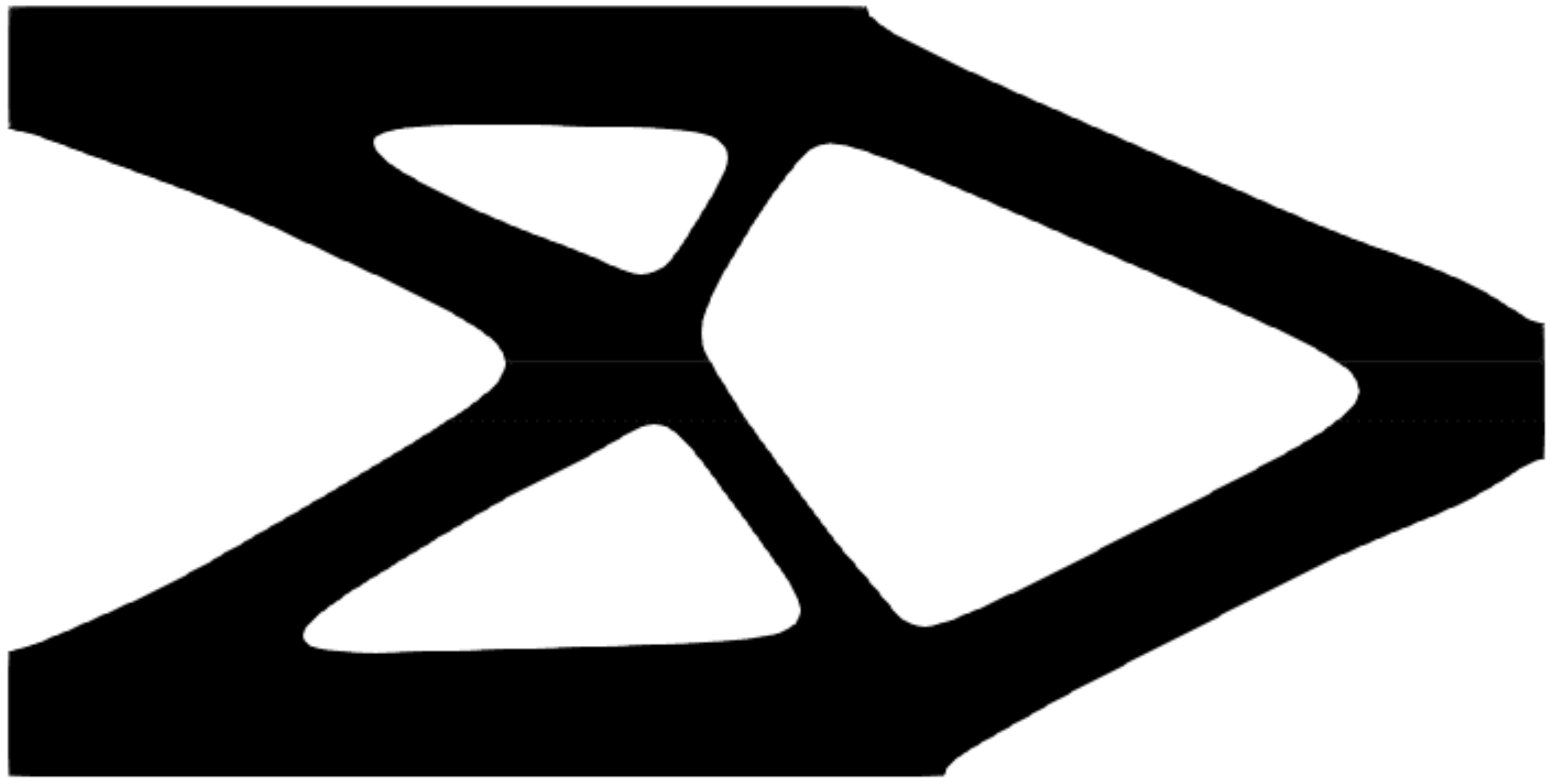}
        \subcaption{Step\,200}
        \label{6-g}
      \end{minipage} 
           \begin{minipage}[t]{0.24\hsize}
        \centering
        \includegraphics[keepaspectratio, scale=0.1]{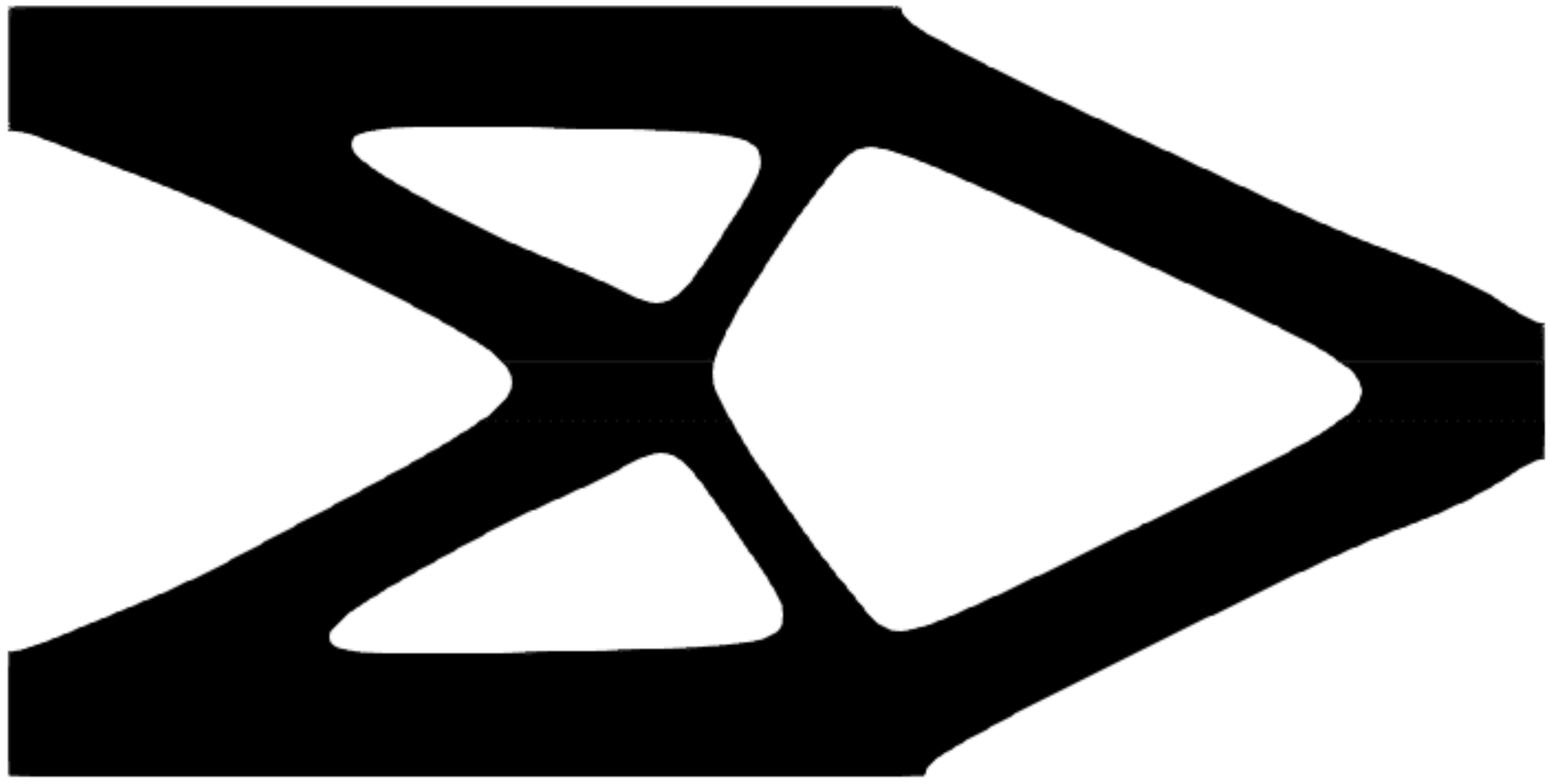}
        \subcaption{Step\,343$^{\#}$}
        \label{6-h}
        \end{minipage} 
        \\
        \begin{minipage}[t]{0.24\hsize}
        \centering
        \includegraphics[keepaspectratio, scale=0.202]{n0.pdf}
        \subcaption{Step\,0}
        \label{8-a}
      \end{minipage} 
      \begin{minipage}[t]{0.24\hsize}
        \centering
        \includegraphics[keepaspectratio, scale=0.1]{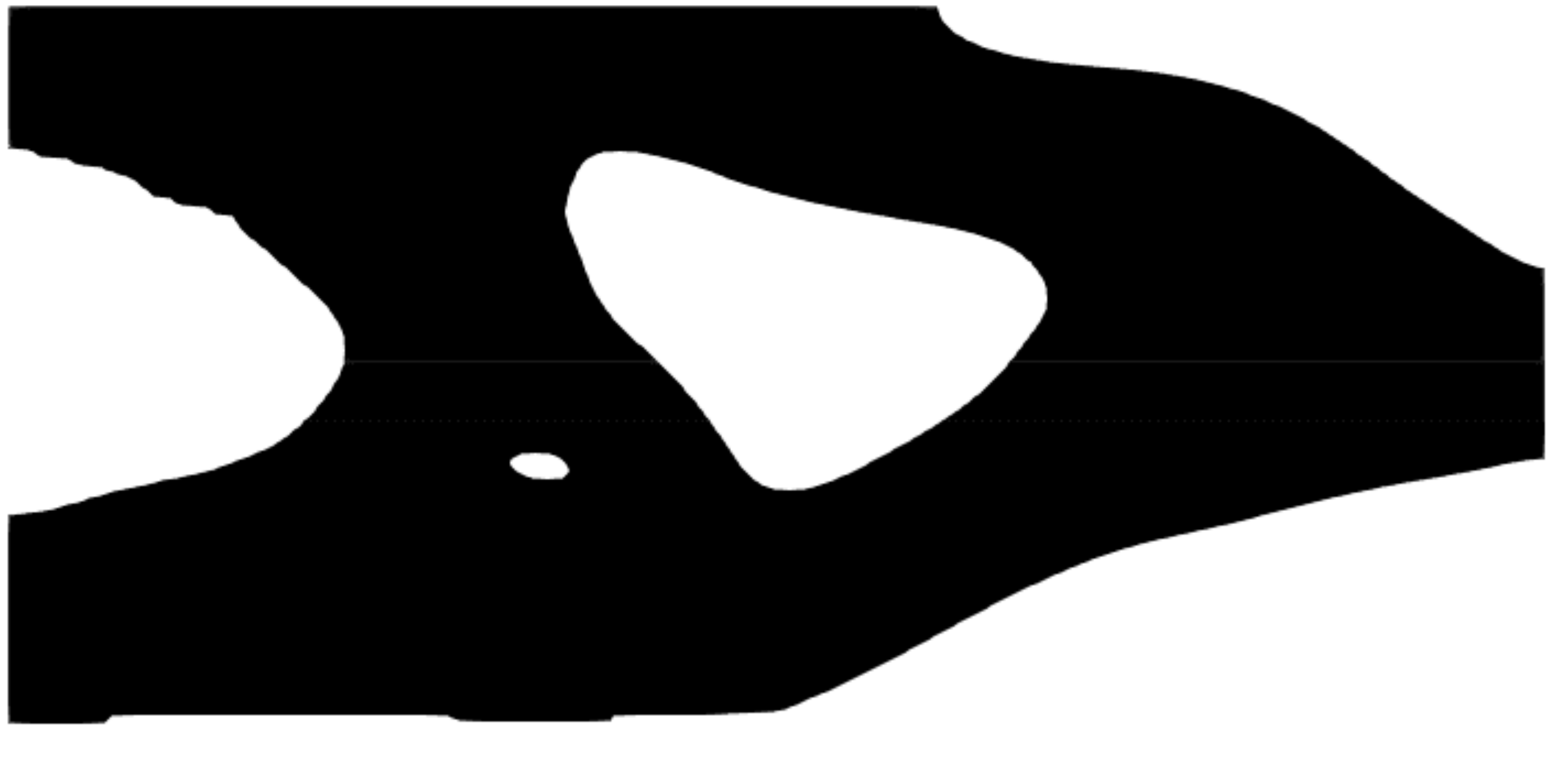}
        \subcaption{Step\,10}
        \label{8-b}
      \end{minipage} 
         \begin{minipage}[t]{0.24\hsize}
        \centering
        \includegraphics[keepaspectratio, scale=0.1]{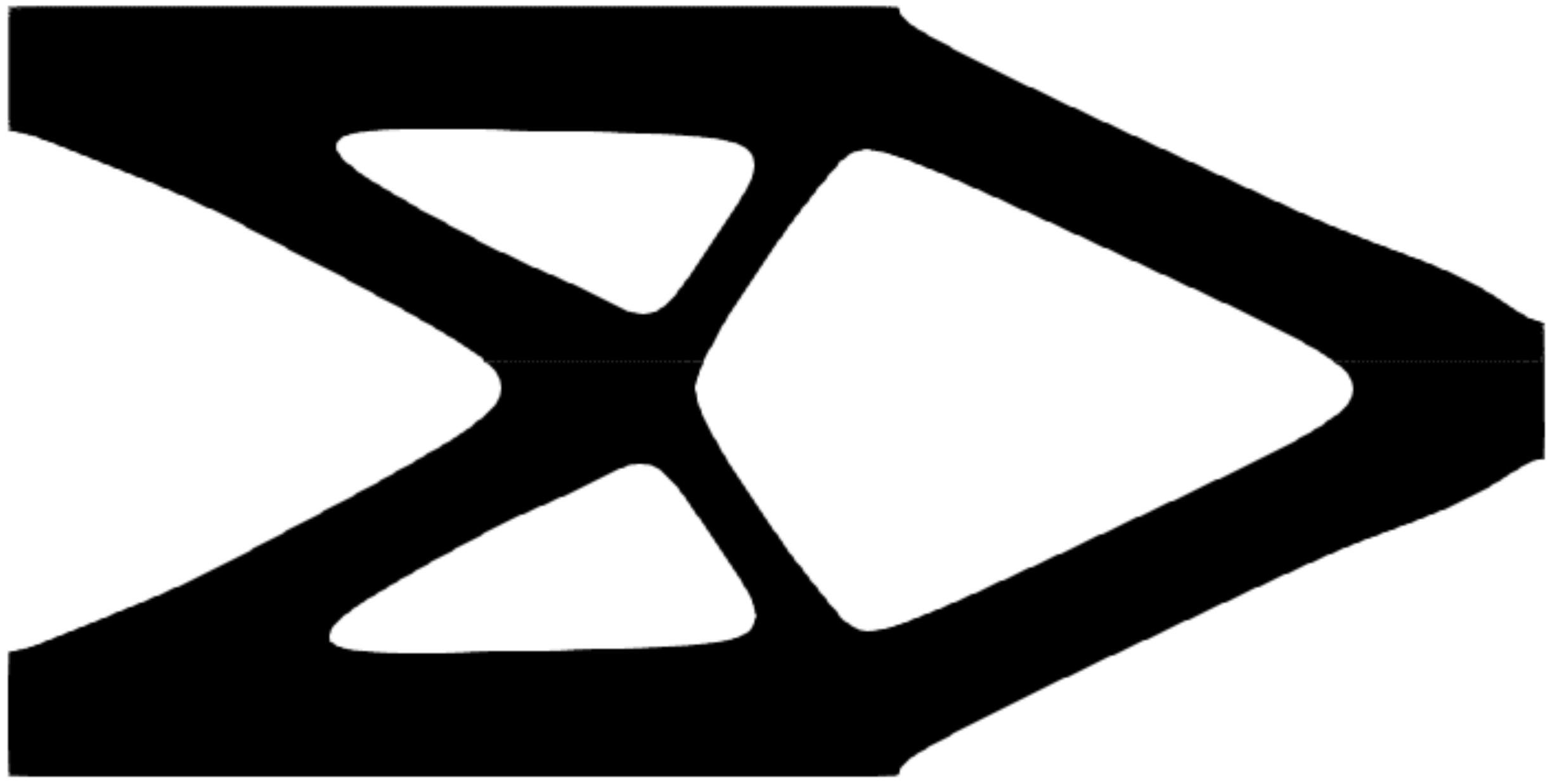}
        \subcaption{Step\,100}
        \label{8-c}
      \end{minipage}
           \begin{minipage}[t]{0.24\hsize}
        \centering
        \includegraphics[keepaspectratio, scale=0.1]{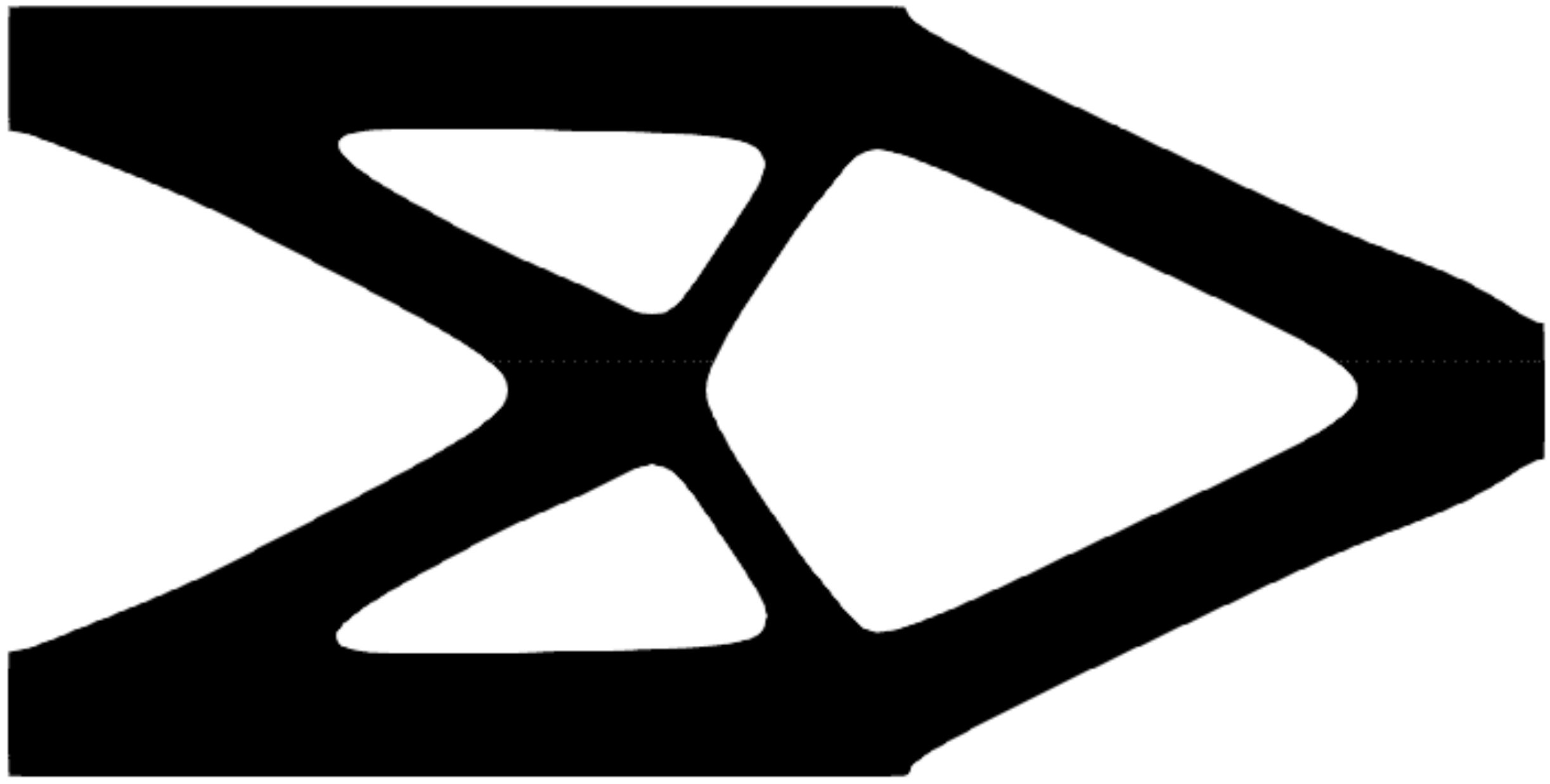}
        \subcaption{Step\,163$^{\#}$}
        \label{8-d}
      \end{minipage}\\
           \end{tabular}
     \caption{
     Configuration $\Omega_{\phi_n}\subset D$ for the case where the initial configuration is the upper domain. 
Figures (a)--(d), (e)--(h) and (i)--(l) represent $\Omega_{\phi_n}\subset D$ using  \eqref{discNLD} for $q=6$, \eqref{discNNLD} for $q=6$ and the combination of \eqref{discDNLD} and \eqref{discNNLD} for $q=6$. The symbol ${}^{ \#}$ implies the final step.   
     }
     \label{fig:nfmc}
  \end{figure}  
Obviously, the method using \eqref{discNLD} converges
faster to the optimal configuration than that with the RDE, and the method using \eqref{discNNLD} converges it even faster.
Thus, one can conclude from Figures \eqref{5-e}--\eqref{5-h} and \eqref{5-i}--\eqref{5-l} that there exists $q\in (1,+\infty)$ such that the algorithm constructed in this section rapidly improves the convergence to the optimal configuration more than that in \S \ref{S:algo}. 

On the other hand, as shown in Figure \ref{fig:nfmc}, all methods converge faster than the method using the RDE for $q=6$ (see \eqref{5-a}--\eqref{5-d} in Figure \ref{fig:nmc}). 
However, the improvement does not seem to be as great as in the previous case $q=5$.

We thus propose a method to further improve the convergence, and then we show numerically that the convergence for $q=6$ can be markedly improved below.  
Compare Figures \ref{fig:mc}, \ref{fig:nmc} with \ref{fig:nfmc}. It can be seen that Figure \ref{fig:nfmc} requires a large amount of updating for LSFs since the initial configuration needs to be changed markedly to obtain the optimized configuration. 
As already mentioned in \S \ref{S:nld}, we recall that, for \eqref{nld} and \eqref{DNLD}, $|v|^{1-p}$ and $|\nabla v|^{2-p}$ correspond to the diffusion coefficients, respectively. In other words, as shown in Figure \ref{fig:plp}, the diffusion coefficient is large near boundary structures if $q>1$ ( i.e.,~the FDE), and larger movements of the shape are expected. In contrast, for the degenerate $p$-Laplace case (i.e.,~$p>2$ in \eqref{DNLD}), the gradient is small near the boundary structures, and similar expectations cannot be made. 
Conversely, since the LSFs are restricted to $[-1,1]$ in Step 7 (see \S \ref{S:algo}), the gradient is likely greater in $[|\phi|<1]\cap[|\phi|\neq 0]$ (see Figure \ref{fig:plp}), and a larger change in shape can be expected than for \eqref{discNLD}. 
\begin{figure}[htbp]
   \hspace*{-5mm} 
    \begin{tabular}{cccc}
      \begin{minipage}[t]{0.32\hsize}
        \centering
        \includegraphics[keepaspectratio, scale=0.09]{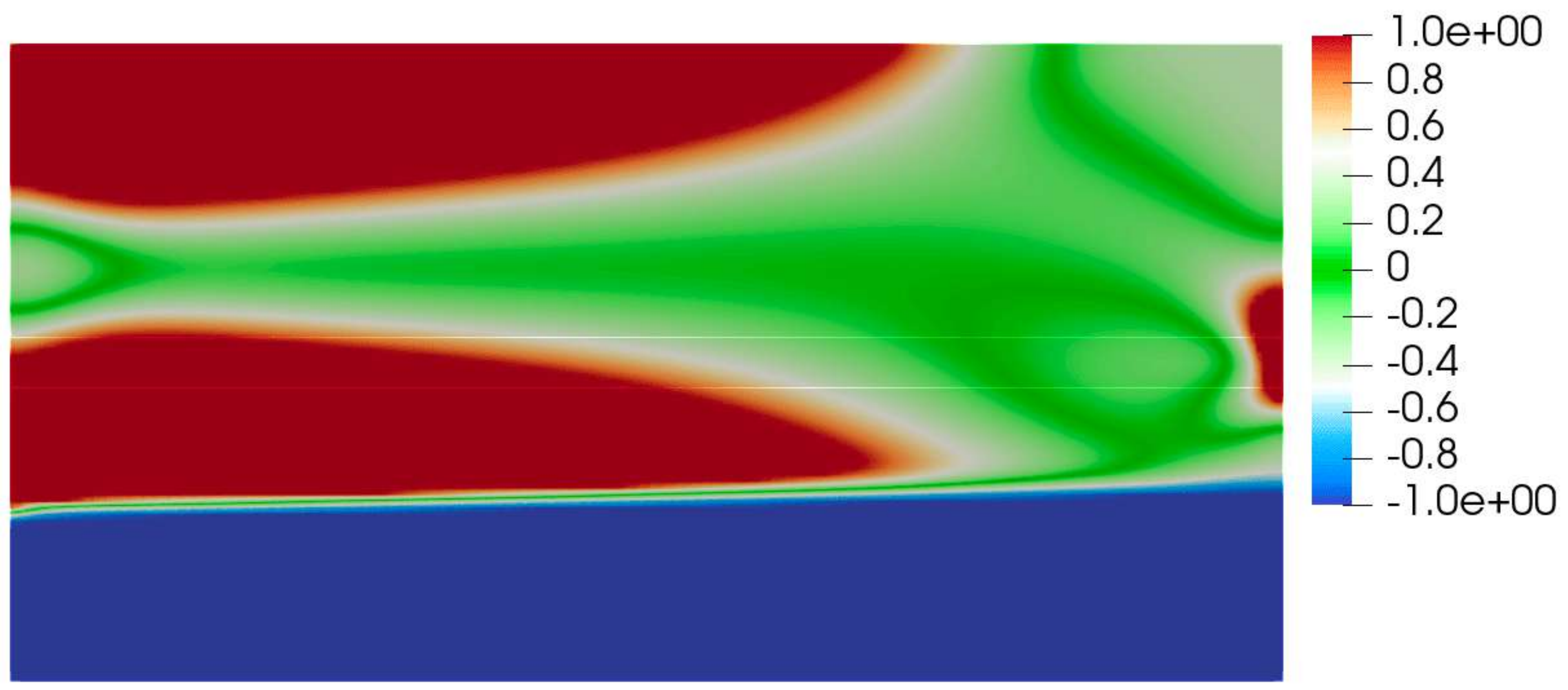}
        \subcaption{$\phi_n\in H^1(D;[-1,1])$}
        \label{7-a}
      \end{minipage} 
         \begin{minipage}[t]{0.32\hsize}
        \centering
        \includegraphics[keepaspectratio, scale=0.09]{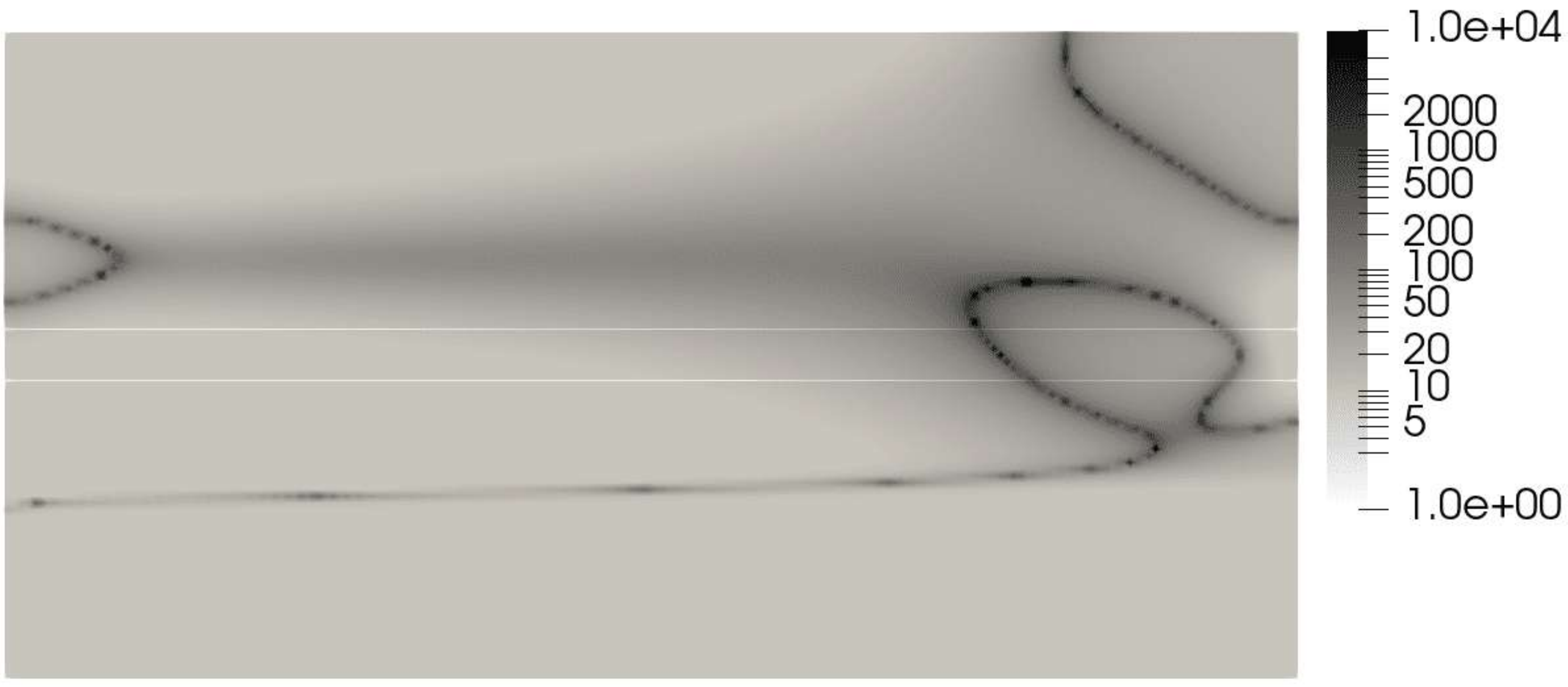}
        \subcaption{$|\phi_n|^{1-q}$ for $q>1$} 
        \label{7-b}
      \end{minipage}
           \begin{minipage}[t]{0.32\hsize}
        \centering
        \includegraphics[keepaspectratio, scale=0.09]{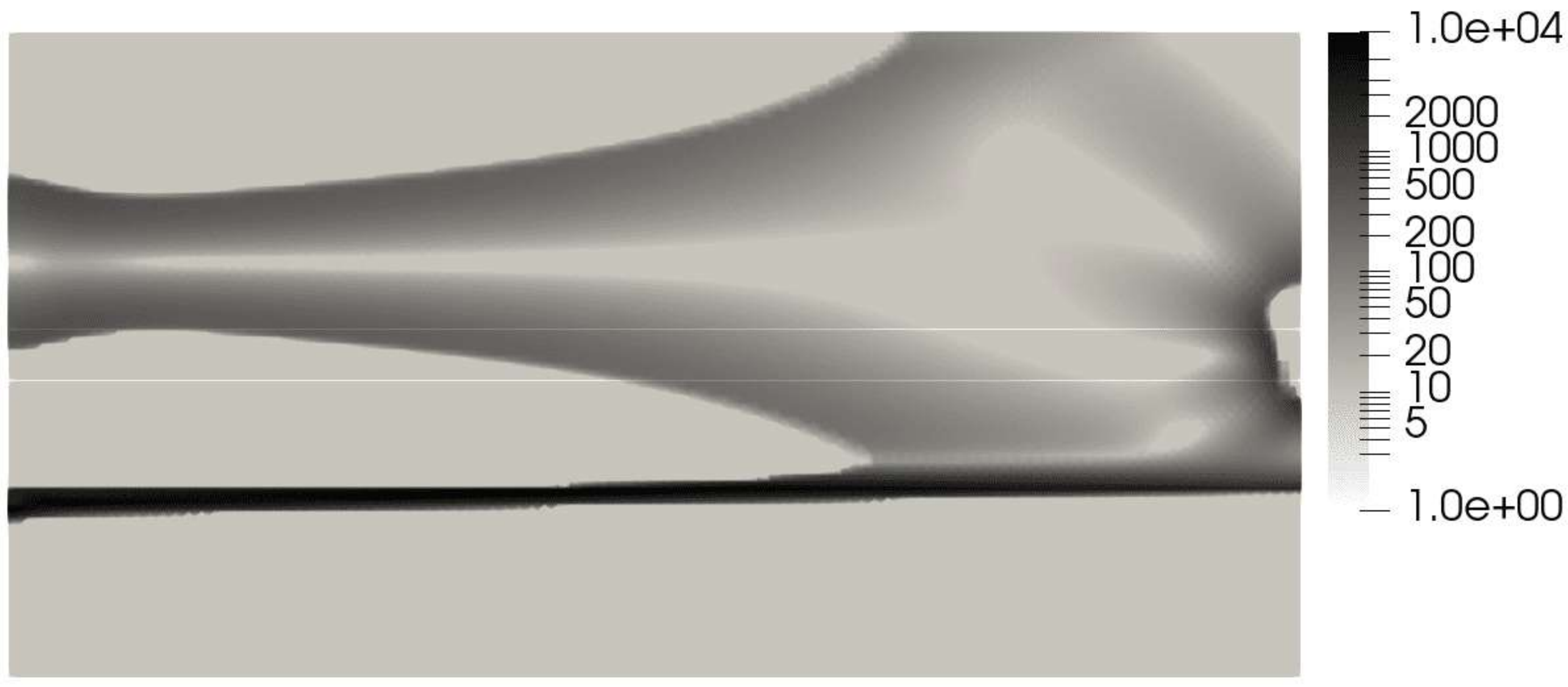}
        \subcaption{$|\nabla \phi_n|^{p-2}$ for $p>2$} 
        \label{7-c}
      \end{minipage}  
    \end{tabular}
    \caption{
    Cut-off LSFs using reaction-diffusion at $n=3$. In (a),
    red, blue and green domains represent the cut-off material domain $[\phi_n=1]$, cut-off void domain $[\phi_n=-1]$, and neighborhood of boundary structure $[|\phi_n|\le 0.2]$, respectively. 
    }
    \label{fig:plp}
\end{figure}
Hence, we establish a method using the following doubly nonlinear diffusion equation 
before applying \eqref{discNNLD}\/{\rm :}
\begin{align}
&
\int_D \tilde{q}(|\phi_{n}(x)|+\xi)^{q-1}\frac{\phi_{n+1}-\phi_{n}}{\varDelta t}(x){\psi}(x) \, \d x\nonumber\\
&\quad+\int_D\tau V_n(x)\nabla \phi_{n+1}(x)\cdot \nabla {\psi}(x)\, \d x 
=
\int_D \rho \mathcal{L}_\eta'(\phi_{n},\lambda_n) \psi(x)\, \d x \label{discDNLD}
\end{align}
for all  $\psi\in V$, where $V_n(x)=\gamma V_{n-1}+(1-\gamma)|\nabla \phi_n|^{p-2}$ and $V_0=\phi_0$.
In this study, we set $p=6$, $q=2.5$ and $\gamma=0.1$, and \eqref{discDNLD} will be used until $10$ steps. 

Figure \ref{8-a}-\ref{8-d} shows the numerical results.
As expected from Figure \ref{7-c}, the configuration on $\partial D$ is similar to the optimized configuration after only $10$ updates of the LSF.  
Furthermore, the proposed method achieves nearly the same convergence to the optimized configuration within only $100$ steps, which is similar to the result for $q=5$. 
This result demonstrates the validity of the proposed method, which is composed of both 
\eqref{discDNLD} and \eqref{discNNLD}.

\section{Conclusion}\label{S:conc}
This study proposed a topology optimization method using (doubly) nonlinear diffusion equations with reaction terms independent of the topological derivative.
Furthermore, 
 the proposed method has been applied to typical minimization problems, and three methods were established. The details are as follows\/{\rm :}
\begin{itemize}
\item[(i)] A ``relaxation method for sensitivity'' 
to derive reaction terms, which are used as an indicator of descent directions for objective functionals in level set-based topology optimization with reaction-diffusion, has been proposed and shown to be theoretically and numerically effective.
\item[(ii)] Using the singularity of the diffusion coefficient for fast diffusion to boundary structures, we theoretically have explained a ``fast convergence method'' to optimized configurations and numerically demonstrated 
the effectiveness of the proposed method. Furthermore, we have shown numerically that even fast convergence can be achieved by combining the result with Nesterov's accelerated gradient method or $p$-laplacian. 
\item[(iii)]
Conversely, using the degeneracy of the diffusion coefficient for slow diffusion to boundary structures and numerically demonstrating that objective functionals converge
for problems with the oscillation near boundary structures, we have developed a ``vanishing oscillation method'', which means that we have extended the range of applicability for the method using reaction-diffusion. 
\end{itemize}

The proposed method is a natural theoretical generalization of \cite{Y10} and will be applied to various topology optimization problems.

\end{document}